\documentclass{amsart}
\usepackage[utf8]{inputenc}

\usepackage{graphicx}
\usepackage[a4paper,top=2cm,bottom=2cm,left=2cm,right=2cm,marginparwidth=1.75cm]{geometry}
\usepackage{standalone}
\usepackage{amsfonts}
\usepackage{amsmath}
\usepackage{amssymb}
\usepackage{tikz-cd}
\usepackage{hyperref}
\usepackage{blindtext}
\usepackage{enumitem}
\usepackage{braket}
\usepackage{mathtools}
\usepackage{environ}
\usepackage{fancyhdr}

\allowdisplaybreaks
\newtheorem{theorem}{Theorem}[subsection]
\newtheorem{proposition}[theorem]{Proposition}
\newtheorem{lemma}[theorem]{Lemma}
\newtheorem{corollary}[theorem]{Corollary}
\newtheorem{definition}[theorem]{Definition}
\newtheorem{example}[theorem]{Example}
\newtheorem{remark}[theorem]{Remark}
\newtheorem{condition}[theorem]{Condition}
\let\olddefinition\definition
\let\oldexample\example
\let\oldremark\remark
\let\oldcondition\condition
\newcommand{\rank}{\text{rank}}
\renewcommand{\definition}{\olddefinition\normalfont}
\renewcommand{\example}{\oldexample\normalfont}
\renewcommand{\remark}{\oldremark\normalfont}
\renewcommand{\condition}{\oldcondition\normalfont}

\newcommand{\theoremname}[1]{\!\textup{\textbf{(#1)}}}

\newcommand\numberthis{\addtocounter{equation}{1}\tag{\theequation}}

\usepackage[backend=bibtex, hyperref=true,maxbibnames=99,sorting=nty]{biblatex}
\newtheorem{innercustomtheorem}{Theorem}
\newenvironment{customtheorem}[1]
  {\renewcommand\theinnercustomtheorem{#1}\innercustomtheorem}
  {\endinnercustomtheorem}
\NewEnviron{restatethis}[1]{%
  \expandafter\xdef\csname restatethis@#1\endcsname{%
    \unexpanded\expandafter{\BODY}%
  }%
  \begin{customtheorem}{\ref{#1}}\BODY\end{customtheorem}%
}
\newcommand{\restate}[1]{%
  \begin{theorem}\label{#1}\csname restatethis@#1\endcsname\end{theorem}%
}

\title{Seiberg-Witten Theory and Topological Recursion}
\author{Wee Chaimanowong}
\date{August 2020}

\address{
School of Mathematics and Statistics, 
University of Melbourne, 
Royal Parade, Parkville, Victoria 3010, Australia
} 

\email{
n.chaimanowong@student.unimelb.edu.au
} 

\begin{document}

\maketitle

\begin{abstract}
Kontsevich-Soibelman (2017) reformulated Eynard-Orantin topological recursion (2007) in terms of Airy structure which provides some geometrical insights into the relationship between the moduli space of curves and topological recursion. In this work, we investigate the analytical approach to this relationship using the Seiberg-Witten family of curves as the main example. In particular, we are going to show that the formula computing the Hitchin systems' Special Kahler's prepotential from the genus zero part of topological recursion as obtained by Baraglia-Huang (2017) can be generalized for a more general family of curves embedded inside a foliated symplectic surface, including the Seiberg-Witten family. Consequently, we obtain a similar formula relating the Seiberg-Witten prepotential to the genus zero part of topological recursion on a Seiberg-Witten curve. 
\end{abstract}

\setcounter{tocdepth}{2}
\tableofcontents{}

\section{Introduction}\label{introductionchapter}
Since its introduction in 1994, \emph{Seiberg-Witten theory} \cite{seiberg1994electric} has been an active research topic in both the physics and mathematics communities. From the physics point of view, it is an exact low-energy effective field theory of an $\mathcal{N}=2$ supersymmetric $SU(N+1)$ gauge theory. From the mathematical point of view, all the physical information can be encoded into an $N$-dimensional family $\mathcal{B}$ of genus $N$ connected compact Riemann surfaces $\Sigma$ known as \emph{Seiberg-Witten curves}. Each Seiberg-Witten curve is equipped with a residueless meromorphic differential form $dS_{SW}$ called the \emph{Seiberg-Witten differential}. The homology group $H_1(\Sigma, \mathbb{C})$ is naturally equipped with a symplectic structure, given by the oriented intersection product. By choosing symplectic basis cycles $A$ and $B$ for the homology group, the corresponding $A$ and $B$-periods of $dS_{SW}$ recover the \emph{electric} $(a^1,\cdots,a^N)$ and \emph{magnetic} $(b_1,\cdots,b_N)$ charges of the theory respectively. 

The supersymmetric origin of Seiberg-Witten theory enforces that the effective action of the theory is completely determined by a holomorphic function $\mathfrak{F}^{SW} = \mathfrak{F}^{SW}(\{a^{i=1,\cdots,N}\},\Lambda)$ known as the \emph{Seiberg-Witten prepotential}. The parameter $\Lambda$ is the energy scale of the theory where gauge couplings become strong. Therefore, technically we should have written $\mathcal{B}$ as $\mathcal{B}(\Lambda)$ to emphasize the dependency of the family on the energy scale $\Lambda$. In the original work \cite{seiberg1994electric}, the electric-magnetic duality was employed to constrain the form of the prepotential. A more mathematical approach to the computation of the prepotential has been studied via Nekrasov's partition function a few years later by various authors \cite{nekrasov2003seiberg, nekrasov2006seiberg, nakajima2004lectures, nakajima2005instanton}. 

Nekrasov's partition function $Z(\{a^{i=1,\cdots,N}\}, \epsilon_1,\epsilon_2,\Lambda)$ is the path-integral over the non-perturbative instanton contributions of a certain $(\epsilon_1,\epsilon_2)$-deformed $\mathcal{N}=2$ supersymmetric $SU(N+1)$ gauge theory. It can be written as an integral over the moduli space of instantons which can be evaluated via Atiyah-Bott equivariant localization formula \cite{atiyah1984moment}. In the end, we find that $Z(\{a^i\}, \epsilon_1,\epsilon_2,\Lambda)$ is a sum of terms indexed by $N$-tuples of Young's diagrams over all possible $N$-tuples of Young's diagrams with boxes of width $-\epsilon_1$ and height $+\epsilon_2$. In the limit $\epsilon_1,\epsilon_2 \rightarrow 0$, Young diagrams become graphs of continuous functions $f$ and there exists the function $f_0$ in which the corresponding term in the summation:
\begin{equation*}
	\mathfrak{F}_{0}(\{a^i\}, \Lambda) = \lim_{\epsilon_1,\epsilon_2\rightarrow 0}\epsilon_1\epsilon_2\log Z(\{a^i\},\epsilon_1,\epsilon_2,\Lambda)
\end{equation*}
will dominate the contribution. It was proven by Nekrasov-Okounkov \cite{nekrasov2006seiberg, nekrasov2003seiberg} that $\mathfrak{F}_0$ is the Seiberg-Witten prepotential $\mathfrak{F}_{SW}$. The same result was independently proven using techniques in algebraic geometry by Nakajima-Yoshioka \cite{nakajima2005instanton,nakajima2004lectures}. 

Mathematically, the Seiberg-Witten prepotential $\mathfrak{F}_{SW}$ contains information about the deformation of Seiberg-Witten curves $\Sigma$ in the family $\mathcal{B}$. A Seiberg-Witten curve $\Sigma$ is represented by a point $[\Sigma] \in \mathcal{B}$. Let us define the \emph{regular locus} $\mathcal{B}^{reg}$ of $\mathcal{B}$ to be the dense subset of points $[\Sigma]\in \mathcal{B}$ such that $\Sigma$ is non-singular. The set $\{a^1 := a^1(\Sigma),\cdots,a^N := a^N(\Sigma)\}$ where $a^i(\Sigma) := \oint_{A_i\subset \Sigma}dS_{SW}$ are $A$-periods of the Seiberg-Witten differential on $\Sigma$ is a system of local holomorphic coordinates on some neighbourhoods $\mathcal{B}_{\Sigma} \subseteq \mathcal{B}^{reg}$ of an arbitrary point $[\Sigma] \in \mathcal{B}^{reg}$ giving $\mathcal{B}^{reg}$ the structure of an $N$-dimensional complex manifold. Let us fix a point $[\Sigma_0] \in \mathcal{B}^{reg}$. The prepotential $\mathfrak{F}^{SW}_{\Sigma_0} : \mathcal{B}_{\Sigma_0}\rightarrow \mathbb{C}$ is a holomorphic function locally defined on some open neighbourhood $\mathcal{B}_{\Sigma_0} \subseteq \mathcal{B}^{reg}$ of $[\Sigma_0]$ which relates the $B$-periods on $\Sigma$ of the Seiberg-Witten differential, $b_i(\Sigma) := \oint_{B_i\subset \Sigma}dS_{SW}$, to its $A$-periods as follows:
\begin{equation*}
	b_i = b_i(\Sigma) = \frac{\partial \mathfrak{F}^{SW}_{\Sigma_0}(\{a^k\},\Lambda)}{\partial a^i}, \qquad i = 1,\cdots, N, \qquad [\Sigma] \in \mathcal{B}_{\Sigma_0}.
\end{equation*}

Let us explore the prepotential from another important viewpoint. Similarly to how we equipped the homology group $H_1(\Sigma,\mathbb{C})$ with a symplectic structure, the cohomology group $H^1(\Sigma,\mathbb{C}) \cong H_1(\Sigma, \mathbb{C})^*$ may be equipped with a symplectic structure given by the oriented intersection of cohomological classes via Poincar\'{e} duality. We can say that the existence of prepotential $\mathfrak{F}^{SW}_{\Sigma_0}$ is also equivalent to the fact that
\begin{equation*}
	\mathcal{L}_{\Sigma_0} := \left\{\left(a^1(\Sigma),\cdots,a^N(\Sigma),b_1(\Sigma), \cdots, b_N(\Sigma)\right)\ |\ [\Sigma] \in \mathcal{B}_{\Sigma_0}\right\} \subset \mathbb{C}^{2g} \cong H^1(\Sigma_0, \mathbb{C})
\end{equation*}
is a Lagrangian submanifold. From this point of view, the prepotential is the generating function of $\mathcal{L}_{\Sigma_0}$. Here we have identified $H^1(\Sigma,\mathbb{C})$ for every $[\Sigma] \in \mathcal{B}_{\Sigma_0}$ with $\mathbb{C}^{2g}$. Consider the symplectic vector bundle $\mathcal{H}\rightarrow \mathcal{B}_{\Sigma_0}$ equipped with the \emph{Gauss-Manin} connection $\nabla_{GM}$ and with the fiber over $[\Sigma] \in \mathcal{B}_{\Sigma_0}$ given by $\mathcal{H}_{\Sigma} := H^1(\Sigma, \mathbb{C})$. Then $\mathcal{L}_{\Sigma_0} \subset \mathcal{H}_{\Sigma_0} = H^1(\Sigma_0, \mathbb{C})$ is the image of the embedding of $\mathcal{B}_{\Sigma_0}$ into the fiber $\mathcal{H}_{\Sigma_0}$ given by sending the cohomology class of $dS_{SW}$ over $[\Sigma] \in \mathcal{B}_{\Sigma_0}$ to $[\Sigma_0]$ via the parallel transport. 

In fact, the existence of the Seiberg-Witten prepotential $\mathfrak{F}^{SW}_{\Sigma_0}$ is a consequence of an important property of the Seiberg-Witten differential, namely, any infinitesimal variations of $dS_{SW}$ in the family $\mathcal{B}$ using a certain choice of connection is a global holomorphic form on $\Sigma$. Based on this key property, we may consider a generalized notion of `Seiberg-Witten theory' as follows. Let $\mathcal{B}$ be an arbitrary $g$ dimensional family of genus $g$ curves $\Sigma$ and each $\Sigma$ is equipped with a residueless meromorphic differential $dS$ such that 
\begin{equation}\label{intro.generalsw}
	\frac{\partial dS}{\partial u^i} = \text{holomorphic form on $\Sigma$}, \qquad i = 1,\cdots, g
\end{equation}
with respect to some choices of connections, where $\{u^1,\cdots,u^g\}$ are local coordinates of $\mathcal{B}$. Let $\mathcal{B}_{\Sigma_0} \subset \mathcal{B}^{reg}$ be some small open neighbourhood of $[\Sigma_0]\in \mathcal{B}^{reg}$. The image of the embedding 
\begin{equation}\label{intro.embedding}
	[\Sigma] \mapsto (\mathbf{a}^1(\Sigma),\cdots,\mathbf{a}^g(\Sigma),\mathbf{b}_1(\Sigma),\cdots,\mathbf{b}_g(\Sigma)) \in \mathcal{H}_{\Sigma_0} = H^1(\Sigma_0,\mathbb{C}),
\end{equation}
given by sending the cohomology class of $dS$ over $[\Sigma] \in \mathcal{B}_{\Sigma_0}$ to $[\Sigma_0]$ via parallel transport, is a Lagrangian submanifold $\mathcal{L}_{\Sigma_0}$. Here, $\mathbf{a}^i = \mathbf{a}^i(\Sigma), \mathbf{b}_i = \mathbf{b}_i(\Sigma)$ are $A, B$-periods of $dS$ and the prepotential $\mathfrak{F}_{\Sigma_0} = \mathfrak{F}_{\Sigma_0}(\{\mathbf{a}^k\})$ exists on $\mathcal{B}_{\Sigma_0}$ as the generating function of $\mathcal{L}_{\Sigma_0}$: $\mathbf{b}_i = \frac{\partial \mathfrak{F}_{\Sigma_0}(\{\mathbf{a}^k\})}{\partial \mathbf{a}^i}$. In the more general context of \emph{Whitham hierarchy} \cite{itoyama1997prepotential, marshakov1999seiberg}, instead of just a single energy scale parameter $\Lambda$, we can consider an infinite number of \emph{time} variables $\{T^{k = 0,1,2,\cdots}\}$ such that $\frac{\partial dS}{\partial T^k} = d\Omega_k$ for some meromorphic differentials $d\Omega_k$ on $\Sigma$. In this context, the main problem to consider is the following: given a set of meromorphic differentials $\{d\Omega_{k=0,1,2,\cdots}\}$ on $\Sigma$ in the family of curves $\mathcal{B}(\{T^{k=0,1,2,\cdots}\})$, find (if possible) the meromorphic differential $dS$ on $\Sigma$ satisfying (\ref{intro.generalsw}) and $\frac{\partial dS}{\partial T^k} = d\Omega_k$.

In this work, we aim to study some mathematical aspects of the Seiberg-Witten family of curves and the prepotential from the point of view of \emph{topological recursion}. In the bigger picture, we will consider the complex analytic moduli space $\mathcal{B}$ of smooth compact complex genus $g$ curves $\Sigma$ embedded inside a foliated symplectic surface $S$ tangent to the leaves of the foliation $\mathcal{F}$ at finitely many points with tangency order $1$. Let us call any curve $[\Sigma] \in \mathcal{B}$ a  $\mathcal{F}$-\emph{transveral curve} and call points where $\Sigma$ is tangent to $\mathcal{F}$ the \emph{ramification points} of $\Sigma$. We will examine the deformation of curves in $\mathcal{B}$, using the Kontsevich-Soibelman approach to topological recursion \cite{kontsevich2017airy} with Seiberg-Witten family of curves as the main example.

A \emph{spectral curve} is a Riemann surface $\Sigma$ with some distinct points $\{r_{\alpha \in Ram}\}$ called \emph{ramification points} together with some additional data. Topological recursion is a recursive procedure which takes a spectral curve as input and produces symmetric meromorphic multi-differentials $\omega_{g,n}$ which are sections of $(\Omega^1_{\Sigma})^{\boxtimes n}$ on $(\Sigma\setminus \{r_{\alpha \in Ram}\})^n$. For instance, $\omega_{g,1}$ is a meromorphic differential on $\Sigma$. Topological recursion was introduced by Eynard-Orantin \cite{eynard2007invariants, eynard2008algebraic} around 2007-2008 with a wide variety of applications in areas ranging from mathematical physics to enumerative geometry. For example, it has been shown \cite{eynard2007invariants, eynard2008algebraic} by comparing the topological recursion relations with a particular choice of the spectral curve to loop equations that the corresponding output $\omega_{g,n}$ are given by the correlation functions in Matrix models. Similarly, let the spectral curve be given by an Airy curve with some additional data. It has been shown \cite{eynard2014invariants, eynard2011intersection, dunin2014identification} that coefficients of the Laurent expansion of $\omega_{g,n}$ centered at the only ramification point is related to the intersection number of divisor classes of universal cotangent line bundles $\{\psi_1,\cdots,\psi_n\}$ (known as $\psi$-classes) on $\bar{\mathcal{M}}_{g,n}$. The generating function of the intersection numbers was conjectured and later proven to be the tau function of KdV hierarchy by Witten and Kontsevich respectively \cite{witten1990two, kontsevich1992intersection}.

In 2017, Kontsevich-Soibelman \cite{kontsevich2017airy} introduced the notion of \emph{Airy structures} which is given by a Lie algebra $\mathfrak{g}$ generated by a set of quadratic polynomials $\{H_{k\in\mathbb{I}}\}$ of a certain form on a symplectic vector space $W$ where the Lie bracket is naturally given by the Poisson bracket on $W$. They showed that any Airy structure corresponds to a recursion called \emph{abstract topological recursion}. This reduces to the usual Eynard-Orantin topological recursion when considering an example of Airy structure $\{H_{i\in\mathbb{I}}\} := \{(H_{Airy})_{k=1,2,3,\cdots}\}$ and $W := W^{Ram}_{Airy}$ arises from the so-called \emph{residue constraints}. The zero set of the generators of $\mathfrak{g}$ is a Lagrangian submanifold $L \subset W$ by the fact that $\mathfrak{g}$ is closed under the Lie bracket. We call $L$ a \emph{quadratic Lagrangian submanifold} since it is defined by a set of quadratic polynomials. In particular, we let $L^{Ram}_{Airy} := \{(H_{Airy})_{k=1,2,3,\cdots}=0\} \subset W^{Ram}_{Airy}$.

Let us return to the discussion of the particular Airy structure $\{(H_{Airy})_{k=1,2,3,\cdots}\}$ corresponding to topological recursion. Let $Discs^{Ram}$ be a space of $|Ram|$-disjoint union of discs $\mathbb{D}_{\alpha}, \alpha \in Ram$ embedded inside the foliated symplectic surface $S$. Each $\mathbb{D}_\alpha$ are deformed Airy curves which can be parameterized by 
\begin{equation}\label{intro.parameterization}
	z_\alpha \mapsto \left(x_\alpha = z_\alpha^2 + a, y_\alpha = \sum_{k=1}^\infty b_{k\alpha}z_\alpha^k\right)    
\end{equation}
on the local coordinates chart, $(U_\alpha, x_\alpha, y_\alpha)$ of $S$ with the ramification point $r_\alpha = (x_\alpha, y_\alpha) = (a,b_{0\alpha}) \in U_\alpha$. Given an annulus on a disc $\mathbb{D}_\alpha$ centered at $r_\alpha$, the set of all Laurent series which converge on the annulus naturally has a structure of an infinite dimensional symplectic vector space. The infinite-dimensional symplectic vector space $W^{Ram}_{Airy}$ is given by the direct sum these vector spaces over $\alpha \in Ram$. The space $Discs^{Ram}$ can be thought of as the local version of the moduli space of curves $\mathcal{B}$. Introduce a trivial vector bundle $W^{Ram}\rightarrow Discs^{Ram}$ with fiber over any point given by $W^{Ram}_{Airy}$ equipped with a certain connection $\nabla$. Let us fix a reference point $t_0 \in Discs^{Ram}$ and let $Discs^{Ram}_{t_0} \ni t_0$ be a certain subset of $Discs^{Ram}$ analogous to $\mathcal{B}_{\Sigma_0}$. The important observation is that there exists a local section $\theta$ of $W^{Ram}$ on $Discs^{Ram}_{t_0}$ such that the parallel transport of $\theta_t \in W^{Ram}_t$ from any point $t\in Discs^{Ram}_{t_0}$ to $t_0$ gives an embedding of $Discs^{Ram}_{t_0}$ into the fiber $W^{Ram}_{t_0}$ with the image contained in the quadratic Lagrangian submanifold $L^{Ram}_{Airy}$. The section $\theta$ can be thought of like a local version of $dS$ or a Seiberg-Witten differential.

At first glance, this on-going discussion seems to have little to do with the deformation of curves or Seiberg-Witten theory. But we note that an infinitesimal deformation of a smooth curve $\Sigma$ will produce a global holomorphic section of the normal bundle on $\Sigma$. If $\Sigma$ is embedded inside a symplectic surface, then by applying the symplectic form to the global section of the normal bundle, we produce a global holomorphic form on $\Sigma$ which naturally arises from the infinitesimal deformation of $\Sigma$. In particular, let $\mathcal{B}$ be the moduli space of $\mathcal{F}$-transversal genus $g$ curves, then we have the map $\pmb{\phi}_{\Sigma} : T_{[\Sigma]}\mathcal{B}\rightarrow \Gamma(\Sigma, \Omega^1_\Sigma) \hookrightarrow \mathcal{H}_\Sigma$ for each $[\Sigma] \in \mathcal{B}$. It is possible to check that $d_{\nabla_{GM}}\pmb{\phi} = 0$ where $d_{\nabla_{GM}}$ denotes the exterior covariant derivative using the Gauss-Manin connection. Let us fix a reference point $[\Sigma_0] \in \mathcal{B}$. By the Pointcar\'{e} Lemma, there exists an open neighbourhood $\mathcal{B}_{\Sigma_0}$ of $[\Sigma_0]$ with a local section $\pmb{\theta}$ of $\mathcal{H}$ on $\mathcal{B}_{\Sigma_0}$ such that $d_{\nabla_{GM}}\pmb{\theta} = \pmb{\phi}$. Then the parallel transport of $\pmb{\theta}_{\Sigma} \in \mathcal{H}_{\Sigma}$ from any point $[\Sigma] \in \mathcal{B}_{\Sigma_0}$ to $[\Sigma_0]$ defines an embedding of $\mathcal{B}_{\Sigma_0}$ into $\mathcal{H}_{\Sigma_0}$ as a Lagrangian submanifold $\mathcal{L}_{\Sigma_0}$. This is exactly the same embedding given in (\ref{intro.embedding}) with
\begin{equation*}
	\mathbf{a}^i := \mathbf{a}^i(\Sigma) = \oint_{A_i}\pmb{\theta}_{\Sigma}, \qquad \mathbf{b}_i := \mathbf{b}_i(\Sigma) = \oint_{B_i}\pmb{\theta}_\Sigma = \frac{\partial \mathfrak{F}_{\Sigma_0}(\{\mathbf{a}^k\})}{\partial\mathbf{a}^i}, \qquad i = 1,\cdots, g,
\end{equation*}
where the holomorphic function $\mathfrak{F}_{\Sigma_0} = \mathfrak{F}_{\Sigma_0}(\{\mathbf{a}^k\})$ defined on $\mathcal{B}_{\Sigma_0}$ is the generating function of $\mathcal{L}_{\Sigma_0}$ called the prepotential.
For each $[\Sigma] \in \mathcal{B}_{\Sigma_0}$, we may consider a lift of the cohomology class $\pmb{\theta}_{\Sigma}$ to a differential form on $\Sigma$. Such a lift is not unique since an addition by $df$ for any meromorphic function $f$ on $\Sigma$ will not change its cohomology class. However, the key idea of Kontsevich-Soibelman is there exists exactly one lift $\theta_\Sigma$ of $\pmb{\theta}_\Sigma$ such that the direct sum of Laurent series expansions of $\theta_\Sigma$ at each ramification point coincides with $\theta_t \in W^{Ram}_t$. Here, $t\in Discs^{Ram}_{t_0}$ is given in term of $[\Sigma] \in \mathcal{B}_{\Sigma_0}$ by the disjoint union of some open neighbourhood $\mathbb{D}_\alpha(\Sigma) \subset \Sigma \subset S$ which we shall denote by $t = \gamma(\Sigma)$. The image of $\theta_t$ under the parallel transport from $t$ to $t_0$ is contained inside $L^{Ram}_{Airy}\subset W^{Ram}_{t_0}$.

In the end, we get a family of differential forms $\theta_\Sigma$ which relates the Lagrangian submanifold $\mathcal{L}_{\Sigma_0} \subset \mathcal{H}_{\Sigma_0}$ with the quadratic Lagrangian submanifold $L^{Ram}_{Airy} \subset W^{Ram}_{t_0}$ and each $\theta_\Sigma$ can be thought of as a generalized Seiberg-Witten differential $dS$ satisfying (\ref{intro.generalsw}). We can understand the relationship between $\mathcal{L}_{\Sigma_0}$ and $L^{Ram}_{Airy}$ more geometrically. The map $[\Sigma] \mapsto \gamma(\Sigma)$ embeds $\mathcal{B}_{\Sigma_0}$ as a subset of $Discs^{Ram}_{t_0}$ and $\mathcal{L}_{\Sigma_0}$ is a cohomological image in $\mathcal{H}_{\Sigma_0}$ of the cross-section of $L^{Ram}_{Airy} \subset W^{Ram}_{Airy}$ over $\mathcal{B}_{\Sigma_0}\subset Discs^{Ram}_{t_0}$. The quadratic Lagrangian submanifold $L^{Ram}_{Airy}$ may be regarded as \emph{universal} in the following sense. Any two Langrangian submanifolds $\mathcal{L}_{\Sigma_0}, \mathcal{L}'_{\Sigma'_0}$ arise from two different moduli spaces $\mathcal{B}, \mathcal{B}'$ of $\mathcal{F}$-transversal curves with the same number of genus and ramification points will correspond to two different cross-sections of $L^{Ram}_{Airy}$. 

From the idea described above, we obtain the main result as follows:

\begin{restatethis}{prepotentialvsomega0ntheorem}
	Let $\mathcal{B}$ be the moduli space of $\mathcal{F}$-transversal genus $g$ curves in a foliated symplectic surface $S$. Then for some small open neighbourhood $\mathcal{B}_{\Sigma_0} \subset \mathcal{B}$ of the reference point $[\Sigma_0] \in \mathcal{B}$, the prepotential $\mathfrak{F}_{\Sigma_0} = \mathfrak{F}_{\Sigma_0}(\mathbf{a}^1,\cdots, \mathbf{a}^g)$ can be expressed as
	\begin{align*}
		\mathfrak{F}_{\Sigma_0}&(\mathbf{a}^1,\cdots,\mathbf{a}^g) = \mathfrak{F}_{\Sigma_0}(\mathbf{a}^1_0,\cdots,\mathbf{a}^g_0)\\ 
		&+ \sum_{i=1}^g(\mathbf{a}^i - \mathbf{a}^i_0)\frac{\partial \mathfrak{F}_{\Sigma_0}}{\partial \mathbf{a}^i}(\mathbf{a}^1_0,\cdots,\mathbf{a}^g_0) + \frac{1}{2}\sum_{i,j=1}^g(\mathbf{a}^i - \mathbf{a}^i_0)(\mathbf{a}^j - \mathbf{a}^j_0)\frac{\partial^2\mathfrak{F}_{\Sigma_0}}{\partial\mathbf{a}^i\partial\mathbf{a}^j}(\mathbf{a}^1_0,\cdots,\mathbf{a}^g_0)\\
		&+ \sum_{n=3}^\infty\frac{1}{n!}\left(\frac{1}{2\pi i}\right)^{n-1}\sum_{i_1,\cdots,i_n = 1}^g\left(\oint_{p_1\in B_{i_1}}\cdots \oint_{p_n \in B_{i_n}}\omega_{0,n}(p_1,\cdots,p_n)\right)(\mathbf{a}^{i_1}-\mathbf{a}^{i_1}_0)\cdots (\mathbf{a}^{i_n}-\mathbf{a}^{i_n}_0)\label{intro.prepotentialvsomega0nformula}\numberthis
	\end{align*}
	for all $(\mathbf{a}^1,\cdots, \mathbf{a}^g) = [\Sigma]\in \mathcal{B}_{\Sigma_0}$. Where $\mathbf{a}_0^k := \mathbf{a}^k(\Sigma_0)$ and the multi-differentials $\left\{\omega_{g,n} \in \Gamma\left((\Sigma_0\setminus \cup_{\alpha\in Ram} r_\alpha)^n, (\Omega_{\Sigma_0}^1)^{\boxtimes n}\right)\right\}$ are produced from topological recursion on $\Sigma_0$.
\end{restatethis}

This theorem is a joint work with Paul Norbury, Michael Swaddle and Mehdi Tavakol which has been proved previously in \cite{chaimanowong2020airystructures}, however in this paper we are going to approach it from an analytical point of view as oppose to the formal approach taken by \cite{chaimanowong2020airystructures}. The interesting point about Theorem \ref{prepotentialvsomega0ntheorem} is that the prepotential $\mathfrak{F}_{\Sigma_0}$ contains information about the deformation of curves $\Sigma$ within the moduli space $\mathcal{B}$ whereas $\omega_{0,n}$ is produced from topological recursion, which is done on a single curve $\Sigma_0$ without any knowledge on the moduli space it belongs to.

We are going to revisit in detail all the discussions about Airy structures, deformations of curves, and topological recursion later in Section \ref{airystructurechapter} along with the proof of Theorem \ref{prepotentialvsomega0ntheorem}.
To apply this result to the Seiberg-Witten family of curves we need to find a foliated symplectic surface $S$ such that each Seiberg-Witten curve can be embedded as a $\mathcal{F}$-transversal curve. It turns out that such a foliated symplectic surface exists as we are going to show in Section \ref{swexamplechapter}, hence there is a nice relationship between the Seiberg-Witten prepotential $\mathfrak{F}^{SW}_{\Sigma_0} = \mathfrak{F}^{SW}_{\Sigma_0}(\{a^k\}, \Lambda)$ and the genus zero part of topological recursion on Seiberg-Witten curve with $\omega_{0,1} = dS_{SW}$. 

The formula (\ref{intro.prepotentialvsomega0nformula}) can also be obtained via the Rauch variational formula \cite{Fay73, eynard2007invariants, baraglia2017special}. For example, Baraglia-Huang \cite{baraglia2017special} studied the moduli space $\mathcal{B}$ of curves $\Sigma$ arising from the \emph{Hitchin systems}, where each $\Sigma$ is embedded inside the total space of the cotangent bundle $T^*\mathcal{C}$ of some compact Riemann surface $\mathcal{C}$. They showed that the Rauch variational formula holds in this case and that the formula (\ref{intro.prepotentialvsomega0nformula}) follows. Theorem \ref{prepotentialvsomega0ntheorem} generalizes the result by Baraglia-Huang to the case where $S$ is an arbitrary foliated symplectic surface instead of $T^*\mathcal{C}$. Moreover, even in the $S = T^*\mathcal{C}$ case, our approach of proving the formula via the idea of Kontsevich-Soibelman should shed light on some interesting underlying geometries between the Lagrangian submanifold in $\mathcal{H}_{\Sigma_0} = H^1(\Sigma_0, \mathbb{C})$ given by the cohomology of the tautological 1-form and the quadratic Lagrangian submanifolds with $\omega_{0,n}$ as the generating function.

\subsection{Seiberg-Witten Family of Curves}\label{intro.swsection}
In this section, we give a quick summary of the needed mathematical details of the Seiberg-Witten family of curves. More comprehensive reviews of this material can be found abundantly in the literature such as \cite{seiberg1994electric, itoyama1997prepotential, marshakov1999seiberg, nakajima2004lectures}. 
The information of $SU(N+1)$ Seiberg-Witten theory can be encoded into a family $\mathcal{B}$ of genus $g = N$ \emph{Seiberg-Witten curves} $\Sigma$ given by
\begin{equation}\label{swcurvedefn}
	\Lambda^{g+1}\left(w + \frac{1}{w}\right) = z^{g+1} + u^gz^{g-1} + ... + u^1 =: P(z;u)
\end{equation}
where $\{u^1,...,u^g\}$ are the coordinates of $\mathcal{B}$. Occasionally, we will write the curve $\Sigma$ as $\Sigma(u)$ to emphasize its dependence on $u = (u^1,\cdots,u^g)$. We may regard (\ref{swcurvedefn}) as a plane curve equation, in that case, we have $w,z \in \mathbb{C}$ and $\Sigma(u) \subset \mathbb{C}^2$ as a non-compact curve. The compactification of $\Sigma(u)$ is obtained by adding two points $\infty_+ = (w,z) = (\infty, \infty)$ and $\infty_- = (w,z) = (0,\infty)$. In this way, we regard $\Sigma(u)$ defined in (\ref{swcurvedefn}) as a compact smooth curve. The map $\pi : \Sigma(u)\rightarrow \mathbb{P}^1, (w,z) \mapsto z$ is a double cover that shows that $\Sigma(u)$ is a hyperelliptic curve with the hyperelliptic involution given by $(w,z) \leftrightarrow (1/w,z)$. An alternative way to express (\ref{swcurvedefn}) is 
\begin{equation*}
	y^2 = P^2(z;u) - 4\Lambda^{2g+2}\qquad \text{where} \qquad y := \Lambda^{g+1}\left(w - \frac{1}{w}\right)
\end{equation*}
which makes it even more obvious that $\Sigma(u)$ is a hyperelliptic curve with the hyperelliptic involution $(y,z) \leftrightarrow (-y,z)$. The $2g+2$ ramification points of $\pi$ are given by $(y,z) = (0, z^{i\pm})$ where $z^{i\pm}, i = 1,\cdots,g+1$ are the roots of the polynomial $P(z;u) \mp 2\Lambda^{g+1} = 0$. We can check that the curve $\Sigma(u)$ is smooth unless $P'(z^{i+};u) = 0$ or $P'(z^{i-};u) = 0$. In other words, we may define the family of smooth Seiberg-Witten curves to be
\begin{equation*}
	\mathcal{B}^{reg} := \left\{u = (u^1,\cdots,u^g) \in \mathcal{B}\ |\ \text{$P^2(z;u) - 4\Lambda^{2g+2} = 0$ has $2g+2$ distinct roots}\right\} \subset \mathcal{B}. 
\end{equation*}
We note that $\mathcal{B}^{reg}$ is dense inside $\mathcal{B}$ and since we will mainly be interested in smooth Seiberg-Witten curves in our works, we will relabel $\mathcal{B}^{reg}$ as $\mathcal{B}$ for simplicity.

We define the \emph{Seiberg-Witten differential} to be the meromorphic differential form
\begin{equation*}
	dS_{SW} := z\frac{dw}{w} = \frac{zP'(z;u)dz}{y}
\end{equation*}
defined on each Seiberg-Witten curve $\Sigma(u)$ with residueless poles of order $2$ at $\infty_{\pm} \in \Sigma(u)$. Occasionally we will write $dS_{SW}(\Sigma(u))$ if we want to emphasize the specific $dS_{SW}$ defined on $\Sigma(u)$. Let us check that the variations of $dS_{SW}(\Sigma(u))$ within the family $\mathcal{B}$ produce global homolomorphic forms on $\Sigma(u)$, hence $dS_{SW}$ satisfies (\ref{intro.generalsw}). As we have mentioned, a choice of connection must be specified in order for the derivatives $\frac{\partial}{\partial u^i}$ to makes sense, in this case let us keep $w = const$, we have 
\begin{equation}
	\frac{\partial dS}{\partial u^i}\Big|_{w = const} = \frac{\partial z}{\partial u^i}\Big|_{w = const}\frac{dw}{w} = -\frac{z^{i-1}}{P'(z;u)}\frac{dw}{w} = \frac{z^{i-1}dz}{y}, \qquad i = 1,\cdots,g.
\end{equation}
which are holomorphic on $\Sigma(u)$. We define
\begin{equation*}
	a^i := \oint_{A_i}dS_{SW}, \qquad b_i := \oint_{B_i}dS_{SW}, \qquad i = 1,\cdots, g.
\end{equation*}
We note that $\{a^1,\cdots,a^g\}$ gives another set of holomorphic coordinates for $\mathcal{B}$. We can see this by noting that $\left(\frac{\partial a^i}{\partial u^j}\right)$ is invertible because if $\sum_{i=1}^gc_i\frac{\partial a^i}{\partial u^j} = \oint_{A_k}\sum_{i=1}^g\frac{c_iz^{i-1}dz}{y} = 0$ for all $j, k = 1,\cdots, g$ and for some constants $c_i \in \mathbb{C}$ then $\sum_{i=1}^g\frac{c_iz^{i-1}dz}{y} = 0$ because it is a holomorphic form with zero $A$-periods. It follows that $\frac{\partial dS_{SW}}{\partial a^i} = \omega_i$ are normalized holomorphic differentials:
\begin{equation*}
	\oint_{A_j}\omega_i = \delta_j^i, \qquad \oint_{B_j}\omega_i = \frac{\partial b_i}{\partial a^j} =: \tau_{ij},
\end{equation*}
where $b_i = b_i(a^1,\cdots,a^g)$ is treated as a holomorphic function of $\{a^{i=1,\cdots,g}\}$. The standard argument involving Riemann bilinear identity shows that $\tau_{ij} = \tau_{ji}$:
\begin{equation*}
	0 = \frac{1}{2\pi i}\sum_{k=1}^g\left(\oint_{A_k}\omega_i\oint_{B_k}\omega_j - \oint_{B_k}\omega_i\oint_{A_k}\omega_j\right) = \frac{1}{2\pi i}\left(\tau_{ji} - \tau_{ij}\right).
\end{equation*}
In other words, $\frac{\partial b_i}{\partial a^j} = \frac{\partial b_j}{\partial a^i}$, therefore, by Poincar\'{e} lemma given any $[\Sigma_0] \in \mathcal{B}$ there exists an open neighbourhood $\mathcal{B}_{\Sigma_0} \ni [\Sigma_0]$ with a holomorphic function $\mathfrak{F}^{SW}_{\Sigma_0} = \mathfrak{F}^{SW}_{\Sigma_0}(\{a^k\},\Lambda)$ such that 
\begin{equation*}
	b_i = \frac{\partial \mathfrak{F}^{SW}_{\Sigma_0}(\{a^k\},\Lambda)}{\partial a^i}.
\end{equation*}
We call $\mathfrak{F}^{SW}_{\Sigma_0} = \mathfrak{F}^{SW}_{\Sigma_0}(\{a^k\},\Lambda)$ the \emph{Seiberg-Witten prepotential}. The expansion of $\mathfrak{F}_{\Sigma_0}^{SW}$ can be computed via Nekrasov’s partition function \cite{nekrasov2006seiberg,okounkov2006random} or Nakajima’s blow-up formula \cite{nakajima2004lectures,nakajima2005instanton}  or by computing the odd periods of $M_{A_g}\otimes QH^*(\mathbb{P}^1)$ Frobenius manifold \cite{dubrovin2004almost}.

\section*{Acknowledgements}
I would like to express my deepest gratitude to Paul Norbury and Mehdi Tavakol for supervising me throughout this project. I also wish to thank Michael Swaddle and Ga\"{e}tan Borot for many useful discussions. Part of this work was done during my visit to the Max Planck Insitute for Mathematics, Bonn, I would like to thank them for their hospitality and for partially funded this project. This work is primarily funded by the Melbourne Research Scholarship.

\section{Airy Structures and Deformation of Curves}\label{airystructurechapter}

In this section, we summarize the work of Kontsevich-Soibelman \cite{kontsevich2017airy} on Airy structures, topological recursion, and the moduli space of curves in a foliated symplectic surface. In particular, there are 3 important ideas from \cite{kontsevich2017airy} which we are going to explore:
\begin{enumerate}
    \item The quantization of Airy structures leads to a generalized version of topological recursion. The usual Eynard-Orantin topological recursion corresponds to an example of Airy structure given by the so-called \emph{residue constraints}.
    \item A deformation of an Airy curve $y^2 = x$, i.e. a \emph{disc}, inside a foliated symplectic surface produces a differential form satisfying the residue constraints.
    \item The cohomology class, arising from the deformation of a curve in a foliated symplectic surface, can be lifted to the differential form on the curve, which locally agrees near each ramification point with the differential form produced from the deformation of a disc.
\end{enumerate}

Connecting these ideas leads to a nice relationship between the prepotential and the genus zero part of topological recursion which is the content of Theorem \ref{prepotentialvsomega0ntheorem}, the main result of this section. We are going to start in Section \ref{deformationofcurvessection} by reviewing some basic deformation theory of curves $\Sigma$ in a foliated symplectic surface $(S,\Omega_S,\mathcal{F})$, following \cite[Section 4]{kontsevich2017airy}. This eventually leads to an embedding of a Lagrangian submanifold in $H^1(\Sigma,\mathbb{C})$ and establishes the notion of \emph{prepotential} in this general context. Section \ref{airystructuressection} reviews definitions and basic results concerning Airy structures, including the corresponding quadratic Lagrangian submanifold, quantization of Airy structures, and abstract topological recursion, all of which can be found in Section 1-3 of \cite{kontsevich2017airy}. An important example of an Airy structure arises from residue constraints as introduced in \cite[Section 3.3]{kontsevich2017airy} will be discussed in Section \ref{residueconstraintairystructuresection}. 

We have been following the material in \cite{kontsevich2017airy} closely so far, and we will be for the rest of this section. However, for the remaining sections, with an exception of Section \ref{relationstotrsection}, we are going to approach the subject from an analytical perspective rather than a more algebraic treatment of \cite{kontsevich2017airy}. In particular, we will replace all formal power series with power series converges on some annulus. We believe that this should be a more natural way to arrive at our main result involving the prepotential which is a holomorphic function. 

We study the analytic version of residue constraints Airy structure in Section \ref{theembeddingofdiscssection}. In particular, we look at the deformation of analytic discs and show that an embedding of a neighbourhood in the space of all discs gives an analytic solution to the residue constraints, as explained in \cite[Section 6]{kontsevich2017airy}. The connection between the `global' deformations, as studied in Section \ref{deformationofcurvessection}, with the `local' deformations, studied in terms of discs in Section \ref{theembeddingofdiscssection}, will be discussed in Section \ref{localtoglobalsection} analytically, with a summary given by Proposition \ref{localvsglobalproposition} \cite[Proposition 7.1.2]{kontsevich2017airy}. In Section \ref{relationstotrsection} we review the well-known result from \cite[Section 3]{kontsevich2017airy} how Eynard-Orantin topological recursion arises from the residue constraints Airy structure. At last, we prove our main result, Theorem \ref{prepotentialvsomega0ntheorem}, in Section \ref{prepotentialandtrsection}. We argue that when specializing to an example of a family of curves arises from the \emph{Hitchin systems} \cite{hitchin1987stable, baraglia2017special}, our result agrees with the result previously obtained by Baraglia-Huang \cite[Theorem 7.4]{baraglia2017special} using a different technique. An application of Theorem \ref{prepotentialvsomega0ntheorem} to the Seiberg-Witten family of curves will also be briefly mentioned before a more extensive discussion in Section \ref{swexamplechapter}.

\subsection{Deformation of Curves in Foliated Symplectic Surfaces}\label{deformationofcurvessection}
\subsubsection{Foliated symplectic surfaces}
Let us start by fixing some frequently used definitions which should be standard in complex geometry:
\begin{definition}\label{symplecticmanifolddefinition}
A $2n$-dimensional \emph{holomorphic sympletic manifold} $(X,\Omega_X)$ is a $2n$-dimensional complex manifold $X$ equipped with a holomorphic closed non-degenerates $2$-form $\Omega_X$, i.e. at any point $x \in X$, $\Omega_{X,x} : T_xX \times T_xX \rightarrow \mathbb{C}$ is a non-degenerate anti-symmetric bilinear form on the holomorphic tangent space $T_xX$.
\end{definition}
\begin{definition}\label{lagrangiansubmanifolddefinition}
Let $(X,\Omega_X)$ be a holomorphic symplectic manifold. A complex submanifold $L \subset X$ is a \emph{holomorphic Lagrangian submanifold} if for every $x \in L$, $T_xL \subset T_xX$ is a maximal isotropic subspace, i.e. $T_xL$ is a maximal subspace of $T_xX$ such that $\Omega_{X,x}|_{T_xL} = 0$. 
\end{definition}

Since no other types of symplectic manifolds or Lagrangian submanifolds will be considered in this work, we are going to refer to $(X,\Omega)$ and $L \subset X$ in the above definitions simply as \emph{symplectic manifold} and \emph{Lagrangian submanifold} respectively, omitting the term `holomorphic'. Definition \ref{symplecticmanifolddefinition} and Definition \ref{lagrangiansubmanifolddefinition} also make sense when $X$ is infinite-dimensional, in that case any Lagrangian submanifold $L \subset X$ will also be infinite-dimensional. On the other hand, when $X$ is finite-dimensional, it is common to replace the maximal isotropic subspace condition in Definition \ref{lagrangiansubmanifolddefinition} with $\dim L = \frac{1}{2}\dim X$. If $(X, \Omega_X)$ and $L \subset X$ are complex vector spaces, then we also call them a \emph{symplectic vector space} and a \emph{Lagrangian subspace} respectively.

\begin{definition}\label{foliationdefinition}
An $m$-dimensional holomorphic foliation $\mathcal{F}$ on an $n$-dimensional complex manifold $X$ is a decomposition of $X$ into a disjoint union of connected $(n-m)$-dimensional complex submanifolds $\{L_\alpha\}$ called the \emph{leaves} of the foliation satisfying the following property: At every point, $x \in X$ there exists an open neighbourhood $U \ni x$ and a system of local holomorphic coordinates $(x^1,\cdots,x^n)$ such that for each leaf $L_\alpha$, each component of $L_\alpha \cap U$ is given by $\{x^{m+1} = const,\cdots,x^n = const\}$.
\end{definition}
In the case where $X$ is a surface, $n = 2$ we choose $m = 1$ and we will refer to a $1$-dimensional foliation $\mathcal{F}$ of $X$ simply as a foliation $\mathcal{F}$ of $X$.  The leaves of a $1$-dimensional foliation $\mathcal{F}$ of $X$ are Lagrangian for dimensional reasons.

\begin{definition}\label{fomegachartdefinition}
Let $(S, \Omega_S, \mathcal{F})$ be a foliated symplectic surface $S$ with a symplectic form $\Omega_S$ and a foliation $\mathcal{F}$.
Let $U\subseteq S$ be an open subset with a coordinates system $(x,y)$ such that $\Omega_S|_U = dx\wedge dy$ and the foliation $\mathcal{F}$ is given locally on $U$ by the equations $x = const$. We call $(x,y)$ a $(\mathcal{F}, \Omega_S)$-local coordinate system and call $(U,x,y)$ a $(\mathcal{F}, \Omega_S)$-chart.
\end{definition}

\begin{lemma}\theoremname{\cite[Section 4.1]{kontsevich2017airy}}\label{fomegachartlemma}
Given any point $p \in S$ there exists a $(\mathcal{F}, \Omega_S)$-chart $(U_p,x_p,y_p), U_p \subseteq S$ containing $p$.
\end{lemma}
\begin{proof}
By Definition \ref{foliationdefinition}, we choose a local coordinate system, $(\bar{x}_p,\bar{y}_p)$ on some open set $\bar{U}_p \ni p$ such that leaves of the foliation $\mathcal{F}$ are given by $\bar{x}_p = const$. Using this coordinate system, the symplectic form can be written as $\Omega_S = \Omega_S(\bar{x}_p,\bar{y}_p)d\bar{x}_p\wedge d\bar{y}_p$. Consider a coordinates transformation of the form
\begin{equation*}
\bar{x}_p = x_p, \qquad \bar{y}_p = G(x_p, y_p),
\end{equation*}
so that the leaves of foliation are still given by $x_p = const$. Without loss of generality, let us suppose that $(x_p,y_p) = (0,0)$ is the point $p\in S$. The symplectic form becomes
\begin{equation*}
    \Omega_S = \Omega_S\left(x_p, G(x_p, y_p)\right)dx_p\wedge\left(\frac{\partial G}{\partial x_p}dx_p + \frac{\partial G}{\partial y_p}dy_p\right) = \Omega_S\left(x_p, G(x_p,y_p)\right)\frac{\partial G}{\partial y_p}dx_p\wedge dy_p.
\end{equation*}
Therefore, we need to find $G = G(x_p,y_p)$ such that $\Omega_S(x_p, G(x_p, y_p))\frac{\partial G}{\partial y_p} = 1$. Integrating this equation, we get that $G$ is implicitly defined by $\tilde{\Omega}_S(x_p,y_p,G) := \int_0^G\Omega_S(x_p,\tau)d\tau - y_p + C = 0$ for some constant $C$. Let us pick $C = 0$, although this choice is arbitrary. We have $\tilde{\Omega}_S(0,0,0) = 0$ and $\frac{\partial\tilde{\Omega}_S}{\partial G}(0,0,0) = \Omega_S(0,0) \neq 0$ because the symplectic form is non-degenerates by definition. Therefore, we can apply the complex implicit function theorem to conclude that there exists $G = G(x_p,y_p)$ holomorphic on some open neighbourhood $U$ of $p = (x_p = 0,y_p = 0)$.
\end{proof}

We note that the $(\mathcal{F}, \Omega_S)$-local coordinates in the neighbourhood of a point $p \in S$ is not unique. In fact, given a $(\mathcal{F}, \Omega_S)$-local coordinates system $(x, y)$ then the new coordinates system $(\bar{x}, \bar{y})$ is also a $(\mathcal{F}, \Omega_S)$-local coordinates system if and only if it is related to $(x,y)$ under the transformation of type
\begin{equation}\label{coordtransformationpreservingomegaandfoliation}
    \bar{x} = F(x), \qquad \bar{y} = \frac{y}{F'(x)} + G'(x),
\end{equation}
where $F$ and $G$ are locally defined holomorphic functions and $F'(x(p)) \neq 0$. In particular, $F$ is any function locally bi-holomorphic in a neighbourhood of $p$.

\subsubsection{Deformation of $\mathcal{F}$-transversal curves}\label{deformationofcurvesinfamilysubsection}
\begin{definition}\theoremname{\cite[Definition 4.1.2]{kontsevich2017airy}}\label{ftransversaldefinition}
A smooth compact complex curve $\Sigma$ embedded inside a foliated symplectic surface $(S,\Omega_S,\mathcal{F})$ is called $\mathcal{F}$-\emph{transversal} if $\Sigma$ is tangent to the leaves of $\mathcal{F}$ at only finitely many points, with tangency order $1$. We call points on a $\mathcal{F}$-transversal curve $\Sigma$ where $\Sigma$ is tangent to a leaf of foliation the \emph{ramification points} $\{r_{\alpha \in Ram(\Sigma)}\}$ where $Ram(\Sigma)$ is an index set.
\end{definition}

\begin{remark}
Given a Riemann surface $\Sigma$ and a holomorphic map $f : \Sigma \rightarrow \mathbb{P}^1$, the term `ramification point' traditionally refers to a point $r \in \Sigma$ such that $f$ can be locally written as $f(z) = z^n$, for some $n \in \mathbb{Z}_{\geq 2}$, using some holomorphic coordinates charts $(U,z)$ and $(V, w)$ around $r$ and $f(r)$ respectively. For us, instead of a map $f:\Sigma\rightarrow \mathbb{P}^1$, we have the foliation $\mathcal{F}$. By Lemma \ref{fomegachartlemma}, a $(\mathcal{F},\Omega_S)$-chart $(U_p,x_p,y_p)$ can be chosen around any point $p \in \Sigma$, hence the foliation locally induces a map $q \in \Sigma\cap U_p \mapsto x_p(q) \in \mathbb{C}$.  Our usage of the term `ramification point' follows from the fact that each $r_\alpha, \alpha \in Ram$ are ramification points of these locally foliation induced maps. Additionally, the `tangency order $1$' condition means the ramification index of each ramification point is $2$. Note that a different choice of $(\mathcal{F},\Omega_S)$-chart $(U_p, \bar{x}_p, \bar{y}_p)$ will give a different local map $q \mapsto \bar{x}_p(q)$ but ramification points will be the same. This is because $\bar{x}_p = F_p(x_p)$ for some function $F_p$ such that $F_p'(x_p(q)) \neq 0, q \in U_p$ by (\ref{coordtransformationpreservingomegaandfoliation}), hence $d\bar{x}_p|_\Sigma = F_p'(x_p)dx_p|_\Sigma$ means $d\bar{x}_p|_\Sigma = 0$ exactly when $dx_p|_\Sigma = 0$.
\end{remark}

Let $\mathcal{B}$ be the complex analytic moduli space of $\mathcal{F}$-transversal genus $g$ curves in $(S,\Omega_S,\mathcal{F})$. We are going to show later that the dimension of $\mathcal{B}$ is $g$. We write $[\Sigma] \in \mathcal{B}$ for the point in $\mathcal{B}$ corresponding to a $\mathcal{F}$-transversal curve $\Sigma$. Since the number of ramification points is constant for every $[\Sigma] \in \mathcal{B}$, we will simply write the index set as $Ram = Ram(\Sigma)$. For any $p \in \Sigma \subseteq S$ we find a $(\mathcal{F}, \Omega_S)$-chart $(U_p, x_p, y_p)$ containing $p$ according to Lemma \ref{fomegachartlemma}. Let us think of $x_p, y_p$ as functions $\Sigma\cap U_p \rightarrow \mathbb{C}$, the ramification points $\{r_{\alpha \in Ram}\}$ are exactly where $dx_p|_\Sigma$ vanishes with an order $1$ zero. Consequently, if $r_\alpha \notin \Sigma \cap U_p$ for all $\alpha \in Ram$ then $dx_p|_\Sigma \neq 0$ on $\Sigma \cap U_p$ and we may assume that $U_p$ is chosen small enough such that $x_p$ is a local coordinate on $\Sigma \cap U_p$ according to the inverse function theorem. On the other hand, if $p = r_\alpha \in \Sigma, \alpha \in Ram$, we denote by $(U_\alpha, x_\alpha, y_\alpha)$ a $(\mathcal{F},\Omega_S)$-chart such that $\Sigma\cap U_\alpha$ is bi-holomorphic to a simply-connected open subset of $\mathbb{C}$ with a local coordinate $z_\alpha := \sqrt{x_\alpha - x_\alpha(r_\alpha)}$.
\begin{definition}\label{standardlocalcoordinatesdefinition0}
Given a $(\mathcal{F}, \Omega_S)$-chart $(U_\alpha, x_\alpha, y_\alpha)$, we call the local coordinate $z_\alpha := \sqrt{x_\alpha - x_\alpha(r_\alpha)}$ on an open neighbourhood $\Sigma\cap U_\alpha$ of the ramification point $r_\alpha$ a \emph{standard local coordinate}.
\end{definition}
The parameterization of $\Sigma\cap U_\alpha$ using the standard local coordinate takes the form
\begin{equation}
    (x_\alpha = a_\alpha(\Sigma) + z_\alpha^2, y_\alpha = b_{0\alpha}(\Sigma) + b_{1\alpha}(\Sigma)z_\alpha + b_{2\alpha}(\Sigma)z^2_\alpha + ...), \qquad a_\alpha(\Sigma), b_{k\alpha}(\Sigma) \in \mathbb{C},
\end{equation}
where $b_{1\alpha}(\Sigma) \neq 0$ by the tangency order $1$ condition. Therefore, we have $dy_p|_\Sigma(r_\alpha) \neq 0$ and $y_p$ is a local coordinate on some neighbourhoods of $r_\alpha$. In particular, since $\Sigma$ is compact, we can always cover $\Sigma$ with a finite collection of $(\mathcal{F},\Omega_S)$-charts $\mathcal{U} = \{(U_{p\in \sigma}, x_{p\in \sigma}, y_{p\in \sigma})\}$, where $\sigma$ is a finite index set, such that either $x_p$ or $y_p$ is a local coordinate on each $\Sigma \cap U_p$.

The following fact will be useful to us later:
\begin{lemma}\theoremname{\cite[Section 4.1 and 6.2]{kontsevich2017airy}}\label{picoveringmaplemma}
Given $[\Sigma] \in \mathcal{B}$, there exists a $(\mathcal{F},\Omega_S)$-chart $(U_\alpha, x_\alpha, y_\alpha)$ containing the ramification point $r_\alpha, \alpha \in Ram$ of $\Sigma$ such that the parameterization of $\Sigma\cap U_\alpha$ using the the standard local coordinate $z_\alpha$ takes the form $(x_\alpha = z_\alpha^2, y_\alpha = z_\alpha)$. The $(\mathcal{F}, \Omega_S)$-chart $(U_\alpha, x_\alpha, y_\alpha)$ is unique up to the $\mathbb{Z}_3$ automorphism action $x\mapsto \xi^2 x, y\mapsto \xi y$, where $\xi$ is a third-root of unity.
\end{lemma}
\begin{proof}
Let $(U_\alpha, x_\alpha, y_\alpha)$ be a $(\mathcal{F}, \Omega_S)$-chart around $r_\alpha$ which exists due to Lemma \ref{fomegachartlemma} and $z_\alpha := \sqrt{x_\alpha - x_\alpha(r_\alpha)}$ be the standard local coordinate on $\Sigma \cap U_\alpha$. We suppose that the parameterization of $\Sigma\cap U_\alpha$ is given by $(x_\alpha = z^2_\alpha + a, y_\alpha = z_\alpha \psi(z^2_\alpha) + \phi(z^2_\alpha))$ for some $\psi$ and $\phi$ holomorphic on $\Sigma\cap U_\alpha$. We wish to find $F$ and $G$ such that the new $(\mathcal{F}, \Omega_S)$-local coordinates $(\bar{x}_\alpha, \bar{y}_\alpha)$ given by 
\begin{equation*}
    \bar{x}_\alpha = F(x_\alpha), \qquad \bar{y}_{\alpha} = \frac{y_\alpha}{F'(x_\alpha)} + G'(x_\alpha)
\end{equation*}
with the new standard local coordinate $\bar{z}_\alpha := \sqrt{\bar{x}_\alpha - \bar{x}_\alpha(r_\alpha)} = \sqrt{F(z_\alpha^2 + a) - F(a)}$ satisfies
\begin{equation*}
    \bar{z}^2_\alpha = F(z_\alpha^2 + a), \qquad \bar{z}_\alpha = \frac{z_\alpha \psi(z_\alpha^2) + \phi(z_\alpha^2)}{F'(z_\alpha^2 + a)} + G'(z_\alpha^2 + a).
\end{equation*}
This is equivalent to 
\begin{equation}\label{discafterchangecoordinates}
    F^{-1}(0) = a, \qquad F'\left(z_\alpha^2 + F^{-1}(0)\right)\left(\left(F\left(z_\alpha^2 + F^{-1}(0)\right)\right)^{1/2} - G'\left(z_\alpha^2 + F^{-1}(0)\right)\right) = z_\alpha \psi(z_\alpha^2) + \phi(z_\alpha^2).
\end{equation}
The second equation above can be solved by separately equating the odd and even terms. The odd terms yields:
\begin{equation}\label{findingf}
    F'\left(z^2_\alpha + F^{-1}(0)\right)F\left(z^2_\alpha + F^{-1}(0)\right)^{1/2} = z_\alpha\psi(z^2_\alpha)
    \implies \frac{1}{3}\frac{d}{dz_\alpha}\left(F\left(z^2_\alpha + F^{-1}(0)\right)\right)^{3/2} = z^2_\alpha\psi(z^2_\alpha).
\end{equation}
Both sides can be expanded as power series in $z^2$, starting with the $z^2$ term, so $F(z^2 + a)$ and hence $F(x)$ can be found by comparing coefficients. Similarly, we can define $G(x) := \frac{\phi(x - F^{-1}(0))}{F'(x)}$. Note that if $\left(F(z^2+a)\right)^{1/2}$ satisfies (\ref{findingf}), then $\xi \left(F(z^2 + a)\right)^{1/2}$ also does for any third root of unity $\xi$.
\end{proof}

The tangent space $T_{[\Sigma]}\mathcal{B}$ describes the infinitesimal deformations of the curve $\Sigma$ inside $S$. It is isomorphic to the space of global sections of the holomorphic normal bundle $\nu_\Sigma \cong TS|_{\Sigma}/T\Sigma$, where $TS$ and $T\Sigma$ are holomorphic tangent bundles of $S$ and $\Sigma$ respectively. In general, given any holomorphic vector bundle $\mathcal{V}$ on a complex manifold $X$, we denote the set of holomorphic sections of $\mathcal{V}$ by $\Gamma(X,\mathcal{V})$. The availability of the symplectic form $\Omega_S$ gives the following result:
\begin{lemma}\theoremname{\cite[Section 4.2]{kontsevich2017airy}}
Let $\Omega^1_{\Sigma}$ be the holomorphic canonical line bundle on $\Sigma$, then the symplectic form $\Omega_S$ induces the isomorphism
$\Gamma(\Sigma, \nu_\Sigma) \cong \Gamma(\Sigma,\Omega^1_{\Sigma})$. Consequently, $\dim T_{[\Sigma]}\mathcal{B} = \dim \Gamma(\Sigma, \Omega^1_\Sigma) = g$, hence $\mathcal{B}$ is $g$-dimensional.
\end{lemma}
\begin{proof}
We establish the isomorphism $\Gamma(\Sigma, \nu_\Sigma) \cong \Gamma(\Sigma,\Omega^1_{\Sigma})$ by $v \mapsto \Omega_S(v, .)|_{\Sigma}$. Any holomorphic section of the normal bundle $v \in \Gamma(\Sigma,TS|_{\Sigma}/T\Sigma)$ can be represented as a collection of sections $\{v_{p\in \sigma} \in \Gamma(V_p, TS|_\Sigma)\}$ such that $v_p - v_{p'} \in \Gamma(V_p\cap V_{p'}, T\Sigma)$ for some open cover $\{V_{p\in \sigma}\}$ of $\Sigma$. We note that for each $p \in \sigma$, $\Omega_S(v_p,.)|_{V_p}$ is a holomorphic form on $V_p \subset \Sigma$ as it is a pull-back of the contraction between $\Omega_S$ and $v_p$ which are both holomorphic. Since $\Omega_S(t,.)|_{\Sigma}$ for all $t \in \Gamma(\Sigma, T\Sigma)$, we can patch together $\Omega_S(v_p,.)|_{V_p}$ for each $p \in \sigma$ to get a global section $\Omega_S(v,.)|_{\Sigma} \in \Gamma(\Sigma, \Omega^1_\Sigma)$.

The map $v \mapsto \Omega_S(v, .)|_{\Sigma}$ is injective because if $\Omega_S(v, .)|_\Sigma = 0$ then by the fact that $\Sigma \subseteq S$ is a Lagrangian submanifold, we have $v(p) \in T_p\Sigma$ for all $p \in \Sigma$. Therefore, $v = 0\in \Gamma(\Sigma, \nu_\Sigma)$. For surjectivity, given a holomorphic differential form $\omega \in \Gamma(\Sigma, \Omega^1_\Sigma)$, we cover $\Sigma$ with a collection of $(\mathcal{F},\Omega_S)$-charts $\mathcal{U} = \{(U_{p\in \sigma}, x_{p\in \sigma}, y_{p\in \sigma})\}$ such that either $x_p$ or $y_p$ is a local coordinate on $\Sigma \cap U_p$. Let us assume without loss of generality that $x_p$ is a local coordinate on $\Sigma\cap U_p$. Then in general $\omega|_{\Sigma \cap U_p} = \omega_p(x_p)dx_p$ and by taking $v_p := -\omega_p(x_p)\frac{\partial}{\partial y_p} \in \Gamma(\Sigma\cap U_p, TS|_{\Sigma})$ we have $\Omega_S(v_p,.)|_{\Sigma\cap U_p} = \omega|_{\Sigma\cap U_p}$. For $p,p' \in \sigma$ we have on $U_p \cap U_{p'}$ that 
\begin{equation*}
    \Omega_S(v_p - v_{p'},.)|_{\Sigma \cap U_p\cap U_{p'}} = \left(\omega|_{\Sigma\cap U_p}\right)|_{\Sigma \cap U_p\cap U_{p'}} - \left(\omega|_{\Sigma\cap U_{p'}}\right)|_{\Sigma \cap U_p\cap U_{p'}} = 0.
\end{equation*}
Therefore, $(v_p - v_{p'})|_{\Sigma \cap U_p \cap U_{p'}} \in \Gamma(\Sigma\cap U_p\cap U_{p'}, T\Sigma)$ and $\{v_{p\in\sigma}\}$ represents a global section of a normal bundle $v \in \Gamma(\Sigma, \nu_{\Sigma})$ such that $\Omega_S(v,.) = \omega$.
\end{proof}

Alternatively, since the symplectic form $\Omega_S$ is a global section of the holomorphic canonical line bundle $\Omega_S^2$ of $S$ with no zeros or poles, we have from the adjunction formula: $\Omega^1_\Sigma \cong \Omega^2_S|_\Sigma \otimes \nu_\Sigma$ that $\Omega^1_\Sigma \cong \nu_\Sigma$. It follows that $\Gamma(\Sigma, \nu_\Sigma) \cong \Gamma(\Sigma, \Omega^1_\Sigma)$ and the map $v \mapsto \Omega_S(v,.)$ explicitly express this isomorphism. 

\begin{definition}
Consider the \emph{universal family} of $\mathcal{F}$-transversal curves:
\begin{equation*}
    Z := \{([\Sigma], p) \in \mathcal{B}\times S\ |\ p \in \Sigma\} \subseteq \mathcal{B}\times S,
\end{equation*}
let $\pi : Z \rightarrow \mathcal{B}$ be the canonical projection and let $\mathbb{C}$ be a constant sheaf on $Z$. Then we define the holomorphic symplectic vector bundle $(\mathcal{H}\rightarrow \mathcal{B}, \Omega_{\mathcal{H}}, \nabla_{GM})$ where $\mathcal{H} := R^1\pi_*\mathbb{C}$ is a holomorphic vector bundle with fiber over each point $[\Sigma] \in \mathcal{B}$ given by $\mathcal{H}_{\Sigma} = H^1(\Sigma, \mathbb{C})$. $\nabla_{GM}$ is the \emph{Gauss-Manin} connection and the symplectic form $\Omega_{\mathcal{H}}$ is given by 
\begin{equation}\label{omegahdefinition}
    \Omega_{\mathcal{H}}([\xi_1],[\xi_2]) := \iint_\Sigma \xi_1\wedge \xi_2
\end{equation}
where $\xi_1, \xi_2$ are smooth differential forms representing $[\xi_1], [\xi_2] \in \mathcal{H}_\Sigma$.
\end{definition}

For any open subset, $\mathcal{B}'\subseteq \mathcal{B}$ we denote by $\Gamma(\mathcal{B}', \mathcal{H})$ the set of holomorphic sections of $\mathcal{H}$ over $U$.

\begin{remark}\label{homologycohomologyremark}
The natural symplectic form $\Omega^*_{\mathcal{H}}$ on $H_1(\Sigma, \mathbb{C})$ can be defined using the oriented intersection product of homology classes as follows. For any $[C_1], [C_2] \in H_1(\Sigma, \mathbb{Z})$, the intersection product is given by $[C_1]\cdot [C_2] = [C_1 \cap C_2] \in H_0(\Sigma, \mathbb{Z}) \cong \mathbb{Z}$ which can be extended linearly to the intersection product on $H_1(\Sigma, \mathbb{C}) = H_1(\Sigma, \mathbb{Z})\otimes \mathbb{C}$. Then, for any $[C_1], [C_2] \in H_1(\Sigma,\mathbb{C})$, we define $\Omega^*_{\mathcal{H}}([C_1], [C_2]) := [C_1]\cdot [C_2]$. Typically, we denote a symplectic basis of $H_1(\Sigma,\mathbb{C})$ by $A$ and $B$ cycles: $\{[A_1],\cdots,[A_g],[B_1],\cdots,[B_g]\}$,
\begin{equation*}
    \Omega_{\mathcal{H}}^*([A_i],[A_j]) = \Omega_{\mathcal{H}}^*([B_i],[B_j]) = 0, \qquad \Omega_{\mathcal{H}}^*([A_i],[B_j]) = \delta_{ij}.
\end{equation*}

On $H^1(\Sigma, \mathbb{C})$, the symplectic form $\Omega_{\mathcal{H}}$ is defined by the intersection product of cohomology classes via Poincar\'{e} duality $P.D. : H^1(\Sigma,\mathbb{C}) \cong H_1(\Sigma,\mathbb{C})$. In other words, given $[\xi_1], [\xi_2] \in H^1(\Sigma, \mathbb{C})$ then (\ref{omegahdefinition}) is equivalent to $\Omega_{\mathcal{H}}([\xi_1],[\xi_2]) := \Omega^*_{\mathcal{H}}(P.D.[\xi_1],P.D.[\xi_2])$. Note also that the expression for $\Omega_{\mathcal{H}}([\xi_1],[\xi_2])$ as given in (\ref{omegahdefinition}) only depends on the cohomology classes of $\xi_1$ and $\xi_2$, because $\iint_{\Sigma}df\wedge \xi_2 = \iint_{\Sigma}d(f\xi_2) = 0$ by Stokes Theorem.
\end{remark}

\begin{remark}\label{htrivializationtremark}
Suppose that the $A, B$ symplectic basis cycles $\{[A_1],\cdots,[A_g],[B_1],\cdots,[B_g]\}$ of $H_1(\Sigma, \mathbb{C})$ are chosen for $[\Sigma] \in \mathcal{B}$. Since $H^1(\Sigma, \mathbb{C}) \cong H_1(\Sigma, \mathbb{C})^*$, a cohomology class $[\xi] \in H^1(\Sigma, \mathbb{C})$ is determined by its values on $A$ and $B$ cycles. In particular, the cohomology class $[\xi]$ represented by a closed differential form $\xi$ is given by the $A$ and $B$ periods of $\xi$ and we identify $\mathcal{H}_\Sigma = H^1(\Sigma, \mathbb{C}) \xrightarrow{\cong} \mathbb{C}^{2g}$ by 
\begin{equation*}
    H^1(\Sigma,\mathbb{C}) \ni [\xi] \mapsto \left(\oint_{A_1}\xi,\cdots, \oint_{A_g}\xi,\oint_{B_1}\xi,\cdots, \oint_{B_g}\xi\right) \in \mathbb{C}^{2g}.
\end{equation*}
It can happen as we move from $[\Sigma]$ along some closed path in $\mathcal{B}$, letting $A, B$ cycles deform continuously, that once we return to the point $[\Sigma]$ the $A, B$ cycles will differ from the original by some Dehn twist actions. On the other hand, given a sufficiently small open neighbourhood $\mathcal{B}_{\Sigma} \subseteq \mathcal{B}$, we will be able to choose $A, B$ symplectic basis cycles consistently throughout $\mathcal{B}_\Sigma$. The Gauss-Manin connection $\nabla_{GM}$ can be described locally on $\mathcal{B}_{\Sigma}$ as the identification of any two cohomology classes on two different fibers $[\xi_1] \in \mathcal{H}_{\Sigma_1}, [\xi_2] \in \mathcal{H}_{\Sigma_2}$, if they have the same $A$ and $B$ periods: $\oint_{A_k \subset \Sigma_1}\xi_1 = \oint_{A_k\subset \Sigma_2}\xi_2, \oint_{B_k\subset \Sigma_1}\xi_1 = \oint_{B_k\subset \Sigma_2}\xi_2$. This gives a local trivialization $\mathcal{H}|_{\mathcal{B}_{\Sigma}}\cong \mathbb{C}^{2g}\times\mathcal{B}_{\Sigma}$, verifying that $\mathcal{H}$ is a vector bundle.
\end{remark}

The following gives an alternative expression for $\Omega_{\mathcal{H}}$:
\begin{lemma}\theoremname{Riemann bilinear identity}\label{symplecticformonhlemma}
Let $[\xi_1], [\xi_2] \in H^1(\Sigma,\mathbb{C})$ then
\begin{equation*}
    \iint_\Sigma \xi_1\wedge \xi_2 = \frac{1}{2\pi i} \sum_{k=1}^g\left(\oint_{A_k}\xi_2\oint_{B_k}\xi_1 - \oint_{B_k}\xi_2\oint_{A_k}\xi_1\right)
\end{equation*}
where $\xi_1, \xi_2$ are smooth differential forms representing $[\xi_1], [\xi_2]$.
\end{lemma}
\begin{proof}
First, we show that for any $[\xi] \in H^1(\Sigma,\mathbb{C})$ has a representative $\xi = \alpha + \bar{\beta}$ where $\alpha$ is holomorphic ($\bar{\partial}\alpha = 0$) and $\bar{\beta}$ is anti-holomorphic ($\partial \bar{\beta} = 0$). A smooth differential form $\xi'$ representing $[\xi]$ can be decomposed as $\xi' = \alpha' + \bar{\beta'}$ where $\alpha'$ is a smooth $(1,0)$-form and $\bar{\beta'}$ is a smooth $(0,1)$-form. Then $d\xi' = (\partial+\bar{\partial})\xi' = 0$ implies $\bar{\partial}\alpha' = -\partial\bar{\beta'}$. Since Riemann surfaces are examples of K\"{a}hler manifolds, $\partial\bar{\partial}$-Lemma \cite[Proposition 6.17]{voisin2002} tells us that there exists a smooth function $f$ such that $\bar{\partial}\alpha' = \partial\bar{\partial}f = -\partial\bar{\beta'}$. We can see that $\xi := \xi' + df = \alpha + \bar{\beta}$ where $\alpha := \alpha' + \partial f$ and $\bar{\beta} = \beta' + \bar{\partial} f$ is the desired representative. 

The rest of the proof follows from Stokes' theorem and the standard argument leading to the Riemann bilinear identity. For example, we have
\begin{equation*}
    \iint_{\Sigma}\alpha \wedge \bar{\beta} = \iint_{\Delta} d\phi\wedge \bar{\beta} = \iint_{\Delta}d(\phi\bar{\beta}) = \oint_{\partial \bar{\Delta}}\phi\bar{\beta} = \frac{1}{2\pi i} \sum_{k=1}^g\left(\oint_{A_k}\bar{\beta}\oint_{B_k}\alpha - \oint_{B_k}\bar{\beta}\oint_{A_k}\alpha\right)
\end{equation*}
where $\Delta := \Sigma \setminus \cup_{k=1}^g\left(A_k\cup B_k\right)$ is a simply-connected domain and $\alpha = df = \partial f$ for some holomorphic function $\phi$ on $\Delta$ by Poincar\'{e} Lemma.
\end{proof}

Let us fix a \emph{reference point} $[\Sigma_0] \in \mathcal{B}$ and an open neighbourhood $\mathcal{B}_{\Sigma_0} \subseteq \mathcal{B}$ of $[\Sigma_0]$. Later we will explain more precisely how $\mathcal{B}_{\Sigma_0}$ is selected (see Condition \ref{howtochoosebsigma0condition}). But for now, we simply require $\mathcal{B}_{\Sigma_0}$ to be contractible and sufficiently small such that it can be covered by a single coordinate chart $(\mathcal{B}_{\Sigma_0}, u^1,...,u^g)$. We will denote by $\mathcal{H}\rightarrow \mathcal{B}_{\Sigma_0}$ the vector bundle given by the restriction of $\mathcal{H}$ to $\mathcal{B}_{\Sigma_0}$, which is trivial because $\mathcal{B}_{\Sigma_0}$ is contractible.
Let $T\mathcal{B}$ and $T^*\mathcal{B}$ denote the holomorphic tangent and cotangent bundles of $\mathcal{B}$ respectively. Following \cite[Section 5.2]{kontsevich2017airy}, we introduce $\pmb{\phi} \in \Gamma(\mathcal{B}_{\Sigma_0},T^*\mathcal{B}\otimes \mathcal{H})$ as follows. At any point $[\Sigma] \in \mathcal{B}_{\Sigma_0}$ we define the map $\pmb{\phi}_{\Sigma}$ via the following diagram: 
\begin{equation*}
    \begin{tikzcd}
    T_{[\Sigma]}\mathcal{B}_{\Sigma_0} \arrow{r}{\cong}\arrow[bend left=20]{rrr}{\pmb{\phi}_\Sigma} & \Gamma(\Sigma, TS|_\Sigma/T\Sigma) \arrow{r}{\cong} & \Gamma(\Sigma, \Omega^1_\Sigma) \arrow[hookrightarrow]{r}{[.]} & H^1(\Sigma, \mathbb{C}) \cong \mathcal{H}_\Sigma.
    \end{tikzcd}
\end{equation*}

Let us discuss the first map $T_{[\Sigma]}\mathcal{B}_{\Sigma_0} \xrightarrow{\cong} \Gamma(\Sigma, TS|_\Sigma/T\Sigma)$.
We map $v = \sum_{k=1}^g v^k\frac{\partial}{\partial u^k}$ to a global normal vector field $n_v \in \Gamma(\Sigma, TS|_\Sigma/T\Sigma)$ which can be defined in terms of $TS|_\Sigma$-sections on each $(\mathcal{F}, \Omega_S)$-chart $(U_p, x_p, y_p)$:
\begin{equation}\label{nvlocal}
    n_v|_{\Sigma\cap U_p} := \sum_{k=1}^gv^k\left(\frac{\partial x_p}{\partial u^k}\frac{\partial}{\partial x_p} + \frac{\partial y_p}{\partial u^k}\frac{\partial}{\partial y_p}\right) \in \Gamma(\Sigma\cap U_p, TS|_\Sigma).
\end{equation}
Taking the quotient of $n_v|_{\Sigma\cap U_p}$ on each $(U_p, x_p, y_p)$ and patch each of them together, we get a global section $n_v$. Finally, we define $\pmb{\phi}_\Sigma(v) := [\phi_\Sigma(v)] \in \mathcal{H}_\Sigma$ where $\phi_\Sigma(v) := \Omega_S(n_v,.)|_\Sigma \in \Gamma(\Sigma, \Omega^1_\Sigma)$. Let $v_k := \frac{\partial}{\partial u^k}$, then we can equivalently write this as
\begin{equation}\label{definingphi}
    \phi_\Sigma := \sum_{k=1}^g\Omega_S(n_{v_k},.)|_\Sigma du^k \in T^*_{[\Sigma]}\mathcal{B}\otimes \Gamma(\Sigma, \Omega^1_\Sigma), \qquad \pmb{\phi}_\Sigma := [\phi_\Sigma] = \sum_{k=1}^g[\Omega_S(n_{v_k},.)|_\Sigma]du^k \in T^*_{[\Sigma]}\mathcal{B}\otimes\mathcal{H}_\Sigma.
\end{equation}

\begin{remark}
The image $\text{im}\pmb{\phi}_{\Sigma} = \Gamma(\Sigma, \Omega^1_{\Sigma})$ is a Lagrangian subspace of $\mathcal{H}_{\Sigma}$ for every $[\Sigma] \in \mathcal{B}_{\Sigma_0}$.
\end{remark}

At first glance, the expression (\ref{nvlocal}) seems ambiguous because it depends on the choice of a connection, in other words, on how the derivatives of $x_p$ and $y_p$ are taken with respect to $u^k$. Suppose that $\Sigma \cap U_p = \{(x_p, y_p)\in U_p \ |\ \mathcal{P}_p(x_p,y_p;u) = 0\}$ for some polynomial $\mathcal{P}_p$. For any $q \in \Sigma \cap U_p$, we have $t \in T_q\Sigma \subset T_qS$ if $\iota_t d\mathcal{P}_p = 0$ where $d\mathcal{P}_p = \frac{\partial \mathcal{P}_p}{\partial x_p}dx_p + \frac{\partial \mathcal{P}_p}{\partial y_p}dy_p$. The variation of the curve $\Sigma$ in the $u^k$-direction in $\mathcal{B}$ gives 
\begin{equation}\label{ppequation}
    \frac{\partial \mathcal{P}_p}{\partial u^k} + \frac{\partial \mathcal{P}_p}{\partial x_p}\frac{\partial x_p}{\partial u^k} + \frac{\partial \mathcal{P}_p}{\partial y_p}\frac{\partial y_p}{\partial u^k} = 0.
\end{equation}
Choosing a connection is equivalent to choosing $\frac{\partial x_p}{\partial u^k}$ and $\frac{\partial y_p}{\partial u^k}$ satisfying this equation. We need to check that if connections are chosen differently on each $(U_p,x_p,y_p)$ then 
\begin{equation*}
    \left(n_v|_{\Sigma \cap U_p}\right)|_{\Sigma\cap U_p\cap U_{p'}} - \left(n_v|_{\Sigma \cap U_{p'}}\right)|_{\Sigma\cap U_p\cap U_{p'}} \in \Gamma(\Sigma \cap U_p \cap U_{p'}, T\Sigma)
\end{equation*}
and therefore we can patch the quotient of each $n_v|_{\Sigma \cap U_p}$ to form a global section $n_v \in \Gamma(\Sigma,\nu_{\Sigma})$. The following lemma resolves this concern:
\begin{lemma}\label{normallocalsectionslemma}
Suppose we obtain $(n_v|_{\Sigma\cap U_p})_i, i = 1,2$ using any two different choices of connections, then $(n_v|_{\Sigma\cap U_p})_1 - (n_v|_{\Sigma\cap U_p})_2 \in \Gamma(\Sigma\cap U_p, T\Sigma)$.
\end{lemma}
\begin{proof}
Without loss of generality, let us assume that $v = \frac{\partial}{\partial u^k}$. For $i = 1,2$, let
\begin{equation*}
    (n_v|_{\Sigma\cap U_p})_i := \left(\frac{\partial x_p}{\partial u^k}\right)_i\frac{\partial}{\partial x_p} + \left(\frac{\partial y_p}{\partial u^k}\right)_i\frac{\partial}{\partial y_p}
\end{equation*}
where
\begin{equation}\label{ppequationi}
    \frac{\partial \mathcal{P}_p}{\partial u^k} + \frac{\partial \mathcal{P}_p}{\partial x_p}\left(\frac{\partial x_p}{\partial u^k}\right)_i + \frac{\partial \mathcal{P}_p}{\partial y_p}\left(\frac{\partial y_p}{\partial u^k}\right)_i = 0.
\end{equation}
If we subtract (\ref{ppequationi}) with $i = 2$ from (\ref{ppequationi}) with $i = 1$ then 
\begin{equation*}
    \left(\left(\frac{\partial x_p}{\partial u^k}\right)_1 - \left(\frac{\partial x_p}{\partial u^k}\right)_2\right)\frac{\partial \mathcal{P}_p}{\partial x_p} + \left(\left(\frac{\partial y_p}{\partial u^k}\right)_1 - \left(\frac{\partial y_p}{\partial u^k}\right)_2\right)\frac{\partial \mathcal{P}_p}{\partial y_p} = 0.
\end{equation*}
The lemma follows as we recognize the left-hand-side to be $\iota_{(n_v|_{\Sigma\cap U_p})_1 - (n_v|_{\Sigma\cap U_p})_2}d\mathcal{P}_p$. 
\end{proof}

The next lemma tells us that $\nabla_{GM}[\xi] = \sum_{k=1}^g\left[\frac{\partial \xi}{\partial u^k}\right]du^k$ is independent of the choice of the connection and the differential form $\xi$ we choose to represent $[\xi]$. 

\begin{lemma}\label{integralandderivativelemma}
Let $\xi_\Sigma$ be a holomorphic family of meromorphic differential forms defined for each $[\Sigma] \in \mathcal{B}_{\Sigma_0}$ on an open subsets $D_\Sigma \subseteq \Sigma$ and let $\left(\frac{\partial \xi_\Sigma}{\partial u^k}\right)_i, i = 1,2$ be the derivative $\frac{\partial}{\partial u^k}$ taken using two different choices of connections. Then $\left(\frac{\partial \xi_\Sigma}{\partial u^k}\right)_1 - \left(\frac{\partial \xi_\Sigma}{\partial u^k}\right)_2$ is an exact meromorphic differential on $D_\Sigma$.
In particular, for any closed contour $C \subseteq D_\Sigma$ which deforms homotopically over $\mathcal{B}_{\Sigma_0}$ and avoiding all poles of $\xi_\Sigma$ for each $[\Sigma] \in \mathcal{B}_{\Sigma_0}$, we have 
\begin{equation}\label{integralandderivativeformula}
    \frac{\partial}{\partial u^k}\oint_C\xi_\Sigma = \oint_C \left(\frac{\partial}{\partial u^k}\xi_\Sigma\right)_{\nabla}
\end{equation}
independent of the the choice of connection $\nabla$ the derivative $\frac{\partial}{\partial u^k}\xi_\Sigma$ is taken with respect to.
\end{lemma}
\begin{proof}
First, let us consider the case in which $D_\Sigma \subseteq \Sigma \cap U_p$ where $(U_p, x_p, y_p)$ is a $(\mathcal{F}, \Omega_S)$-chart with $dx_p|_{\Sigma U_p}\neq 0$. We write in general $\xi_\Sigma|_{\Sigma \cap U_p} = \xi_{\Sigma,p}(x_p,u)dx_p$. Then we have for $i = 1,2$:
\begin{align*}
    \left(\frac{\partial \xi_\Sigma}{\partial u^k}\right)_i &= \frac{\partial \xi_{\Sigma,p}}{\partial u^k}dx_p + \left(\frac{\partial x_p}{\partial u^k}\right)_i\frac{\partial \xi_{\Sigma,p}}{\partial x_p}dx_p + \xi_{\Sigma,p}d\left(\frac{\partial x_p}{\partial u^k}\right)_i\\
    &= \frac{\partial \xi_{\Sigma,p}}{\partial u^k}dx_p + \left(\frac{\partial x_p}{\partial u^k}\right)_i\frac{\partial \xi_{\Sigma,p}}{\partial x_p}dx_p - \left(\frac{\partial x_p}{\partial u^k}\right)_i\frac{\partial \xi_{\Sigma,p}}{\partial x_p}dx_p + d\left(\xi_p\left(\frac{\partial x_p}{\partial u^k}\right)_i\right) = \frac{\partial \xi_{\Sigma,p}}{\partial u^k}dx_p + d\left(\xi_p\left(\frac{\partial x_p}{\partial u^k}\right)_i\right).
\end{align*}
Evidently, we have 
\begin{equation*}
    \left(\frac{\partial \xi_\Sigma}{\partial u^k}\right)_1 - \left(\frac{\partial \xi_\Sigma}{\partial u^k}\right)_2 = d\left(\xi_{\Sigma,p}\left(\frac{\partial x_p}{\partial u^k}\right)_1 - \xi_{\Sigma,p}\left(\frac{\partial x_p}{\partial u^k}\right)_2\right)
\end{equation*}
which is exact as claimed. In general we cover $D_\Sigma$ with a collection of $(\mathcal{F},\Omega_S)$-charts $\{(U_{p\in\sigma}, x_{p\in\sigma}, y_{p\in\sigma})\}$ and check that the function $\xi_{\Sigma,p}\left(\frac{\partial x_p}{\partial u^k}\right)_1 - \xi_{\Sigma,p}\left(\frac{\partial x_p}{\partial u^k}\right)_2$ on each $\Sigma\cap U_p$ patches together to give us a function $\Xi_{\Sigma,12}$ defined on $D_\Sigma$ such that $d\Xi_{\Sigma,12} = \left(\frac{\partial \xi_\Sigma}{\partial u^k}\right)_1 - \left(\frac{\partial \xi_\Sigma}{\partial u^k}\right)_2$. 

To see that (\ref{integralandderivativeformula}) is true, we dissect $C$ into closed segments $C_i, i = 1,\cdots,N$ such that $C_i \subset \Sigma\cap U_p$, then apply the Leibniz integral rule on each segment $C_i$. Since the endpoints of each $C_i$ vary with $u^k$ we also get the boundary term contribution, but the sum of these contributions from each $C_i$ will be zero because $C$ is closed. Note that there is no ambiguity in what connection to use on the right-hand-side of (\ref{integralandderivativeformula}) because the result will only be differed by an exact differential, which vanishes on a closed contour $C$.
\end{proof}

\begin{lemma}\theoremname{\cite[Section 5.1]{kontsevich2017airy}}\label{hisaflatbundlelemma}
The symplectic vector bundle $(\mathcal{H}\rightarrow \mathcal{B}, \Omega_{\mathcal{H}}, \nabla_{GM})$ is flat and $\Omega_{\mathcal{H}}$ is $\nabla_{GM}$-covariantly constant.
\end{lemma}
\begin{proof}
Let $d_{\nabla_{GM}} : \Gamma(\mathcal{B}_{\Sigma_0}, \Omega^r_{\mathcal{B}}\otimes \mathcal{H}) \rightarrow \Gamma(\mathcal{B}_{\Sigma_0}, \Omega^{r+1}_{\mathcal{B}}\otimes \mathcal{H})$ be the exterior covariant derivative. The condition that $\nabla_{GM}$ is flat is the same as the vanishing of the curvature tensor $R_{\nabla_{GM}} = d_{\nabla_{GM}}\circ d_{\nabla_{GM}} = 0$. For $\psi = s[\xi] \in \Gamma(\mathcal{B}_{\Sigma_0}, \Omega^r_{\mathcal{B}}\otimes \mathcal{H}), s \in \Gamma(\mathcal{B}_{\Sigma_0}, \Omega^r_{\mathcal{B}})$ and $[\xi] \in \Gamma(\mathcal{B}_{\Sigma_0}, \mathcal{H})$, we have 
\begin{equation*}
    d_{\nabla_{GM}}\psi = (ds)[\xi] + (-1)^rs\wedge\nabla_{GM}[\xi] = (ds)[\xi] + (-1)^rs\wedge\sum_{k=1}^g\left[\frac{\partial \xi}{\partial u^k}\right]du^k
\end{equation*}
where we have used Lemma \ref{integralandderivativelemma}: $\nabla_{GM}[\xi] = \sum_{k=1}^g\left[\frac{\partial \xi}{\partial u^k}\right]du^k$. Therefore,
\begin{align*}
    d^2_{\nabla_{GM}}\psi &= (d^2s)[\xi] + (-1)^r(ds)\wedge\sum_{k=1}^g\left[\frac{\partial \xi}{\partial u^k}\right]du^k\\
    &\qquad + (-1)^{r+1}(ds)\wedge\sum_{k=1}^g\left[\frac{\partial \xi}{\partial u^k}\right]du^k + (-1)^rs\wedge\sum_{k,l = 1}^g\left[\frac{\partial \xi}{\partial u^k\partial u^l}\right]du^k\wedge du^l
    = 0
\end{align*}
which means $\nabla_{GM}$ is a flat connection. 

Because the vector bundle $\mathcal{H}\rightarrow \mathcal{B}$ is flat, parallel transport on $\mathcal{H}\rightarrow \mathcal{B}_{\Sigma_0}$ using the Gauss-Manin connection is path-independent. Let us denote by $\Gamma_{[\Sigma]}^{[\Sigma']}[\xi] \in \mathcal{H}_{\Sigma'}$ the parallel transport of $[\xi] \in \mathcal{H}_{\Sigma}$ from the point $[\Sigma] \in \mathcal{B}_{\Sigma_0}$ to $[\Sigma'] \in \mathcal{B}_{\Sigma_0}$.
Since parallel transport preserves $A$ and $B$-periods, we have $\Omega_{\mathcal{H}}(\Gamma_{[\Sigma]}^{[\Sigma']}[\xi_1], \Gamma_{[\Sigma]}^{[\Sigma']}[\xi_2]) = \Omega_{\mathcal{H}}([\xi_1],[\xi_2])$ proving that $\Omega_{\mathcal{H}}$ is $\nabla_{GM}$-covariantly constant.
\end{proof}

We will find in Section \ref{localtoglobalsection} that the connection given by taking derivatives keeping the direction of foliation leaves constant is particularly important and we denote it by $\nabla_{\mathcal{F}}$. Let us write down $\phi_{\Sigma}(v)$ and $\pmb{\phi}_{\Sigma}(v)$ using $\nabla_{\mathcal{F}}$. Let us cover $\Sigma$ with a collection of $(\mathcal{F},\Omega_S)$-charts $\mathcal{U} = \{(U_{p\in \sigma}, x_{p\in \sigma}, y_{p\in \sigma})\}$. On each $\Sigma\cap U_p$, we write $\nabla_{\mathcal{F}} = \sum_{k=1}^gdu^k\frac{\partial}{\partial u^k}|_{x_p = const}$. Straightforwardly, we have
\begin{align*}
    \phi_\Sigma|_{\Sigma\cap U_p} &= \sum_{k=1}^g\Omega_S(n_{v_k},.)|_{\Sigma\cap U_p}du^k = \sum_{k=1}^g\Omega_S\left(\frac{\partial x_p}{\partial u^k}\Big|_{x_p = const}\frac{\partial}{\partial x_p} + \frac{\partial y_p}{\partial u^k}\Big|_{x_p = const}\frac{\partial}{\partial y_p},.\right)\Big|_{\Sigma\cap U_p}du^k\\
    &= \sum_{k=1}^g\Omega_S\left(\frac{\partial y_p}{\partial u^k}\Big|_{x_p = const}\frac{\partial}{\partial y_p},.\right)\Big|_{\Sigma\cap U_p}du^k = -\sum_{k=1}^g\left(\frac{\partial y_p}{\partial u^k}\Big|_{x_p = const}dx_p|_{\Sigma\cap U_p}\right)du^k\\
    &= -\nabla_\mathcal{F}(y_pdx_p|_{\Sigma\cap U_p}),
\end{align*}
and so $\phi_\Sigma(v)|_{\Sigma\cap U_p} = -\iota_{v_k}\nabla_{\mathcal{F}}(y_pdx_p|_{\Sigma\cap U_p})$. Lastly, we simply put $\pmb{\phi}_{\Sigma}(v) = [\phi_\Sigma(v)] \in\mathcal{H}_\Sigma$.

\begin{remark}
Note that, from (\ref{ppequation}) we have $\frac{\partial y_p}{\partial u^k}|_{x_p = const} = -\frac{\partial \mathcal{P}_p/\partial u^k}{\partial \mathcal{P}_p/\partial y_p}$. Using
\begin{equation*}
    d\mathcal{P}_p|_{\Sigma\cap U_p} = \frac{\partial \mathcal{P}_p}{\partial x_p}dx_p|_{\Sigma\cap U_p} + \frac{\partial \mathcal{P}_p}{\partial y_p}dy_p|_{\Sigma\cap U_p} = 0,
\end{equation*}
we can see that $\frac{\partial \mathcal{P}_p}{\partial y_p}$ vanishes with order $1$ zero exactly at ramification points because $dx_p|_{\Sigma\cap U_p}$ does. Suppose that $\Sigma\cap U_p$ contains a ramification point $r_\alpha$, then $\frac{\partial y_p}{\partial u^k}|_{x_p = const}$ has a simple pole at $r_\alpha$. However, $\frac{\partial y_p}{\partial u^k}|_{x_p = const}dx_p|_{\Sigma\cap U_p}$ remains finite at $r_\alpha$ because the simple pole cancels out with the simple zero of $dx_p|_{\Sigma\cap U_p}$. From this analysis, we conclude that $\phi_\Sigma(v)$ is holomorphic on $\Sigma$, i.e. $\phi_\Sigma(v) \in \Gamma(\Sigma, \Omega^1_\Sigma)$, as expected since the definition of $\phi_\Sigma$ is independent of our choice of the connection (see Lemma \ref{normallocalsectionslemma}). We refer to the proof of Lemma \ref{phiongtophionhandwlemma} for a version of this argument repeated using the standard local coordinate $z_\alpha$.
\end{remark}

Let us compute the exterior covariant derivative of $\phi_\Sigma$ using $\nabla_{\mathcal{F}}$, we have
\begin{align*}
    d_{\nabla_{\mathcal{F}}}\phi_\Sigma|_{\Sigma\cap U_p} &= d_{\nabla_{\mathcal{F}}}\left(\sum_{k=1}^g\Omega_S(n_{v_k},.)|_\Sigma du^k\right)\Big|_{\Sigma \cap U_p} = d_{\nabla_\mathcal{F}}\left(-\nabla_\mathcal{F}(y_pdx_p|_{\Sigma\cap U_p})\right)\\
    &= \sum_{k,l=1}^g\left(\frac{\partial^2y_p}{\partial u^k\partial u^l}\Big|_{x_p=const}dx_p|_{\Sigma\cap U_p}\right)du^k\wedge du^l = 0
\end{align*}
for any $(\mathcal{F}, \Omega_S)$-chart $(U_p, x_p, y_p) \in \mathcal{U}$. This means $\nabla_{\mathcal{F}}\phi_\Sigma = 0$ and therefore,
\begin{equation*}
    d_{\nabla_{GM}}\pmb{\phi}_\Sigma = d_{\nabla_{GM}}\left[\phi_\Sigma\right] = \left[d_{\nabla_{\mathcal{F}}}\phi_\Sigma\right] = 0,
\end{equation*}
where we have applied Lemma \ref{integralandderivativelemma} at the second equality. Since $\mathcal{B}_{\Sigma_0}$ is contractible, by the Poincar\'{e} Lemma (see Lemma \ref{flatvectorbundlepoincarelemma} below), $d_{\nabla_{GM}}\pmb{\phi} = 0$ on $\mathcal{B}_{\Sigma_0}$ implies that there exists
\begin{equation}\label{definingthetainh}
\text{$\pmb{\theta} \in \Gamma(\mathcal{B}_{\Sigma_0}, \mathcal{H})$ such that $\nabla_{GM}\pmb{\theta} = \pmb{\phi}$}.
\end{equation}

\subsubsection{The embedding of $\mathcal{B}_{\Sigma_0}$ into $\mathcal{H}_{\Sigma_0}$ as a Lagrangian submanifold}\label{embeddingofbinhsubsection}
Let us consider the map $\pmb{\Phi}_{\Sigma_0}$ embedding the neighbourhood $\mathcal{B}_{\Sigma_0} \subseteq \mathcal{B}$ into a fiber $\mathcal{H}_{\Sigma_0}$ at $[\Sigma_0] \in \mathcal{B}_{\Sigma_0}$. Following \cite{kontsevich2017airy}, we first define the affine connection $\nabla_{GM} + \pmb{\phi}$ on $\mathcal{H}$:
\begin{equation*}
    \left(\nabla_{GM} + \pmb{\phi}\right)[\xi] := \nabla_{GM}[\xi] + \pmb{\phi}
\end{equation*}
for any section $[\xi] \in \Gamma(\mathcal{B}_{\Sigma_0}, \mathcal{H})$.
The map $\pmb{\Phi}_{\Sigma_0} : \mathcal{B}_{\Sigma_0}\rightarrow \mathcal{H}_{\Sigma_0}$ is defined by sending a point $[\Sigma] \in \mathcal{B}_{\Sigma_0}$ to the image of the zero vector $0_\Sigma \in \mathcal{H}_\Sigma$ under the parallel transport from point $[\Sigma_0]$ to $[\Sigma]$ using the affine connection $\nabla_{GM} +  \pmb{\phi}$. In general, given $[\Sigma], [\Sigma_1] \in \mathcal{B}_{\Sigma_0}$ the parallel transport of $0_{\Sigma_1} \in \mathcal{H}_{\Sigma_1}$ to $\mathcal{H}_{\Sigma}$ is unique, path-independent and it is given by
\begin{equation*}
    v_{\Sigma_1}(\Sigma) := -\pmb{\theta}_{\Sigma} + \Gamma_{[\Sigma_1]}^{[\Sigma]}\pmb{\theta}_{\Sigma_1}.
\end{equation*}
It follows that
\begin{equation*}
    \pmb{\Phi}_{\Sigma_0}(\Sigma) := v_{[\Sigma]}(\Sigma_0) = -\pmb{\theta}_{\Sigma_0} + \Gamma_{[\Sigma]}^{[\Sigma_0]}\pmb{\theta}_{\Sigma}.
\end{equation*}
We denote the image of $\pmb{\Phi}_{\Sigma_0}$ by
\begin{equation*}
    \mathcal{L}_{\Sigma_0} := \text{im}\pmb{\Phi}_{\Sigma_0} \subset \mathcal{H}_{\Sigma_0}.
\end{equation*}
Let $v = \sum_{k=1}^gv^k\frac{\partial}{\partial u^k} \in T_{[\Sigma]}\mathcal{B}_{\Sigma_0}$ be any tangent vector, then 
\begin{equation*}
    \iota_{v}d\pmb{\Phi}_{\Sigma_0}(\Sigma) = \sum_{k=1}^gv^k\frac{\partial \pmb{\Phi}_{\Sigma_0}(\Sigma)}{\partial u^k} = \Gamma_{[\Sigma]}^{[\Sigma_0]}\iota_{v}\nabla_{GM}\pmb{\theta}_{\Sigma} = \Gamma_{[\Sigma]}^{[\Sigma_0]}\iota_{v}\pmb{\phi}_{\Sigma}
\end{equation*}
is a tangent vector to $\mathcal{L}_{\Sigma_0} \subset \mathcal{H}_{\Sigma_0}$. Since $v \mapsto \iota_vd\pmb{\Phi}_{\Sigma_0}$ is injective because $\pmb{\phi}_{\Sigma}$ is injective, by assuming that $\mathcal{B}_{\Sigma_0}$ is small enough we have that $\pmb{\Phi}_{\Sigma_0} : \mathcal{B}_{\Sigma_0} \rightarrow \mathcal{H}_{\Sigma_0}$ embeds $\mathcal{B}_{\Sigma_0}$ as a submanifold in $\mathcal{H}_{\Sigma_0}$.
\begin{lemma}\theoremname{\cite[Section 5.1]{kontsevich2017airy}}\label{lsigma0islagrangianlemma}
$\mathcal{L}_{\Sigma_0}$ is a Lagrangian submanifold of $\mathcal{H}_{\Sigma_0}$.
\end{lemma}
\begin{proof}
Since $\mathcal{H}_{\Sigma_0} = H^1(\Sigma_0, \mathbb{C}) \cong \mathbb{C}^{2g}$ and $\dim \mathcal{L}_{\Sigma_0} = \dim\mathcal{B}_{\Sigma_0} = \frac{1}{2}\dim\mathcal{H}_{\Sigma_0}$ it remains for us to check that $\Omega_{\mathcal{H}}|_{\mathcal{L}_{\Sigma_0}} = 0$. For $v_i \in T_{[\Sigma]}\mathcal{B}_{\Sigma_0}, i = 1,2$ we have $\iota_{v_i}d\pmb{\Phi}_{\Sigma_0} \in T_{\pmb{\Phi}_{\Sigma_0}(\Sigma)}\mathcal{L}_{\Sigma_0}$ and
\begin{equation*}
    \Omega_{\mathcal{H}}(\iota_{v_1}d\pmb{\Phi}_{\Sigma_0}, \iota_{v_2}d\pmb{\Phi}_{\Sigma_0}) = \Omega_{\mathcal{H}}(\Gamma_{[\Sigma]}^{[\Sigma_0]}\iota_{v_1}\pmb{\phi}_{\Sigma}, \Gamma_{[\Sigma]}^{[\Sigma_0]}\iota_{v_2}\pmb{\phi}_{\Sigma}) = \Omega_{\mathcal{H}}(\iota_{v_1}\pmb{\phi}_{\Sigma},\iota_{v_2}\pmb{\phi}_{\Sigma}) = 0
\end{equation*}
because the image $\text{im} \pmb{\phi}_{\Sigma} = \Gamma(\Sigma, \Omega^1_{\Sigma})$ is a Lagrangian subspace of $\mathcal{H}_{\Sigma}$. 
\end{proof}

Let us consider a trivialization of $\mathcal{H}\rightarrow \mathcal{B}_{\Sigma_0}$ by choosing $A, B$-symplectic basis cycles of $H_1(\Sigma,\mathbb{C})$ for each $[\Sigma] \in \mathcal{B}_{\Sigma_0}$ and identify $\mathcal{H}_{\Sigma} \cong \mathbb{C}^{2g}$ as in Remark \ref{htrivializationtremark}: $[\xi] \mapsto \left(\oint_{A_1}\xi,\cdots, \oint_{A_g}\xi,\oint_{B_1}\xi,\cdots, \oint_{B_g}\xi\right) \in \mathbb{C}^{2g}$.
Let us define
\begin{equation*}
    \mathbf{a}^i(\Sigma) := \oint_{A_i}\pmb{\theta}_{\Sigma},\qquad \mathbf{b}_i(\Sigma) := \oint_{B_i}\pmb{\theta}_{\Sigma},\qquad i = 1,...,g.
\end{equation*}
We will often write $\mathbf{a}^i(\Sigma)$ and $\mathbf{b}_i(\Sigma)$ simply as $\mathbf{a}^i$ and $\mathbf{b}_i$ when they do not cause any ambiguity. 
\begin{lemma}\label{aiscoordforblemma}
$\{\mathbf{a}^1,...,\mathbf{a}^g\}$ is a coordinate system for $\mathcal{B}_{\Sigma_0}$.
\end{lemma}
\begin{proof}
The matrix $\left(\frac{\partial \mathbf{a}^i}{\partial u^j}\right)$ is invertible because $\frac{\partial \mathbf{a}^i}{\partial u^j} = \frac{\partial}{\partial u^j}\oint_{A_i}\pmb{\theta}_{\Sigma} = \oint_{A_i}\iota_{\frac{\partial}{\partial u^j}}\nabla_{GM}\pmb{\theta}_{\Sigma} = \oint_{A_i}\iota_{\frac{\partial}{\partial u^j}}\pmb{\phi}_\Sigma$. So if there exists $(v^i) \in \mathbb{C}^g$ such that $\sum_{j=1}^gv^j\frac{\partial \mathbf{a}^i}{\partial u^j} = 0, i = 1,...,g$ then $\phi_\Sigma(v) = \sum_{j=1}^gv^j\iota_{\frac{\partial}{\partial u^j}}\phi_{\Sigma} \in \Gamma(\Sigma, \Omega^1_\Sigma), v := \sum_{i=1}^gv^i\frac{\partial}{\partial u^j}$ is a holomorphic form on $\Sigma$ with zero $A$-periods, therefore $\phi_\Sigma(v) = 0$. But we know that $v \mapsto \iota_v\phi_\Sigma$ is injective, therefore $v = 0$. Hence $\mathbf{a}^i := \mathbf{a}^i(u^1,...,u^g), i = 1,...,g$ is a holomorphic coordinates transformation.
\end{proof}

So $\mathbf{b}_i$ can be written as a function of $\{\mathbf{a}^i\}$. We can write $\pmb{\Phi}_{\Sigma_0} : \mathcal{B}_{\Sigma_0} \rightarrow \mathcal{H}_{\Sigma_0}$ using this trivialization as
\begin{equation*}
    [\Sigma] := (\mathbf{a}^1(\Sigma),...,\mathbf{a}^g(\Sigma)) \mapsto \pmb{\Phi}(\Sigma) = (\mathbf{a}^1(\Sigma) - \mathbf{a}^1(\Sigma_0),...,\mathbf{a}^g(\Sigma) - \mathbf{a}^g(\Sigma_0),\mathbf{b}^1(\Sigma) - \mathbf{b}^1(\Sigma_0),...,\mathbf{b}^g(\Sigma) - \mathbf{b}^g(\Sigma_0))
\end{equation*}
and 
\begin{equation*}
    \frac{\partial\pmb{\Phi}_{\Sigma_0}(\Sigma)}{\partial \mathbf{a}^i} = \left(0,...,1,...,0,\frac{\partial \mathbf{b}_1}{\partial \mathbf{a}^i},...,\frac{\partial \mathbf{b}_g}{\partial \mathbf{a}^i}\right) \in T_{\pmb{\Phi}_{\Sigma_0}(\Sigma)}\text{im}\pmb{\Phi}_{\Sigma_0}, \qquad i = 1,...,g.
\end{equation*}
It follows from Lemma \ref{lsigma0islagrangianlemma} that for any $i,j = 1,...,g$ we have
\begin{equation*}
    0 = \Omega_{\mathcal{H}}\left(\frac{\partial\pmb{\Phi}_{\Sigma_0}(\Sigma)}{\partial \mathbf{a}^i}, \frac{\partial\pmb{\Phi}_{\Sigma_0}(\Sigma)}{\partial \mathbf{a}^j}\right) = \frac{\partial \mathbf{b}_i}{\partial \mathbf{a}^j} - \frac{\partial \mathbf{b}_j}{\partial \mathbf{a}^i}.
\end{equation*}
Therefore, by an application of Poincar\'{e} lemma, there exists a holomorphic function $\mathfrak{F}_{\Sigma_0} = \mathfrak{F}_{\Sigma_0}(\mathbf{a}^1,...,\mathbf{a}^g)$ on $\mathcal{B}_{\Sigma_0}$ such that
\begin{equation}\label{abperiodsandprepotential}
    \mathbf{b}_i = \frac{\partial \mathfrak{F}_{\Sigma_0}(\mathbf{a}^1,...,\mathbf{a}^g)}{\partial \mathbf{a}^i},\qquad i = 1,...,g.
\end{equation}
The section $\pmb{\theta} \in \Gamma(\mathcal{B}_{\Sigma_0}, \mathcal{H})$ is uniquely determined up to an addition of a parallel section $\mathbf{c} \in \Gamma(\mathcal{B}_{\Sigma_0},\mathcal{H}), \nabla_{GM}\mathbf{c} = 0$. Therefore, given the moduli space $\mathcal{B}$ of $\mathcal{F}$-transversal curves, the function $\mathfrak{F}_{\Sigma_0}$ is unique up to an addition of a linear term. 
\begin{definition}\label{prepotentialdefinition}
Given the moduli space $\mathcal{B}$ of $\mathcal{F}$-transversal curves inside a foliated symplectic surface $(S,\Omega_S,\mathcal{F})$, the \emph{prepotential} $\mathfrak{F}_{\Sigma_0} = \mathfrak{F}_{\Sigma_0}(\mathbf{a}^1,\cdots,\mathbf{a}^g)$ is a holomorphic function locally defined on the open neighbourhood $\mathcal{B}_{\Sigma_0} \ni [\Sigma_0]$ which relates $A$ and $B$-periods of $\pmb{\theta}_{\Sigma} \in \mathcal{H}_{\Sigma}$ according to (\ref{abperiodsandprepotential}). 
\end{definition}
The prepotential is the generating function of the Lagrangian submanifold $\mathcal{L}_{\Sigma_0} := \text{im}\pmb{\Phi}_{\Sigma_0} \subset \mathcal{H}_{\Sigma_0}$.
In other words, the prepotential $\mathfrak{F}_{\Sigma_0}$ naturally arises from the study of the deformation of curves $\Sigma$ embedded in a foliated symplectic surface $(S, \Omega_S, \mathcal{F})$ around some neighbourhood $\mathcal{B}_{\Sigma_0}$ of $[\Sigma_0]$ in the moduli space $\mathcal{B}$. 

\subsubsection{Generalization}
Previously, we considered a particular vector bundle $\mathcal{H}\rightarrow \mathcal{B}$ and an embedding of a neighbourhood $\mathcal{B}_{\Sigma_0}$ of $[\Sigma_0]\in \mathcal{B}$ into a fiber $\mathcal{H}_{\Sigma_0}$ using the affine connection $\nabla_{GM} + \pmb{\phi}$. We showed that the image is a Lagrangian submanifold. This type of embedding can be done with a more general class of symplectic vector bundles. Let us end this section by giving a discussion of the general setting. First, let us state and proof a version of Poincar\'{e} Lemma for holomorphic flat vector bundles:

\begin{lemma}\label{flatvectorbundlepoincarelemma}
Let $(\mathcal{E}\rightarrow X, \nabla)$ be a $n$-dimensional holomorphic vector bundle over a contractible $m$ dimensional complex manifold $X$ with a flat connection $\nabla$ and let $d_{\nabla} : \Omega^r_X\otimes \mathcal{E} \rightarrow \Omega^{r+1}_X\otimes \mathcal{E}$ be the corresponding exterior covariant derivative. Then given a section $\xi \in \Gamma(X,\Omega^r_X\otimes \mathcal{E})$ such that $d_{\nabla}\xi = 0$, there exists a section $\eta \in \Gamma(X,\Omega^{r-1}_X\otimes \mathcal{E})$ such that $\xi = d_{\nabla}\eta$. 
\end{lemma}
\begin{proof}
Let us pick a basis $\{e_{i,x_0}\}$ of $\mathcal{E}_{x_0}$ at some point $x_0 \in X$ and let the section $e_i \in \Gamma(X,\mathcal{E})$ be given at any $x \in X$ by $e_{i,x} := \Gamma_{x_0}^xe_{i,x_0}$, where $\Gamma_{x_0}^x$ is the parallel transport on $\mathcal{E}$ from $x_0$ to $x$ using the flat connection $\nabla$. Since $\mathcal{E}\rightarrow X$ is flat, the parallel transport with respect to the connection $\nabla$ is path-independent. On the other hand, the existence and uniqueness of the parallel transport of $e_{i,x_0}$ from $x_0$ to $x$ along any path follows from Picard–Lindel\"{o}f theorem. It follows that $e_{i,x} := \Gamma_{x_0}^xe_{i,x_0} \in \mathcal{E}_x$ is well-defined for $x \in X$.
Therefore, we have parallel sections $\nabla e_i \in \Gamma(X,\mathcal{E})$, $\nabla e_i = 0$, and $\{e_{i,x}\}$ is a basis of the fiber $\mathcal{E}_x$ for all $x\in X$. A section $\xi\in \Gamma(X, \Omega^r_X\otimes \mathcal{E})$ can be expressed using the local coordinates $\{x^1,\cdots, x^m\}$ of $X$ as $\xi :=  \sum_{k=1}^n\xi_ke_k$ where $\xi_k := \sum_{i_1,\cdots,i_r = 1}^m\xi_{i_1\cdots i_r;k}(x)dx^{i_1}\wedge \cdots \wedge dx^{i_r} \in \Gamma(X,\Omega^r_X)$. Using $\nabla e_k = 0$, we have that
\begin{equation*}
    d_\nabla \xi = \sum_{k=1}^n\left((d\xi_k)e_k + (-1)^r\xi_k\wedge\nabla e_k\right) = \sum_{k=1}^n(d\xi_k)e_k.
\end{equation*}
Consequently, if $d_\nabla \xi = 0$ then $d\xi_k = 0$ for all $k = 1,\cdots,n$. Since $X$ is contractible, we can apply the usual Poincar\'{e} Lemma for differential forms to $\xi_k$ and conclude that there exists $\eta_k \in \Gamma(X, \Omega^{r-1}_X)$ such that $d\eta_k = \xi_k$. It is now a matter of simple calculation to verify that $d_\nabla \eta = \xi$ where $\eta = \sum_{k=1}^n\eta_ke_k$.
\end{proof}

Let us now present the general version of the embedding $\mathcal{B}_{\Sigma_0} \rightarrow \mathcal{H}_{\Sigma_0}$:

\begin{proposition}\label{generalembeddingtofiberproposition}
Let $(\mathcal{E}\rightarrow X, \Omega_{\mathcal{E}}, \nabla, \phi)$ be a $2n$-dimensional holomorphic symplectic vector bundle over a $n$-dimensional complex manifold $X$ with a flat connection $\nabla$, a $\nabla$-covariantly constant symplectic form $\Omega_{\mathcal{E}}$ and $\phi \in \Gamma(X, \Omega^1_X\otimes \mathcal{E})$, $d_\nabla \phi = 0$ such that for each $x \in X$, $\phi_x : T_xX\rightarrow \mathcal{E}_x$ embeds $T_xX$ as a Lagrangian subspace of $\mathcal{E}_x$. Fix a reference point $x_0 \in X$. There exists an open neighbourhood $X_{x_0} \subseteq X$ of $x_0$ such that the map $\Phi_{x_0} : X_{x_0}\rightarrow \mathcal{E}_{x_0}$, given by sending $x \in X_{x_0}$ to the image of $0_x \in \mathcal{E}_x$ under the parallel transport from $x$ to $x_0$ using the affine connection $\nabla + \phi$, is a well-defined embedding of $X_{x_0}$ into a Lagrangian submanifold $\mathcal{L}_{x_0} := \text{im}\Phi_{x_0} \subseteq \mathcal{E}_{x_0}$.
\end{proposition}
\begin{proof}
Let $X_{x_0} \subseteq$ be a contractible open neighbourhood of $x_0 \in X$. Since $d_\nabla\phi = 0$, by Lemma \ref{flatvectorbundlepoincarelemma}, there exists a section $\theta \in \Gamma(X_{x_0}, \mathcal{E})$ such that $d_\nabla\theta = \nabla \theta = \phi$. For $x_1\in X_{x_0}$, the section $v_{x_1} \in \Gamma(X_{x_0}, \mathcal{E}), v_{x_1}(x) := -\theta_x + \Gamma_{x_1}^x\theta_{x_1}$ satisfies $v_{x_1}(x_1) = 0$ and $\nabla v_{x_1}(x) = \phi_x$. For all $x_1, x \in X_{x_0}$, the parallel transport of $0_{x_1} \in \mathcal{E}_{x_1}$ to $\mathcal{E}_x$ using the affine connection $\nabla + \phi$ is path-independent and it is uniquely given by $v_{x_1}(x)$. Therefore, the map $\Phi_{x_0} : X_{x_0} \rightarrow \mathcal{E}_{x_0}$ is given by
\begin{equation*}
    \Phi_{x_0}(x) := v_{x}(x_0) = -\theta_{x_0} + \Gamma_{x}^{x_0}\theta_x.
\end{equation*}
Let us compute the derivatives of $\Phi_{x_0}$. Suppose that we can cover $X_{x_0}$ with a single coordinates chart $(X_{x_0}, x^1,...,x^n)$ then 
\begin{align*}
    d\Phi_{x_0}(x) &= \sum_{i=1}^ndx^i\frac{\partial \Phi_{x_0}(x)}{\partial x^i} = \sum_{i=1}^ndx^i\frac{\partial}{\partial x^i}\Gamma_x^{x_0}\left(\sum_{k=1}^{2n}\theta_{k}(x)e_{k,x}\right) = \sum_{i=1}^ndx^i\frac{\partial}{\partial x^i}\left(\sum_{k=1}^{2n}\theta_{k}(x)e_{k,x_0}\right)\\
    &= \sum_{i=1}^ndx^i\sum_{k=1}^{2n}\frac{\partial \theta_k(x)}{\partial x^i}e_{k,x_0} = \sum_{i=1}^ndx^i\Gamma_{x}^{x_0}\left(\sum_{k=1}^{2n}\frac{\partial \theta_k(x)}{\partial x^i}e_{k,x}\right) = \Gamma_{x}^{x_0}\nabla\left(\sum_{k=1}^{2n}\theta_{k}(x)e_{k,x}\right) = \Gamma_x^{x_0}\phi_x.
\end{align*}
For every $v \in T_xX_{x_0}$, we have $\iota_vd\Phi_{x_0}(x) \in T_{\Phi_{x_0}(x)}\mathcal{L}_{x_0}$. Furthermore, given $v_1,v_2 \in T_xX_{x_0}$ we have
\begin{equation*}
    \Omega_{\mathcal{E}}(\iota_{v_1}d\Phi_{x_0}(x), \iota_{v_2}d\Phi_{x_0}(x)) = \Omega_{\mathcal{E}}(\Gamma_x^{x_0}\iota_{v_1}\phi_x, \Gamma_x^{x_0}\iota_{v_2}\phi_x) = \Omega_{\mathcal{E}}(\iota_{v_1}\phi_x, \iota_{v_2}\phi_x) = 0
\end{equation*}
by the assumption that $\phi_x$ embeds $T_xX$ as a Lagrangian subspace of $\mathcal{E}_x$. Since $\phi_x$ is injective, the rank of $d\Phi_{x_0}(x)$ is $n$ for all $x \in X_{x_0}$. It follows that $\dim \mathcal{L}_{x_0} = \frac{1}{2}\dim \mathcal{E}_{x_0}$, hence $\mathcal{L}_{x_0} \subseteq \mathcal{E}_{x_0}$ is a Lagrangian submanifold. 
\end{proof} 

The proof of Proposition \ref{generalembeddingtofiberproposition} only works when $\mathcal{E}\rightarrow X$ is finite-dimensional, hence the result need not hold when the dimension is infinite. Nevertheless, Proposition \ref{generalembeddingtofiberproposition} gives us a guideline of what to expect when we consider the analogous embedding in the infinite-dimensional case. We will see this in Section \ref{theembeddingofdiscssection} where we consider the infinite-dimensional vector bundle $W^{\epsilon, M}\rightarrow Discs^M$.

\subsection{Airy Structures}\label{airystructuressection}

In this section, we review some basic terminologies and facts about Airy Structures as introduced in \cite{kontsevich2017airy}
.
\subsubsection{Topological vector spaces and symplectic structures}
\begin{definition}
A $k$-\emph{topological vector space} $V$ is a vector space equipped with a Hausdorff topology over a topological field $k$ such that 
\begin{equation}
    \text{vector addition : } V\times V \rightarrow V, \qquad \text{scalar multiplication : } k\times V \rightarrow V
\end{equation}
are continuous.
\end{definition}
We need a topology on $V$ to define the dual space $V^* := Hom(V,k)$ of $V$ as the vector space of continuous linear functionals on $V$. Typically, we equip $V^*$ with the weakest topology such that the map $f \in V^* \mapsto \phi_v(f) := f(v) \in k$ is continuous for all $v \in V$, making $V^*$ also a topological vector space. Therefore, we have $\phi_v \in V^{**} := Hom(V^*, k)$ by definition, where $V^{**}$ is the vector space of continuous linear functionals on $V^*$. We will not assume vector spaces in this section to be finite-dimensional. To help with the notation when listing the basis or coordinates we introduce an index set $\mathbb{I}$. For example, we write the basis of $V$ as $\{e_{i\in \mathbb{I}}\}$ and take $\mathbb{I} = \{1,...,n\}$ if $V$ is finite $n$-dimensional or $\mathbb{I} = \mathbb{Z}_{>0}$ if $V$ is countably infinite-dimensional.

\begin{remark}
The topology of $V$ will only be important to us when $n$ is infinite because when $n$ is finite, linear functionals on $V$ are automatically continuous. 
\end{remark}

\begin{definition}
A symplectic topological vector space $(W,\Omega)$ is a pair consisting of a topological vector space $W$ and a continuous antisymmetric bilinear form $\Omega : W\times W \rightarrow k$ such that $w \mapsto \Omega(w,.)$ is an isomorphism $W \xrightarrow{\cong} W^*$. We call $\Omega$ the \emph{symplectic form} of $W$.
\end{definition}

Let us introduce the example of symplectic topological vector space $(W,\Omega)$ which will be important in defining Airy structures. We will exclusively work with $k := \mathbb{C}$ equipped with a discrete topology. Let $V$ be a topological vector space with a discrete topology and a basis $\{e^{i\in \mathbb{I}}\}$ over $\mathbb{C}$. Any vector $v 
\in V$ can be written as $v = \sum_{i \in \mathbb{I}}y_ie^i$ where only a finite number of $y_i \in \mathbb{C}$ are non-zero. Let $\{f_{i\in \mathbb{I}}\} \subseteq V^*$ be the set of linearly independent vectors such that $f_i(e^j) = \delta_i^j$. Then elements of $V^*$ are $f = \sum_{i\in \mathbb{I}}x^if_i$ for $x^i \in \mathbb{C}$. We call $\{y_{i\in \mathbb{I}}\}$ and $\{x^{i\in \mathbb{I}}\}$ coordinates of $V$ and $V^*$ respectively. Equip $V^*$ with the weakest topology such that the map $f\mapsto \phi_v(f) := f(v) \in \mathbb{C}$ is continuous for all $v \in V$. The open sets in $V^*$ take the form $f + U_{\mathbb{J}} \subseteq V^*$ where $f \in V^*, \mathbb{J} \subseteq \mathbb{I}$ is a finite set and.
\begin{equation}\label{ujsubspace}
    U_{\mathbb{J}} := \left\{\sum_{i\in \mathbb{I}\setminus\mathbb{J}}x^if_i\ |\ x^i \in \mathbb{C}\right\} \subseteq V^*.
\end{equation}

\begin{remark}
If $V$ is finite dimensional then $\{f_{i\in \mathbb{I}}\}$ is a basis of $V^*$. When the dimension is infinite, $\{f_{i\in \mathbb{I}}\}$ is a \emph{Schauder basis} of $V^*$ not an ordinary basis because their finite linear combination do not span $V^*$. However, let us also refer to $\{f_{i\in\mathbb{I}}\}$ as a basis of $V^*$ to simplify the language. 
\end{remark}

Finally, we define $W := V\oplus V^*$ with standard symplectic form $\Omega := \sum_{i\in \mathbb{I}} dy_i\wedge dx^i$ where $dy_i := f_i, dx^i := \phi_{e^i}$, 
\begin{equation}\label{efcanonicalbasis}
    \Omega(e^i,e^j) = \Omega(f_i,f_j) = 0, \qquad \Omega(e^i,f_j) = \delta^i_j.
\end{equation}
$\{e^{i\in\mathbb{I}}, f_{i\in \mathbb{I}}\}$ is a basis of $W$. We call $\{e^{i\in\mathbb{I}}, f_{i\in \mathbb{I}}\}$, or any basis of $W$ satisfying (\ref{efcanonicalbasis}), a \emph{canonical basis} of $W$. To show that $(W,\Omega)$ is a symplectic topological vector space we will need the following:

\begin{lemma}
Let $V$ be a discrete topological vector space with basis $\{e^{k\in \mathbb{I}}\}$ over a discrete topological field $\mathbb{C}$, then $V^{**} \cong V$.
\end{lemma}
\begin{proof}
If $V$ is finite-dimensional, this is trivial. Suppose that $V$ is infinite-dimensional and let $\phi \in V^{**} = Hom(V^*, \mathbb{C})$ be a continuous linear functional. For any sequence $s : \mathbb{Z}_{>0} \rightarrow V^*, n\mapsto s_n$ converges to $f \in V^*$, we have $\phi(f) = \phi\left(\lim_{n\rightarrow \infty} s_n\right) = \lim_{n\rightarrow \infty}\phi(s_n)$ by the continuity. Let us embed $\mathbb{Z}_{>0}$ into the index set $\mathbb{I}$. Suppose that $f := \sum_{k \in \mathbb{Z}_{>0} \subseteq \mathbb{I}}x^kf_k$, where $f_i(e^j) = \delta_i^j$ and consider the sequence $s_n := \sum_{k=1}^{n}x^kf_k + \Lambda_{n+1} f_{n+1}$ converges to $f$ in the topology of $V^*$ for any arbitrary sequence $\{\Lambda_n\}$.

By the continuity of $\phi$ we have 
\begin{equation*}
    \phi(f) = \phi(\lim_{n\rightarrow \infty} s_n) = \lim_{n\rightarrow \infty} \phi(s_n) = \lim_{n\rightarrow \infty}\phi\left(\sum_{k=1}^nx^kf_k + \Lambda_{n+1} f_{n+1}\right) = \lim_{n\rightarrow \infty}\left(\sum_{k=1}^nx^k\phi(f_k)+\Lambda_{n+1}\phi(f_{n+1})\right).
\end{equation*}
But the sequence $\{\Lambda_n\}$ is arbitrary, therefore, the only way for the above to be true is if $\phi(f_n)$ is only non-zero for finitely many $n \in \mathbb{Z}_{>0}\subseteq \mathbb{I}$. Since the embedding of $\mathbb{Z}_{>0}$ into $\mathbb{I}$ is also arbitrary, we conclude that $\phi(f_k)$ is non-zero for only finite number of $k \in \mathbb{I}$.

The topology on $V^{**}$ is given by the weakest topology such that the map $\phi \mapsto \phi(f)$ is continuous for all $f \in V^*$, which is not hard to check to be a discrete topology. The isomorphism $V^{**}\cong V$ is given by $e^i \leftrightarrow \phi_{e^i}, i \in \mathbb{I}$ where $\phi_{e^i}(f_j) := f_j(e^i) = \delta_j^i$, proving the lemma.
\end{proof}

Let us finish showing that $(W,\Omega)$ is a symplectic topological vector space. We have that $e^i \mapsto \Omega(e^i,.) = \phi_{e^i}$ is an isomorphism $V\xrightarrow{\cong}V^{**}$ and $f_i \mapsto \Omega(f_i,.) = -f_i$ is an isomorphism $V^*\xrightarrow{\cong}V^*$. It follows that $(W,\Omega)$ is a symplectic topological vector space because $w \mapsto \Omega(w,.)$ gives an isomorphism $W = V\oplus V^* \xrightarrow{\cong} V^{**}\oplus V^* = (V^*\oplus V)^* = W^*$. The topological symplectic vector space $W$ is an example of a \emph{Tate space}.

\begin{definition}\theoremname{Tate space \cite{drinfeld2006infinite}}
A Tate space is a topological vector space over a discrete field $k$ which is isomorphic to $V_1\oplus V_2^*$ for some discrete topological $k$-vector spaces $V_1$ and $V_2$.
\end{definition}

\begin{remark}\label{xycoordsarevectorsremark}
The coordinate $x^i$ of $V^*$ can be thought of as a linear functional $x^i \in Hom(V^*, \mathbb{C}) = V^{**} \cong V$. Equivalently, $x^i$ is identified with the vector $e^i \in V$ via the isomorphism $e^i \mapsto \Omega(e^i,.)$. The reason why we are keeping both notations is that they carry different conceptual meanings. Similarly, the coordinate $y_i$ can be identified with the vector $-f_i \in V^*$. Together, $\{x^{i\in\mathbb{I}}, y_{i\in\mathbb{I}}\}$ is a coordinate system of $W$.
\end{remark}

Given two vector spaces over a field $k$, a \emph{tensor product} is a frequently used operation to construct a new vector space. Formally, the tensor product $V_1\otimes V_2$ of vector spaces $V_1$ and $V_2$ is uniquely defined up to isomorphism as the vector space satisfying a certain universal property. More concretely, suppose that $V_1$ and $V_2$ are potentially infinite-dimensional vector spaces with basis $\{e^{i \in \mathbb{I}_1}_1\}$ and $\{e^{i \in \mathbb{I}_2}_2\}$ respectively, then $V_1\otimes V_2$ is a vector space with a basis $\{e^i_1\otimes e^j_2\ |\ (i,j) \in \mathbb{I}_1\times \mathbb{I}_2\}$. In other words, 
\begin{equation}\label{tensorprodofvectorspacesdefinition}
V_1\otimes V_2 = \left\{\sum_{(i,j) \in \mathbb{I}_1\times \mathbb{I}_2} v_{ij}e^i_1\otimes e^j_2\ |\ v_{ij}\in k \text{ is non-zero for finitely many pairs of } (i,j)\in \mathbb{I}_1\times \mathbb{I}_2\right\}.
\end{equation}

\begin{remark}
The important message of (\ref{tensorprodofvectorspacesdefinition}) is that only finite linear combinations of tensor products of vectors are allowed in $V_1\otimes V_2$ even though $V_1$ or $V_2$ are possibly infinite-dimensional. 
\end{remark}

$V_1\otimes V_2$ is different from \emph{topological complete tensor product} $V_1\hat{\otimes}V_2$ mentioned in \cite{kontsevich2017airy}. Let us briefly explain the meaning of $V_1\hat{\otimes}V_2$. Given a topological vector space $V$, the \emph{completion} of $V$ is the unique complete topological vector space $\hat{V}$ such that $V$ can be embedded densely into $\hat{V}$ \cite[Theorem 5.2]{treves2016topological}. Let $V_1, V_2$ be topological vector spaces. We assign the topology to $V_1\otimes V_2$ such that it is the unique object satisfying the universal properties of the tensor product in the category of $k$-topological vector spaces \cite{glockner2004tensor}. For instance, when $V_1$ and $V_2$ are locally convex, then this topology on $V_1\otimes V_2$ is the \emph{projective topology} \cite[Definition 43.2]{treves2016topological}.
\begin{definition}
The \emph{topological complete tensor product} $V_1\hat{\otimes} V_2$ is the topological completion of the tensor product $V_1\otimes V_2$.
\end{definition}

Let us give an explicit presentation of $W\hat{\otimes}W$ as an example, where $W = V\oplus V^*$ is a Tate space. Assign a topology to $W\otimes W$ by declaring open sets around $w\in W$ to be $w + U_{\mathbb{J}_1\mathbb{J}_2\mathbb{J}_3} \subseteq W\otimes W$ where $\mathbb{J}_{k=1,2} := (\mathbb{J}_k^{i\in\mathbb{I}})$, $\mathbb{J}^i_1, \mathbb{J}^i_2, \mathbb{J}_3 \subseteq \mathbb{I}$ are finite sets and
\begin{equation*}
    U_{\mathbb{J}_1\mathbb{J}_2\mathbb{J}_3} := \sum_{i \in \mathbb{I}}\left(U_{\mathbb{J}^i_1}\otimes \mathbb{C}e^i + \mathbb{C}e^i\otimes U_{\mathbb{J}^i_2}\right) + \sum_{\mathbb{J}_3\subseteq \mathbb{I}_1, \mathbb{I}_2 \subseteq \mathbb{I}} U_{\mathbb{I}_1}\otimes U_{\mathbb{I}_2} \subseteq W\otimes W
\end{equation*}
where $U_{\mathbb{J}} \subseteq V^*$ are vector subspaces given by (\ref{ujsubspace}). Since open neighborhoods of zero $U_{\mathbb{J}_1\mathbb{J}_2\mathbb{J}_3}$ are subspaces of $W\otimes W$ in this case, the completion can be computed using the inverse limit:
\begin{align*}
    W\hat{\otimes} W &:= \varprojlim W\otimes W/U_{\mathbb{J}_1\mathbb{J}_2\mathbb{J}_3}\\
    &= \left\{\sum_{i,j\in\mathbb{I}}a_{ij}e^i\otimes e^j + \sum_{i,j\in \mathbb{I}}(b_{1i}^je^i\otimes f_j + b_{2j}^if_i\otimes e^j) + \sum_{i,j\in\mathbb{I}}c^{ij}f_i\otimes f_j\ |\ \begin{array}{c}
    a_{ij}, b_{1i}^j, b_{2i}^j, c^{ij} \in\mathbb{C}\\
    a_{ij} \neq 0 \text{ for finitely }\\
    \text{ many pairs of } (i,j) \in \mathbb{I}\times \mathbb{I}, \\
    \text{ and for each } i \in \mathbb{I}, b_{1j}^i, b_{2j}^i \neq 0\\
    \text{ for finitely many } j \in \mathbb{I}
    \end{array}\right\}.
\end{align*}
For $k=1,2$, let us write $\mathbb{J}'_k \subseteq \mathbb{J}_k$ if $\mathbb{J}'^i_k \subseteq \mathbb{J}^i_k$ for all $i \in \mathbb{I}$. The inverse limit is taken over the directed poset of tuples $(\mathbb{J}_1,\mathbb{J}_2,\mathbb{J}_3)$ with with partial order $(\mathbb{J}'_1,\mathbb{J}'_2,\mathbb{J}'_3) \leq (\mathbb{J}_1,\mathbb{J}_2,\mathbb{J}_3)$ if $\mathbb{J}'_i\subseteq \mathbb{J}_i, \forall i = 1,2,3$.

The $n$-th order topological complete tensor product $W^{\hat{\otimes} n} := W\hat{\otimes}\cdots \hat{\otimes} W$ can be defined in a similar way. We denote by $\widehat{Sym}_{n}(W)$ the subspace of symmetric tensors in $W^{\hat{\otimes}n}$ and let $\widehat{Sym}_{\leq n}(W) := \bigoplus_{i=0}^n\widehat{Sym}_i(W)$.
Using Remark \ref{xycoordsarevectorsremark}, we can think of $\prod_{i=0}^\infty\widehat{Sym}_i(W)$ as the space of formal polynomials on $W$ and denote it by $\mathbb{C}[[W]] = \mathbb{C}[[x^{i\in \mathbb{I}}, y_{i\in \mathbb{I}}]]$. In particular, elements of $\mathbb{C}[[W]]$ are formal series of the form
\begin{align*}
    \sum_{n = 0}^\infty\sum_{\substack{n_1 + n_2 = n\\n_1,n_2 \geq 0}}\sum_{\substack{i_1,\cdots,i_{n_1}\in\mathbb{I}\\j_1,\cdots,j_{n_2}\in\mathbb{I}}}P_{i_1,\cdots,i_{n_1}}^{j_1,\cdots,j_{n_2}}x^{i_1}\cdots x^{i_{n_1}}y_{j_1}\cdots y_{j_{n_2}}, \qquad \begin{array}{c}
    \text{Where $P_{i_1,\cdots,i_{n_1}}^{j_1,\cdots,j_{n_2}} \in \mathbb{C}$ and }\\
    \text{for each $(j_1,\cdots,j_{n_2}) \in \mathbb{I}^{n_2}$, $P_{i_1,\cdots,i_{n_1}}^{j_1,\cdots,j_{n_2}}\neq 0$}\\
    \text{for finitely many $(i_1,\cdots,i_{n_1}) \in \mathbb{I}^{n_1}$.}
    \end{array}
\end{align*}
The symplectic form $\Omega$ defines a Poisson bracket on $\mathbb{C}[[W]]$ given by 
\begin{equation}\label{xycanonicalcoordinates}
    \{x^i, x^j\} = \{y_i, y_j\} = 0, \qquad \{y_i, x^j\} = \delta^j_i
\end{equation}
and extended to arbitrary formal polynomials by the Leibniz rule. We call coordinates $\{x^{i\in \mathbb{I}}, y_{i\in \mathbb{I}}\}$ satisfying (\ref{xycanonicalcoordinates}), or equivalently coordinates such that $\Omega$ takes the canonical form: $\Omega = \sum_{i\in \mathbb{I}}dy_i\wedge dx^i$, a \emph{canonical coordinates system} of $W$.

\subsubsection{Pre-Airy structures}
Let $V$ be a topological vector space with a discrete topology over the discrete topological field $\mathbb{C}$.
\begin{definition}\theoremname{\cite[Definition 2.6.1]{kontsevich2017airy}}
A \emph{classical pre-Airy structure} on $V$ is a tuple $(V,A,B,C)$ where 
\begin{equation*}
    A = (a_{ijk})_{i,j,k \in \mathbb{I}} \in V\otimes V \otimes V, \qquad B = (b^k_{ij})_{i,j,k \in \mathbb{I}} : V \rightarrow V\otimes V, \qquad C = (c_i^{jk})_{i,j,k \in \mathbb{I}} : V\otimes V \rightarrow V. 
\end{equation*}
A \emph{quantum pre-Airy structure} on $V$ is a tuple $(V,A,B,C,\epsilon)$ where $(V,A,B,C)$ is a classical pre-Airy structure and $\epsilon = (\epsilon_i) \in V$.
\end{definition}

\begin{remark}
Equivalently, a (quantum) pre-Airy structure is given by any $A = (a_{ijk})_{i,j,k \in \mathbb{I}}, B = (b^k_{ij})_{i,j,k \in \mathbb{I}}, C = (c_i^{jk})_{i,j,k \in \mathbb{I}}$ and $\epsilon = (\epsilon_i)$ such that: $a_{ijk}, \epsilon_i \in \mathbb{C}$ are non-zero for only finitely many $(i,j,k) \in \mathbb{I}\times \mathbb{I}\times \mathbb{I}$, $b_{ij}^k \in \mathbb{C}$ is non-zero for only finitely many $(i,j) \in \mathbb{I}\times \mathbb{I}$ for each $k \in \mathbb{I}$ and $c_i^{jk} \in \mathbb{C}$ is non-zero for only finitely many $i \in \mathbb{I}$ for each $(j,k) \in \mathbb{I}\times \mathbb{I}$.
\end{remark}

Given a quantum pre-Airy structure $(V,A,B,C,\epsilon)$ there exists a recursive procedure called \emph{Abstract Topological Recursion (ATR)} \cite[Section 2.5]{kontsevich2017airy} produces a set of numbers $\{S_{g,n;i_1,...,i_n} \in \mathbb{C}\ |\ i_k \in \mathbb{I}\}$. Let the basis of $V$ be $\{e^{i\in \mathbb{I}}\}$. We package the set of numbers $\{S_{g,n;i_1,...,i_n}\}$ into the tensor
\begin{equation*}
    S_{g,n} := \sum_{i_1,...,i_n \in \mathbb{I}}\frac{1}{n!}S_{g,n;i_1,...,i_n}e^{i_1}\otimes ... \otimes e^{i_n} \in V^{\otimes n} = V\otimes ... \otimes V.
\end{equation*}
As the name suggests, ATR is related to topological recursion. We will see later this section the quantum pre-Airy structure such that the ATR reduces to the usual topological recursion. 

Let us now describe ATR in details. We set the initial conditions to be
\begin{equation*}
    S_{0,3} := 2A \in V\otimes V\otimes V, \qquad S_{1,1} := \epsilon \in V, \qquad S_{0,n\leq 2} = S_{g\geq 0, 0} = 0.
\end{equation*}
Then $S_{g,n}$ for $g \geq 0, n\geq 4$ or $g\geq 2, n \geq 1$ or $g\geq 1, n \geq 2$ are computed recursively as follows
\begin{align*}\label{tratr}
    S_{g,n;i_1,...,i_n} &= 2\sum_{k=2}^n\sum_{j\in \mathbb{I}}b_{i_1i_k}^{j}S_{g,n-1;ji_{\{2,...,n\}\setminus\{k\}}} + \sum_{\substack{g_1+g_2=g\\I_1\coprod I_2 = \{2,...,n\}}}\sum_{j_1,j_2 \in \mathbb{I}}c_{i_1}^{j_1j_2}S_{g_1,1+|I_1|;j_1i_{I_1}}S_{g_2,1+|I_2|;j_2i_{I_2}}\\
    &\qquad + \sum_{j_1,j_2 \in \mathbb{I}}c_{i_1}^{j_1j_2}S_{g-1,n+1;j_1j_2i_2...i_n}.\numberthis
\end{align*}
Note that the right-hand-side of (\ref{tratr}) only involve $S_{g',n'}$ with $2g'-2+n' < 2g-2+n$, therefore the recursion is guaranteed to terminate.
\begin{remark}\label{atrisfiniteremark}
The right-hand-side of (\ref{tratr}) will indeed define an element of $V^{\otimes n}$. Let $v := \sum_{k\in \mathbb{I}}v_ke^k \in V, u :=\sum_{k\in \mathbb{I}}u_ke^k \in V$, then by definition we have $\sum_{i,j,k\in \mathbb{I}}b^k_{ij}v_ke^i\otimes e^j = B(v) \in V\otimes V$ and $\sum_{i,j,k \in \mathbb{I}}c^{jk}_iv_ju_ke^i = C(v\otimes u) \in V$. Note that all summations came from contractions of a vectors in $V$ and $V^*$, so they have at most finitely non-zero terms even though the carnality of the index set $\mathbb{I}$ may be infinite.

Given $g,n$ and suppose that $S_{g',n'} \in V^{\otimes n'}$ for all $2g'-2 + n' < 2g-2+n$ then it follows that the first and second terms in (\ref{tratr}) are well defined as elements of $V^{\otimes n}$. Finally, the last term of (\ref{tratr}) is a well-defined element of $V^{\otimes n}$ because $S_{g-1,n+1} \in V^{\otimes n+1}$, a non-completed tensor products of $V$, so it contains only finite linear combination of tensor products of vectors in $V$. More explicitly, we can write $\sum_{j_1,j_2\in\mathbb{I}}S_{g-1,n+1;j_1j_2i_2...i_n}e^{j_1}\otimes e^{j_2}$ as a finite sum $\sum_{k=1}^Nv_{1,k}\otimes v_{2,k}$ for some $v_{1,k}, v_{2,k} \in V$ and $N \in \mathbb{Z}_{\geq 0}$. Then
\begin{align*}
    \sum_{i_1\in \mathbb{I}}\sum_{j_1,j_2 \in \mathbb{I}}c^{j_1,j_2}_{i_1}S_{g-1,n+1;j_1j_2i_2...i_n}e^{i_1} &= C\left(\sum_{j_1,j_2\in\mathbb{I}}S_{g-1,n+1;j_1j_2i_2...i_n}e^{j_1}\otimes e^{j_2}\right)\\
    &= C\left(\sum_{k=1}^Nv_{1,k}\otimes v_{2,k}\right) = \sum_{k=1}^NC\left(v_{1,k}\otimes v_{2,k}\right) \in V.
\end{align*}
Overall, the output $S_{g,n}$ is a well-defined element of $V^{\otimes n}$.
\end{remark}

\begin{definition}\theoremname{\cite[Definition 2.6.1]{kontsevich2017airy}}
A  classical pre-Airy \emph{sub-structure} $(U,A,B,C)$ of $(V,A,B,C)$ is a classical pre-Airy structure given by a vector subspace $U \subseteq V$ compatible with tensors $A,B,C$ in the sense that
\begin{equation*}
    A \in U\otimes U\otimes U, \qquad B|_U : U \rightarrow U\otimes U, \qquad C|_{U\otimes U} : U\otimes U \rightarrow U.
\end{equation*}
A quantum pre-Airy sub-structure $(U,A,B,C,\epsilon)$ of $(V,A,B,C,\epsilon)$ is a quantum pre-Airy structure given by a classical pre-Airy sub-structure $(U,A,B,C)$ of $(V,A,B,C)$ such that $\epsilon \in U$.

A (quantum or classical) pre-Airy structure on $V$ is called \emph{primitive} if it does not contain a non-trivial (quantum or classical) pre-Airy sub-structure $(U,A,B,C,\epsilon)$ with $U\subsetneq V$.
\end{definition}

\begin{remark}\label{atrandprestructuresremark}
Given two quantum pre-Airy structures $(V_i, A_i, B_i, C_i, \epsilon_i)$ and their ATR outputs $\{S^{(i)}_{g,n} \in V^{\otimes n}_i\}, i = 1,2$ the important feature is the following. If both quantum pre-Airy structures $(V_i, A_i, B_i, C_i, \epsilon_i)$, $i = 1,2$ contain the same quantum pre-Airy substructure $(V,A,B,C,\epsilon)$, i.e. $V \subseteq V_1, V_2$ and
\begin{equation*}
    A_1 = A = A_2, \qquad B_1|_{V} = B = B_2|_{V}, \qquad C_1|_{V\otimes V} = C = C_2|_{V\otimes V}, \qquad \epsilon_1 = \epsilon = \epsilon_2
\end{equation*}
then their ATR outputs will be the same \cite[Section 2.6]{kontsevich2017airy}:
\begin{equation*}
    S^{(1)}_{g,n} = S^{(2)}_{g,n} = S_{g,n} \in V^{\otimes n} \subseteq V^{\otimes n}_i,\qquad i = 1,2
\end{equation*}
where $\{S_{g,n}\in V^{\otimes n}\}$ is the output of ATR using the pre-Airy structure $(V,A,B,C,\epsilon)$.
\end{remark}

\subsubsection{Airy structures}
In this section we let $W = V\oplus V^*$ be a Tate space and let $\{x^{i\in \mathbb{I}}, y_{i\in \mathbb{I}}\}$ be canonical coordinates of $W$.
\begin{definition}\theoremname{\cite[Definition 2.1.1]{kontsevich2017airy}}
A \emph{classical Airy structure} on $V$ is a classical pre-Airy structure $(V,A,B,C)$ such that the collection of \emph{at-most quadratic polynomials} $\{H_{i\in \mathbb{I}} \in \widehat{Sym}_{\leq 2}(W)\}$, where
\begin{equation}\label{airystructuredefn}
    H_i := -y_i + \sum_{j,k \in \mathbb{I}}a_{ijk}x^jx^k + 2\sum_{j,k \in \mathbb{I}}b_{ij}^kx^jy_k + \sum_{j,k\in \mathbb{I}}c^{jk}_iy_jy_k,
\end{equation}
generates the vector space $\mathfrak{g} := \bigoplus_{i\in \mathbb{I}}\mathbb{C}\cdot H_i \subset \widehat{Sym}_{\leq 2}(W)$ which is closed under the Poisson bracket : $\{H_i,H_j\} = \sum_{k\in \mathbb{I}}g^k_{ij}H_k \in \mathfrak{g}$ for some \emph{structure constants} $g^k_{ij} \in \mathbb{C}$. In other words, the structure constants $g^k_{ij}$ are the matrix entries of the operator $\mathfrak{g}$ is a Lie algebra with Lie bracket given by the Poisson bracket.
\end{definition}

\begin{remark}\label{symmetryabcremark}
We can always assume that the tensors $A = (a_{ijk})$ and $(c_i^{jk})$ are symmetric in the indices $j,k$ because their anti-symmetric part would vanish in the expression of $H_i$ after summing against $x^jx^k$ and $y_jy_k$. On the other hand, the tensor $B = (b_{jk}^i)$ does not possess any symmetry in the indices $j,k$.

In fact, given a classical Airy structure $\{H_{i\in \mathbb{I}} \in \widehat{Sym}_{\leq 2}(W)\}$, it can be shown (see \cite[Section 2.2]{andersen2017abcd}, \cite[Section 2.1]{kontsevich2017airy}) from $\{H_i,H_j\} = \sum_{k\in \mathbb{I}}g^k_{ij}H_k$ that $a_{ijk}$ is symmetric in the indices $i,j$ and that $g_{ij}^k = 2(b^k_{ij}-b^k_{ji})$. In other words, $a_{ijk}$ is symmetric in all indices $i,j,k$ and the structure constants $g_{ij}^k$ are components of the tensor $(g^k_{ij})_{i,j,k\in \mathbb{I}} : V\rightarrow V\otimes V$ which is entirely determined by the tensor $B$.
\end{remark}

\begin{remark}
Note that each $H_i$ may be treated as a function $H_i : W \rightarrow \mathbb{C}$. Denote by $H_i(w)$ the evaluation of $H_i$ at a point $w := \sum_{i\in \mathbb{I}}(x^if_i + y_ie^i) \in W$ where $x^i, y_i \in \mathbb{C}$, then $H := \sum_{i\in\mathbb{I}}H_i(w)e^i \in V$.
\end{remark}

Let us define the Weyl algebra $\mathcal{D}^\hbar(W) := \mathbb{C}\left[\left[x^{i\in \mathbb{I}}, \hbar\partial_{i\in \mathbb{I}}\right]\right]\left[\left[\hbar\right]\right]$ where $\hbar\partial_i := \hbar\frac{\partial}{\partial x^i}$. We introduce grading on $\mathcal{D}^\hbar(W)$ by setting $\deg \hbar = 2, \deg x^i = 1$ and $\deg \hbar\partial_i = 1$. Let $\mathcal{D}^\hbar_n(W)$ be the vector subspace of degree $n$ elements in $\mathcal{D}^\hbar(W)$ and let $\mathcal{D}^\hbar_{\leq n}(W) := \bigoplus_{i=0}^n \mathcal{D}^\hbar_i(W)$. We also introduce the following notations when working with $\mathcal{D}^\hbar(W)$:
\begin{definition}
Let $\hat{f}, \hat{g} \in \mathcal{D}^\hbar(W)$. We denote \emph{applying the operator $\hat{f}$ on $\hat{g}$} by $\hat{f}(\hat{g})$. We denote \emph{multiplying the operator $\hat{f}$ with the operator $\hat{g}$} by $\hat{f}\hat{g}$. For example, $\hbar\partial_i(x^j) = \hbar\delta^j_i$, and $\hbar\partial_ix^j = \hbar\delta_i^j + x^j\hbar\partial_i$. 
\end{definition}
Conceptually, $\mathcal{D}^\hbar(W)$ is the quantization of $\mathbb{C}[[W]]$ with $y_i$ replaced by $\hbar\partial_i$ and the Poisson bracket $\{.,.\}$ replaced by the commutator $\frac{1}{\hbar}[.,.]$, where $[\hat{f}, \hat{g}] := \hat{f}\hat{g} - \hat{g}\hat{f}, \hat{f}, \hat{g} \in \mathcal{D}^\hbar(W)$. We define the \emph{classical limit} to be the map $cl : \mathcal{D}^\hbar(W) \rightarrow \mathbb{C}[[W]]$, given by $\hbar \mapsto 0, x^i\mapsto x^i$ and $\hbar\partial_i \mapsto y_i$. 

\begin{definition}\theoremname{\cite[Definition 2.2.1]{kontsevich2017airy}}
A \emph{quantum Airy structure} on $V$ is a collection $\{\hat{H}_{i\in \mathbb{I}}\in \mathcal{D}^\hbar_{\leq 2}(W)\}$ such that $\hat{\mathfrak{g}} := \bigoplus_{i\in \mathbb{I}}\mathbb{C}\cdot \hat{H}_i \subset \mathcal{D}^\hbar_{\leq 2}(W)$ is a Lie algebra with Lie bracket given by the commutator $\frac{1}{\hbar}[.,.]$. In other words, $\frac{1}{\hbar}[\hat{H}_1, \hat{H}_2] = \sum_{k\in\mathbb{I}}g^k_{ij}\hat{H}_k \in \hat{\mathfrak{g}}$ for some constants $g^k_{ij} \in \mathbb{C}$.

A classical Airy structure $\{H_{i\in \mathbb{I}}\}$ is called the \emph{classical limit} of $\{\hat{H}_{i\in \mathbb{I}}\}$ if $H_i = cl(\hat{H}_i)$ in which case we also have $\{H_i, H_j\} = cl\left(\frac{1}{\hbar}[\hat{H}_i, \hat{H}_j]\right) = \sum_{k \in \mathbb{I}}g^k_{ij}H_k$. We say that $\{\hat{H}_{i\in \mathbb{I}}\}$ \emph{quantizes} or it is a \emph{quantization} of $\{H_{i\in\mathbb{I}}\}$.
\end{definition}

\begin{remark}\label{abcepsilonrelationsremark}
The condition $\frac{1}{\hbar}[\hat{H}_i, \hat{H}_j] = \sum_{k\in \mathbb{I}}g^k_{ij}\hat{H}_k$ can be written as relations among tensors $A,B,C$ and $\epsilon$, \cite{andersen2017abcd}. If we ignored all relations involving $\epsilon$, then the remaining are relations among $A,B$ and $C$ arise from the classical condition $\{H_i,H_j\} = \sum_{k\in \mathbb{I}}g^k_{ij}H_k$.
\end{remark}

\begin{remark}
Quantization is more difficult than just replacing $x^i$ by $x^i$ and $y_i$ by $\hbar\partial_i$ because $\mathbb{C}[[W]]$ is a commutative algebra whereas $\mathcal{D}^\hbar(W)$ is non-commutative, hence there is an ambiguity in ordering $x^i$ and $\hbar\partial_i$. For instance, it is unclear whether the term $x^iy_j$ should be quantized as $x^i\hbar\partial_j$ or $\hbar\partial_jx^i = \hbar\delta_j^i + x_i\hbar\partial_j$. In general, suppose that $\hat{f}, \hat{g}, \hat{h} \in \mathcal{D}^\hbar_{\leq 2}(W)$ with classical limit $f = cl(\hat{f}), g = cl(\hat{g}), h = cl(\hat{h}) \in \widehat{Sym}_{\leq 2}(W)$ then $\{\hat{f},\hat{g}\} = \hat{h}$ implies $\{f,g\} = cl\left(
\frac{1}{\hbar}[\hat{f},\hat{g}]\right) = h$ but on the other hand, if $\{f,g\} = h$ then $cl\left(
\frac{1}{\hbar}[\hat{f},\hat{g}]\right) = \{f,g\} + c = h+c$ for some constant $c \in \mathbb{C}$. Therefore, the quantization of $\{H_{i\in \mathbb{I}}\}$ where $H_i$ are given by (\ref{airystructuredefn}) and $\{H_i,H_j\} = \sum_{k\in \mathbb{I}}g^k_{ij}H_k$ amount to finding $\{\epsilon_{i\in \mathbb{I}} \in \mathbb{C}\}$ such that
\begin{equation}\label{quantumairystructuredefn}
    \hat{H}_i := -\hbar\partial_i + \sum_{j,k \in \mathbb{I}}a_{ijk}x^jx^k + 2\hbar\sum_{j,k \in \mathbb{I}}b^k_{ij}x^j\partial_k + \hbar^2\sum_{j,k \in \mathbb{I}}c^{jk}_i\partial_j\partial_k + \hbar\epsilon_i
\end{equation}
satisfying $\frac{1}{\hbar}[\hat{H}_i, \hat{H}_j] = \sum_{k\in\mathbb{I}}g^k_{ij}\hat{H}_k$. 

If follows from (\ref{quantumairystructuredefn}) we can write a quantum Airy structure on $V$ as a tuple $(V,A,B,C,\epsilon)$. In other words, a quantum Airy structure is a quantum pre-Airy structure $(V,A,B,C,\epsilon)$ such that $\hat{\mathfrak{g}} := \bigoplus_{i\in \mathbb{I}}\mathbb{C}\cdot \hat{H}_i$, where $\hat{H}_i$ given by (\ref{quantumairystructuredefn}), is a Lie algebra.

Given a classical Airy structure $\{H_{i\in\mathbb{I}}\}$, then neither the uniqueness nor existence of its quantization is guaranteed in general. It turns out that the sufficient condition for the existence of quantization of $\{H_{i\in \mathbb{I}}\}$ is given by the vanishing of the second cohomology of Lie algebra $H^2(\mathfrak{g}, \mathbb{C})$. On the other hand, the necessary and sufficient condition for the uniqueness of the quantization is given by the vanishing of $H^1(\mathfrak{g},\mathbb{C})$ \cite{kontsevich2017airy}.
\end{remark}

We will often refer to classical (pre-)Airy structures simply as (pre-)Airy structures unless it is important to emphasize their distinction from quantum (pre-)Airy structures.

There is a correspondence between Airy structures and \emph{quadratic Lagrangian submanifolds} $L := \{H_{i\in\mathbb{I}} = 0\}$ of $W$, i.e. Lagrangian submanifolds defined by at-most quadratic polynomials in $W$. When $V$ is finite-dimensional, this correspondence is easy:
\begin{lemma}\theoremname{\cite[Section 2.1]{kontsevich2017airy}}\label{airystrandquadlagrangianlemma}
Let $V$ be an $n$-dimensional vector space and $\{H_1,...,H_n\}$ is a collection of at-most quadratic polynomials on $W = V\oplus V^*$ in the form given in (\ref{airystructuredefn}). Then $L := \{H_{i=1,\cdots,n} = 0\} \subset W$ is a Lagrangian submanifold if and only if the collection $\{H_1,...,H_n\}$ gives an Airy structure. 
\end{lemma}
\begin{proof}
Suppose that $\{H_1,...,H_n\}$ is an Airy structure.
Let $v_{H_i}$ be a Hamiltonian vector field corresponding to $H_i$, i.e. $\iota_{v_{H_i}}\Omega = dH_i$. Then $v_{H_i}(x) \in T_xL$ for all $x \in L$ because $\iota_{v_{H_i}(x)}dH_j(x) = \{H_i, H_j\}(x) = 0$. Because $\dim L = n$, we conclude that $\{v_{H_1}(x),...,v_{H_n}(x)\}$ is a basis of $T_xL$. Finally, $\Omega(v_{H_i},v_{H_j}) = \{H_i,H_j\} = 0$ on $L$ and $\dim L = \frac{1}{2}\dim W$, therefore $L \subset W$ is a Lagrangian submanifold.

Conversely, suppose that $L\subset W$ is a Lagrangian submanifold. Then for all $v \in T_xL$ we have $\Omega(v_{H_i}(x),v) = \iota_vdH_i(x) = 0$. Because $T_xL \subset W$ is a Lagrangian subspace, i.e. a maximal isotropic subspace, it follows that $v_{H_i} \in T_xL$. Then $\{H_i, H_j\}(x) = \Omega(v_{H_i}, v_{H_j}) = 0$ for all $x \in L$, which means $\{H_i, H_j\}$ belongs to an ideal generated by $\{H_1,...,H_n\}$. But each $H_i$ is at most quadratic, so $\{H_i, H_j\}$ also is at most quadratic, we conclude that $\{H_i, H_j\} = \sum_{k=1}^ng^k_{ij}H_k$ for some constants $g_{ij}^k \in \mathbb{C}$.
\end{proof}

\begin{remark}\label{lagrangiancomplementremark}
When $V$ is finite dimensional, we define the tangent space to $L$ at $x \in L \subset W$ to be
\begin{equation*}
    T_xL := \{v \in W\ |\ \iota_vdH_i(x) = 0, i = 1,\cdots,n\}.
\end{equation*}
It is more subtle how to interpret this definition when $V$ is infinite-dimensional, except at $x = 0 \in W$. It is easy to see from the form of each $H_i$ in (\ref{airystructuredefn}) even when $V$ is infinite-dimensional that $H_{i}(0) = 0$ and $dH_i(0) = dy_i$ for all $i\in \mathbb{I}$, therefore we define $T_0L := \{y_{i\in \mathbb{I}} = 0\} = V^* \subset W$ which is a Lagrangian subspace of $W$. Since $V\subset W$ is also a Lagrangian subspace of $W$ and $W = V\oplus V^* \cong V\oplus T_0L$ we call $V$ a \emph{Lagrangian complement of $T_0L$}. 
\end{remark}

When $V$ is infinite-dimensional the arguments used in the proof of Lemma \ref{airystrandquadlagrangianlemma} will no longer hold. Instead, what we mean by quadratic Lagrangian submanifold is that there exists a generating function $S_0 \in \prod_{n=1}^\infty Sym_n(V)$ such that
\begin{equation}\label{s0isgeneratingfunction}
    H_i\left(\{x^{i\in \mathbb{I}}\}, \left\{y_{i\in \mathbb{I}} := \frac{\partial S_0}{\partial x^i}\right\}\right) = 0, \qquad i \in \mathbb{I}.
\end{equation}
We are going to show this in a moment. Keep in mind that $\partial_iS_0 = \frac{\partial S_0}{\partial x^i} \in \prod_{n=1}^\infty Sym_n(V)$ is a formal derivative and (\ref{s0isgeneratingfunction}) should be understood as: the image of $H_i$ under the homomorphism $\widehat{Sym}_{\leq 2}(W) \rightarrow \prod^\infty_{n=1}Sym_n(V), x^i\mapsto x^i, y_i \mapsto \partial_iS_0$ is zero.

Note that we can think of a quantum Airy structures $\{\hat{H}_{i\in \mathbb{I}}\}$ as a quantization of a quadratic Lagrangian submanifold $L := \{H_{i\in\mathbb{I}} = 0\}$. Let us consider solving for $\psi_L$ such that $\psi_L^{-1}\hat{H}_i(\psi_L) = 0$ where $\psi_L$ is given by the formal expression
\begin{equation}\label{psiwavefunction}
    \psi_L := \exp\left(\sum_{g = 0}^\infty \hbar^{g-1}S_g\right), \qquad S_g =  \sum_{n=0}^\infty S_{g,n} \in \prod_{n=1}^\infty Sym_n(V), 
\end{equation}
such that $\psi_L^{-1}(\hbar\partial_i)\psi_L := \hbar\partial_i + \partial_i\left(\sum_{g=0}^\infty\hbar^gS_g\right)$ and each $S_{g,n}$ are homogeneous polynomial of degree $n$:
\begin{equation}\label{sgnpolynomial}
    S_{g,n} := \sum_{i_1,...,i_n\in \mathbb{I}}\frac{1}{n!}S_{g,n;i_1,\cdots,i_n}x^{i_1}\cdots x^{i_n} \in Sym_n(V).
\end{equation}
Observe that $S_{g,n} \in Sym_n(V)$ without completion of symmetric tensor product and $S_g$ is a formal sum of $S_{g,n}$. The following Proposition guarantee the existence and uniqueness of the solution $\psi_L$:

\begin{theorem}\theoremname{\cite[Theorem 2.4.2]{kontsevich2017airy}}\label{quantizationofltheorem}
Let $\{\hat{H}_{i\in\mathbb{I}}\}$ be a quantum Airy structure. There exists a unique solution of the form $\psi_L$ of the form (\ref{psiwavefunction}) such that $\psi^{-1}_L\hat{H}_i(\psi_L) = 0, \forall i \in \mathbb{I}$.
\end{theorem}
\begin{proof}
First of all, we note that if $\psi := \exp\left(\sum_{g=0}^\infty \hbar^{g-1}\sum_{n=0}^\infty S_{g,n}\right)$ for an arbitrary $S_{g,n} \in Sym_n(V)$ then $\psi^{-1}\hat{H}_i\psi \in \mathcal{D}^\hbar(W)$. To verify this we need to check that for each $g$ and $n$, the degree $n$ term of the coefficient of $\hbar^g$ in $\psi^{-1}\hat{H}_i\psi(1)$ belongs to $Sym_n(V)$. These type of terms came from
\begin{equation}\label{hgnformula}
    H^{(g,n)}_i := -\hbar^g\partial_iS_{g,n+1} + 
    \hbar^g\sum_{j,k\in \mathbb{I}}b^k_{ij}x^j\partial_kS_{g,n} + \hbar^g\sum_{\substack{g_1+g_2=g\\n_1+n_2=n+2}}\sum_{j,k\in \mathbb{I}}c^{jk}_i\partial_jS_{g_1,n_1}\partial_kS_{g_2,n_2} + \hbar^g\sum_{j,k\in \mathbb{I}}c^{jk}_i\partial_j\partial_kS_{g-1,n+2}.
\end{equation}
The first term belongs to $Sym_n(V)\hbar^g$. Since $\sum_{k\in\mathbb{I}}b^k\partial_kS_{g,n}$ is the same as contracting the tensor $B$ with $S_{g,n}$, by Remark \ref{atrisfiniteremark} the second term also belongs to $Sym_n(V)\hbar^g$. Similarly, by Remark \ref{atrisfiniteremark} we also conclude that the third and the fourth terms also belong to $Sym_n(V)\hbar^g$. This verifies our claim that $\psi^{-1}\hat{H}_i\psi \in \mathcal{D}^\hbar(W)$. Remark \ref{atrisfiniteremark} tells us that $H^{(g,n)} := \sum_{i\in \mathbb{I}}H^{(g,n)}_ie^i \in V^{\otimes (n+1)}$.

The proof is by induction in $g$ and for each fixed $g$ we proceed with induction in $n$. Suppose that we have found the solution $S_g$ for all $g = 0,...,g_0-1$ and for $S_{g_1}$ up to degree $n_0$ term in the sense that if $\psi^{(g_0,n_0)}_L := \exp\left(\sum_{g=0}^{g_0-1}\hbar^{g-1}S_g + \hbar^{g_0-1}\sum_{n=1}^{n_0}S_{g_0,n}\right)$ then 
\begin{align*}\label{hhatg0n0}
    \hat{H}^{(g_0,n_0)}_i &:= (\psi^{(g_0,n_0)}_L)^{-1}\hat{H}_i\psi^{(g_0,n_0)}_L\\
    &= -\hbar\partial_i + H_i^{(g_0,n_0)} + \hbar^{g_0}x^{\geq n_0+1} + \hbar^{\geq g_0+1}x^{\geq 0} + 2\hbar\sum_{j,k\in \mathbb{I}}b^k_{ij}x^j\partial_k + \hbar^2\sum_{j,k\in \mathbb{I}}c^{jk}_i\partial_j\partial_k \in \mathcal{D}^\hbar(W),\numberthis
\end{align*}
where $H^{(g_0,n_0)}_i := \hbar^{g_0}x^{=n_0} \in Sym_n(V)\hbar^{g_0}$. Since $\{\hat{H}_{i\in \mathbb{I}}\}$ is a quantum Airy structure we can check that $\{\hat{H}_{i\in \mathbb{I}}^{(g_0,n_0-1)}\}$ satisfies the similar commutation relations 
\begin{align*}\label{commutationrelationforhhat}
    \frac{1}{\hbar}[\hat{H}_i^{(g_0,n_0)}, \hat{H}_j^{(g_0,n_0)}] &= \frac{1}{\hbar}\left(\hat{H}_i^{(g_0,n_0)}\hat{H}_j^{(g_0,n_0)} - \hat{H}_j^{(g_0,n_0)}\hat{H}_i^{(g_0,n_0)}\right)
    = \frac{1}{\hbar}(\psi_L^{(g_0,n_0)})^{-1}\left(\hat{H}_i\hat{H}_j - \hat{H}_j\hat{H}_i\right)\psi_L^{(g_0,n_0)}\\
    &= (\psi_L^{(g_0,n_0)})^{-1}\left(\frac{1}{\hbar}[\hat{H}_i, \hat{H}_j]\right)\psi_L^{(g_0,n_0)} = \sum_{k\in \mathbb{I}}g^{k}_{ij}\hat{H}_k^{(g_0,n_0)}.\numberthis
\end{align*}
The left hand side of (\ref{commutationrelationforhhat}) contains the term $\partial_j \left(H_i^{(g_0,n_0)}\right) - \partial_i \left(H_j^{(g_0,n_0)}\right) = \hbar^{g_0}x^{=n_0-1}$ but the right hand side only contain $\hbar^{\geq g_0}x^{\geq n_0}$, so it must be the case that $\partial_j \left(H_i^{(g_0,n_0)}\right) = \partial_i \left(H_j^{(g_0,n_0)}\right)$. In other words, there exists a symmetric tensor $S_{g_0,n_0+1} := \frac{1}{(n_0+1)!}\sum_{i_1,...,i_{n_0}\in\mathbb{I}}S_{g_0,n_0;i_1,...,i_{n_0+1}}x^{i_1}\cdots x^{i_{n_0+1}} \in Sym_{n_0+1}(V)$ such that $H_i^{(g_0,n_0)} = \hbar^{g_0}\partial_i S_{g_0,n_0+1}$. 

For the base case of the induction, $(g_0 = 0)$, we can look at the first non-trivial $S_{0,n_0}$ which is $S_{0,3} = 2A \in Sym_3(V)$. By setting $\psi^{(0,3)}_L = \exp\left(\frac{1}{\hbar}S_{0,3}\right)$ we can see that $(\psi_L^{(0,3)})^{-1}\hat{H}_i\psi_L^{(0,3)}$ takes the form of (\ref{hhatg0n0}) with $(g_0,n_0) = (0,3)$. 
\end{proof}

Taking the limit $\hbar\rightarrow 0$ of the expression $\psi^{-1}_L\hat{H}_i\psi_L(1) = 0$ is the same as replacing every $\hbar\partial_i$ in $\hat{H}_i$ with $\partial_iS_0$. Therefore, we obtain (\ref{s0isgeneratingfunction}), showing that $S_0 = \sum_{n=0}^\infty S_{0,n} \in \prod_{n=1}^\infty Sym_n(V)$ is the generating function of $L = \{H_i = 0\}$. It is in this sense that $\psi_L$ is considered a quantization of $L$. The following Corollary of Theorem \ref{quantizationofltheorem} makes the relationship between quantum Airy structures and ATR becomes clear:
\begin{corollary}\label{genfunccorollary}
If a quantum pre-Airy structure $(V,A,B,C,\epsilon)$ is a quantum Airy structure then the output $S_{g,n} \in V^{\otimes n}$ of the corresponding ATR are symmetric tensors. Moreover, if $\psi_L := \exp\left(\sum_{g=0}^\infty \hbar^{g-1}\sum_{n=0}^\infty S_{g,n}\right)$ then $\psi^{-1}_L\hat{H}_i(\psi_L) = 0$.
\end{corollary}
\begin{proof}
Theorem \ref{quantizationofltheorem} tells us that the solution to $\psi_L^{-1}\hat{H}_i(\psi_L) = 0$ of the form $\psi_L := \exp\left(\sum_{g=0}^\infty \hbar^{g-1}\sum_{n=0}^\infty S_{g,n}\right)$ where $S_{g,n} \in Sym_n(V)$ exists. The equation $\psi^{-1}_L\hat{H}_i\psi_L(1) = 0$ implies that the right-hand side of (\ref{hgnformula}) must vanish. In particular, the coefficient of each $x^{i_1}\cdots x^{i_n}, i_k \in \mathbb{I}$ must vanish and this gives us the recursion relations (\ref{tratr}).
\end{proof}

Finally, the notion of sub-structures we introduced for (quantum) pre-Airy structures translates straightforwardly to the notion of sub-structures for (quantum) Airy structures:
\begin{definition}\theoremname{\cite[Section 2.6]{kontsevich2017airy}}
A classical Airy sub-structure $(U,A,B,C)$ of a classical Airy structure $(V,A,B,C)$ is a classical pre-Airy sub-structure of a pre-Airy structure $(V,A,B,C)$. A quantum Airy sub-structure $(U,A,B,C,\epsilon)$ of a quantum Airy structure $(V,A,B,C,\epsilon)$ is a quantum pre-Airy sub-structure of a quantum pre-Airy structure $(V,A,B,C,\epsilon)$.
\end{definition}
\begin{lemma}\theoremname{\cite[Section 2.6]{kontsevich2017airy}}
An Airy sub-structure of an Airy structure is an Airy structure. A quantum Airy sub-structure of a quantum Airy structure is a quantum Airy structure
\end{lemma}
\begin{proof}
According to Remark \ref{abcepsilonrelationsremark}, certain relations among $A,B,C$ and $\epsilon$ are satisfied if and only if $A,B,C$ and $\epsilon$ define a (quantum) Airy structure. The fact that we can restrict these tensors to a subspace $U \subset V$ will not affect these relations.
\end{proof}

\subsubsection{Gauge transformation of (quantum) Airy structures}\label{gaugetransformationsubsection}
Given an Airy structure $(V,A,B,C)$, it defines a quadratic Lagrangian submanifold $L = \{H_{i\in \mathbb{I}} = 0\} \subset W = V\oplus V^*$. Suppose that $\{e^{i\in \mathbb{I}}, f_{i\in \mathbb{I}}\}$ and $\{x^{i\in\mathbb{I}}, y_{i\in\mathbb{I}}\}$ are the given canonical basis and canonical coordinates of $W$ respectively. As we have mentioned in Remark \ref{lagrangiancomplementremark}, the vector subspace $V \subset W$, with a basis $\{e^{i\in\mathbb{I}}\}$, is a Lagrangian complement of $T_0L = \left\{\sum_{i\in\mathbb{I}}x^kf_i\right\} \subset W$. But $V$ is not a unique Lagrangian complement of $T_0L$. Any other Lagrangian complement $\bar{V} \subset W$ of $T_0L$ corresponds to the change of a canonical basis $\{e^{i\in\mathbb{I}}, f_{i\in\mathbb{I}}\} \mapsto \{\bar{e}^{i\in\mathbb{I}}, \bar{f}_{i\in\mathbb{I}}\}$, $\bar{e}^i := e^i + \sum_{j\in \mathbb{I}}s^{ij}f_j, \bar{f}_i := f_i$ of $W$, where $(s^{ij}) : V\otimes V \rightarrow \mathbb{C}$ is an arbitrary symmetric tensor. Note that $\{\bar{e}^{i\in\mathbb{I}}\}$ is a new basis spanning the new Lagrangian complement $\bar{V}$ of $T_0L$ which remains fixed. The corresponding change of canonical coordinates on $W$ is given by $\bar{x}^i = x^i - \sum_{j \in \mathbb{I}}s^{ij}y_j, \bar{y}_i = y_i$. The symplectic form retains its canonical form $\Omega = \sum_{i \in \mathbb{I}}d\bar{x}^i\wedge d\bar{y}_i$ in this new coordinates and $T_0L \cong W/\bar{V} \cong \bar{V}^*$, therefore $W \cong \bar{V}\oplus \bar{V}^*$. Expressing the Airy structure $\{H_{i\in \mathbb{I}}\}$ using the new coordinates $\{\bar{x}^{i\in \mathbb{I}}, y_{i\in\mathbb{I}}\}$ in the usual form (\ref{airystructuredefn}) we obtain the new tensors $\bar{A} = (\bar{a}_{ijk}), \bar{B} = (\bar{b}_{ij}^k), \bar{C} = (\bar{c}_i^{jk})$.

Both Airy structures $(V,A,B,C)$ and $(\bar{V}, \bar{A}, \bar{B}, \bar{C})$ are the same in the sense that they define the same Lagrangian submanifold $L \subset W$, but we have chosen to described them differently using the different choices of Lagrangian complement $\bar{V}$. Hence we call $(V,A,B,C)\mapsto (\bar{V}, \bar{A}, \bar{B}, \bar{C})$ a \emph{gauge transformation} of Airy structures.

We can also explore gauge transformation associate with freedom in choosing a basis $\bar{e}^i := \sum_{j\in \mathbb{I}}d^i_je^j$ of $V$ and $\bar{f}_i := \sum_{j\in \mathbb{I}}c^j_if_j$ of $V^*$. The corresponding change of coordinates given by $\bar{x}^i = \sum_{j\in \mathbb{I}}d^i_jx^j$ and $\bar{y}_i = \sum_{j\in\mathbb{I}}c^j_iy_j$ respectively.
We require $(c^i_j) : V\rightarrow V$ and $(d^i_j) : V\rightarrow V$ to be the inverse of each other, i.e. $\sum_{k\in\mathbb{I}}d^k_ic_k^j = \delta_i^j$ and $\sum_{k\in\mathbb{I}}c^k_id^j_k = \delta^j_i$, so that the coordinates $\{\bar{x}^{i\in\mathbb{I}}, \bar{y}_{i\in\mathbb{I}}\}$ remain canonical for $W = V\oplus V^*$. Let us also redefine the basis for $\mathfrak{g} = \bigoplus_{i\in\mathbb{I}}\mathbb{C}\cdot H_i$ by $H_i \mapsto \bar{H}_i := \sum_{j\in\mathbb{I}}c^j_iH_j$. The commutation relations become 
\begin{equation*}
    \{\bar{H}_i, \bar{H}_j\} = \sum_{k\in\mathbb{I}}\bar{g}^k_{ij}\bar{H}_k, \qquad \text{where} \qquad \bar{g}_{ij}^k := g_{pq}^rc^p_ic^q_jd^k_r.
\end{equation*}
Here and the remaining of this section, Einstein summation notations will be used where appropriate. Note that $\bar{g}^k_{ij} := g_{pq}^rc^p_ic^q_jd^k_r$ only involve finite summations since $g^i_{jk}$ are components of the tensor $(g^i_{jk}) : V\rightarrow V\otimes V$ as explained in Remark \ref{symmetryabcremark} (see also Remark \ref{atrisfiniteremark}). Hence $\bar{g}^i_{jk}$ is well-defined and we have a tensor $(\bar{g}^i_{jk}) : V\rightarrow V\otimes V$.

\begin{remark}\label{cdsremark}
For a fixed $i \in \mathbb{I}$, $c^i_j$ will be non-zero for only finitely many $j \in \mathbb{I}$. It follows that the elements of the inverse matrix must satisfy the same properties. Namely, given an $i \in \mathbb{I}$, $d^i_j$ will be non-zero for only finitely many $j\in \mathbb{I}$. On the other hand, $s^{ij}$ can be non-zero for infinitely many pairs of $(i,j) \in \mathbb{I}^2$.
\end{remark}

As a summary, let us give an explicit formula for the gauge transformation $(V,A,B,C)\mapsto (\bar{V}, \bar{A}, \bar{B}, \bar{C})$ associating with the change of a canonical basis $\{e^{i\in\mathbb{I}}, f_{i\in\mathbb{I}}\} \mapsto \{\bar{e}^{i\in\mathbb{I}}, \bar{f}_{i\in\mathbb{I}}\}$ of $W$:
\begin{equation}\label{efcanonicalbasistransformation0}
    \bar{e}^i := \sum_{j\in\mathbb{I}}d_j^i\left(e^j + \sum_{k\in\mathbb{I}}s^{jk}f_k\right), \qquad \bar{f}_i := \sum_{j\in\mathbb{I}}c_i^jf_j,
\end{equation}
where $(c^i_j) : V\rightarrow V$ and $(d^i_j) : V\rightarrow V$ are the inverse of each other and $(s^{ij}) : V\times V \rightarrow \mathbb{C}$. It is easy to check that
\begin{equation*}
    \Omega(\bar{e}^i,\bar{e}^j) = \Omega(\bar{f}_i,\bar{f}_j) = 0, \qquad \Omega(\bar{e}^i,\bar{f}_j) = \delta^i_j.
\end{equation*}

The corresponding change of a canonical coordinates is given by:
\begin{equation}\label{xycanonicalcoordtransformation0}
    \bar{x}^i := \sum_{j=1}^\infty d^i_j\left(x^j 
    - \sum_{k=1}^\infty s^{jk}y_k\right), \qquad \bar{y}_i := \sum_{j=1}^\infty c^j_iy_j,
\end{equation}
and it is easy to check that 
\begin{equation*}
    \{\bar{x}^i, \bar{x}^j\} = \{\bar{y}_i, \bar{y}_j\} = 0, \qquad \{\bar{y}_i, \bar{x}^j\} = \delta^j_i.
\end{equation*}

Denote the transformation (\ref{efcanonicalbasistransformation0}) by $T : W \rightarrow W$. Let $\bar{V} \subset W$ be the vector subspace spanned by $\{\bar{e}^{i\in\mathbb{I}}\}$:
\begin{equation}\label{vgaugetransformation}
    \bar{V} := T(V).
\end{equation}

Substituting the inverse coordinates transformation $x^i = \sum_{j\in\mathbb{I}}\left(c^i_j\bar{x}^j + \sum_{k\in\mathbb{I}}\bar{y}_jd^j_ks^{ki}\right), y_i = \sum_{j\in\mathbb{I}}d^j_i\bar{y}_j$, into (\ref{airystructuredefn}) and collecting terms so that $\bar{H}_i := \sum_{j\in\mathbb{I}}c^j_iH_j$ takes the form
\begin{equation*}
    \bar{H}_i := -\bar{y}_i + \sum_{j,k \in \mathbb{I}}\bar{a}_{ijk}\bar{x}^j\bar{x}^k + 2\sum_{j,k \in \mathbb{I}}\bar{b}_{ij}^k\bar{x}^j\bar{y}_k + \sum_{j,k\in \mathbb{I}}\bar{c}^{jk}_i\bar{y}_j\bar{y}_k,
\end{equation*}
we find that $(V,A,B,C)\mapsto (\bar{V},\bar{A},\bar{B},\bar{C})$ is given by \cite[Section 2.1]{kontsevich2017airy}
\begin{align}
    \bar{a}_{i_1i_2i_3} &:= a_{j_1j_2j_3}c^{j_1}_{i_1}c^{j_2}_{i_2}c^{j_3}_{i_3}\label{agaugetransfromation}\\
    \bar{b}^{i_3}_{i_1i_2} &:= b^{j_3}_{j_1j_2}c^{j_1}_{i_1}c^{j_2}_{i_2}d^{i_3}_{j_3} + a_{j_1j_2p}s^{pj_3}c^{j_1}_{i_1}c^{j_2}_{i_2}d^{i_3}_{j_3}\label{bgaugetransfromation}\\
    \bar{c}^{i_2i_3}_{i_1} &:= c^{j_2j_3}_{j_1}c^{j_1}_{i_1}d^{i_2}_{j_2}d^{i_3}_{j_3} + b^{j_2}_{j_1p}s^{pj_3}c^{j_1}_{i_1}d^{i_2}_{j_2}d^{i_3}_{j_3} + b^{j_3}_{j_1q}s^{qj_2}c^{j_1}_{i_1}d^{i_2}_{j_2}d^{i_3}_{j_3} + a_{j_1pq}s^{pj_2}s^{qj_3}c^{j_1}_{i_1}d^{i_2}_{j_2}d^{i_3}_{j_3}\label{cgaugetransfromation}.
\end{align}

Similarly, we study \emph{gauge transformations} of quantum Airy structures $(V,A,B,C,\epsilon)\mapsto (\bar{V},\bar{A},\bar{B},\bar{C},\bar{\epsilon})$ by making the following substitutions:
\begin{equation*}
    x^i \mapsto \hat{x}^i := \sum_{j\in\mathbb{I}}\left(c^i_j\bar{x}^j + \sum_{k\in\mathbb{I}}s^{ik}d^j_k\hbar\bar{\partial}_j\right), \qquad \hbar \partial_i = \hbar\frac{\partial}{\partial x^i} \mapsto \hat{y}_i := \sum_{j\in \mathbb{I}}d^j_i\hbar\bar{\partial}_j = \sum_{j\in \mathbb{I}}d^j_i\hbar\frac{\partial}{\partial \bar{x}^j}
\end{equation*}
in (\ref{quantumairystructuredefn}) and collecting terms so that $\bar{\hat{H}} := \sum_{j\in\mathbb{I}}c^j_i\hat{H}_j$ takes the form: 
\begin{equation*}
    \bar{\hat{H}}_i := -\hbar\bar{\partial}_i + \sum_{j,k \in \mathbb{I}}\bar{a}_{ijk}\bar{x}^j\bar{x}^k + 2\hbar\sum_{j,k \in \mathbb{I}}\bar{b}^k_{ij}\bar{x}^j\bar{\partial}_k + \hbar^2\sum_{j,k \in \mathbb{I}}\bar{c}^{jk}_i\bar{\partial}_j\bar{\partial}_k + \hbar\bar{\epsilon}_i.
\end{equation*}
We find that the transformation $A\mapsto \bar{A},B \mapsto \bar{B},C \mapsto \bar{C}$ are given by (\ref{agaugetransfromation})-(\ref{cgaugetransfromation}) and $\epsilon \mapsto \bar{\epsilon}$ is given by \cite[Section 2.2]{kontsevich2017airy}
\begin{equation}\label{epsilongaugetransformation}
    \bar{\epsilon}_i := \epsilon_lc^l_i + a_{ljk}c^l_is^{jk}.
\end{equation}
Because the commutation relations of $\hat{x}^i$ and $\hat{y}_j$ in $\mathcal{D}^\hbar(W)$ are the same as those of $x^i$ and $\hbar\partial_j$ we conclude that $(\bar{V}, \bar{A}, \bar{B}, \bar{C}, \bar{\epsilon})$ is a quantization of $(\bar{V}, \bar{A}, \bar{B}, \bar{C})$.

\begin{remark}\label{gaugetransformationofpreairystrremark}
Suppose that $(V,A,B,C,\epsilon)$ is a quantum pre-Airy structure. Given $(c^i_j) : V\rightarrow V$, $(d^i_j) = (c^i_j)^{-1} : V\rightarrow V$ and $(s^{ij}) : V\times V \rightarrow \mathbb{C}$, then $(V, A, B, C, \epsilon) \mapsto (\bar{V}, \bar{A}, \bar{B}, \bar{C}, \bar{\epsilon})$ given by (\ref{vgaugetransformation})-(\ref{epsilongaugetransformation}) is also a quantum pre-Airy structure transformation. We call this a \emph{gauge transformation} of quantum pre-Airy structure.
\end{remark}

\begin{remark}\label{gaugetransfrespectssubstructuresremark}
The gauge transformation respects sub-structures in the following sense \cite[Section 2.6]{kontsevich2017airy}. We write $\bar{V} := T(V)$. Suppose that $(U,A,B,C,\epsilon)$ is a quantum (pre-)Airy sub-structure of $(V,A,B,C,\epsilon)$ then $(\bar{U},\bar{A}, \bar{B}, \bar{C},\bar{\epsilon})$ is a quantum (pre-)Airy sub-structure of $(\bar{V}, \bar{A}, \bar{B}, \bar{C}, \bar{\epsilon})$ where $\bar{U} := T(U)$ and $T : W \rightarrow W$ denotes the transformation (\ref{efcanonicalbasistransformation0}).
\end{remark}

\subsection{The Residue Constraints Airy Structure}\label{residueconstraintairystructuresection}
In this section, we are going to review the residue constraints Airy structure \cite[Section 3.3]{kontsevich2017airy}. It is the Airy structure that will be most interesting to us due to its connection with topological recursion as we will explain in Section \ref{relationstotrsection}.  Consider a countably infinite dimensional (let $\mathbb{I} := \mathbb{Z}_{>0}$) vector space
\begin{equation*}
    V_{Airy} := z^{-1}\mathbb{C}[z^{-1}]\frac{dz}{z}
\end{equation*}
of meromorphic differential forms in the neighborhood of $z = 0$ and equipped with the discrete topology. The set $\left\{e^{k\in\mathbb{I}} := z^{-k}\frac{dz}{z}\right\}$ is a basis of $V_{Airy}$. The continuous dual space is given by 
\begin{equation*}
    V^*_{Airy} := z\mathbb{C}[[z]]\frac{dz}{z}
\end{equation*}
where a vector $f \in V^*_{Airy}$ can be thought of as a linear functional $v\mapsto Res_{z=0}\left(f\int v\right) \in \mathbb{C}$ on $V_{Airy}$. The set $\left\{f_{k\in\mathbb{I}} := kz^{k}\frac{dz}{z}\right\}$ is a basis of $V^*_{Airy}$, we can check that $f_i(e^j) = \delta_i^j$. The corresponding Tate space is
\begin{equation*}
    W_{Airy} := V_{Airy}\oplus V^*_{Airy} = \left\{\eta \in \mathbb{C}((z))dz\ |\ Res_{z=0}\eta = 0\right\}
\end{equation*}
with the symplectic form 
\begin{equation*}
\Omega_{Airy}(f,g) := Res_{z=0}\left(f\int g\right).    
\end{equation*}
Any typical element $w$ of $W_{Airy}$ can be expressed as $w := \sum_{k\in \mathbb{I}}\left(x^kf_k + y_ke^k\right) = \sum_{k=1}^\infty\left(y_kz^{-k}\frac{dz}{z} + kx^kz^k\frac{dz}{z}\right)$, where $y_k \in \mathbb{C}$ is non-zero for an only finite number of $k$. The symbol $\int g$ in the definition of $\Omega_{Airy}(f,g)$ denotes a formal indefinite integral of $g$. For instance, let us write $g := \sum_{k=1}^\infty\left(y_kz^{-k}\frac{dz}{z} + kx^kz^k\frac{dz}{z}\right)$ then $\int g := \sum_{k=1}^\infty\left(-\frac{1}{k}y_kz^{-k} + x^kz^k\right) + const$. Note that the integration constant does not introduce any ambiguity in the expression of $\Omega_{Airy}(f,g)$ since the residue of $f$ is zero. It is easy to check that $f \mapsto \Omega_{Airy}(f,.)$ gives an isomorphism $W\cong W^*$ and in particular, $e^k\mapsto \Omega_{Airy}(e^k,.)$ gives an isomorphism $V_{Airy}\cong V^{**}_{Airy}$. 

Define the \emph{residue constraints Airy structure} $\{H_{i\in\mathbb{I}}\}$ on $V_{Airy}$ by 
\begin{align*}\label{resconstraintsairystructure}
    (H_{Airy})_{2n}(w) &= Res_{z = 0}\left(\left(z - \frac{w(z)}{2zdz}\right) z^{2n}d(z^2)\right),\\
    (H_{Airy})_{2n-1}(w) &= \frac{1}{2}Res_{z = 0}\left(\left(z - \frac{w(z)}{2zdz}\right)^2 z^{2n-2}d(z^2)\right),\qquad n = 1,2,3,...\numberthis
\end{align*}
where $w(z) := \sum_{k\in \mathbb{I}}\left(x^kf_k + y_ke^k\right) = \sum_{k=1}^\infty\left(y_kz^{-k}\frac{dz}{z} + kx^kz^k\frac{dz}{z}\right) \in W_{Airy}$. By expanding (\ref{resconstraintsairystructure}) into formal polynomials in $\widehat{Sym}_{\leq 2}(W_{Airy})$ of coordinates $\{x^{i\in\mathbb{I}}, y_{i\in\mathbb{I}}\}$ of $W_{Airy}$, we find that $\{(H_{Airy})_{i\in\mathbb{I}}\}$ can be expressed in the standard form (\ref{airystructuredefn}) as \cite[Section 3.3]{kontsevich2017airy}:
\begin{equation}\label{resconstraintsairyhamiltonian}
    (H_{Airy})_{2n} = -J_{2n},\qquad (H_{Airy})_{2n-1} = -J_{2n-1}+\frac{1}{4}\sum_{k=-\infty}^\infty J_{2k-1}J_{2n-2k-3}, \qquad n = 1,2,3,...
\end{equation}
where $J_{+k} := y_k, J_{-k} := kx^k$. The quantization of Airy structure $\{(H_{Airy})_{i\in\mathbb{I}}\}$ exists and it is given by \cite[Section 3.3]{kontsevich2017airy}
\begin{equation}\label{quantumresconstraintsairyhamiltonian}
    (\hat{H}_{Airy})_{2n} = -\hat{J}_{2n},\qquad (\hat{H}_{Airy})_{2n-1} = -\hat{J}_{2n-1}+\frac{1}{4}\sum_{k=-\infty}^\infty \hat{J}_{2k-1}\hat{J}_{2n-2k-3} + \frac{\hbar}{16}\delta_{n,2}, \qquad n = 1,2,3,...
\end{equation}
where $\hat{J}_{+k} := \hbar\frac{\partial}{\partial x^k}, \hat{J}_{-k} := kx^k$. By reading the coefficients of (\ref{quantumresconstraintsairyhamiltonian}), we found that the tensors $A_{Airy},B_{Airy},C_{Airy}$ and $\epsilon_{Airy}$ are given by \cite[Section 3.3]{kontsevich2017airy}
\begin{align}
    A_{Airy} = \frac{1}{4}e^1\otimes e^1\otimes e^1, &\qquad \text{or} \qquad (a_{Airy})_{ijk} = \frac{1}{4}\delta_{1i}\delta_{1j}\delta_{1k}\label{explicitaairy}\\
    B_{Airy} : e^n \mapsto \frac{1}{4}\sum_{\substack{n_1,n_2 \geq 1\\ n_1 + n_2 = n+3\\ n_1\ odd}}n_2e^{n_1}\otimes e^{n_2}, &\qquad \text{or} \qquad (b_{Airy})_{ij}^k = \frac{1}{4}\sum_{\substack{k_1,k_2 \geq 1\\ k_1 + k_2 = k+3\\ k_1\ odd}}k_2\delta_{k_1i}\delta_{k_2j}\label{explicitbairy}\\
    C_{Airy} : e^{n_1}\otimes e^{n_2} \mapsto \begin{cases}
    0, & n_1+n_2\ \text{odd}\\
    \frac{1}{4}e^{n_1+n_2+3}, & n_1+n_2\ \text{even} 
    \end{cases}&\qquad \text{or} \qquad (c_{Airy})_i^{jk} = \begin{cases}
    0, & j+k\ \text{odd}\label{explicitcairy}\\
    \frac{1}{4}\delta_{i,j+k+3}, & j+k\ \text{even} 
    \end{cases}\\
    \epsilon_{Airy} = \frac{1}{16}e^3, &\qquad \text{or} \qquad (\epsilon_{Airy})_i = \frac{1}{16}\delta_{i,3}\label{explicitepsilonairy}.
\end{align}
\begin{remark}\label{residueconstaintsabcremark}
We can also express $A_{Airy},B_{Airy},C_{Airy}$ tensors in term of a residue of meromorphic forms by substituting $w(z) := \sum_{k\in\mathbb{I}}(x^kf_k(z) + y_ke^k(z))$ into (\ref{resconstraintsairystructure}):
\begin{align*}
    (a_{Airy})_{ijk} &= Res_{z=0}\left(-\frac{1}{4i}\frac{1}{z(dz)^2}f_{i}(z)f_{j}(z)f_{k}(z)\right)\delta_{i,odd}\\
    (b_{Airy})_{ij}^{k} &= Res_{z=0}\left(-\frac{1}{4i}\frac{1}{z(dz)^2}f_{i}(z)f_{j}(z)e^{k}(z)\right)\delta_{i,odd}\\
    (c_{Airy})_{i}^{jk} &= Res_{z=0}\left(-\frac{1}{4i}\frac{1}{z(dz)^2}f_{i}(z)e^{j}(z)e^{k}(z)\right)\delta_{i,odd},
\end{align*}
where $\delta_{i,odd} = \sum_{n=1}^\infty\delta_{i,2n-1}$ is $1$ if $i$ is an odd integer and $0$ if $i$ is an even integer. Sometime it is easier to work with $A_{Airy},B_{Airy},C_{Airy}$ in this form. It is less straightforward to express $\epsilon_{Airy}$ as a residue.
\end{remark}
One last important point to note is that $(V_{Airy}, A_{Airy}, B_{Airy}, C_{Airy}, \epsilon_{Airy})$ contains a primitive quantum Airy sub-structure $((V_{Airy})_{odd}, A_{Airy}, B_{Airy}, C_{Airy}, \epsilon_{Airy})$ where 
\begin{equation*}
    (V_{Airy})_{odd} := z^{-1}\mathbb{C}[z^{-2}]\frac{dz}{z}.
\end{equation*}
\subsubsection{The Lagrangian `submanifold' $L_{Airy}\subset W_{Airy}$}
As we have mentioned in the last section, the interpretation of an Airy structure as a quadratic Lagrangian submanifold is less obvious when dimension is infinite. Let us show how $L_{Airy} := \{(H_{Airy})_{i\in \mathbb{I}} = 0\} \subset W_{Airy}$ highlights this difficulty. 
\begin{lemma}\label{lairyistrivialinwairylemma}
\begin{equation*}
    \{w \in W_{Airy}\ |\ (H_{Airy})_{i\in\mathbb{I}}(w) = 0\} = \{x^1 = y_1 = y_2 = y_3 = \cdots = 0\} \subset W_{Airy}.
\end{equation*}
\end{lemma}
In other words, the solution to $(H_{Airy})_{i\in\mathbb{I}} = 0$ in $W_{Airy}$ are the trivial ones. Therefore, by interpreting $L_{Airy}$ as a subset of $W_{Airy}$ we will have $L_{Airy}$ as a linear subspace of $W_{Airy}$ instead of a non-linear submanifold as the name \emph{quadratic} suggests.

\begin{proof}
Since $(H_{Airy})_{2n} = 0$ implies $y_{2n} = 0$, it remains for us to show that $x^1 = y_{2n-1} = 0$ for all $n = 1,2,3,...$. Since $w = \sum_{k=1}^\infty\left(y_kz^{-k}\frac{dz}{z} + kx^kz^k\frac{dz}{z}\right) \in W_{Airy}$ means there are finitely many non-zero $y_{i\in \mathbb{I}}$, let us assume that $y_{2n+1} = 0$ for all $n \geq N$ for some $N \in \mathbb{Z}_{>0}$. Then it is easy to check that $(H_{Airy})_{4n+1} = \frac{1}{4}y^2_{2n-1} = 0$, which implies $y_{2n-1} = 0$. Therefore $y_{2n-1} = 0$ for all $n = 1,2,3,...$. Finally, $y_1 = 0$ implies $(H_{Airy})_1 = \frac{1}{4}(x^1)^2 = 0$. Therefore $x^1 = 0$ and the lemma is proven.
\end{proof}

This Lemma tells us that any solution $w = \sum_{k=1}^\infty\left(y_kz^{-k}\frac{dz}{z} + kx^kz^k\frac{dz}{z}\right)$ to $(H_{Airy})_i(w) = 0$ with $x^1 \neq 0$ must consist of infinitely many non-zero $y_i$ and so $w$ does not belong to $W_{Airy}$. According to Theorem \ref{quantizationofltheorem} or Corollary \ref{genfunccorollary}, there exists $\psi_{L_{Airy}} := \exp\left(\sum_{g=0}^\infty \hbar^{g-1}(S_{Airy})_g\right)$ such that $(\psi_{L_{Airy}})^{-1}(\hat{H}_{Airy})_i(\psi_{L_{Airy}}) = 0$. In particular, we can still find the generating function $(S_{Airy})_0$ such that 
\begin{equation*}
(H_{Airy})_i\left(\{x^{i\in \mathbb{I}}\}, \left\{y_{i\in \mathbb{I}} := \frac{\partial (S_{Airy})_0}{\partial x^i}\right\}\right) = 0, \qquad i \in \mathbb{I},
\end{equation*}
although we will have $\frac{\partial S_0}{\partial x^i} \neq 0$ for infinitely many $i\in\mathbb{I}$. In this sense, $L_{Airy}$ can be thought of formally as a Lagrangian submanifold of $W_{Airy}$.

Nevertheless, we define $T_0L_{Airy} := z\mathbb{C}[[z]]\frac{dz}{z} \subset W_{Airy}$. It is clear that $T_0L_{Airy}$ is a Lagrangian subspace of $W_{Airy}$ and $V_{Airy}$ is a Lagrangian complement of $T_0L_{Airy}$.

\subsection{The Embedding of Discs and The Analytic Residue Constraints}\label{theembeddingofdiscssection}
Essentially, our work in this section is an analytic approach to \cite[Section 6]{kontsevich2017airy}. As we have seen, the quadratic Lagrangian submanifold $L_{Airy} \subset W_{Airy}$ defined by the residue constraints can only be interpreted formally and not as an actual submanifold. In this section, we will study the residue constraints from an analytic perspective. We show that there exists non-trivial solution $L^M_{Airy} := \{(H_{Airy})_{i=1,2,3,\cdots} = 0\}$ of residue constraints in a modified version of $W_{Airy}$, which will be introduced as $W^{\epsilon_,M}_{Airy}$. We also will see how the residue constraints are related to the deformation of \emph{discs} embedded in the foliated symplectic surface $(\mathbb{C}^2, dx\wedge dy, \{x = const\})$. Another motivation for us to take the analytic approach instead of the formal one is because it should make the connection between Airy structures and the prepotential $\mathfrak{F}_{\Sigma_0}$ more transparent as $\mathfrak{F}_{\Sigma_0}$ is a holomorphic function.

We organize the content of this section as follows. In Section \ref{spaceofdiscssubsection}, we introduce the space of analytic discs $Discs^M$ (similar to the space of formal discs $Discs$ in \cite[Section 6]{kontsevich2017airy}) and the vector bundle $W^{\epsilon, M}\rightarrow Discs^M$. Some interesting properties of $W^{\epsilon, M}$ will be noted. In Section \ref{analyticresconstraintssubsection}, we show how the residue constraints Airy structure can be considered analytically. For example, when writing $(H_{Airy})_i$ in the standard form (\ref{airystructuredefn}), all summations converge absolutely and there is no ambiguity in summation order. On the other hand, not all gauge transformations considered in Section \ref{gaugetransformationsubsection} are allowed in the analytic context, and we propose a more restricted class of gauge transformations. We also explore how some of the analytic solutions in $L^M_{Airy}$ can be expressed in terms of the genus zero part of the ATR with quantum Airy structure $(V_{Airy}, A_{Airy}, B_{Airy}, C_{Airy})$, which we have discussed previously as a formal power series. Finally, we define in Section \ref{embeddingofdiscssubsection} an embedding map $\Phi_{t_0} : Discs^{\epsilon, M}_{t_0}\rightarrow W^{\epsilon, M}_{t_0}$ analogous to $\pmb{\Phi}_{\Sigma_0} : \mathcal{B}_{\Sigma_0}\rightarrow \mathcal{H}_{\Sigma_0}$ defined in Section \ref{deformationofcurvessection} and show that the image of $\Phi_{t_0}$ is given by convergence Laurent series satisfying the residue constraints. 

\subsubsection{The space of discs}\label{spaceofdiscssubsection}
We begin by introducing the notion of analytic discs and the trivial weak symplectic vector bundle $(W^{\epsilon, M}\rightarrow Discs^M, \Omega,\nabla)$. Unlike the formal discs and symplectic vector bundle $W\rightarrow Discs$ introduced in \cite[Section 6]{kontsevich2017airy}, the convergence condition of series on some annulus will be enforced in our version. The upshot is that $(W^{\epsilon, M}\rightarrow Discs^M, \Omega, \nabla)$ is flat (Lemma \ref{wisflatlemma}), covariantly constant (Lemma \ref{symplecticformiscovariantlyconstantlemma}) and the parallel transport of $\xi_1 \in W^{\epsilon, M}_{t_1}$ to $W^{\epsilon, M}_{t_2}$, for any two points $t_1,t_2 \in Discs^M$ sufficiently close, is given by (Lemma \ref{paralleltransportinwlemma})
\begin{equation*}
    \xi_2 := \exp\left((a(t_2) - a(t_1))\mathcal{L}_{\frac{1}{2z}\partial_z}\right)(\xi_1) \in W^{\epsilon, M}_{t_2},
\end{equation*}
where $\exp(a\mathcal{L}_{\frac{1}{2z}\partial_z})$ is an operator exponential of the Lie derivative $\frac{1}{2z}\partial_z$. If $|a(t_2) - a(t_1)|$ is too large then $\xi_2$ may not have the desired convergence properties to be in $W^{\epsilon, M}_{t_2}$. This is why it is important for $t_1$ and $t_2$ to be sufficiently close. In fact, most of the contents from Lemma \ref{expoperatorlemma} until Lemma \ref{wisflatlemma} are dedicated to discussing the properties of $\exp(a\mathcal{L}_{\frac{1}{2z}\partial_z})$ and in particular, the precise notion for $t_1$ and $t_2$ to be considered `sufficiently close'. However, these details are technical and can be skipped in the first reading. On the other hand, Lemma \ref{paralleltransportinwlemma} is the key result that will be used in the construction of $\Phi_{t_0} : Discs^{\epsilon, M}_{t_0} \rightarrow W^{\epsilon, M}_{t_0}$ later in Section \ref{embeddingofdiscssubsection}. 

In the simplest terms, a disc is an embedding $t : \mathbb{D}\cong \mathbb{D}_t \hookrightarrow \mathbb{C}^2$ such that $\mathbb{D}_t$ has a tangency of order $1$ to a leaf of the foliation at $r := t(0)$, where $\mathbb{D} \subset \mathbb{C}$ is an open neighbourhood of $0 \in \mathbb{C}$.
Using the standard local coordinate $z := \sqrt{x - x(r)}$, in general, the embedding $t$ is a parameterization of $\mathbb{D}_t \subset \mathbb{C}^2$ which can be written in the form
\begin{equation*}
    t = (x = a + z^2, y = b_0 + b_1z + b_2z^2 + b_3z^3 + ...), \qquad a, b_k \in \mathbb{C}, 
\end{equation*}
where $b_1 \neq 0$. The study of the embedding of discs helps us understand at the local level the deformation of $\mathcal{F}$-transversal curves $\Sigma$ embedded inside a foliated symplectic surface $(S,\Omega_S,\mathcal{F})$. As explained in Section \ref{deformationofcurvessection}, given a curve $[\Sigma] \in \mathcal{B}$ then we can find a $(\mathcal{F},\Omega_S)$-chart $U_\alpha$ around any ramification point $r_\alpha, \alpha \in Ram$ of $\Sigma$. Because $U_\alpha$ is symplectomorphic to an open subset of the foliated symplectic surface $(\mathbb{C}^2, dx\wedge dy, \{x = const\})$, therefore $\Sigma \cap U_\alpha$ can be treated as a disc embedded in $(\mathbb{C}^2, dx\wedge dy, \{x = const\})$. We will compare the global deformation of $\Sigma \subset S$ to the deformation of discs embedded around each ramification point of $\Sigma$ later in Section \ref{localtoglobalsection}. 

Because discs are only defined locally, it follows that their deformation space is much larger than that of a $\mathcal{F}$-transversal curve $\Sigma$. Let us now define the space of discs:

\begin{definition}
    Let $\mathbb{D}_M := \{z \in \mathbb{C}\ |\ |z| < M\}$, then
    \begin{equation*}
        Discs^M := \left\{t : \mathbb{D}_M \xrightarrow{\cong}\mathbb{D}_{t,M}\hookrightarrow \mathbb{C}^2\ |\ t = \left(x = a(t) + z^2, y = \sum_{k=0}^\infty b_k(t)z^k\right), 
        \begin{array}{c}
             b_1(t) \neq 0 \text{ and } \\
             \sum_{k=0}^\infty b_k(t)z^k \text{ converges }\\
             \text{ for $|z| < \bar{M}$ for some $M < \bar{M}$.}
        \end{array}\right\}.
    \end{equation*}
\end{definition}
 
We can think of $Discs^M$ as a subset of the infinite dimensional vector space of infinite tuple of numbers $(a(t),b_0(t),b_1(t),\cdots)$. It is almost a vector subspace except for the condition $b_1(t) \neq 0$. Nevertheless, for $t_1, t_2 \in Discs^M$ and $\lambda_1, \lambda_2 \in \mathbb{C}$ such that $\lambda_1 b_1(t_1) + \lambda_2 b_1(t_2) \neq 0$ we define
\begin{equation*}
    \lambda_1 t_1 + \lambda_2 t_2 := \left(x = \lambda_1 a(t_1) + \lambda_2 a(t_2) + z^2, y = \sum_{k=0}^\infty (\lambda_1 b_k(t_1) + \lambda_2 b_k(t_2))z^k\right) \in Discs^M.
\end{equation*}
We define the \emph{kinematic tangent space} $T_tDiscs^M$ as guided by \cite{kriegl1997convenient}:
\begin{equation*}
    T_tDiscs^M := \left\{ v = A\frac{\partial}{\partial a} + \sum_{k=0}^\infty B_k\frac{\partial}{\partial b_k}\ |\ A, B_k \in \mathbb{C}, \begin{array}{c}
             \sum_{k=0}^\infty B_kz^k \text{ converges }\\
             \text{ for $|z| < \bar{M}$ for some $M < \bar{M}$.}
        \end{array}\right\}.
\end{equation*}
More precisely, $T_tDiscs^M$ is the set of $v \in Discs^M$ such that $\gamma(\tau) := t + \tau v$ defines a path $\gamma : (-\delta, \delta) \rightarrow Discs^M$ for some $\delta > 0$. The tangent space $T_tDiscs^M$ tells us information about the infinitesimal deformations of $\mathbb{D}_{t,M}$.

We are going to introduce the vector bundle $W^{\epsilon, M}\rightarrow Discs^M$ which can be thought of as a local counter-part of the vector bundle $\mathcal{H}\rightarrow \mathcal{B}$ defined in Section \ref{deformationofcurvessection}. A fiber of $W^{\epsilon, M}$ over each point $t \in Discs^M$ is given by $W^{\epsilon, M}_{Airy}$, the analytic version of $W_{Airy}$. Unlike $W_{Airy}$ which is a symplectic topological vector space, $W^{\epsilon, M}_{Airy}$ will be a \emph{weak symplectic vector space}.
\begin{definition}
A weak symplectic (topological) vector space $(W,\Omega)$ is a pair of a (topological) vector space $W$ and a (continuous) antisymmetric bilinear form $\Omega : W\times W \rightarrow \mathbb{C}$ such that $w \mapsto \Omega(w,.)$ is injective $W\hookrightarrow W^*$. We call $\Omega$ the \emph{weak symplectic form} of $W$.
\end{definition}
\begin{definition}
For any $\epsilon, M \in \mathbb{R}$ such that $0 < \epsilon < M$, the weak symplectic vector space $(W^{\epsilon,M}_{Airy},\Omega_{Airy})$ is given by
\begin{equation*}
    W_{Airy}^{\epsilon,M} :=  \left\{\xi \in z^{-1}\mathbb{C}[[z^{-1}]]\frac{dz}{z}\oplus z\mathbb{C}[[z]]\frac{dz}{z}\ |\ \begin{array}{c}
    \text{$\xi(z)$ converges when $\bar{\epsilon} < |z| < \bar{M}$}\\ 
    \text{for some $\bar{\epsilon} < \epsilon < M < \bar{M}$}\end{array}
    \right\}.
\end{equation*}
and the symplectic form $\Omega_{Airy}$ is given by:
\begin{equation*}
    \Omega_{Airy}(\xi_1,\xi_2) := Res_{z=0}\left(\xi_1\int\xi_2\right).
\end{equation*}
The definition of the symplectic form $\Omega_{Airy}$ is exactly the same as the one on $W_{Airy}$, hence we are denoting them using the same notation. Let $(W^{\epsilon, M}\rightarrow Discs^M, \Omega, \nabla)$ be the trivial weak symplectic vector bundle, $W^{\epsilon, M} := Discs^M\times W^{\epsilon, M}_{Airy}$. The fiber over each point $t \in Discs^M$ of $W^{\epsilon, M}$ is given by the weak symplectic vector space $(W^{\epsilon, M}_t, \Omega_t) := (W^{\epsilon, M}_{Airy}, \Omega_{Airy})$. The connection $\nabla$ is given by:
\begin{align*}\label{connectionofw}
    \nabla &:= d - (da)\mathcal{L}_{\frac{1}{2z}\partial_z} = da\frac{\partial}{\partial a} + \sum_{k=0}^\infty db_k\frac{\partial}{\partial b_k} - (da)\mathcal{L}_{\frac{1}{2z}\partial_z}\\
    &: \Gamma(Discs^M, W^{\epsilon,M})\rightarrow \Gamma(Discs^M, Hom(TDiscs^M,W^{\epsilon, M}))\numberthis
\end{align*}
where $\mathcal{L}_{\frac{1}{2z}\partial_z} \in End(W^{\epsilon,M}_t)$ denotes the Lie derivative operator.
\end{definition}

\begin{remark}
The trivial vector bundle $W^{\epsilon, M}\rightarrow Discs^M$ should be compared to the trivial vector bundle $W\rightarrow Discs$ introduced in \cite[Section 6]{kontsevich2017airy} where $Discs$ is the space of formal discs and each fiber of $W$ is given by $W_t := W_{Airy} = \{\eta \in \mathbb{C}((z))dz\ |\ Res_{z=0}\eta = 0\}$.
\end{remark}

\begin{remark}\label{laurentseriesgeneralremark}
The following are useful elementary facts when working with any $\xi \in W^{\epsilon, M}_{Airy}$. If a Laurent series $\sum_{k=-\infty}^{+\infty} a_kz^k$ converges for $\bar{\epsilon} < |z| < \bar{M}$ then it converges absolutely and uniformly for $\epsilon < |z| < M$ for any $\epsilon > \bar{\epsilon}$ and $M < \bar{M}$. From the uniform convergence, it follows that the integration of a Laurent series can be done term-by-term. Similarly, using the Cauchy integral formula, we find that the differentiation of a Laurent series can also be done term-by-term.
\end{remark}

\begin{remark}
Let us explain why the symplectic form $\Omega_{Airy} : W^{\epsilon, M}_{Airy}\times W^{\epsilon, M}_{Airy}\rightarrow \mathbb{C}$ is well-defined. If $\xi_2 = \sum_{k=1}^\infty\left((y_2)_kz^{-k}\frac{dz}{z} + (x_2)^kkz^k\frac{dz}{z}\right)$ converges to a holomorphic differential on an annulus $\bar{\epsilon} < |z| < \bar{M}$ with $\bar{\epsilon} < \epsilon < M <\bar{M}$ then the indefinite integral $\int \xi_2$ converges to a holomorphic function on the same annulus. It follows that $\xi_1\int \xi_2$ is holomorphic on some annulus containing $\epsilon < |z| < M$ and hence its residue is well-defined as a coefficient of $\frac{dz}{z}$ in the Laurent expansion of $\xi_1\int \xi_2$. The integration constant arising from $\int \xi_2$ will not introduce any ambiguity in $Res_{z=0}(\xi_1\int \xi_2)$ because the residue of $\xi_1$ is zero.
\end{remark}

\begin{remark}\label{topologyofwremark}
It is possible to assign a topology to $W^{\epsilon, M}$. For example, we can equip $Discs^M$ and $W^{\epsilon, M}$ with the topology induced from the following norms:
\begin{equation*}
    \Vert t \Vert_{Discs^M} := \sqrt{|a(t)|^2 + \left(\sup_{\epsilon < |z| < M}\left|\sum_{k=0}^\infty b_k(t)z^k\right|\right)^2}, \qquad \Vert \xi \Vert_{W^{\epsilon, M}} := \sup_{\epsilon < |z| < M}\left|\xi(z)\frac{z}{dz}\right|
\end{equation*}
and equip $W^{\epsilon, M} := Discs^M\times W^{\epsilon, M}_{Airy}$ with the product topology. The topology of $W^{\epsilon, M}$ will play no role later. Note that for $W^{\epsilon, M}_{Airy}$ to be a weak symplectic vector space, a topology is not needed. The stronger condition of symplectic requires a topology because an infinite-dimensional vector space is always strictly smaller than its algebraic dual, hence a topology is needed for the condition $W\cong W^{*}$ to be satisfied. The fact that $W^{\epsilon, M}_{Airy}$ is weak symplectic instead of symplectic will not lead to any major consequences. 
\end{remark}

\begin{remark}\label{discsglobalsectionremark}
We follow \cite{kriegl1997convenient, lang2012differential} in defining a connection $\nabla$ (as given in (\ref{connectionofw})) on an infinite-dimensional vector bundle over an infinite-dimensional manifold. Let us explain more precisely the meaning of $\Gamma(Discs^M, W^{\epsilon, M})$ and $\Gamma(Discs^M, Hom(TDiscs^M, W^{\epsilon, M}))$ as follows. 

We call a section $\xi$ of $W^{\epsilon, M}$ differentiable if $\iota_{v}d\xi = A\frac{\partial \xi}{\partial a} + \sum_{k=0}^\infty B_k\frac{\partial \xi}{\partial b_k} \in W^{\epsilon, M}$ for all $v := A\frac{\partial}{\partial a} + \sum_{k=0}^\infty B_k\frac{\partial}{\partial b_k} \in T_tDiscs^M$ and call $\xi$ smooth if it is infinitely differentiable. We define $\Gamma(Discs^M, W^{\epsilon, M})$ to be the set of smooth global sections of $W^{\epsilon, M}$. Let $\Gamma(Discs^M, TDiscs^M)$ be the set of sections $t \mapsto A(t)\frac{\partial}{\partial a} + \sum_{k=0}^\infty B_k(t)\frac{\partial}{\partial b_k} \in T_tDiscs^M$ such that $\sum_{k=0}^\infty B_k(t)z^kdz \in \Gamma(Discs^M, W^{\epsilon, M})$. Finally, we let $\Gamma(Discs^M, Hom(TDiscs^M, W^{\epsilon, M}))$ be a set of linear maps $\Gamma(Discs^M, TDiscs^M)\rightarrow \Gamma(Discs^M, W^{\epsilon, M})$.

The term $\mathcal{L}_{\frac{1}{2z}\partial_z}$ in the definition of $\nabla$ naturally arises from the deformation of discs via the variation of $a$ keeping the coordinate $x$ of $\mathbb{C}^2$ constant. This is because $x = z^2 + a$. Therefore, the connection $\nabla$ of $W^{\epsilon, M}$ is indeed the local version of $\nabla_{GM}$ on $\mathcal{H}$.
\end{remark}

Let us examine the operator $\mathcal{L}_{\frac{1}{2z}\partial_z}$ more closely. In general, we have that $\mathcal{L}_{h(z)\partial_z}\in End(W^{\epsilon,M}_t) = End(W^{\epsilon,M}_{Airy})$ for any $h(z) \in z^{-1}\mathbb{C}[[z^{-1}]]\oplus \mathbb{C}[[z]]$ which converges absolutely for $\bar{\epsilon} < |z| < \bar{M}$ for some $\bar{\epsilon} < \epsilon < M < \bar{M}$. To see this, let $\xi = f(z)\frac{dz}{z} \in W^{\epsilon,M}_{Airy}$ be given such that $f(z)$ converges absolutely for $\bar{\bar{\epsilon}} < |z| < \bar{\bar{M}}$ for some $\bar{\bar{\epsilon}} < \epsilon < M < \bar{\bar{M}}$. Then the Lie derivative
\begin{equation}\label{liederivativeexample}
    \mathcal{L}_{h(z)\partial_z}\xi(z) = (d\iota_{h(z)\partial_z} + \iota_{h(z)\partial_z}d)(f(z)dz) = d(h(z)f(z))
\end{equation}
is holomorphic for $\max(\bar{\epsilon},\bar{\bar{\epsilon}}) < |z| < \min(\bar{M}, \bar{\bar{M}})$. The Laurent series expansion of $\mathcal{L}_{h(z)\partial_z}\xi(z)$ also will not have the $\frac{dz}{z}$ term because $\mathcal{L}_{h(z)\partial_z}\xi(z) = d(h(z)f(z))$ is a total differential and we have $\oint_{|z| = \epsilon}\mathcal{L}_{h(z)\partial_z}\xi(z) = 0$. This shows us that $\mathcal{L}_{h(z)\partial_z}\xi(z) \in W^{\epsilon,M}_{Airy}$.

For any $a \in \mathbb{C}$, we also define an exponential operator $\exp(a\mathcal{L}_{h(z)\partial_z})$ by
\begin{equation}\label{expoperatordefn}
    \exp(a\mathcal{L}_{h(z)\partial_z})(\xi) := \sum_{k=0}^\infty \frac{a^k}{k!}(\mathcal{L}_{h(z)\partial_z})^k(\xi), \qquad \xi \in W^{\epsilon,M}_{Airy}.
\end{equation}
The operator $\exp(a\mathcal{L}_{h(z)\partial_z})$ can be interpreted as a flow after `time' $a \in \mathbb{C}$ generated by $\mathcal{L}_{h(z)\partial_z}$.
Note that although $\mathcal{L}_{h(z)\partial_z} \in End(W^{\epsilon,M}_{Airy})$, its exponential $\exp(a\mathcal{L}_{h(z)\partial_z})$ will not be in $End(W^{\epsilon,M}_{Airy})$ in general. For instance, the flow $\exp(a\mathcal{L}_{h(z)\partial_z})\xi$ of $\xi \in W^{\epsilon,M}_{Airy}$ can fail to have the desired convergence properties and `fall out of' $W^{\epsilon,M}_{Airy}$. Let us consider the case where $h(z) := z^k, k \in \mathbb{Z}$ in the following

\begin{lemma}\label{expoperatorlemma}
    Given $\xi = f(z)dg(z) \in W^{\epsilon,M}_{Airy}$ converges absolutely for $\bar{\epsilon} < |z| < \bar{M}$ where $\bar{\epsilon} < \epsilon < M < \bar{M}$. Then for $a \in \mathbb{C}$ with a sufficiently small $|a|$:
    \begin{equation*}
        |a| < \begin{cases}
        \displaystyle\min\left(\frac{\epsilon^{1-k} - \bar{\epsilon}^{1-k}}{1-k},\frac{\bar{M}^{1-k} - M^{1-k}}{1-k}\right), &\qquad k \neq 1, \\
        \displaystyle\min\left(\log\frac{\epsilon}{\bar{\epsilon}}, \log\frac{\bar{M}}{M}\right), &\qquad k = 1
        \end{cases}
    \end{equation*}
    we have
    \begin{equation*}
        \exp(a\mathcal{L}_{z^k\partial_z})\xi(z) = \left\{\begin{aligned}
        &f\left((z^{1-k} + (1-k)a)^{\frac{1}{1- k}}\right)dg\left((z^{1-k} + (1-k)a)^{\frac{1}{1-k}}\right) &, k \neq 1\\
        &f(e^az)dg(e^az) &, k = 1
        \end{aligned}\right\} \in W^{\epsilon,M}_{Airy}.
    \end{equation*}
    In fact, $\exp(a\mathcal{L}_{z^k\partial_z})\xi(z)$ converges absolutely for 
    \begin{equation*}
        \left\{\begin{aligned} (\bar{\epsilon}^{1-k} + (1-k)|a|)^{\frac{1}{1-k}} &< |z| < (\bar{M}^{1 - k} - (1-k)|a|)^{\frac{1}{1-k}} &, k \neq 1\\
        \bar{\epsilon}|e^a|^{-1} &< |z| < \bar{M}|e^a|^{-1} &, k = 1
        \end{aligned}\right\}.
    \end{equation*}
\end{lemma}
\begin{proof}
For $k \neq 1$ we consider a $z$-plane as a $|k-1|$-fold cover of a $w$-plane where $w := z^{1-k}$. We have $z^k\partial_z = (1-k)\partial_w$ and $\mathcal{L}_{z^k\partial_z} = \mathcal{L}_{(k-1)\partial_w} = (k-1)\mathcal{L}_{\partial_w}$. For any holomorphic functions $f_1 = f_1(w),f_2 = f_2(w)$ we have $\mathcal{L}_{\partial_w}(f_1(w)df_2(w)) = (\mathcal{L}_{\partial_w}f_1(w))df_2(w) + f_1(w)d(\mathcal{L}_{\partial_w}f_2(w)) = (\partial_w f_1(w))df_2(w) + f_1(w)d(\partial_wf_2(w))$. We recognize $\exp((k-1)a\mathcal{L}_{\partial_w})$ as the shift operator $\exp(\mathcal{L}_{(1-k)a\partial_w})(f_1(w)df_2(w)) = f_1(w + (1-k)a)df_2(w + (1-k)a)$. Choose a branch for $z$ and write $\xi = f(z)dg(z)$ as a differential form on $w$-plane $\xi = f(w^{\frac{1}{1 - k}})dg(w^{\frac{1}{1 - k}})$, then
\begin{align*}
    \exp(a\mathcal{L}_{z^k\partial_z})(f(z)dg(z)) &= f\left((w + (1-k)a)^\frac{1}{1-k}\right)dg\left((w + (1-k)a)^\frac{1}{1-k}\right) \\
    &= f\left((z^{1-k} + (1-k)a)^\frac{1}{1-k}\right)dg\left((z^{1-k} + (1-k)a)^\frac{1}{1-k}\right) =: f_a(z)dg_a(z).
\end{align*}
We have that $f_a(z)dg_a(z)$ would be holomorphic on the chosen branch-cut whenever $\bar{\epsilon}<|(z^{1-k}+(1-k)a)^{\frac{1}{1-k}}| < \bar{M}$ or equivalently $(\bar{\epsilon}^{1-k} + (1 - k)|a|)^{\frac{1}{1-k}} < |z| < (\bar{M}^{1 - k} - (1 - k)|a|)^{\frac{1}{1-k}}$. The condition $|a| < \min\left(\frac{\epsilon^{1-k} - \bar{\epsilon}^{1-k}}{1-k},\frac{\bar{M}^{1-k} - M^{1-k}}{1-k}\right)$ ensures that $(\bar{\epsilon}^{1-k} + (1-k)|a|)^{\frac{1}{1-k}} < \epsilon$ and $(\bar{M}^{1 - k} - (1-k)|a|)^{\frac{1}{1-k}} > M$. On the other hand, $(z^{1-k} + (1 - k)a)^{\frac{1}{1-k}} = z(1 + (1-k)az^{k-1})^{\frac{1}{1-k}}$ is a well-defined function and can be expanded as a Laurent series converges in $z$ for $(|1-k||a|)^\frac{1}{1-k} < |z|$ if $1 - k > 0$ and for $|z| < (|1-k||a|)^\frac{1}{1-k}$ if $1 - k < 0$. These conditions automatically follow if $(\bar{\epsilon}^{1-k} + (1-k)|a|)^{\frac{1}{1-k}} < |z| < (\bar{M}^{1 - k} - (1-k)|a|)^{\frac{1}{1-k}}$ since $(|1-k||a|)^{\frac{1}{1-k}} < (\bar{\epsilon}^{1-k} + (1-k)|a|)^{\frac{1}{1-k}}$ when $1 - k > 0$ and $(|1-k||a|)^{\frac{1}{1-k}} > (\bar{M}^{1-k} - (1-k)|a|)^{\frac{1}{1-k}}$ when $1 - k < 0$.

Since the series $\xi \in W^{\epsilon,M}_{Airy}$ doesn't contain the term $\frac{dz}{z}$, by substituting $z = w^{\frac{1}{1-k}}$ we also do not have the term $\frac{1}{1-k}\frac{dw}{w}$ in $f(w^{\frac{1}{1-k}})dg(w^{\frac{1}{1-k}})$. The operator $\exp(a\mathcal{L}_{z^k\partial_z}) = \exp((k-1)a\mathcal{L}_{\partial_w})$ shifts $w$ to $w + (k-1)a$ and will not create the term $\frac{dw}{w}$ therefore, we do not have the term $\frac{dz}{z}$ in the Laurent series of $\exp(a\mathcal{L}_{z^k\partial_z})\xi(z)$.

For $k = 1$, we note that $(\mathcal{L}_{z\partial_z})^k(z) = z$ and $\exp(a\mathcal{L}_{z\partial_z})(\xi(z)) = f(e^az)dg(e^az)$ follows directly for (\ref{expoperatordefn}). Then we can see that if $|a| < \min\left(\log\frac{\epsilon}{\bar{\epsilon}}. \log\frac{\bar{M}}{M}\right)$ we have $\bar{\epsilon} < |e^az| < \bar{M}$ whenever $\bar{\epsilon}|e^a|^{-1} < \epsilon < |z| < M < \bar{M}|e^a|^{-1}$
\end{proof}

As an immediate consequence of Lemma \ref{expoperatorlemma} we have the following useful statement:
\begin{corollary}\label{expoperatorcorollary}
Let $k\neq 1$, $a \in \mathbb{C}$ and define
 \begin{equation*}
    \epsilon(a) := (\epsilon^{1-k} - (1-k)|a|)^{\frac{1}{1-k}}, \qquad M(a) := (M^{1 - k} + (1-k)|a|)^{\frac{1}{1-k}}.
\end{equation*}
Then we have $\exp\left(a\mathcal{L}_{z^k\partial_z}\right) \in Hom_{\mathbb{C}}\left(W^{\epsilon(a), M(a)}_{Airy}, W^{\epsilon, M}_{Airy}\right)$.
\end{corollary}
\begin{lemma}\label{liederivativesplemma}
Let $k \in \mathbb{Z}$ then $\mathcal{L}_{z^k\partial_z} \in sp(W^{\epsilon,M}_{Airy})$. 
Moreover, if $\xi_1,\xi_2 \in W^{\epsilon,M}_{Airy}$ and $a \in \mathbb{C}$ are given such that $\exp(a\mathcal{L}_{z^k\partial_z})\xi_1, \exp(a\mathcal{L}_{z^k\partial_z})\xi_2 \in W^{\epsilon,M}_{Airy}$ then
\begin{equation*}
    \Omega_{Airy}\left(\exp(a\mathcal{L}_{z^k\partial_z})\xi_1, \exp(a\mathcal{L}_{z^k\partial_z})\xi_2\right) = \Omega_{Airy}(\xi_1,\xi_2).
\end{equation*}
\end{lemma}
\begin{proof}
Using (\ref{liederivativeexample}) it is straightforward to check that
\begin{align*}
    \Omega_{Airy}\left(\mathcal{L}_{z^k\partial_z}z^m\frac{dz}{z}, z^n\frac{dz}{z}\right) &+ \Omega_{Airy}\left(z^m\frac{dz}{z}, \mathcal{L}_{z^k\partial_z}z^n\frac{dz}{z}\right)\\
    &= \Omega_{Airy}\left((m + k - 1)z^{m+k-1}\frac{dz}{z}, z^n\frac{dz}{z}\right) + \Omega_{Airy}\left(z^m\frac{dz}{z}, (n + k - 1)z^{n+k-1}\frac{dz}{z}\right)\\
    &= \frac{m + k - 1}{n}\delta_{-m-k+1,n} + \delta_{-m,n+k-1} = 0.
\end{align*}
Since differentiation of a power series can be done term-by-term (Remark \ref{laurentseriesgeneralremark}), it follows that 
\begin{equation*}
    \Omega_{Airy}\left(\mathcal{L}_{z^k\partial_z}\xi_1, \xi_2\right) + \Omega_{Airy}\left(\xi_1, \mathcal{L}_{z^k\partial_z}\xi_2\right) = 0
\end{equation*}
for any $\xi_1,\xi_2 \in W^{\epsilon, M}_{Airy}$.
Therefore $\mathcal{L}_{z^k\partial_z} \in sp(W^{\epsilon,M}_{Airy})$, proving the first part of the Lemma. Using Lemma \ref{expoperatorlemma} we check that $\frac{d}{da}\exp(a\mathcal{L}_{z^k\partial_z})\xi_i = \mathcal{L}_{z^k\partial_z}\exp(a\mathcal{L}_{z^k\partial_z})\xi_i, i = 1,2$. It follows that
\begin{align*}
    \frac{d}{da}\Omega_{Airy}&\left(\exp(a\mathcal{L}_{z^k\partial_z})\xi_1, \exp(a\mathcal{L}_{z^k\partial_z})\xi_2\right)\\
    &= \Omega_{Airy}\left(\mathcal{L}_{z^k\partial_z}\exp(a\mathcal{L}_{z^k\partial_z})\xi_1, \exp(a\mathcal{L}_{z^k\partial_z})\xi_2\right) + \Omega_{Airy}\left(\exp(a\mathcal{L}_{z^k\partial_z})\xi_1, \mathcal{L}_{z^k\partial_z}\exp(a\mathcal{L}_{z^k\partial_z})\xi_2\right)\\
    &= 0.
\end{align*}
Which means $\Omega_{Airy}\left(\exp(a\mathcal{L}_{z^k\partial_z})\xi_1, \exp(a\mathcal{L}_{z^k\partial_z})\xi_2\right)$ is constant and must equal to its value at $a = 0$ which is $\Omega_{Airy}(\xi_1,\xi_2)$.
\end{proof}

For this section, we will only need $\mathcal{L}_{z^k\partial_z}$, and its exponentiation when $k = -1$. According to Lemma \ref{expoperatorlemma}, given $\xi \in W^{\epsilon,M}_{Airy}$ converges absolutely for $\bar{\epsilon} < |z| < \bar{M}, \bar{\epsilon} < \epsilon < M < \bar{M}$ and given $a \in \mathbb{C}, |a| < \min(\epsilon^2 - \bar{\epsilon}^2, \bar{M}^2 - M^2, \bar{\epsilon}^2, \bar{M}^2)$ we have $\exp(a\mathcal{L}_{\frac{1}{2z}\partial_z})\xi(z) = f(\sqrt{z^2 + a})dg(\sqrt{z^2 + a}) \in W^{\epsilon,M}_{Airy}$ since it is holomorphic for $\sqrt{\bar{\epsilon}^2 + |a|} < |z| < \sqrt{\bar{M}^2 - |a|}$. In particular, it is straightforward to check that
\begin{equation}
    \exp(a\mathcal{L}_{\frac{1}{2z}\partial_z})(z) = z + \frac{a}{2z} + \sum_{k = 2}^\infty \frac{a^k}{k!}\frac{(-1)^{k-1}(2k-3)!!}{2^{k}z^{2k-1}},
\end{equation}
which can be recognized as a Laurent series expansion of $\sqrt{z^2 + a} = z\sqrt{1 + \frac{a}{z^2}}$ for $|z| > \epsilon > \sqrt{|a|}$.

We will now show that the vector bundle $(W^{\epsilon, M}\rightarrow Discs^M, \Omega, \nabla)$ is flat and the symplectic form $\Omega$ is $\nabla$-covariantly constant.

\begin{lemma}\label{wisflatlemma}
The vector bundle $W^{\epsilon,M} \rightarrow Discs^M$ with the connection $\nabla$ is flat.
\end{lemma}
\begin{proof}
We define the exterior covariant derivative 
\begin{equation*}
    d_{\nabla} : \Gamma\left(Discs^M, Hom\left(\bigwedge^rTDiscs^M, W^{\epsilon,M}\right)\right) \rightarrow \Gamma\left(Discs^M, Hom\left(\bigwedge^{r+1}TDiscs^M, W^{\epsilon,M}\right)\right)
\end{equation*}
by
\begin{equation*}
    d_\nabla \psi := (ds)\xi + (-1)^rs\wedge \nabla(\xi) = (ds)\xi - (-1)^rs\wedge da\mathcal{L}_{\frac{1}{2z}\partial_z}(\xi)
\end{equation*}
where $\psi = s\xi, s \in \Gamma(Discs^M, \Omega^r_{Discs^M})$ and $\xi \in W^{\epsilon,M}_{Airy}$, then extend by linearity. We note that $d_\nabla = \nabla$ for $r = 0$.
The condition that $\nabla$ is flat is the same as the vanishing of curvature tensor $R_\nabla = d_\nabla\circ d_\nabla = 0$ and it is easy to check that
\begin{equation*}
    d^2_\nabla \psi = (d^2s)\xi - (-1)^{r+1}(ds)\wedge da\mathcal{L}_{\frac{1}{2z}\partial_z}(\xi) - (-1)^rds\wedge da \mathcal{L}_{\frac{1}{2z}\partial_z}(\xi) + s\wedge da\wedge da (\mathcal{L}_{\frac{1}{2z}\partial_z})^2(\xi) = 0.
\end{equation*}
This concludes that $\nabla$ is flat.
\end{proof}

\begin{lemma}\label{paralleltransportinwlemma}
Pick $t_1, t_2 \in Discs^M$ such that $|a(t_1) - a(t_2)|$ is sufficiently small, then
the parallel transport of $\xi \in W^{\epsilon,M}_{t_1}$ to the fiber $W^{\epsilon,M}_{t_2}$ exists, path-independent and given by $\xi_{t_2} = \exp((a(t_2) - a(t_1))\mathcal{L}_{\frac{1}{2z}\partial_z})\xi$. In fact, $\xi_t = \exp((a(t) - a(t_1))\mathcal{L}_{\frac{1}{2z}\partial_z})(\xi)$ is a parallel vector field.
\end{lemma}
\begin{proof}
The path-independence of the parallel transport follows from the fact that $\nabla$ is flat. The condition that $|a(t_1) - a(t_2)|$ being sufficiently small is needed to ensure that $\exp((a(t_2)-a(t_1))\mathcal{L}_{\frac{1}{2z}\partial_z})(\xi) \in W^{\epsilon,M}_{t_2}$ according to Lemma \ref{expoperatorlemma}.
Finally, we can check that 
\begin{equation*}
    \nabla \xi_t = da\frac{\partial}{\partial a}\exp((a(t) - a(t_1))\mathcal{L}_{\frac{1}{2z}\partial_z})\xi - da\mathcal{L}_{\frac{1}{2z}\partial_z}\exp((a(t)-a(t_1))\mathcal{L}_{\frac{1}{2z}\partial_z})(\xi) = 0.
\end{equation*}
\end{proof}

\begin{lemma}\label{symplecticformiscovariantlyconstantlemma}
The symplectic form $\Omega$ is $\nabla$-covariantly constant.
\end{lemma}
\begin{proof}
For any $\xi_1, \eta_1 \in W^{\epsilon,M}_{t_1}$ we have from Lemma \ref{paralleltransportinwlemma} that their parallel transport to $W^{\epsilon,M}_{t_2}$ are $\xi_2 = \exp((a(t_2) - a(t_1))\mathcal{L}_{\frac{1}{2z}\partial_z})(\xi_1), \eta_2 = \exp((a(t_2) - a(t_1))\mathcal{L}_{\frac{1}{2z}\partial_z})(\eta_1) \in W^{\epsilon,M}_{t_2}$. Therefore, $\Omega_{t_2}(\xi_2, \eta_2) = \Omega_{t_1}(\xi_1,\eta_1)$ by Lemma \ref{liederivativesplemma}.
\end{proof}

\subsubsection{The analytic residue constraints}\label{analyticresconstraintssubsection}
Observe that the residue constraints Hamiltonians:
\begin{align*}
    (H_{Airy})_{2n}(w) &= Res_{z = 0}\left(\left(z - \frac{w(z)}{2zdz}\right) z^{2n}d(z^2)\right),\\
    (H_{Airy})_{2n-1}(w) &= \frac{1}{2}Res_{z = 0}\left(\left(z - \frac{w(z)}{2zdz}\right)^2 z^{2n-2}d(z^2)\right),\qquad n = 1,2,3,...
\end{align*}
as given in Section \ref{residueconstraintairystructuresection} is not only well-defined when $w(z) \in W_{Airy}$, but also when $w(z) \in W^{\epsilon, M}_{Airy}$ is a Laurent series converging on some annulus containing $\epsilon < |z| < M$. Define $L_{Airy}^M := \{(H_{Airy})_{i=1,2,3,\cdots} = 0\} \subset W_{Airy}^{\epsilon, M}$. As we will see later, unlike its counter-part $L_{Airy} \subset W_{Airy}$ (see Lemma \ref{lairyistrivialinwairylemma}), the submanifold $L_{Airy}^M$ contains many non-trivial solutions of $(H_{Airy})_{i=1,2,3,\cdots} = 0$ because an infinite number of terms with negative power in $z$ is allowed. Define the tangent space to $L_{Airy}^M$ at any point $p \in L_{Airy}^M$ by
\begin{equation*}
    T_pL_{Airy}^M := \left\{w \in W^{\epsilon, M}_{Airy}\ |\ \iota_wd(H_{Airy})_i = 0, i = 1,2,3,\cdots\right\} \subset W^{\epsilon, M}_{Airy}.
\end{equation*}
In particular,
\begin{equation*}
    T_0L_{Airy}^M := \left\{\xi \in z\mathbb{C}[[z]]\frac{dz}{z}\ |\ \begin{array}{c}
    \text{$\xi(z)$ converges for $|z| < \bar{M}$}\\ 
    \text{for some $\bar{M} > M$}\end{array} \right\} \subset W^{\epsilon, M}_{Airy}, \qquad T_0L^M_{Airy} \subset T_0L_{Airy} \cong V^*_{Airy}.
\end{equation*}
is a Lagrangian subspace of $W^{\epsilon, M}_{Airy}$. Let us also define 
\begin{equation*}
    V_{Airy}^\epsilon := \left\{\xi \in z^{-1}\mathbb{C}[[z^{-1}]]\frac{dz}{z}\ |\ \begin{array}{c}
    \text{$\xi(z)$ converges for $|z| > \bar{\epsilon}$}\\ 
    \text{for some $\bar{\epsilon} < \epsilon$}\end{array}\right\} \subset W^{\epsilon, M}_{Airy}\qquad V_{Airy} \subset V^\epsilon_{Airy}.
\end{equation*}
$V^\epsilon_{Airy}$ is also a Lagrangian subspace of $W^{\epsilon, M}_{Airy}$, in fact, it is a Lagrangian complement of $T_0L^M_{Airy}$ : $W^{\epsilon, M}_{Airy} = V^\epsilon_{Airy} \oplus T_0L^M_{Airy}$. This is similar to $W_{Airy} = V_{Airy}\oplus V_{Airy}^*$, but in this case, $T_0L_{Airy}^M \subsetneq (V^\epsilon_{Airy})^*$. We let $\{e^{k=1,2,3,\cdots}, f_{k=1,2,3,\cdots}\}$ to be a canonical basis of $W^{\epsilon, M}_{Airy}$:
\begin{equation*}
    \Omega_{Airy}(e^i,e^j) = \Omega_{Airy}(f_i,f_j) = 0, \qquad \Omega_{Airy}(e^i,f_j) = \delta^i_j,
\end{equation*}
where $\{e^{k = 1,2,3,\cdots} := z^{-k}\frac{dz}{z}\}$ is a basis of $V_{Airy}^\epsilon$ and $\{f_{k= 1,2,3,\cdots} := kz^k\frac{dz}{z}\}$ is a basis of $T_0L_{Airy}^M$.

Among many topics we have discussed regarding Airy structures in Section \ref{airystructuressection}, we will mostly be interested in the genus zero ATR output and the gauge transformation in the context of the quantum residue constraints Airy structure introduced in Section \ref{residueconstraintairystructuresection}. Since we only have considered these topics from the formal perspective in Section \ref{airystructuressection} and \ref{residueconstraintairystructuresection}, the goal of this section is to revisit these topics from an analytical perspective within the context of $\{(H_{Airy})_{i=1,2,3,\cdots}\}$. We will show that most things we have learned previously continue to hold, possibly with some minor adjustments. Most difficulties came from the following: In $W_{Airy}$, the Tate topology dictates that the contraction between vectors in $V_{Airy}$ and $V^*_{Airy}$ always result in finite sums since vectors in $V_{Airy}$ are finite linear combinations of $\{e^{i\in \mathbb{I}}\}$. On the other hand, $W^{\epsilon, M}_{Airy}$ is not a Tate space and $w(z) \in W^{\epsilon, M}_{Airy}$ in general can be written as $w(z) = \sum_{k=1}^\infty\left(y_kz^{-k}\frac{dz}{z} + kx^kz^k\frac{dz}{z}\right)$ with infinitely many non-zero $x^k$ and $y_k$. It turns out that the lack of finiteness can be replaced with absolute uniform convergence. Here are summaries of the key points:
\begin{enumerate}
\item By expressing $(H_{Airy})_i$ in the form (\ref{airystructuredefn}), we get multiple terms with double infinite summation which could lead to some issues such as divergences and ambiguity of summation order. It turns out that there is no problem because each $(H_{Airy})_i$ is given by the residue of some convergence Laurent series, hence each double infinite summation actually converges absolutely (see Lemma \ref{standardformanalyticresconstairystrlemma}).
\item In Section \ref{gaugetransformationsubsection}, gauge transformations are given by $(c^i_j)$, its inverse $(d^i_j)$, and $(s^{ij})$ where $c^i_j$ is non-zero for only finitely many $j = 1,2,3,\cdots$ for each fixed $i$. In the analytic setting, we need to check that canonical coordinates transformation (\ref{xycanonicalcoordtransformation0}) follows from the canonical basis transformation (\ref{efcanonicalbasistransformation0}). In Section \ref{gaugetransformationsubsection} this is easy because there are only finitely many terms in each summation, but now we have infinite sums and getting (\ref{xycanonicalcoordtransformation0}) from (\ref{efcanonicalbasistransformation0}) requires interchanging the orders of these infinite sums. Moreover, we also need to make sure that the summation of the type $\sum_{i=1}^\infty c^i_jy_i$ which appears in (\ref{xycanonicalcoordtransformation0}) will converge, as $y_i$ may be non-zero for an infinitely many $i=1,2,3,\cdots$. We will see that if $(c^i_j), (d^i_j)$ and $(s^{ij})$ are further restricted to those satisfying Condition \ref{analyticcdscondition}, then all of these concerns will be taken care of. In particular, expressing $(H_{Airy})_i$ in the transformed canonical coordinates will give us exactly the same gauge transformation of $A_{Airy}, B_{Airy}, C_{Airy}$ as given in Section \ref{gaugetransformationsubsection} (see Proposition \ref{analyticgaugetransformationproposition}).
\item We discussed in Section \ref{airystructuressection} that $(H_{Airy})_k(\{x^{i \in \mathbb{I}}\}, \{y_{i\in \mathbb{I}} := \partial S_0/\partial x^i\}) = 0$ where $S_0 \in \prod_{n=1}^\infty Sym_n(V_{Airy})$ is the genus zero output of the ATR using $(V_{Airy}, A_{Airy}, B_{Airy}, C_{Airy}, \epsilon_{Airy})$. Even though $S_0$ was introduced as a formal series, we will see in Proposition \ref{analyticatrproposition} that it can be interpreted analytically and that its derivatives are relevant to the solution of analytic residue constraints.
\end{enumerate}
For the remainder of this section, we present the technical detail leading to each point above. It is possible to skip these detail without losing any continuity to Section \ref{embeddingofdiscssubsection} where we will resume to discuss the embedding map $\Phi_{t_0} : Discs^{\epsilon, M}_{t_0}\rightarrow W^{\epsilon, M}_{t_0}$.

\begin{lemma}\label{standardformanalyticresconstairystrlemma}
Let $w(z) := \sum_{k=1}^\infty(y_ke^k(z) + x^kf_k(z)) = \sum_{k=1}^\infty\left(y_kz^{-k}\frac{dz}{z} + kx^kz^k\frac{dz}{z}\right) \in W^{\epsilon, M}_{Airy}$. Then
\begin{equation*}
    (H_{Airy})_i = -y_i + \sum_{j,k = 1}^\infty(a_{Airy})_{ijk}x^jx^k + 2\sum_{j,k = 1}^\infty (b_{Airy})_{ij}^kx^jy_k + \sum_{j,k = 1}^\infty (c_{Airy})^{jk}_iy_jy_k,
\end{equation*}
where $A_{Airy} = ((a_{Airy})_{ijk}), B_{Airy} = ((b_{Airy})_{ij}^k), C_{Airy} = ((c_{Airy})_i^{jk})$ are given by (\ref{explicitaairy})-(\ref{explicitcairy}) and all the summations converge absolutely. 
\end{lemma}
\begin{proof}
Note that there are possibly infinitely many non-zero $y_{k=1,2,3,\cdots}$ and $x^{k=1,2,3,\cdots}$ therefore each $\sum_{j,k = 1}^\infty$ are double infinite summations. To check that each double infinite summations in $(H_{Airy})_i$ converges absolutely we note that $(H_{Airy})_{2n} = -y_{2n}$ and
\begin{equation}\label{resconstrainsairystrintexpression}
    (H_{Airy})_{2n-1}(w) = \frac{1}{4\pi i}\oint_{|z| = \epsilon}\left(\left(z - \frac{\sum_{k=1}^\infty\left(y_ke^k(z) + x^kf_k(z)\right)}{2zdz}\right)^2z^{2n-2}d(z^2)\right)
\end{equation}
where $\sum_{k=1}^\infty(y_ke^k(z) + x^kf_k(z)) = \sum_{k=1}^\infty\left(y_kz^{-k}\frac{dz}{z} + kx^kz^k\frac{dz}{z}\right)$ converges absolutely uniformly for $\epsilon < |z| < M$. Adding or multiplying convergence Laurent series a finite number of times will produce another convergence Laurent series. We can expand (\ref{resconstrainsairystrintexpression}) and examine each term more closely. For example
\begin{align*}
     \frac{1}{2\pi i}\oint_{|z| = \epsilon}&\left(-\frac{1}{4i}\frac{1}{z(dz)^2}f_i(z)\left(\sum_{j=1}^\infty x^jf_j(z)\right)\left(\sum_{k=1}^\infty y_ke^k(z)\right)\right)\delta_{i,odd}\\
    &= \sum_{j,k=1}^\infty Res_{z=0}\left(-\frac{1}{4i}\frac{1}{z(dz)^2}f_i(z)f_j(z)e^k(z)\right)\delta_{i, odd}x^jy_k = \sum_{j,k=1}^\infty(b_{Airy})^k_{ij}x^jy_k,
\end{align*}
where we have used Remark \ref{residueconstaintsabcremark} to identify the residue with $(b_{Airy})^k_{ij}$. Note that $\sum_{j=1}^\infty\sum_{k=1}^\infty x^jy_kf_j(z)e^k(z)$ converges absolutely uniformly for $\epsilon < |z| < M$ as it is the product of $\sum_{j=1}^\infty x^jf_j(z)$ and $\sum_{k=1}^\infty y_ke^k(z)$. Therefore term-by-term integration is allowed and $\sum_{j,k=1}^\infty (b_{Airy})^k_{ij}x^jy_k$ converges absolutely. Similarly, the sum $\sum_{j,k = 1}^\infty (a_{Airy})_{ijk}x^jx^k$ and $\sum_{j,k=1}^\infty(c_{Airy})^{jk}_iy_jy_k$ also converges absolutely.
\end{proof}

\begin{remark}
Note that $-\sum_{i=1}^\infty\frac{\delta_{i,odd}}{4i}\frac{1}{z(dz)^2}f_i(z)e^i(z_1) = \frac{1}{4z(z^2 - z^2_1)}\frac{dz_1}{dz} =: K(z_1,z)$. It follows that
\begin{equation*}
    \sum_{i,j,k=1}^\infty(b_{Airy})^k_{ij}x^jy_ke^i(z_1) = \left(\frac{1}{2\pi i}\oint_{|z| = \bar{\epsilon}}\frac{\left(\sum_{j=1}^\infty x^jf_j(z)\right)\left(\sum_{k=1}^\infty y_ke^k(z)\right)}{4z(z^2 - z_1^2)dz}\right)dz_1 \in V^\epsilon_{Airy}.
\end{equation*}
Similarly, we have $\sum_{i,j,k=1}^\infty (a_{Airy})_{ijk}x^jx^ke^i(z_1), \sum_{i,j,k=1}^\infty (c_{Airy})_{i}^{jk}y_jy_ke^i(z_1) \in V^\epsilon_{Airy}$.
Therefore, we obtain the analytic analogues to the definition of tensors $A_{Airy},B_{Airy},C_{Airy}$:
\begin{equation*}
    A_{Airy} : T_0L_{Airy}^M\otimes T_0L_{Airy}^M \rightarrow V_{Airy}^\epsilon, \qquad B_{Airy} : T_0L_{Airy}^M\otimes V_{Airy}^\epsilon \rightarrow V_{Airy}^\epsilon,\qquad C_{Airy} : V^\epsilon_{Airy}\otimes V^\epsilon_{Airy}\rightarrow V^\epsilon_{Airy}.
\end{equation*}
We can recognise $K(z_1,z)$ to be the recursive kernel in the topological recursion of an Airy curve $y^2 = x$ with $\omega_{0,2}(z_1,z_2) = \frac{dz_1dz_2}{(z_1 - z_2)^2}$ and $\omega_{0,1}(z) = ydx$.
\end{remark}

Let us examine the analytic version of the gauge transformation introduced in Section \ref{gaugetransformationsubsection}. Consider the change of a canonical basis $\{e^{k=1,2,3,\cdots}, f_{k=1,2,3,\cdots}\}\mapsto \{\bar{e}^{k=1,2,3,\cdots}, \bar{f}_{k=1,2,3,\cdots}\}$ of $W^{\epsilon, M}_{Airy}$ where
\begin{equation}\label{efcanonicalbasistransformation}
    \bar{e}^i := \sum_{j=1}^\infty d^i_j\left(e^j 
    + \sum_{k=1}^\infty s^{jk}f_k\right), \qquad \bar{f}_i := \sum_{j=1}^\infty c^j_if_j,
\end{equation}
where $(c^j_i), (d^j_i)$ and $(s^{ij})$ are tensors satisfying the following conditions:
\begin{condition}\label{analyticcdscondition}~
\begin{description}
    \item[C\ref{analyticcdscondition}.1] $(c^j_i)$ and $(d^j_i)$ are inverse of each other: $\sum_{k=1}^\infty c^k_id^j_k = \delta_i^j$ and $\sum_{k=1}^\infty d^k_ic^j_k = \delta_i^j$.
    \item[C\ref{analyticcdscondition}.2] $\sum_{j=1}^\infty c^j_if_j = \sum_{j=1}^\infty jc^j_iz^j\frac{dz}{z} \in T_0L_{Airy}^M$ and $\sum_{j=1}^\infty d^j_if_j = \sum_{j=1}^\infty jd^j_iz^j\frac{dz}{z} \in T_0L^M_{Airy}$. 
    \item[C\ref{analyticcdscondition}.3] There exists $I \in \mathbb{Z}_{>0}$ such that for all $j = 1,2,3,\cdots$ we have $c^j_i = d^j_i = \delta^j_i$ for all $i > I$.
    \item[C\ref{analyticcdscondition}.4] $s^{ij}$ is symmetric and $\sum_{i,j=1}^\infty s^{ij}f_i(z_1)f_j(z_2) = \sum_{i,j=1}^\infty ijs^{ij}\frac{dz_1dz_2}{z_1z_2}$ is a power series in two variables $z_1, z_2$ converges for $|z_1|, |z_2| < \bar{M}$ for some $\bar{M} > M$ (hence a holomorphic symmetric bi-differential).
\end{description}
\end{condition}

\begin{remark}
Compare the conditions on $(c^i_j), (d^i_j)$ and $(s^{ij})$ defined here with those in Section \ref{gaugetransformationsubsection}, especially in Remark \ref{cdsremark}. We find that the analytic version of gauge transformation is more restrictive. This is because $V^\epsilon_{Airy}$ is a larger space than $V_{Airy}$, so loosely speaking there are more ways for things to go wrong. In other words, any analytic gauge transformation will be a gauge transformation as considered in Section \ref{gaugetransformationsubsection} but not the other way around. 

Note that, we do not need to specify $\sum_{j=1}^\infty d^i_je^j \in V^\epsilon_{Airy}$ and $\sum_{j=1}^\infty c^i_je^j \in V^\epsilon_{Airy}$ as another condition, because C\ref{analyticcdscondition}.3 ensures this. If one of $(c^i_j)$ or $(d^i_j)$ satisfies C\ref{analyticcdscondition}.2 or C\ref{analyticcdscondition}.3 then so will the other one because they are the inverse of each other. This can be seen using the formula for the inverse of block matrices.
\end{remark}

\begin{lemma}\label{cdsanalyticlemma}
Suppose that $(c^j_i), (d^j_i)$ and $(s^{ij})$ satisfy the Condition \ref{analyticcdscondition}, then we have the following
\begin{enumerate}
    \item Let $\sum_{k=1}^\infty x^kf_k \in T_0L^M_{Airy}$ and $\sum_{k=1}^\infty y_ke^k \in V^\epsilon_{Airy}$ then $\sum_{j=1}^\infty\left(\sum_{i=1}^\infty y_id^i_j\right)e^j, \sum_{j=1}^\infty\left(\sum_{i=1}^\infty y_ic^i_j\right)e^j \in V^\epsilon_{Airy}$ and $\sum_{j=1}^\infty\left(\sum_{i=1}^\infty x^id^j_i\right)f_j, \sum_{j=1}^\infty\left(\sum_{i=1}^\infty x^ic^j_i\right)f_j \in T_0L_{Airy}^M$ converges absolutely uniformly.
    \item Let $\sum_{k=1}^\infty y_ke^k \in V^\epsilon_{Airy}$ and Let $(v^i)$ be an arbitrary vector such that $\sum_{k=1}^\infty y_kv^k$ converges then\\ $\sum_{i=1}^\infty y_i\left(\sum_{j=1}^\infty d^i_jv^j\right) = \sum_{j=1}^\infty\left(\sum_{i=1}^\infty y_i d^i_j\right)v^j$ and $\sum_{i=1}^\infty y_i\left(\sum_{j=1}^\infty c^i_jv^j\right) = \sum_{j=1}^\infty\left(\sum_{i=1}^\infty y_i c^i_j\right)v^j$.
    \item Let $\sum_{k=1}^\infty u_ke^k \in V^\epsilon_{Airy}$ and $\sum_{k=1}^\infty v_ke^k \in V^\epsilon_{Airy}$ then $\sum_{i,j=1}^\infty s^{ij}u_iv_j$ converges absolutely.
    \item Let $\sum_{k=1}^\infty y_ke^k\in V^\epsilon_{Airy}$ then $\sum_{i=1}^\infty y_i\left(\sum_{j=1}^\infty s^{ij}f_j\right) = \sum_{i,j=1}^\infty y_is^{ij}f_j \in T_0L^M_{Airy}$ converges absolutely uniformly.
\end{enumerate}
\end{lemma}
\begin{proof}
\begin{enumerate}
    \item Using C\ref{analyticcdscondition}.2, we have that $\sum_{i=1}^\infty y_id^i_j = \frac{1}{2\pi i}\oint_{|z| = \epsilon}\left(\left(\sum_{i=1}^\infty\frac{y_i}{z^i}\right)\int\left(\sum_{k=1}^\infty d^k_jf_k(z)\right)\right)$ because of the uniform convergences of each series in the integral. Therefore $\sum_{i=1}^\infty y_id^i_j$ converges for all $j = 1,2,3,\cdots$. By C\ref{analyticcdscondition}.3, there exists $J$ such that $d^i_j = \delta^i_j$ for all $j > J$, therefore $\sum_{j=1}^\infty\sum_{i=1}^\infty y_id^i_je^j = \sum_{j=1}^J\left(\sum_{i=1}^\infty y_id^i_j\right)e^j + \sum_{j=J+1}^\infty y_je^j \in V^\epsilon_{Airy}$. The remaining claims follow similarly.
    \item This is a straightforward consequence of C\ref{analyticcdscondition}.3:
    \begin{align*}
        \sum_{i=1}^\infty y_i\left(\sum_{j=1}^\infty d^i_jv^j\right) &= \sum_{i=1}^\infty y_i\left(\sum_{j=1}^Jd^i_jv^j\right) + \sum_{i=1}^\infty y_i\left(\sum_{j=J+1}^\infty\delta^i_jv^j\right) = \sum_{j=1}^J\left(\sum_{i=1}^\infty y_id^i_j\right)v^j + \sum_{i=J+1}^\infty y_iv^i\\
        &= \sum_{j=1}^J\left(\sum_{i=1}^\infty y_id^i_j\right)v^j + \sum_{j=J+1}^\infty \left(\sum_{i=1}^\infty y_id^i_j\right)v^j = \sum_{j=1}^\infty\left(\sum_{i=1}^\infty y_id^i_j\right)v^j.
    \end{align*}
    Similarly, $\sum_{i=1}^\infty y_i\left(\sum_{j=1}^\infty c^i_jv^j\right) = \sum_{j=1}^\infty\left(\sum_{i=1}^\infty y_i c^i_j\right)v^j$.
    \item Note that
    \begin{align}
        \sum_{i,j=1}^\infty s^{ij}u_iv_j = \left(\frac{1}{2\pi i}\right)^2\oint_{|z_1| = \epsilon}\oint_{|z_2| = \epsilon}\left(\sum_{k_1=1}^\infty \frac{u_{k_1}}{k_1z_1^{k_1}}\right)\left(\sum_{i,j=1}^\infty s^{ij}f_i(z_1)f_j(z_2)\right)\left(\sum_{k_2=1}^\infty \frac{v_{k_2}}{k_2z_2^{k_2}}\right).
    \end{align}
    To obtain the first equality, we have used the fact that each series in the integral converges uniformly for $\bar{\epsilon} < |z_1|, |z_2| < \bar{M}$ (see C\ref{analyticcdscondition}.4) for some $\bar{\epsilon} < \epsilon$ and $\bar{M} > M$. By symmetry, the second equality follows.
    \item The proof is mostly identical to the part 3.
\end{enumerate}
\end{proof}

Using Lemma \ref{cdsanalyticlemma} to help compute the symplectic pairings, we can verify that $\{\bar{e}^{k=1,2,3,\cdots}, \bar{f}_{k=1,2,3,\cdots}\}$ as given in (\ref{efcanonicalbasistransformation}) is a canonical basis of $W^{\epsilon, M}_{Airy}$:
\begin{equation*}
    \Omega_{Airy}(\bar{e}^i,\bar{e}^j) = \Omega_{Airy}(\bar{f}_i,\bar{f}_j) = 0, \qquad \Omega_{Airy}(\bar{e}^i,\bar{f}_j) = \delta^i_j.
\end{equation*}

\begin{proposition}\label{analyticgaugetransformationproposition}
Let us consider the change of a canonical basis $\{e^{k=1,2,3,\cdots}, f_{k=1,2,3,\cdots}\}\mapsto \{\bar{e}^{k=1,2,3,\cdots}, \bar{f}_{k=1,2,3,\cdots}\}$ of $W^{\epsilon, M}_{Airy}$ as given in (\ref{efcanonicalbasistransformation}) with $(c^j_i), (d^j_i)$ and $(s^{ij})$ satisfying Condition \ref{analyticcdscondition}. Given $w := \sum_{k=1}^\infty(y_ke^k + x^kf_k) \in W^{\epsilon, M}_{Airy}$, then by setting
\begin{equation}\label{xyanalyticcanonicalcoordinatestransformation}
    \bar{x}^i := \sum_{j=1}^\infty d^i_j\left(x^j 
    - \sum_{k=1}^\infty s^{jk}y_k\right), \qquad \bar{y}_i := \sum_{j=1}^\infty c^j_iy_j,
\end{equation}
we have $\sum_{k=1}^\infty (\bar{y}_k\bar{e}^k + \bar{x}^k\bar{f}_k)$ converges absolutely uniformly to $w$. Let $(\bar{H}_{Airy})_i := \sum_{j=1}^\infty c^j_i(H_{Airy})_j$, then
\begin{equation*}
    (\bar{H}_{Airy})_i = -\bar{y}_i + \sum_{j,k = 1}^\infty(\bar{a}_{Airy})_{ijk}\bar{x}^j\bar{x}^k + 2\sum_{j,k = 1}^\infty (\bar{b}_{Airy})_{ij}^k\bar{x}^j\bar{y}_k + \sum_{j,k = 1}^\infty (\bar{c}_{Airy})^{jk}_i\bar{y}_j\bar{y}_k,
\end{equation*}
where $\bar{A}_{Airy} = ((\bar{a}_{Airy})_{ijk}), \bar{B}_{Airy} = ((\bar{b}_{Airy})_{ij}^k), \bar{C}_{Airy} = ((\bar{c}_{Airy})_i^{jk})$ are given by the gauge transformation of $A_{Airy}, B_{Airy}, C_{Airy}$ according to (\ref{agaugetransfromation})-(\ref{cgaugetransfromation}).
\end{proposition}
\begin{proof}
Substitute $e^i = \sum_{j=1}^\infty\left(c^i_j\bar{e}^j - \sum_{k=1}^\infty\bar{f}_jd^j_ks^{ki}\right)$ and $f_i = \sum_{j=1}^\infty d^j_i\bar{f}_j$ into $\sum_{k=1}^\infty(y_ke^k + x^kf_k)$. Then using Lemma \ref{cdsanalyticlemma} to help interchange infinite summations when needed and look at coefficients of $\bar{e}^{i=1,2,3,\cdots}$, $\bar{f}_{i=1,2,3,\cdots}$, we have (\ref{xyanalyticcanonicalcoordinatestransformation}). The fact that $\sum_{k=1}^\infty (\bar{y}_k\bar{e}^k + \bar{x}^k\bar{f}_k)$ converges absolutely uniformly follows from part 1 and part 4 of Lemma \ref{cdsanalyticlemma}.

Substitute the inverse of (\ref{xyanalyticcanonicalcoordinatestransformation}): $x^i = \sum_{j=1}^\infty\left(c^i_j\bar{x}^j + \sum_{k=1}^\infty\bar{y}_jd^j_ks^{ki}\right)$ and $y_i = \sum_{j=1}^\infty d^j_i\bar{y}_j$ into $(H_{Airy})_i$. Using Lemma \ref{cdsanalyticlemma}, we find that $(H_{Airy})_{2n} = \sum_{j=1}^\infty d^j_{2n}\bar{y}_j$ and
\begin{equation*}
    (H_{Airy})_{2n-1} = \frac{1}{4\pi i}\oint_{|z| = \epsilon}\left(\left(z - \frac{\sum_{k=1}^\infty(\bar{y}_k\bar{e}^k(z) + \bar{x}^k\bar{f}_k(z))}{2zdz}\right)^2 z^{2n-2}d(z^2)\right).
\end{equation*}
From C\ref{analyticcdscondition}.2, $\sum_{j=1}^\infty c^j_i\delta_{j,odd}z^jdz = \sum_{n=1}^\infty c^{2n-1}_iz^{2n-1}dz$ converges absolutely uniformly for $|z| < \bar{M}$ for some $\bar{M} > M$. Therefore, we can commute the summation in $(\bar{H}_{Airy})_i = \sum_{j=1}^\infty c^j_i(H_{Airy})_j$ with the integral $\oint_{|z| = \epsilon}$. Expand $(\bar{H}_{Airy})_i$ into the standard form, the proposition is proven. 
\end{proof}

Let $\bar{V}^\epsilon_{Airy}$ be the vector space containing all $\sum_{k=1}^\infty\bar{y}_k\bar{e}^k \in W^{\epsilon, M}_{Airy}$ such that $\sum_{k=1}^\infty y_ke^k \in V^\epsilon_{Airy}$, where $\bar{y}_i = \sum_{j=1}^\infty c^j_iy_j$. Equivalently by Lemma \ref{cdsanalyticlemma}, we have that $\sum_{k=1}^\infty y_k\bar{e}^k \in \bar{V}^\epsilon_{Airy}$ if and only if $\sum_{k=1}^\infty y_ke^k \in V^\epsilon_{Airy}$. It follows from Proposition \ref{analyticgaugetransformationproposition} that $W^{\epsilon, M}_{Airy} = \bar{V}^\epsilon_{Airy} \oplus T_0L^M_{Airy}$. In particular, $\bar{V}^\epsilon_{Airy}$ is the new Lagrangian complement of $T_0L_{Airy}^M$ corresponding to the new canonical basis $\{\bar{e}^{k=1,2,3,\cdots}, \bar{f}_{k=1,2,3,\cdots}\}$ of $W^{\epsilon, M}_{Airy}$. 

Finally, let us compare the formal solution $L_{Airy}\subset W_{Airy}$ of $(H_{Airy})_{i=1,2,3,\cdots} = 0$ with the analytic solution $L^M_{Airy}\subset W^{\epsilon, M}_{Airy}$. In particular, we would like to understand how to interpret the generating function $S_0$ of $L_{Airy}$ analytically. The following proposition is written for the residue constraints Airy structure $\{(H_{Airy})_{i=1,2,3,\cdots}\}$ to simplify the notations but its statement and the proof presented are equally valid for any of its gauge transformations $\{(\bar{H}_{Airy})_{i=1,2,3,\cdots}\}$ due to Proposition \ref{analyticgaugetransformationproposition}.
\begin{proposition}\label{analyticatrproposition}
Suppose that we have a set of functions $\{y_{i=1,2,3,\cdots} : T_0L^M_{Airy} \rightarrow \mathbb{C}\}$ such that
\begin{equation*}
    \sum_{k=1}^\infty y_k(\{x^{j=1,2,3,\cdots}\})e^k \in V^\epsilon_{Airy},\qquad \text{and} \qquad (H_{Airy})_k(\{x^{i=1,2,3,\cdots}\}, \{y_{i=1,2,3,\cdots}(\{x^{j=1,2,3,\cdots}\})\}) = 0.
\end{equation*}
Suppose further that $\sum_{k=1}^\infty\mathbf{y}_i(x^1,...,x^K)e^k(z)$, where $\mathbf{y}_i(x^1,...,x^K) := y_i(\{x^1,...,x^K,x^{k > K} = 0\})$, is holomorphic in $K+1$ variables $\{x^1,\cdots,x^K,z^{-1}\}$ on some small open neighbourhoods of $(x^1,\cdots ,x^K) = (0,\cdots,0)$ and $|z| > \epsilon$, for any $K \in \mathbb{Z}_{>0}$. Let $S_0 = \sum_{n=1}^\infty S_{0,n}\in \prod_{n=1}^\infty Sym_n(V_{Airy})$ be the formal power series we obtained from the $g=0$ part of the ATR of the quantum Airy structure $(V_{Airy}, A_{Airy}, B_{Airy}, C_{Airy}, \epsilon_{Airy})$. Then the functions $(\partial_iS_0)|_{x^{k>K}=0}, i = 1,2,3,\cdots$ and $\mathcal{S}_0 = \mathcal{S}_0(x^1,\cdots, x^K) := S_0(\{x^1,\cdots,x^K,x^{k>K}=0\})$ are holomorphic on some small open neighbourhoods of $(x^1,\cdots ,x^K) =(0,\cdots,0)$, for any $K \in \mathbb{Z}_{>0}$. In particular, we have $(\partial_iS_0)|_{x^{k>K}=0} = \mathbf{y}_i, i=1,2,3,\cdots$.
\end{proposition}
\begin{proof}
We express $\mathbf{y}_i$ as a power series:
\begin{equation*}
    \mathbf{y}_i(x^1,\cdots,x^K) = \sum_{n=0}^\infty \mathbf{y}^{(n)}_i(x^1,\cdots,x^K), \qquad 
    \mathbf{y}^{(n)}_i(x^1,\cdots,x^K) = \sum_{i_1,\cdots,i_n = 1}^K\frac{1}{n!} \mathbf{y}_{i;i_1,\cdots,i_n}x^{i_1}\cdots x^{i_n}.
\end{equation*}
By assumption, $\sum_{k=1}^\infty\mathbf{y}_k(x^1,\cdots,x^K)e^k(z)\in V^\epsilon_{Airy}$ is a power series in $\{x^1,\cdots,x^K,z^{-1}\}$ which converges absolutely uniformly for $|z| > \epsilon$ and $|x^{i=1,\cdots,K}| < r_K$ for some $r_K > 0$, hence we are allowed to arbitrarily rearrange its terms. According to (\ref{resconstrainsairystrintexpression}), the left-hand-side of
\begin{equation}\label{hairyiszero}
    (H_{Airy})_i\left(\{x^1,\cdots,x^K,x^{k>K} = 0\}, \{\mathbf{y}_{j=1,2,3,\cdots}(x^1,\cdots,x^K)\}\right) = 0.
\end{equation}
is a residue of a simple expression involving $\sum_{k=1}^\infty\mathbf{y}_k(x^1,\cdots,x^K)e^k$, hence we can also freely rearrange its terms. Therefore, we can solve for $\mathbf{y}_i$ from (\ref{hairyiszero}) by substituting in its power series expression and collecting like terms exactly as we have done in Theorem \ref{quantizationofltheorem} for $g_0 = 0$, even though formal power series has been replaced by an absolutely convergent power series. The Equation (\ref{hairyiszero}) must hold for all $|x^{i=1,\cdots, K}| < r_K$, so the coefficient of each monomial $x^{i_1}\cdots x^{i_n}, i_k = 1,\cdots, K$ must vanish independently. In particular, we have $\mathbf{y}_i^{(n_0)} = (H_{Airy})_i^{(0,n_0)}|_{x^{k > K} = 0} = (\partial_iS_{0,n_0+1})|_{x^{k > K} = 0}$ (see (\ref{hhatg0n0})). Therefore, $(\partial_iS_0)|_{x^{k > K} = 0} = \sum_{n=0}^\infty (\partial_iS_{0,n+1})|_{x^{k>K}=0} = \sum_{n=0}^\infty \mathbf{y}_i^{(n)} = \mathbf{y}_i$ is a holomorphic function in $x^1,\cdots, x^K$ for all $i = 1,2,3,\cdots$. Furthermore, for $i = 1,\cdots, K$ we have $\partial_i\mathcal{S}_0 = (\partial_iS_0)|_{x^{k>K}=0}$. Since $\partial_i\mathcal{S}_0 = \mathbf{y}_i$ are holomorphic for $i=1,\cdots, K$, we conclude that the function $\mathcal{S}_0$ is holomorphic for all $(x^1,\cdots,x^K)$ such that $|x^{i=1,\cdots, K}| < r_K$.
\end{proof}

\subsubsection{The embedding of $Discs^{\epsilon, M}_{t_0}$ into $W^{\epsilon, M}_{t_0}$}\label{embeddingofdiscssubsection}
In this section, the procedure we performed in Section \ref{embeddingofbinhsubsection} with be repeated with the covariantly constant weak symplectic flat vector bundle $(W^{\epsilon, M}\rightarrow Discs^M, \Omega, \nabla)$ instead of $(\mathcal{H}\rightarrow \mathcal{B}_{\Sigma_0}, \Omega_{\mathcal{H}}, \nabla_{GM})$.  Following \cite[Section 6.1]{kontsevich2017airy}, we will define
\begin{equation*}
    \phi \in \Gamma(Discs^{\epsilon,M}_{t_0}, Hom(TDiscs_{t_0}^{\epsilon,M}, W^{\epsilon, M})),
\end{equation*}
and use it to construct the map $\Phi_{t_0}: Discs^{\epsilon, M}_{t_0} \rightarrow W^{\epsilon, M}_{t_0}$ embedding the `neighbourhood' $Discs^{\epsilon, M}_{t_0}$ of $t_0 \in Discs^M$ into the fiber $W^{\epsilon, M}_{t_0}$. Finally, we show that the image of the embedding satisfies the analytic residue constraints. In other words, $\text{im}\Phi_{t_0}$ is contained inside $L^M_{Airy} \subset W^{\epsilon, M}_{t_0} \cong W^{\epsilon, M}_{Airy}$ (see Proposition \ref{phiembeddiscsinwproposition}). 

Let us fix the reference point $t_0 := (x = z^2, y = z) \in Discs^M$ and define the subspace $Discs^{\epsilon, M}_{t_0} \subset Discs^M$ containing $t_0$ given by
\begin{equation*}
    Discs^{\epsilon,M}_{t_0} := \left\{t = \left(x = a(t) + z^2, y = \sum_{k=0}^\infty b_k(t)z^k\right) \in Discs^M\ |\ \begin{array}{c}
             \sum_{k=0}^\infty b_k(t)z^k \text{ converges for }\\
             \text{ $|z| < \bar{M}$ for some $\bar{M} > M$ }\\
             \text{and $\min(\bar{M}^2 - M^2, \epsilon^2) > |a(t) - a(t_0)|$}
        \end{array}\right\}.
\end{equation*}
Conceptually, we think of $Discs^{\epsilon, M}_{t_0}$ as a small neighbourhood of $t_0$ inside $Discs^M$ (although, it is not a neighbourhood according to the topology given in Remark \ref{topologyofwremark}). Similar to before, for any $t \in Discs^{\epsilon, M}_{t_0}$, the kinematic tangent space $T_tDiscs^{\epsilon, M}_{t_0}$ is defined to be the set of $v \in Discs^M$ such that $\gamma(\tau) := t + \tau v$ defines a path $\gamma : (-\delta, \delta) \rightarrow Discs^{\epsilon, M}_{t_0}$ for some $\delta > 0$. In other words, 
\begin{equation*}
    T_tDiscs^{\epsilon, M}_{t_0} := \left\{v = A\frac{\partial}{\partial a} + \sum_{k=0}^\infty B_k\frac{\partial}{\partial b_k} \in T_tDiscs^M\ |\ \begin{array}{c}
             \sum_{k=0}^\infty B_kz^k \text{ converges for }\\
             \text{ $|z| < \bar{M}$ for some $\bar{M} > M$ }\\
             \text{and $\min(\bar{M}^2 - M^2, \epsilon^2) > |a(t) - a(t_0)|$.}
        \end{array}\right\}.
\end{equation*}

\begin{remark}
We will let the definition of $\Gamma(Discs^{\epsilon, M}_{t_0}, W^{\epsilon, M})$, $\Gamma(Discs^{\epsilon, M}_{t_0}, TDiscs^{\epsilon, M}_{t_0})$ and\\ $\Gamma(Discs^{\epsilon,M}_{t_0}, Hom(TDiscs_{t_0}^{\epsilon,M}, W^{\epsilon, M}))$ to be exactly as given in Remark \ref{discsglobalsectionremark} with $Discs^M$ replaced by $Discs^{\epsilon, M}_{t_0}$.
\end{remark}

At any point $t \in Discs^{\epsilon,M}_{t_0}$, the map $\phi_t$ arises from the consideration of the infinitesimal deformations within the moduli space $Discs^{\epsilon, M}_{t_0}$ of the disc $\mathbb{D}_{t,M}\hookrightarrow \mathbb{C}^2$ embedded inside the foliated symplectic surface $(\mathbb{C}^2, \Omega_{\mathbb{C}^2} := dx\wedge dy, \{x=const\})$. It is given by the following diagram:
\begin{equation*}
    \begin{tikzcd}
    T_tDiscs^{\epsilon,M}_{t_0} \arrow{r}{\cong}\arrow[bend left=15]{rrr}{\phi_t} & \displaystyle\varinjlim_{\bar{M} > M(t)}\Gamma(\mathbb{D}_{t,\bar{M}},  T\mathbb{C}^2|_{\mathbb{D}_{t,\bar{M}}}/T\mathbb{D}_{t,\bar{M}}) \arrow{r}{\cong} & \displaystyle\varinjlim_{\bar{M} > M(t)}\Gamma(\mathbb{D}_{t,\bar{M}}, \Omega^1_{\mathbb{D}_{t,\bar{M}}}) \arrow[hookrightarrow]{r} & W_t^{\epsilon(t),M(t)} \subset W_t^{\epsilon, M}.
    \end{tikzcd}
\end{equation*}
Where $M(t)$ and $\epsilon(t)$ are given as follows
\begin{equation*}
    \epsilon(t) := \sqrt{\epsilon^2 - |a(t) - a_0(t)|} \leq \epsilon < M \leq M(t) := \sqrt{M^2 + |a(t) - a_0(t)|}.
\end{equation*}
Note that for all $t \in Discs^{\epsilon, M}_{t_0}$, $\sum_{k=0}^\infty b_k(t)z^k$ converges for $|z| < \bar{M}$ for some $\bar{M}$ where $\bar{M} > M(t) > M$. The isomorphism $T_tDiscs^{\epsilon,M}_{t_0} \xrightarrow{\cong} \varinjlim_{\bar{M} > M(t)}\Gamma(\mathbb{D}_{t,\bar{M}},T\mathbb{C}^2|_{\mathbb{D}_{t,\bar{M}}}/T\mathbb{D}_{t,\bar{M}})$ follows because a tangent vector $v \in T_tDiscs^{\epsilon,M}_{t_0}$ gives an infinitesimal deformation of $\mathbb{D}_{t,\bar{M}}$, $\bar{M} > M(t)$ which can be thought of as a global section $n_v$ of the normal bundle $\nu_{\mathbb{D}_{t,\bar{M}}/\mathbb{C}^2} \cong T\mathbb{C}^2|_{\mathbb{D}_{t,\bar{M}}}/T\mathbb{D}_{t,\bar{M}}$. The map $\Gamma(\mathbb{D}_{t,\bar{M}},T\mathbb{C}^2|_{\mathbb{D}_{t,\bar{M}}}/T\mathbb{D}_{t,\bar{M}}) \xrightarrow{\cong} \Gamma(\mathbb{D}_{t,\bar{M}}, \Omega^1_{\mathbb{D}_{t,\bar{M}}})$ is given by $n_v \mapsto \Omega_{\mathbb{C}^2}(n_v,.)|_{\mathbb{D}_{t,\bar{M}}}$. It is an isomorphism because $\Omega_{\mathbb{C}^2}$ is non-degenerates by definition and $\mathbb{D}_{t,\bar{M}} \subset \mathbb{C}^2$ is obviously a Lagrangian submanifold. This induces the isomorphism of the direct limit 
\begin{equation*}
    \varinjlim_{\bar{M} > M(t)}\Gamma(\mathbb{D}_{t,\bar{M}},  T\mathbb{C}^2|_{\mathbb{D}_{t,\bar{M}}}/T\mathbb{D}_{t,\bar{M}}) \xrightarrow{\cong} \varinjlim_{\bar{M} > M(t)}\Gamma(\mathbb{D}_{t,\bar{M}}, \Omega^1_{\mathbb{D}_{t,\bar{M}}}).
\end{equation*}
Using the standard local coordinates $z$ on $\mathbb{D}_{t,\bar{M}}$, we find that
\begin{equation*}
    \varinjlim_{\bar{M} > M(t)}\Gamma(\mathbb{D}_{t,\bar{M}}, \Omega^1_{\mathbb{D}_{t,\bar{M}}}) \cong \left\{\xi \in z\mathbb{C}[[z]]\frac{dz}{z}\ |\ \begin{array}{c}
    \text{$\xi(z)$ converges }\\ 
    \text{ for $|z| < \bar{M}$ for some $\bar{M} > M(t)$}\end{array}\right\}\hookrightarrow W^{\epsilon(t),M(t)}_t \subset W^{\epsilon, M}_t.
\end{equation*}
\begin{remark}
Observe that, $\varinjlim_{\bar{M} > M(t)}\Gamma(\mathbb{D}_{t,\bar{M}}, \Omega^1_{\mathbb{D}_{t,\bar{M}}})$ is embedded as a Lagrangian subspace of $W^{\epsilon(t),M(t)}_t$, i.e. it is a maximal subspace of $W^{\epsilon(t),M(t)}_t$ such that if $\xi_1,\xi_2 \in \varinjlim_{\bar{M} > M(t)}\Gamma(\mathbb{D}_{t,\bar{M}}, \Omega^1_{\mathbb{D}_{t,\bar{M}}})$ then $\Omega_{Airy}(\xi_1,\xi_2) = 0$. In other words, $\phi_t$ embeds $T_tDiscs^{\epsilon, M}_{t_0}$ as a Lagrangian subspace of $W^{\epsilon(t),M(t)}_t$. On the other hand, it is not always true that $\phi_t$ embeds $T_tDiscs^{\epsilon,M}_{t_0}$ as a Lagrangian subspace of $W^{\epsilon, M}_t$, in fact $\varinjlim_{\bar{M} > M(t)}\Gamma(\mathbb{D}_{t,\bar{M}}, \Omega^1_{\mathbb{D}_{t,\bar{M}}}) \subseteq T_0L^M_{Airy}$ with an equality if and only if $a(t) = a(t_0)$.
\end{remark}
Let us write down $\phi_t$ explicitly. Consider a tangent vector
\begin{equation}\label{tangentvectdiscs}
    v = A \frac{\partial}{\partial a} + \sum_{k=0}^\infty B_k\frac{\partial}{\partial b_k} \in T_tDiscs^{\epsilon,M}_{t_0}
\end{equation}
where $\sum_{k=0}^\infty B_kz^k$ converges absolutely for $|z| < \bar{M}$ for some $\bar{M} > M(t) > M$.
The vector $v$ maps to 
\begin{equation*}
    n_v = A\frac{\partial x}{\partial a}\partial_x + \sum_{k = 0}^\infty B_k\frac{\partial y}{\partial b_k}\partial_y = A\partial_x + \sum_{k=0}^\infty B_kz^k\partial_y \in \varinjlim_{\bar{M} > M(t)}\Gamma(\mathbb{D}_{t,\bar{M}},  TS|_{\mathbb{D}_{t,\bar{M}}}/T\mathbb{D}_{t,\bar{M}}).
\end{equation*}
Therefore, the value of $\phi_t$ on $v$ is given by 
\begin{align*}\label{phioftangentvectdiscs}
    \phi_t(v) &= \Omega_S(n_v,.)|_{\mathbb{D}_{t,\bar{M}}} = \left(Ady - \sum_{k=0}^\infty B_kz^kdx\right)\Big|_{\mathbb{D}_{t,\bar{M}}}\\
    &= \left(A\sum_{k=1}^\infty kb_kz^{k-1} - 2\sum_{k = 0}^\infty B_kz^{k+1}\right)dz \in \varinjlim_{\bar{M} > M(t)}\Gamma(\mathbb{D}_{t,\bar{M}}, \Omega^1_{\mathbb{D}_{t,\bar{M}}}) \hookrightarrow W^{\epsilon(t),M(t)}_t \subset W^{\epsilon, M}_t.\numberthis 
\end{align*}
In other words, 
\begin{equation*}
    \phi_t = \left(\sum_{k=1}^\infty kb_kz^{k-1}dz\right)da - 2\sum_{k=0}^\infty (z^{k+1}dz)db_k,
\end{equation*}
and it is easy to check that $d_\nabla \phi = 0$.

Let us now construct the map $\Phi_{t_0} : Discs^{\epsilon,M}_{t_0}\hookrightarrow W^{\epsilon,M}_{t_0}$ embedding the neighbourhood $Discs^{\epsilon,M}_{t_0}$ of $t_0 = (x = z^2, y = z) \in Discs^M$ into the fiber $W_{t_0}^{\epsilon,M}$ in the same way we did in Section \ref{embeddingofbinhsubsection}. We send the zero vector $0_{t_1} \in W^{\epsilon,M}_{t_1}$ at any point $t_1 \in Discs^{\epsilon,M}_{t_0}$ into $W^{\epsilon,M}_{t_0}$ at $t_0$ via parallel transport using the connection $\nabla + \phi$. Where $(\nabla + \phi)\xi := \nabla\xi + \phi$ for $\xi \in \Gamma(Discs^M, W^{\epsilon, M})$. Since $d_\nabla\phi = 0$, there exists a section $\theta\in \Gamma(Discs^{\epsilon,M}_{t_0}, W^{\epsilon,M})$ such that $d_\nabla \theta = \phi$. We can write down $\theta_t$ explicitly for any $t = \left(x = z^2 + a(t), y = \sum_{k=0}^\infty b_k(t)z^k\right) \in Discs^{\epsilon,M}_{t_0}$:
\begin{equation}\label{definingthetainw}
    \theta_t = -ydx = -\sum_{k=0}^\infty b_k(t)z^k d(z^2) = -2\sum_{k=0}^\infty b_k(t)z^{k+1} dz.
\end{equation}
In general, given $t \in Discs^M$, the parallel transport using $\nabla + \phi$ of $0_{t_1} \in W^{\epsilon,M}_{t_1}$ from $t_1 \in Discs^{\epsilon,M}_t$ to $W^{\epsilon, M}_{t}$ at $t$ is path-independent and it is given by 
\begin{equation*}
    v_{t_1}(t) = -\theta_t + \exp((a(t) - a(t_1))\mathcal{L}_{\frac{1}{2z}\partial_z})\theta_{t_1}.
\end{equation*}
We note that $v_{t_1}(t)$ is well-defined for all $t_1 \in Discs^{\epsilon,M}_t$ because $|a(t) - a(t_1)| < \min(\bar{M}^2 - M^2, \epsilon^2)$ by definition, hence $\exp((a(t) - a(t_1))\mathcal{L}_{\frac{1}{2z}\partial_z})\theta_{t_1} \in W^{\epsilon,M}_{t}$. It follows that
\begin{equation*}
    \Phi_{t_0}(t) := v_t(t_0) = -\theta_{t_0} + \exp((a(t_0) - a(t))\mathcal{L}_{\frac{1}{2z}\partial_z})\theta_{t} \in W^{\epsilon,M}_{t_0},
\end{equation*}
for all $t \in Discs^{\epsilon,M}_{t_0}$.
\begin{proposition}\theoremname{\cite[Section 6.3]{kontsevich2017airy}}\label{phiembeddiscsinwproposition}
Let $t_0 = (x=z^2, y=z) \in Discs^M$. The image of the map $\Phi_{t_0}:Discs^{\epsilon,M}_{t_0} \hookrightarrow W^{\epsilon,M}_{t_0}$ satisfies the residue constraints (\ref{resconstraintsairystructure}), i.e. $\text{\emph{im}}\Phi_{t_0} \subseteq L_{Airy}^M \subset W^{\epsilon, M}_{t_0} \cong W^{\epsilon, M}_{Airy}$.
\end{proposition}
\begin{proof}
At the point $t = (x = z^2 + a(t), y = \sum_{k=0}^\infty b_k(t)z^k) \in Discs^{\epsilon,M}_{t_0}$, we have
\begin{equation*}
    \Phi_{t_0}(t) = zd(z^2) - \exp((a(t_0) - a(t))\mathcal{L}_{\frac{1}{2z}\partial_z})(yd(z^2)).
\end{equation*}
We will show that $\Phi_{t_0}(t) \in W^{\epsilon, M}_{t_0}$ satisfies the residue constraints. In fact, for any $m,n \in \mathbb{Z}_{\geq 0}$ we have
\begin{align*}
    Res_{z = 0}\left(\left(z - \frac{\Phi_{t_0}(t)}{2zdz}\right)^m z^{2n}d(z^2)\right) &= Res_{z = 0}\left(\left(\frac{\exp(-a(t)\mathcal{L}_{\frac{1}{2z}\partial_z})(y)d(z^2)}{d(z^2)}\right)^m z^{2n}d(z^2)\right)\\
    &= Res_{z = 0}\left(\exp(-a(t)\mathcal{L}_{\frac{1}{2z}\partial_z})(y^m) z^{2n}d(z^2)\right)\\
    &= Res_{z = 0}\left(\exp(-a(t)\mathcal{L}_{\frac{1}{2z}\partial_z})\left(y^m (z^2+a(t))^{n}d(z^2)\right)\right) = 0.
\end{align*}
We justify the vanishing of the residue as follows. Since $t \in Discs^{\epsilon, M}_{t_0}$ and $a(t_0) = 0$, we have that $|a(t)| < \epsilon^2$ and $y^m (z^2+a(t))^{n}d(z^2)$ is holomorphic for $|z| < \sqrt{M^2 + |a(t)|}$. In particular it contains no $\frac{dz}{z}$ and it belongs to $W^{0, M(t)}_{Airy} \subset W^{\epsilon, M}_{Airy}$. Therefore
\begin{equation*}
    \exp(-a(t)\mathcal{L}_{\frac{1}{2z}\partial_z})\left(y^m (z^2+a(t))^{n}d(z^2)\right) \in W^{\epsilon,M}_{Airy}
\end{equation*}
by Lemma \ref{expoperatorlemma}, which also contains no $\frac{dz}{z}$ term. In the case where $m = 1,2$, we can see that $\Phi_{t_0}(t)$ satisfies the residue constraints for all $t \in Discs^{\epsilon,M}_{t_0}$ as claimed.
\end{proof}

\begin{example}
Let $t_0 = (x=z^2,y=z)$ and $t = (x = z^2 + a, y = z)$, then
\begin{align*}
    \Phi_{t_0}(t) &= zd(z^2) - \exp(-a\mathcal{L}_{\frac{1}{2z}\partial_z})(zd(z^2))\\
    &= zd(z^2) - zd(z^2) + a\frac{d(z^2)}{2z} + \sum_{k = 2}^\infty \frac{a^k}{k!}\frac{(2k-3)!!}{2^{k}z^{2k-1}}d(z^2) = az\frac{dz}{z} + \sum_{k = 2}^\infty \frac{a^k}{k!}\frac{(2k-3)!!}{2^{k-1}z^{2k-3}}\frac{dz}{z}
\end{align*}
In $J$ coordinates we would have
\begin{equation*}
    J_{-1} = a, \qquad J_{2k-3} = \frac{a^k(2k-3)!!}{k!2^{k-1}},\qquad k > 1
\end{equation*}
and the rest of $J_n$ are zero. We can check that $\{J_n\}$ we have found satisfies the set of equations $\{(H_{Airy})_{i=1,2,3,\cdots} = 0\}$ defining $L_{Airy}^M \subset W^{\epsilon, M}_{Airy}$.
\end{example}

\begin{remark}
Since $(W^{\epsilon, M}\rightarrow Discs^{\epsilon, M}_{t_0},\Omega, \nabla, \phi)$ is a weak symplectic flat vector bundle with $\nabla$-covariantly constant $\Omega$ and $d_{\nabla}\phi = 0$, it would be nice to be able to conclude that $\text{im}\Phi_{t_0}$ is a Lagrangian submanifold using Proposition \ref{generalembeddingtofiberproposition}. Unfortunately, Proposition \ref{generalembeddingtofiberproposition} does not apply because $W^{\epsilon, M}$ has an infinite rank. However, we can show that $\text{im}\Phi_{t_0}$ is isotropic in the sense that 
\begin{equation*}
    T_{\Phi_{t_0}(t)}\text{\emph{im}}\Phi_{t_0} := \left\{\iota_vd\Phi_{t_0}(t) = 0\ |\ v \in T_tDiscs^{\epsilon, M}_{t_0}\right\} \subset W^{\epsilon, M}_{t_0}
\end{equation*}
is an isotropic subspace where $d\Phi_{t_0}(t) = \exp((a(t_0)-a(t))\mathcal{L}_{\frac{1}{2z}\partial_z})\phi_t$. For any $v_1, v_2 \in T_tDiscs^{\epsilon, M}_{t_0}$ we have by Lemma \ref{symplecticformiscovariantlyconstantlemma} and the fact that $\phi_t$ embeds $T_tDiscs^{\epsilon,M}_{t_0}$ as a Lagrangian subspace of $W^{\epsilon(t), M(t)}_t \subset W^{\epsilon,M}_t$ that $\Omega_{t_0}(\iota_{v_1}d\Phi_{t_0}(t), \iota_{v_2}d\Phi_{t_0}(t)) = \Omega_t(\phi_t(v_1), \phi_t(v_2)) = 0$. 
\end{remark}

\subsection{From Local To Global}\label{localtoglobalsection}
In Section \ref{deformationofcurvessection} we studied the complex analytic moduli space $\mathcal{B}$ of $\mathcal{F}$-transversal genus $g$ curves in a foliated symplectic surface $(S,\Omega_S,\mathcal{F})$ and the embedding of an open neighbourhood $\mathcal{B}_{\Sigma_0}$ into a fiber $\mathcal{H}_{\Sigma_0} = H^1(\Sigma_0,\mathbb{C})$ of $\mathcal{H}\rightarrow \mathcal{B}$. All considerations were done at the level of cohomology which gives a global picture of the deformations. In Section \ref{theembeddingofdiscssection} we studied what deformations of $\Sigma$ look like near a ramification point and the embedding of $Discs^{\epsilon, M}_{t_0}$ into the fiber $W^{\epsilon, M}_{t_0}$ of $W^{\epsilon, M}\rightarrow Discs^M$, which is a local picture of the deformations. The idea of this section is to connect the two pictures by lifting each cohomology class in Section \ref{deformationofcurvessection} to a differential form with Laurent series expansions near each ramification point given by a series introduced in Section \ref{theembeddingofdiscssection}.

First of all, since there are more than one ramification points on $\Sigma$ in general, instead of $W^{\epsilon, M}$ we will consider $W^{Ram} := \prod_{\alpha\in Ram}W^{\epsilon_\alpha, M_\alpha}$, a vector bundle of differential forms locally defined on the disjoint union of some annulus on $\Sigma$, each centered at a ramification point. We are going to introduce the vector bundle $G\rightarrow \mathcal{B}_{\Sigma_0} \subseteq \mathcal{B}$ where the fibers $G_\Sigma$ are given by the space of elements of $W^{Ram}$ which can be extended to a holomorphic form on the entire $\Sigma$ with some discs removed $= \Sigma\setminus \cup_{\alpha \in Ram}\mathbb{D}_{\alpha, \epsilon_\alpha}(\Sigma)$. As we about to see, the removal of discs is essential in order to make sense of parallel transports on $G$. There is a map from $G$ to $W^{Ram}$ given by Laurent expansion around ramification points and a map from $G$ to $\mathcal{H}$ given by taking cohomology classes. Therefore, $G$ is an `intermediate' vector bundle between $W^{Ram}$ and $\mathcal{H}$ in this sense. 

We equip $G$ with the connection given by differentiation along a foliation leaf $\nabla_{\mathcal{F}}$. In fact, this is the unique choice of connection for $G$, since applying $\nabla_{\mathcal{F}}$ on any element of $G_{\Sigma}$ can create poles only at ramification points. It follows that the parallel transport on $G$ from $[\Sigma] \in \mathcal{B}_{\Sigma_0}$ to $[\Sigma'] \in \mathcal{B}_{\Sigma_0}$ is given by pulling-back a differential form on $\Sigma$ to $\Sigma'$ along the foliation leaves. Clearly, the pull-back will only be well-defined away from ramification points. For this reason, we are removing a small disc $\mathbb{D}_{\alpha,\epsilon_\alpha}(\Sigma)$ near each ramification point of each $\Sigma$ to avoid the ambiguity of the pull-back between two near-by points $[\Sigma], [\Sigma']$. In other words, a parallel transport will only be possible if $[\Sigma'] \in \mathcal{B}$ is sufficiently close to $[\Sigma] \in \mathcal{B}$ such that each ramification point $r_\alpha(\Sigma')$ (or rather, the point on $\Sigma$ connected to $r_\alpha(\Sigma')$ via a foliation leaf) does not move too far from $r_\alpha(\Sigma)$ and it remains contained inside the discs $\mathbb{D}_{\alpha, \epsilon_\alpha}(\Sigma)$ to be removed from $\Sigma$. This adds another reason why we need $\mathcal{B}_{\Sigma_0}$ to be sufficiently small, a condition first mentioned in Section \ref{deformationofcurvessection} and we have delayed making it precise until this section. We may compare this to the treatment in \cite{chaimanowong2020airystructures} where no discs were removed from $\Sigma$, instead, the \emph{formal neighbourhood} $\hat{\mathcal{B}}_{\Sigma_0}$ of $[\Sigma_0]$ was considered.

Since Section \ref{theembeddingofdiscssection} is an analytic approach to \cite[Section 6]{kontsevich2017airy}, we may think of this section as an analytic approach to \cite[Section 7]{kontsevich2017airy}. We outline the content as follows. After fixing terminologies in Section \ref{basicsetupsanddefinitionssubsection}, we discuss precisely how the open neighbourhood $\mathcal{B}_{\Sigma_0} \subseteq \mathcal{B}$ should be selected. Not surprisingly, the condition $\mathcal{B}_{\Sigma_0}$ needs to satisfy also depends on the sizes of annulus $\{\epsilon_{\alpha \in Ram}\}$, $\{M_{\alpha in Ram}\}$, and the choice of $(\mathcal{F},\Omega_S)$-charts $U_\alpha$ around each ramification point. Then we will introduce the vector bundle $G$ and discuss some of its properties in Section \ref{vectorbundlegsubsection}. Among other things, we will show that $G$ is a coisotropic sub-bundle of $W^{Ram}$. In Section \ref{ramproductresconstraintssubsection}, the residue constraints Airy structure on $W^{Ram}$ will be introduced as the product of residue constraints Airy structure. Then, we will discuss the connection $\nabla_{\mathcal{F}}$ and parallel transport on $G$ in detail in Section \ref{connectionandparalleltransportongsubsection}. The comparison between connection and parallel transport on $\mathcal{H}\rightarrow \mathcal{B}$, $W^{Ram}\rightarrow Discs^{Ram}$, and $G\rightarrow \mathcal{B}_{\Sigma_0}$ will be studied in Section \ref{comparisonofvectorbundlessubsection}.
Finally, in Section \ref{relationtoresconstraintssubsection}, we will define the embedding map $\Phi_{\Sigma_0} : \mathcal{B}_{\Sigma_0} \rightarrow G_{\Sigma_0}$ analogous to $\pmb{\Phi}_{\Sigma_0} : \mathcal{B}_{\Sigma_0}\rightarrow \mathcal{H}_{\Sigma_0}$ defined in Section \ref{embeddingofbinhsubsection} and $\Phi_{t_0} : Discs_{t_0}^{\epsilon, M} \rightarrow W^{\epsilon, M}_{t_0}$ defined in Section \ref{embeddingofdiscssubsection} and argue that the image of $\Phi_{\Sigma_0}$ satisfies the residue constraints. The main result of this section, Proposition \ref{localvsglobalproposition} \cite[Proposition 7.1.2]{kontsevich2017airy} summarizes the relationship between $\Phi_{\Sigma_0}, \pmb{\Phi}_{\Sigma_0}$ and $\Phi_{t_0}$. In fact, since we are working entirely in an analytic framework, we are going to see the relationship between the Lagrangian submanifold $\mathcal{L}_{\Sigma_0} := \text{im}\pmb{\Phi}_{\Sigma_0} \subset \mathcal{H}_{\Sigma_0}$ and $L^{Ram}_{Airy}$, treated as an actual subset of $W^{Ram}_{Airy}$.

\subsubsection{Basic setup and definitions}\label{basicsetupsanddefinitionssubsection}
Let us fix a reference point $[\Sigma_0] \in \mathcal{B}$ and a collection of $(\mathcal{F}, \Omega_S)$-charts $\mathcal{U}_{Ram} := \{(U_{\alpha\in Ram}, x_{\alpha \in Ram}, y_{\alpha \in Ram})\}$ such that each $U_\alpha, \alpha \in Ram$ contains the ramification point $r_\alpha(\Sigma_0)$ of $\Sigma_0$ and $\Sigma_0\cap U_\alpha$ is bi-holomorphic to a simply-connected open subset of $\mathbb{C}$ with the standard local coordinate $z_\alpha := \sqrt{x_\alpha - x_\alpha(r_\alpha)}$. We choose $\mathcal{B}_{\Sigma_0}$ to be sufficiently small such that for all $[\Sigma] \in \mathcal{B}_{\Sigma_0}$, $\Sigma \cap U_\alpha$ contains the ramification point $r_\alpha(\Sigma)$ of $\Sigma$ and it is bi-holomorphic to a simply-connected open subset of $\mathbb{C}$ with the standard local coordinate $z_\alpha := \sqrt{x_\alpha - x_\alpha(r_\alpha(\Sigma))}$. For each $[\Sigma] \in \mathcal{B}_{\Sigma_0}$ the parameterization of $\Sigma\cap U_\alpha$ takes the form
\begin{equation*}
    t_\alpha(\Sigma) = (x_\alpha = a_\alpha(\Sigma) + z^2_\alpha,\ y_\alpha = b_{0\alpha}(\Sigma) + b_{1\alpha}(\Sigma)z_\alpha + b_{2\alpha}(\Sigma)z^2_\alpha + ...).
\end{equation*}

Let us emphasize that coefficients $\{a_\alpha(\Sigma), b_{0\alpha}(\Sigma), b_{1\alpha}(\Sigma), b_{2\alpha}(\Sigma),...\}$ depends on our choice of the $(\mathcal{F},\Omega_S)$-local coordinates $(x_\alpha, y_\alpha)$. In particular, we often choose $(x_\alpha, y_\alpha)$ such that a reference point $[\Sigma_0] \in \mathcal{B}$ is parameterized by $t_{\alpha}(\Sigma_0) = (x_\alpha = z^2_\alpha, y_\alpha = z_\alpha)$.

Pick $\{\epsilon_{\alpha \in Ram} \in \mathbb{R}\}$ and $\{M_{\alpha \in Ram} \in \mathbb{R}\}$ such that $0 < \epsilon_\alpha < M_\alpha$ and define the following discs and annulus on $\Sigma$ for $[\Sigma] \in \mathcal{B}_{\Sigma_0}$ :
\begin{align*}
    \mathbb{D}_{\alpha, \epsilon_\alpha}(\Sigma) &:= \{p \in \Sigma\cap U_\alpha\ |\ |z_\alpha(p)| < \epsilon_\alpha\} \subset \Sigma\cap U_\alpha\\
    \mathbb{D}_{\alpha, M_\alpha}(\Sigma) &:= \{p \in \Sigma\cap U_\alpha\ |\ |z_\alpha(p)| < M_\alpha\} \subset \Sigma\cap U_\alpha\\
    \mathbb{A}_{\alpha, \epsilon_\alpha, M_\alpha}(\Sigma) &:= \mathbb{D}_{\alpha, M_\alpha}(\Sigma) - \bar{\mathbb{D}}_{\alpha, \epsilon_\alpha}(\Sigma).
\end{align*}
We suppose that $\mathcal{B}_{\Sigma_0}$ is chosen to be sufficiently small such that $\bar{\mathbb{D}}_{\alpha,M_\alpha}(\Sigma) \subset \Sigma\cap U_\alpha \subset \Sigma$ for all $[\Sigma] \in \mathcal{B}_{\Sigma_0}$. In other words, $\mathbb{D}_{\alpha,M_\alpha}(\Sigma)$ is properly contained inside the coordinate chart $\Sigma\cap U_\alpha$ for all $[\Sigma]\in \mathcal{B}_{\Sigma_0}$. Consequently, $\sum_{k=0}^\infty b_{k\alpha}(\Sigma)z^k_\alpha$ converges in an open set containing $\mathbb{D}_{\alpha, M_\alpha}(\Sigma)$ for all $[\Sigma] \in \mathcal{B}_{\Sigma_0}$ and so we have the map $\gamma_\alpha : \mathcal{B}_{\Sigma_0} \hookrightarrow Discs^{M_{\alpha}}$ given by $[\Sigma]\mapsto t_\alpha(\Sigma) \in Discs^{M_\alpha}$.

Over $Discs^{M_\alpha}$ for each $\alpha \in Ram$, we have the vector bundle $W^{\epsilon_\alpha, M_\alpha} \rightarrow Discs^{M_\alpha}$ as defined in Section \ref{theembeddingofdiscssection}. Locally, we can study a curve $\Sigma \subset S$ as an embedding of $|Ram|$ discs $\mathbb{D}_{\alpha, M_\alpha}(\Sigma) \subset \Sigma\cap U_\alpha \cong \mathbb{C}^2$ indexed by the ramification index set $Ram$. Therefore, it makes sense to introduce 
\begin{equation*}
    Discs^{Ram} := \prod_{\alpha \in Ram}Discs^{M_\alpha}, \qquad Discs^{Ram}_{t} := \prod_{\alpha \in Ram}Discs^{\epsilon_\alpha, M_\alpha}_{t_{\alpha}}, \qquad W^{Ram} := \prod_{\alpha \in Ram}W^{\epsilon_\alpha, M_\alpha}.
\end{equation*}
In other words, $W^{Ram}\rightarrow Discs^{Ram}$ is a trivial vector bundle where each fiber over $t = (t_{\alpha \in Ram}) \in Discs^{Ram}$ is $W^{Ram}_t = \prod_{\alpha \in Ram}W^{\epsilon_\alpha, M_\alpha}_{t_\alpha} = \prod_{\alpha \in Ram}W^{\epsilon_\alpha, M_\alpha}_{Airy} =: W^{Ram}_{Airy}$. A typical element $\xi \in W^{Ram}_{t}$ can be written as $\xi = \sum_{\alpha \in Ram} [\alpha]\otimes\xi_\alpha$ where $\xi_\alpha \in W^{\epsilon_\alpha, M_\alpha}_{t_\alpha}$. We naturally define the multiplication $W^{Ram}_{Airy}\times W^{Ram}_{Airy} \rightarrow W^{Ram}_{Airy}$ as follows: let $\xi_i = \sum_{\alpha \in Ram}[\alpha]\otimes \xi_{i\alpha}, i = 1,2$ then
\begin{equation*}
    \xi_1\xi_2 := \sum_{\alpha \in Ram}[\alpha]\otimes \xi_{1\alpha}\xi_{2\alpha}.
\end{equation*}
Likewise, for our convenience, if $\xi = \sum_{\alpha \in Ram}[\alpha]\otimes \xi_\alpha$ where $\xi_\alpha \in z^{-1}\mathbb{C}[[z]]dz\oplus\mathbb{C}[[z]]dz$ then $Res_{z=0}\xi := \sum_{\alpha \in Ram}Res_{z=0}\xi_\alpha$. We define the symplectic form $\Omega$ and the connection $\nabla$ for $W^{Ram}$ in an obvious way: $\Omega_t := \Omega_{Airy}$ where
\begin{equation*}
    \Omega_{Airy}(\xi_1,\xi_2) := Res_{z=0}\left(\xi_1\int\xi_2\right) = \sum_{\alpha \in Ram}Res_{z=0}\left(\xi_{1\alpha}\int \xi_{2\alpha}\right)
\end{equation*}
and $\nabla \xi := \sum_{\alpha\in Ram}[\alpha]\otimes \nabla_\alpha \xi_\alpha$ where $\nabla_\alpha$ is a connection for $W^{\epsilon_\alpha, M_\alpha}$ as defined in Section \ref{theembeddingofdiscssection}. It is clear that $(W^{Ram}\rightarrow Discs^{Ram}, \Omega, \nabla)$ is a weak symplectic vector bundle and $\Omega_{Airy}$ is $\nabla$-covariantly constant. The collection of maps $\gamma_\alpha : \mathcal{B}_{\Sigma_0} \hookrightarrow Discs^{M_{\alpha}}$ induces
\begin{equation}\label{gammamapbtodiscs}
    \gamma := \prod_{\alpha\in Ram}\gamma_\alpha : \mathcal{B}_{\Sigma_0} \hookrightarrow Discs^{Ram}.
\end{equation}
Typically, we denote the point $(t_{\alpha\in Ram}(\Sigma_0)) = \gamma([\Sigma_0]) \in Discs^{Ram}$ simply by $t_0 = (t_{0\alpha\in Ram}) \in Discs^{Ram}$ where $t_{0\alpha} := t_\alpha(\Sigma_0) \in Discs^{M_\alpha}$. Moreover, if $\mathcal{B}_{\Sigma_0}$ is sufficiently small then the image of $\gamma$ will even be in $Discs^{Ram}_{t_{0}}$. This fact will be stated more precisely in Lemma \ref{gammamaplemma}.

Using the map $\gamma$ we obtain the pull-back bundle $\gamma^*W^{Ram}\rightarrow \mathcal{B}_{\Sigma_0}$ with the connection $\gamma^*\nabla$. We define the pull-back connection $\gamma^*\nabla$ as follows. For any $\xi \in \Gamma(Discs^{Ram}, W^{Ram})$ we have $\gamma^*\xi \in \Gamma(\mathcal{B}_{\Sigma_0},\gamma^*W^{Ram})$ with $(\gamma^*\xi)_{\Sigma} := \xi_{\gamma(\Sigma)} = \xi_{t(\Sigma)} = \sum_{\alpha \in Ram}[\alpha]\otimes (\xi_{\alpha})_{t_\alpha(\Sigma)}$ and $\gamma^*\nabla$ is defined via the properties $(\gamma^*\nabla)(\gamma^*\xi) := \gamma^*(\nabla\xi)$. Writing this down explicitly, we have
\begin{align*}
    (\gamma^*\nabla)(\gamma^*\xi)_\Sigma &:= \sum_{k = 1}^g\sum_{\alpha \in Ram, l}du^k[\alpha]\otimes\left(\frac{\partial a_\alpha(\Sigma)}{\partial u^k}\frac{\partial}{\partial a_\alpha} + \frac{\partial b_{l\alpha}(\Sigma)}{\partial u^k}\frac{\partial}{\partial b_{l\alpha}} - \frac{\partial a_\alpha(\Sigma)}{\partial u^k}\mathcal{L}_{\frac{1}{2z_\alpha}\partial_{z_\alpha}}\right)(\xi_{\alpha})_{t_\alpha(\Sigma)}\\
    &= \sum_{k=1}^gdu^k\left(\frac{\partial}{\partial u^k} - \frac{\partial a_\alpha(\Sigma)}{\partial u^k}\sum_{\alpha \in Ram}[\alpha]\otimes\mathcal{L}_{\frac{1}{2z_\alpha}\partial_{z_\alpha}}\right)(\gamma^*\xi)_\Sigma.
\end{align*}

\subsubsection{Choosing $\{\epsilon_{\alpha\in Ram}\}, \{M_{\alpha\in Ram}\}, \mathcal{U}_{Ram}$ and $\mathcal{B}_{\Sigma_0}$}\label{choosingepsilonmubsubsection}
We will now write down a precise summary of how $\{\epsilon_{\alpha\in Ram}\}, \{M_{\alpha\in Ram}\}, \mathcal{U}_{Ram} = \{(U_{\alpha\in Ram}, x_{\alpha \in Ram}, y_{\alpha \in Ram})\}$ and $\mathcal{B}_{\Sigma_0}$ can be selected. Once they are selected, they should remain fixed for the entire calculation. Let $[\Sigma_0] \in \mathcal{B}$ be the fixed reference point. For each $\alpha \in Ram$, we pick $\epsilon_\alpha, M_\alpha \in \mathbb{R}, 0 < \epsilon_\alpha < M_\alpha$, a collection of $(\mathcal{F},\Omega_S)$-charts $\mathcal{U}_{Ram} = \{(U_{\alpha\in Ram}, x_{\alpha \in Ram}, y_{\alpha \in Ram})\}$ with $r_\alpha(\Sigma_0) \in U_\alpha$ and an open neighbourhood $\mathcal{B}_{\Sigma_0}$ of $[\Sigma_0]$ satisfying Condition \ref{howtochoosebsigma0condition} below:

\begin{remark}
These conditions are technical and can be skipped in the first reading. They are sufficient but not necessary conditions for the results of this section to hold. It is usually more practical to pick an appropriate $\mathcal{U}_{Ram}$ then select $\{\epsilon_{\alpha \in Ram}\}, \{M_{\alpha \in Ram}\}$ such that $\bar{\mathbb{D}}_{\alpha, M_\alpha}(\Sigma_0) \subset \Sigma_0\cap U_\alpha$ before picking $\mathcal{B}_{\Sigma_0}$ and directly check that it is small enough for all the needed objects, such as $\Phi : \mathcal{B}_{\Sigma_0} \rightarrow G_{\Sigma_0}$, to be defined.
\end{remark}

\begin{condition}\label{howtochoosebsigma0condition}~
\begin{description}
    \item[C\ref{howtochoosebsigma0condition}.1] $\mathcal{B}_{\Sigma_0}$ is a contractible open neighbourhood of $[\Sigma_0]$ which can be covered by a single coordinate chart $(\mathcal{B}_{\Sigma_0}, u^1,\cdots, u^g)$.
    \item[C\ref{howtochoosebsigma0condition}.2] For all $[\Sigma] \in \mathcal{B}_{\Sigma_0}$, the open subset $\Sigma\cap U_\alpha \subset \Sigma$ is bi-holomorphic to a simply-connected open subset of $\mathbb{C}$ with the standard local coordinate $z_\alpha := \sqrt{x_\alpha - x_\alpha(r_\alpha(\Sigma))}$ and $\bar{\mathbb{D}}_{\alpha, M_\alpha}(\Sigma) \subset \Sigma\cap U_\alpha$.
    \item[C\ref{howtochoosebsigma0condition}.3] There exists a collection of $(\mathcal{F},\Omega_S)$-charts $\mathcal{U}_0 = \{(U_{0,p\in\sigma_0},x_{p\in\sigma_0}, y_{p\in\sigma_0})\}$ such that for all $[\Sigma] \in \mathcal{B}_{\Sigma_0}$, the parallel transport on the vector bundle $(G\rightarrow \mathcal{B}_{\Sigma_0}, \nabla_{\mathcal{F}})$ from the fiber $G_{\Sigma_0}$ to the fiber $G_\Sigma$ is possible for any holomorphic form $\xi \in \Gamma(\Sigma_0,\Omega^1_{\Sigma_0}) \subset G_{\Sigma_0}$ according to Lemma \ref{paralleltransportinglemma}. Additionally, the image of $\mathbb{A}_{\alpha, \epsilon_\alpha, M_\alpha}(\Sigma)$ via the map $s_{\Sigma, \Sigma_0}$  (see Definition \ref{ssigmadefinition}) is contained in $\Sigma_0\cap U_\alpha$.
    \item[C\ref{howtochoosebsigma0condition}.4] Let $\mathcal{U}_0 = \{(U_{0,p\in\sigma_0},x_{p\in\sigma_0}, y_{p\in\sigma_0})\}$ be the collection of $(\mathcal{F},\Omega_S)$-charts from C\ref{howtochoosebsigma0condition}.3. For all $[\Sigma] \in \mathcal{B}_{\Sigma_0}$ the parallel transport on the vector bundle $(G\rightarrow \mathcal{B}_{\Sigma_0}, \nabla_{\mathcal{F}})$ from the fiber $G_{\Sigma}$ to the fiber $G_{\Sigma_0}$ is possible according to Lemma \ref{paralleltransportinglemma} for any $\xi \in G_\Sigma$ holomorphic on $\cup_{p \in \sigma_0} \Sigma \cap U_{0,p} \supset \Sigma \setminus \cup_{\alpha \in Ram}\mathbb{D}_{\alpha, \epsilon_\alpha}(\Sigma)$. Additionally, the image of $\mathbb{A}_{\alpha, \epsilon_\alpha, M_\alpha}(\Sigma_0)$ via the map $s_{\Sigma_0, \Sigma}$ is contained in $\Sigma\cap U_\alpha$.
\end{description}
\end{condition}

\begin{remark}
We can summarize the basic idea behind each part in Condition \ref{howtochoosebsigma0condition} as follows:
\begin{description}
\item[C\ref{howtochoosebsigma0condition}.1] Simplifies our discussions by avoiding any complications arises from coordinates transformations, transition functions, and possible monodromies of parallel transport when $\mathcal{B}_{\Sigma_0}$ wrap around a singularity in $\mathcal{B}$.
\item[C\ref{howtochoosebsigma0condition}.2] Helps us deal with neighbourhoods of ramification points of each $\Sigma\in\mathcal{B}_{\Sigma_0}$. In particular, this helps fix the definition of the map $\gamma : \mathcal{B}_{\Sigma_0}\hookrightarrow Discs^{Ram}$, the map $i : G \hookrightarrow W^{Ram}$ (see Equation (\ref{imapdefinition})), the discs $\mathbb{D}_{\alpha,\epsilon_\alpha}(\Sigma), \mathbb{D}_{\alpha,M_\alpha}(\Sigma)$ and annulus $\mathbb{A}_{\alpha, \epsilon_\alpha,M_\alpha}(\Sigma)$.
\item[C\ref{howtochoosebsigma0condition}.3] Helps us define the section $\theta \in \Gamma(\mathcal{B}_{\Sigma_0}, G)$ which will be important in the construction of $\Phi : \mathcal{B}_{\Sigma_0} \rightarrow G_{\Sigma_0}$ embedding the neighbourhood $\mathcal{B}_{\Sigma_0}$ into the fiber $G_{\Sigma_0}$ analogous to what we did in Section \ref{deformationofcurvessection} and Section \ref{theembeddingofdiscssection}.
\item[C\ref{howtochoosebsigma0condition}.4] Makes sure that we are allowed to perform a parallel transport of $\theta_{\Sigma}$ necessary to define the embedding map $\Phi : \mathcal{B}_{\Sigma_0}\rightarrow G_{\Sigma_0}$.
\end{description}
\end{remark}

The following Lemma guarantees that given a reference point $[\Sigma_0] \in \mathcal{B}$, it is always possible to find $\{\epsilon_{\alpha\in Ram}\}, \{M_{\alpha\in Ram}\}, \mathcal{U}_{Ram}$ and $\mathcal{B}_{\Sigma_0}$ that satisfy Condition \ref{howtochoosebsigma0condition}.

\begin{lemma}
For any given $\{\epsilon_{\alpha\in Ram}\}, \{M_{\alpha\in Ram}\}, \mathcal{U}_{Ram}$, there exists an open neighbourhood $\mathcal{B}_{\Sigma_0}$ of $[\Sigma_0]$ small enough such that $\{\epsilon_{\alpha\in Ram}\}, \{M_{\alpha\in Ram}\}, \mathcal{U}_{Ram}$ and $\mathcal{B}_{\Sigma_0}$ satisfy Condition \ref{howtochoosebsigma0condition}.
\end{lemma}
\begin{proof}
This follows from the fact that C\ref{howtochoosebsigma0condition}.1-C\ref{howtochoosebsigma0condition}.4 are all open conditions.
\end{proof}

\subsubsection{The vector bundle $G\rightarrow \mathcal{B}_{\Sigma_0}$ and its properties}\label{vectorbundlegsubsection}
We are going to define the vector bundle $(G\rightarrow \mathcal{B}_{\Sigma_0}, \nabla_{\mathcal{F}})$ in this section (see Definition \ref{gvectorbundledefinition}). The fiber of $G$ over any point $[\Sigma] \in \mathcal{B}_{\Sigma_0}$ is the infinite-dimensional vector space $G_{\Sigma}$ given by Definition \ref{gdefinition} below. By taking a Laurent expansion in the vicinity of each ramification point $r_\alpha, \alpha \in Ram$ we obtain a vector bundle morphism $i: G\hookrightarrow W^{Ram}$, embedding $G$ as a vector sub-bundle of $W^{Ram}$. It will follow from Proposition \ref{gperpproposition} that $G$ is, in fact, a coisotropic vector sub-bundle of $W^{Ram}$. Consider $G^\perp_{\Sigma} \subset G_{\Sigma}$ as a group acting on $W^{Ram}_{t(\Sigma)}$ by vector addition, then we will see in Proposition \ref{sympreductionproposition} that the symplectic reduction of $(W^{Ram}_{t(\Sigma)}, \Omega_{Airy})$ is given by $(\mathcal{H}_{\Sigma}, \Omega_{\mathcal{H}})$. In particular, we have the surjective vector bundle morphism $[.] : G\rightarrow \mathcal{H}$ given by taking the cohomology class. Together, the map $i$ and $[.]$ establish $G \rightarrow \mathcal{B}_{\Sigma_0}$ as an intermediate vector bundle connecting $W^{Ram} \rightarrow Discs^{Ram}$ and $\mathcal{H}\rightarrow \mathcal{B}$ as we will discuss further in Section \ref{comparisonofvectorbundlessubsection} and Section \ref{relationtoresconstraintssubsection}.

\begin{definition}\label{gdefinition}
For any $[\Sigma] \in \mathcal{B}$, the infinite-dimensional vector space $G_{\Sigma}$ is given by
\begin{equation*}
    G_{\Sigma} := \left\{\xi \in \varinjlim_{\bar{\epsilon}_\alpha < \epsilon_\alpha}\Gamma(\Sigma \setminus \cup_{\alpha \in Ram}\bar{\mathbb{D}}_{\alpha,\bar{\epsilon}_\alpha}, \Omega^1_\Sigma)\ |\ \oint_{\partial \bar{\mathbb{D}}_{\alpha,\epsilon_\alpha}}\xi = 0,\ \alpha \in Ram\right\}. 
\end{equation*}
\end{definition}




What we meant is that if $\xi \in G_\Sigma$ then $\xi$ is a holomorphic differential on some open neighbourhood of $\Sigma\setminus \cup_{\alpha \in Ram}\mathbb{D}_{\alpha, \epsilon_\alpha}(\Sigma)$. In particular, $\xi$ is holomorphic on some open set containing $\bar{\mathbb{A}}_{\alpha, \epsilon_\alpha, M_\alpha}(\Sigma)$. 

\begin{remark}\label{gmerremark}
Definition \ref{gdefinition} of $G_{\Sigma}$ should be compared to the original definition of $G_{\Sigma}$ in Equation (38) of \cite{kontsevich2017airy}. We denote the version of $G_{\Sigma}$ in \cite{kontsevich2017airy} as $G^{mer}_\Sigma$ and recall that it is given by
\begin{equation*}
    G^{mer}_\Sigma := \left\{\xi \in \Gamma\left(\Sigma, \Omega^1_\Sigma\left(\sum_{\alpha \in Ram}nr_\alpha\right)\right)\ |\ \text{ for some } n \in \mathbb{Z}_{\geq 0} \text{ and } Res_{r_\alpha}\xi = 0,\ \alpha \in Ram\right\}.
\end{equation*}
In other words, $G^{mer}_{\Sigma}$ is a vector space of residueless meromorphic differential forms on $\Sigma$ with poles at ramification points.
\end{remark}


For each $\alpha \in Ram$, the Laurent series representation of $\xi \in G_\Sigma$ in the standard local coordinate center at $z_\alpha = 0$ converges absolutely for $\bar{\epsilon}_\alpha < |z_\alpha| < \bar{M}_\alpha$ for some $\bar{\epsilon}_\alpha < \epsilon_\alpha < M_\alpha < \bar{M}_\alpha$ and contains no $\frac{dz_\alpha}{z_\alpha}$ term because $\oint_{\partial \bar{\mathbb{D}}_{\alpha, \epsilon_\alpha}}\xi = 0$. Therefore, this define the map $i_\alpha : G_\Sigma \hookrightarrow W^{\epsilon_\alpha, M_\alpha}_{t_\alpha(\Sigma)}$. Performing the Laurent expansion at every ramification point of $\Sigma$, we obtain the map
\begin{equation}\label{imapdefinition}
    i := \sum_{\alpha \in Ram}[\alpha]\otimes i_\alpha : G_\Sigma \hookrightarrow W^{Ram}_{t(\Sigma)}
\end{equation}
for any $[\Sigma] \in \mathcal{B}_{\Sigma_0}$. Hence, instead of writing $i(G_{\Sigma}) \subset W^{Ram}_{t(\Sigma)}$ we will sometime write $G_{\Sigma} \subset W^{Ram}_{t(\Sigma)}$.
\begin{remark}
Let us emphasize that the definition of the map $i$ depends on our choice of the collection of coordinate charts $\mathcal{U}_{Ram}$. Choosing a different $(\mathcal{F},\Omega_S)$-local coordinate system $(x_\alpha, y_\alpha)$ on $U_\alpha$ for any $\alpha \in Ram$ will result in a different map $i$ and hence, a different embedding $G_\Sigma \hookrightarrow W^{Ram}_{t(\Sigma)}$.
\end{remark}
We can apply the symplectic form $\Omega_{Airy}$ on vectors in $G_{\Sigma}$. In fact, using Riemann-bilinear identity we can obtain a nice alternative expression. For $\xi_1, \xi_2 \in G_\Sigma$ we have
\begin{align*}\label{omegaairyong}
    \Omega_{Airy}(i(\xi_1), i(\xi_2)) &= \sum_{\alpha \in Ram}Res_{z_\alpha =0}\left(i_\alpha(\xi_1)\int i_\alpha(\xi_2)\right) = \frac{1}{2\pi i}\sum_{\alpha \in Ram}\oint_{\partial \bar{\mathbb{D}}_{\alpha,\epsilon_\alpha}}\left(\xi_1\int \xi_2\right)\\
    &= \frac{1}{2\pi i} \sum_{k=1}^g\left(\oint_{A_k}\xi_2\oint_{B_k}\xi_1 - \oint_{B_k}\xi_2\oint_{A_k}\xi_1\right).\numberthis
\end{align*}

For each $[\Sigma]\in \mathcal{B}_{\Sigma_0}$, we are going to show that $G_{\Sigma}\subset W^{Ram}_{t(\Sigma)}$ is a coisotropic subspace and that $G_{\Sigma}/G^\perp_{\Sigma} \cong \mathcal{H}_{\Sigma}$. Then, we will define the vector bundle $G\rightarrow \mathcal{B}$. But first, we need some preliminary definitions and results.

\begin{definition}\label{vdefinition}
For any $[\Sigma] \in \mathcal{B}$, the infinite-dimensional vector space $V_\Sigma$ is given by
\begin{equation*}
    V_\Sigma := \left\{\xi \in G_\Sigma\ |\ \oint_{A_i} \xi = 0, i = 1,...,g\right\} \subset G_\Sigma.
\end{equation*}
\end{definition}

Let us quickly recall the following definition of Bergman kernels:

\begin{definition}\label{bergmankerneldefinition}
A \emph{Bergman kernel} is a meromorphic bi-differential $B = B(p,q)$ on $\Sigma\times \Sigma$ with zero residue poles of order $2$ along the diagonal $p = q$ and holomorphic elsewhere. We say that the Bergman kernel $B = B(p,q)$ is \emph{normalized} with respect to the chosen $A$-periods if
\begin{equation*}
    B(p,q) =_{p\sim q} \frac{dz(p)dz(q)}{(z(p)-z(q))^2} + \text{holomorphic}, \qquad \oint_{p \in A_i}B(p,q) = 0, i = 1,\cdots,g.
\end{equation*}
\end{definition}
\begin{remark}
Given a curve $\Sigma$ and a choice of $A,B$ symplectic basis cycles, the normalized Bergman kernel is unique. This is because if $B$ and $B'$ are normalized Bergman kernels then $(B-B')(p,q)$ is holomorphic in $p$ with zero $A$-periods for all $q$. By the explicit construction via Riemann Theta function \cite{eynard2008algebraic, eynard2007invariants} we find that the normalized Bergman kernels also satisfy the following properties:
\begin{equation*}
    B(p,q) = B(q,p), \qquad \oint_{p \in B_i}B(p,q) = 2\pi i \omega_i(q)
\end{equation*}
where $\omega_i$ are the normalized holomorphic forms on $\Sigma$.
\end{remark}

\begin{definition}\label{endifferentialdefinition}
Let $B(p,q)$ be the normalized Bergman kernel on $\Sigma$.
For any $k \in \mathbb{Z}_{\geq 1}$ and $\alpha \in Ram$ we define
the meromorphic differential form on $\Sigma$:
\begin{equation*}
    \bar{e}^{k,\alpha} := \frac{1}{2\pi ik}\oint_{z_\alpha = 0}\frac{B(.,q(z_\alpha))}{z_\alpha^k}.
\end{equation*}
\end{definition}
Using the standard local coordinates $z_\alpha$ and $z_\beta$ of $\Sigma \cap U_\alpha$ and $\Sigma \cap U_\beta$ respectively, $\alpha, \beta \in Ram$, the local series expansion of $B(p,q)$ for $p \in \Sigma\cap U_\alpha$ and $q \in \Sigma\cap U_\beta$ is given by 
\begin{equation}\label{bergmankernelexpansion}
    B(z_\alpha(p), z_\beta(q)) = \frac{\delta_{\alpha\beta} dz_\alpha(p) dz_\beta(q)}{(z_\alpha(p) - z_\beta(q))^2} + \sum_{i,j = 1}^\infty P^{(i,\alpha)(j,\beta)}z_\alpha^{i-1}(p)z_\beta^{j-1}(q)dz_\alpha(p) dz_\beta(q)
\end{equation}
for some symmetric tensor $P^{(i,\alpha)(j,\beta)} = P^{(j,\beta)(i,\alpha)}$. We obtain the following easy fact about $\bar{e}^{k,\alpha}$:

\begin{lemma}
Let $\bar{e}^{k,\alpha}$ be a meromorphic differential form on $\Sigma$ as given by Definition \ref{endifferentialdefinition}. Then $\bar{e}^{k,\alpha} \in V_{\Sigma}$ and 
\begin{equation}\label{ebarintermofef}
    i(\bar{e}^{k,\alpha}) = [\alpha]\otimes z^{-k}_\alpha\frac{dz_\alpha}{z_\alpha} + \sum_{k' \geq 1, \alpha' \in Ram}[\alpha']\otimes \frac{1}{k}P^{(k,\alpha)(k',\alpha')}z_{\alpha'}^{k'-1}\frac{dz_{\alpha'}}{z_{\alpha'}} \in W^{Ram}_{t(\Sigma)}
\end{equation}
\end{lemma}
\begin{proof}
The expansion (\ref{ebarintermofef}) is obtained from integrating the series expansion (\ref{bergmankernelexpansion}). This shows that $\bar{e}^{k,\alpha}$ is holomorphic everywhere except for a pole of order $k+1$ at $r_\alpha$ with zero residue. Deform $A$-cycles to avoid intersection with $r_\alpha$, then $\oint_{p \in A_i}\bar{e}^{k,\alpha}(p) = \frac{1}{2\pi ik}\oint_{z_\alpha = 0}\frac{1}{z_\alpha^k}\oint_{p \in A_i}B(p,q(z_\alpha)) = 0$ because the Bergman kernel is normalized.
\end{proof}

\begin{remark}
Note that the definition of $e^{k,\alpha}$ implicitly depends on our choice of the Bergman kernel (since we are using the normalized Bergman kernel this is equivalent to our choice of $A,B$ symplectic basis cycles) and the standard local coordinate $z_\alpha$ (which is equivalent to our choice of the collection of $(\mathcal{F},\Omega_S)$-charts $\mathcal{U}_{Ram}$).
\end{remark}

\begin{lemma}\label{fromlocalvtoglobalvlemma}
Let $[\Sigma] \in \mathcal{B}$.
If $\xi_\alpha = \sum_{k=1}^\infty \xi_{\alpha,k}z_\alpha^{-k}\frac{dz_\alpha}{z_\alpha} \in V^{\epsilon_\alpha}_{Airy}$ then $\bar{\xi}_\alpha = \sum_{k=1}^\infty \xi_{\alpha,k}\bar{e}^{k,\alpha} \in V_{\Sigma}$. The principal part of $i_{\alpha'}(\bar{\xi}_\alpha)$ is $\delta_{\alpha\alpha'}\xi_\alpha$ and for any path $C \subset \Sigma\setminus \cup_{\alpha \in Ram}\mathbb{D}_{\alpha, \epsilon_\alpha}(\Sigma)$, we have 
\begin{equation*}
    \int_C\bar{\xi}_\alpha = \int_C \sum_{k=1}^\infty\xi_{\alpha,k}\bar{e}^{k,\alpha} = \sum_{k=1}^\infty\xi_{\alpha,k}\int_C \bar{e}^{k,\alpha}.
\end{equation*}
\end{lemma}
\begin{proof}
We have that $\xi_\alpha = \sum_{k=1}^\infty \xi_{\alpha,k}z_\alpha^{-k}\frac{dz_\alpha}{z_\alpha}$ converges absolutely uniformly for $|z_\alpha| > \bar{\epsilon}_\alpha$ for some $\bar{\epsilon}_\alpha < \epsilon_\alpha$. Therefore, 
\begin{equation}
    \sum_{k=1}^\infty \frac{\xi_{\alpha, k}}{kz^{k}_\alpha} = -\sum_{k=1}^\infty\int \xi_{\alpha,k}z^{-k}_\alpha\frac{dz_\alpha}{z_\alpha} = -\int\sum_{k=1}^\infty \xi_{\alpha,k}z_\alpha^{-k}\frac{dz_\alpha}{z_\alpha}
\end{equation}
converges absolutely uniformly for $|z_\alpha| > \bar{\epsilon}_{\alpha}$, because integration of a power series does not change its radius of convergence. Then by using the Bergman kernel $B(p,q)$ on $\Sigma$ we define
\begin{align*}
    \bar{\xi}_{\alpha}(p) := \frac{1}{2\pi i}\oint_{z_\alpha \in \partial \bar{\mathbb{D}}_{\alpha,\bar{\epsilon}_\alpha}}\left(\sum_{k=1}^\infty \frac{\xi_{\alpha, k}}{kz^{k}_\alpha}\right)B(p,q(z_\alpha))
    = \sum_{k=1}^\infty\frac{\xi_{\alpha,k}}{2\pi i k}\oint_{z_\alpha = 0}\frac{B(p,q(z_\alpha))}{z_\alpha^k} = \sum_{k = 1}^\infty\xi_{\alpha,k}\bar{e}^{k,\alpha}(p),
\end{align*}
where the term-by-term integration is allowed due to the absolute uniform convergence. Therefore, $\bar{\xi}_{\alpha}$ is holomorphic for all $p \in \Sigma \setminus \bar{\mathbb{D}}_{\alpha,\bar{\epsilon}_\alpha} \supset \Sigma \setminus \cup_{\alpha' \in Ram}\mathbb{D}_{\alpha',\epsilon_{\alpha'}}$ but diverges as $p$ approaches the integration path $\partial \bar{\mathbb{D}}_{\alpha,\bar{\epsilon}_\alpha}$. The principal part of $\bar{e}^{k,\alpha}$ is $z^{-k}_\alpha\frac{dz_\alpha}{z_\alpha}$ so the principal part of $i_{\alpha'}(\bar{\xi}_\alpha)$ is $\xi_{\alpha}$ if $\alpha' = \alpha$ and zero otherwise. Given any path $C\subset \Sigma \setminus \cup_{\alpha\in Ram}\mathbb{D}_{\alpha, \epsilon_\alpha}(\Sigma)$ we have 
\begin{align*}
    \int_{C}\bar{\xi}_\alpha(p) &= \int_{p\in C} \left( \frac{1}{2\pi i}\oint_{z_\alpha \in \partial\bar{\mathbb{D}}_{\alpha, \bar{\epsilon}_\alpha}}\left(\sum_{k=1}^\infty \frac{\xi_{\alpha, k}}{kz^k_\alpha}\right)B(p,q(z_\alpha))\right) = \frac{1}{2\pi i}\oint_{z_\alpha \in \partial\bar{\mathbb{D}}_{\alpha, \bar{\epsilon}_\alpha}}\left(\sum_{k=1}^\infty \frac{\xi_{\alpha, k}}{kz^k_\alpha}\right)\int_{p\in C}B(p,q(z_\alpha))\\
    &= \sum_{k=1}^\infty\frac{\xi_{\alpha,k}}{2\pi i k}\oint_{z_\alpha \in \partial\bar{\mathbb{D}}_{\alpha, \bar{\epsilon}_\alpha}}\int_{p\in C}\frac{B(p,q(z_\alpha))}{z_\alpha^k} = \sum_{k=1}^\infty \xi_{\alpha,k}\int_{p\in C}\frac{1}{2\pi i k}\oint_{z_\alpha \in \partial\bar{\mathbb{D}}_{\alpha, \bar{\epsilon}_\alpha}}\frac{B(p,q(z_\alpha))}{z_\alpha^k}\\
    &= \sum_{k=1}^\infty \xi_{\alpha,k}\int_{p\in C}\bar{e}^{k,\alpha}(p).
\end{align*}
We have used the fact that $C$ do not intersect with $\partial \bar{\mathbb{D}}_{\epsilon, \bar{\epsilon}_\alpha}$, so the order of integrations can be permuted. The integration can be done term-by-term because $\sum_{k=1}^\infty \xi_{\alpha,k}z^{-k}_\alpha \frac{dz_\alpha}{z_\alpha}$ converges absolutely uniformly for $|z_\alpha| > \bar{\epsilon}_\alpha$. Let $C = \partial\bar{\mathbb{D}}_{\alpha,\epsilon_\alpha}$ then $\oint_{\partial\bar{\mathbb{D}}_{\alpha,\epsilon_\alpha}}\bar{\xi}_\alpha = \sum_{k=1}^\infty \xi_{\alpha,k}\oint_{\partial\bar{\mathbb{D}}_{\alpha,\epsilon_\alpha}}\bar{e}^k = 0$. On the other hand, let $C = A_i$ then $\oint_{A_i}\bar{\xi}_\alpha = \sum_{k=1}^\infty\xi_{\alpha,k}\oint_{A_i}\bar{e}^{k,\alpha} = 0$. Therefore, we conclude that $\bar{\xi}_\alpha \in V_\Sigma$.
\end{proof}

Recall the definition of $L^M_{Airy} \subset W^{\epsilon, M}_{Airy}$ from Section \ref{analyticresconstraintssubsection} and define $L^{Ram}_{Airy} := \prod_{\alpha \in Ram}L^{M_\alpha}_{Airy} \subset W^{Ram}_{Airy}$. We note that $T_0L^{Ram}_{Airy} = \prod_{\alpha \in Ram}T_0L^{M_\alpha}_{Airy}$, or more explicitly:
\begin{equation*}
    T_0L^{Ram}_{Airy} = \left\{\xi = \sum_{\alpha \in Ram}[\alpha]\otimes \xi_\alpha\ |\ \xi_\alpha(z_\alpha) \in z_\alpha\mathbb{C}[[z_\alpha]]\frac{dz_\alpha}{z_\alpha}, \begin{array}{c}
    \text{$\xi_\alpha(z_\alpha)$ converges for $|z_\alpha| < \bar{M}_\alpha$}\\ 
    \text{for some $\bar{M}_\alpha > M_\alpha$}\end{array}\right\} \subset W^{Ram}_{Airy},
\end{equation*}
is a Lagrangian subspace of $W^{Ram}_{Airy}$.

\begin{lemma}\label{visgplusllemma}
For any $[\Sigma] \in \mathcal{B}$ we have
\begin{equation*}
     W^{Ram}_{t(\Sigma)} = W^{Ram}_{Airy} = V_\Sigma \oplus T_0L^{Ram}_{Airy}.
\end{equation*}
Additionally, if $\xi \in V_\Sigma$ and the principal part of $i_\alpha(\xi)$ is $\sum_{k=1}^\infty \xi_{\alpha, k}z^{-k}_\alpha \frac{dz_\alpha}{z_\alpha} \in V^{\epsilon_\alpha}_{Airy}$ then it can be written as $\xi = \sum_{\alpha\in Ram}\sum_{k=1}^\infty \xi_{\alpha, k}\bar{e}^{k,\alpha}$.
\end{lemma}
\begin{proof}
If $\xi \in V_\Sigma$ is such that $i_\alpha(\xi)$ has no principal part for all $\alpha \in Ram$ then it must be holomorphic on $\Sigma$. Therefore, $\oint_{A_i}\eta = 0, i = 1,\cdots,g$ which implies that $\eta = 0$. This shows $V_\Sigma \cap T_0L^{Ram}_{Airy} = \{0\}$.
To show that $V_\Sigma + T_0L^{Ram}_{Airy} = W^{Ram}_{Airy}$ we consider any given $\xi = (\xi_{\alpha\in Ram}) \in W^{Ram}_{Airy}$ where $\xi_\alpha \in z_\alpha^{-1}\mathbb{C}[[z_\alpha^{-1}]]\frac{dz_\alpha}{z_\alpha}\oplus z_\alpha\mathbb{C}[[z_\alpha]]\frac{dz_\alpha}{z_\alpha}$ converges absolutely for $\bar{\epsilon}_\alpha < |z_\alpha| < \bar{M}_\alpha$ for some $\bar{\epsilon}_\alpha < \epsilon_\alpha < M_\alpha < \bar{M}_\alpha$. Let $\xi_{\alpha,P} = \sum_{k=1}^\infty z^{-k}_\alpha\xi_{\alpha,k}\frac{dz_\alpha}{z_\alpha} \in V^{\epsilon_\alpha}_{Airy}$ be the principal part of $\xi_\alpha$. By Lemma \ref{fromlocalvtoglobalvlemma} we have $\bar{\xi}_P := \sum_{\alpha \in Ram}\bar{\xi}_{\alpha,P} = \sum_{\alpha \in Ram}\sum_{\alpha,k}\xi_{k,\alpha}\bar{e}^{k,\alpha} \in V_\Sigma$ and the principal part of $i_\alpha(\bar{\xi})$ is $\xi_{\alpha,P}$. On the other hand, because $\Sigma \in \mathcal{B}_{\Sigma_0}$ we have that $i_\alpha(\bar{\xi})$ converges absolutely for $\bar{\epsilon}_\alpha < |z| < \bar{\bar{M}}_\alpha$ where $\bar{\bar{M}}_\alpha$ is such that $\mathbb{D}_{\alpha, M_\alpha}(\Sigma) \subset \mathbb{D}_{\alpha, \bar{\bar{M}}_\alpha}(\Sigma) \subset \Sigma\cap U_\alpha$. It follows that $\xi_\alpha - i_\alpha(\bar{\xi}_P)$ has zero principal part and converges absolutely for  $|z_\alpha| < \min(\bar{M}_\alpha, \bar{\bar{M}}_\alpha)$. Therefore $\xi - i(\bar{\xi}_P) \in T_0L^{Ram}_{Airy}$ which proves that $W^{Ram}_{Airy} = V_{\Sigma}\oplus T_0L^{Ram}_{Airy}$. Besides, if $\xi \in V_\Sigma$, then $\xi - \bar{\xi}_P$ is holomorphic on $\Sigma$ with $\oint_{A_i}(\xi - \bar{\xi}_P) = 0$ for all $i = 1,...,g$ which means $\xi = \bar{\xi}_P = \sum_{\alpha \in Ram}\sum_{k=1}^\infty \xi_{\alpha, k}\bar{e}^{k,\alpha}$.
\end{proof}

\begin{corollary}
For any $[\Sigma] \in \mathcal{B}$ we have
\begin{equation*}
    G_{\Sigma} + T_0L^{Ram}_{Airy} = W^{Ram}_{Airy}.
\end{equation*}
\end{corollary}
\begin{proof}
This follows from Lemma \ref{visgplusllemma} and the fact that $V_\Sigma \subset G_\Sigma$.
\end{proof}

\begin{corollary}\label{gisvoplush0corollary}
For any $[\Sigma] \in \mathcal{B}$, we have,
\begin{equation*}
    G_\Sigma = V_\Sigma \oplus \Gamma(\Sigma, \Omega^1_\Sigma).
\end{equation*}
In particular, any $\xi \in G_\Sigma$ with the principal part of $i_\alpha(\xi)$ given by $\sum_{k=1}^\infty \xi_{\alpha, k}z_\alpha^{-k}\frac{dz_\alpha}{z_\alpha}$ and $a^i := \oint_{A_i}\xi$ can be written as $\xi = \sum_{\alpha \in Ram}\sum_{k=1}^\infty \xi_{\alpha, k}\bar{e}^{k,\alpha} + \sum_{k=1}^g a^k\omega_k$, where $\omega_i$ are the normalized holomorphic forms on $\Sigma$.
\end{corollary}
\begin{proof}
In fact, $G_\Sigma = V_\Sigma \oplus \Gamma(\Sigma, \Omega^1_\Sigma)$ follows directly from the Definition \ref{gdefinition} and Definition \ref{vdefinition} of $G_\Sigma$ and $V_\Sigma$. This is because for any $\xi \in G_{\Sigma}$, we have $\bar{\xi} := \xi - \sum_{k=1}^g a^k\omega_k \in V_\Sigma$, hence $V_\Sigma + \Gamma(\Sigma, \Omega^1_\Sigma) = G_\Sigma$. We also know that $V_\Sigma\cap \Gamma(\Sigma, \Omega^1_\Sigma) = \{0\}$ because any global holomorphic form with zero $A$-periods is zero. The rest of the Corollary follows from applying Lemma \ref{visgplusllemma} to $\bar{\xi} \in V_\Sigma$.
\end{proof}

\begin{corollary}\label{vislagrangiancomplementcorollary}
$V_\Sigma$ is a Lagrangian complement of $T_0L^{Ram}_{Airy}$ in $W^{Ram}_{Airy}$. 
\end{corollary}
\begin{proof}
We already know from Lemma \ref{visgplusllemma} that $V_\Sigma$ is a complement of $T_0L^{Ram}_{Airy}$.
Using (\ref{omegaairyong}), it is clear that for any $\xi_1, \xi_2 \in V_{\Sigma}$ we would have $\Omega_{Airy}(\xi_1,\xi_2) = 0$ because $\oint_{A_i}\xi_j = 0, i = 1,...,g, j = 1,2$. Suppose there exists $\xi_2 \in W^{Ram}_{Airy}$ such that $\Omega_{Airy}(\xi_1, \xi_2) = 0$ for all $\xi_1 \in V_{\Sigma}$. Then by Lemma \ref{visgplusllemma} we can write $\xi_2 = v + l$ where $v \in V_\Sigma, l \in T_0L^{Ram}_{Airy}$, but then $\Omega_{Airy}(\xi_1,l) = \Omega_{Airy}(\xi_1, v+l) = \Omega_{Airy}(\xi_1, \xi_2) = 0, \forall \xi_1 \in V_{\Sigma}$ because $\Omega_{Airy}(\xi_1, v) = 0$. Since we also have $\Omega_{Airy}(l',l) = 0$ for all $l' \in T_0L^{Ram}_{Airy}$, this can only be true if $l = 0$, i.e. $\xi_2 \in V_\Sigma$, because $\Omega_{Airy}$ is non-degenerates. So we have that $V_\Sigma$ is a Lagrangian subspace of $W^{Ram}_{Airy}$ i.e. $V^\perp_\Sigma = V_\Sigma$. 
\end{proof} 

\begin{proposition}\theoremname{\cite[Proposition 7.1.3]{kontsevich2017airy}}\label{gperpproposition}
For any $[\Sigma]\in \mathcal{B}$, the vector space $G_\Sigma$ is a coisotropic subspace of $W^{Ram}_{t(\Sigma)}$ with the perpendicular subspace $G^\perp_\Sigma$ given by
\begin{equation*}
    G^\perp_{\Sigma} = \varinjlim_{\bar{\epsilon}_\alpha < \epsilon_\alpha}d\Gamma(\Sigma\setminus \cup_{\alpha \in Ram}\bar{\mathbb{D}}_{\alpha, \bar{\epsilon}_\alpha}, \mathcal{O}_{\Sigma}) \subseteq G_\Sigma.
\end{equation*}
\end{proposition}
\begin{proof}
From $V_\Sigma \subset G_\Sigma$ and Corollary \ref{vislagrangiancomplementcorollary} we have that $G^\perp_{\Sigma} \subset V^\perp_\Sigma = V_\Sigma \subset G_\Sigma$. 
Let us find $G^\perp_{\Sigma}$. If $\xi \in G^\perp_\Sigma$ then $\xi \in G_\Sigma$ which means $\oint_{\partial \bar{\mathbb{D}}_{\alpha, \epsilon_\alpha}}\xi = 0$. Using (\ref{omegaairyong}) we can see that the only way $\Omega_{Airy}(\xi, \eta) = 0, \forall \eta \in G_\Sigma$ can be possible is if $\oint_{A_i}\xi = 0, \oint_{B_i}\xi = 0, i = 1,...,g$. We can now define $f(p) := \int_{p_0}^p\xi$ which is a holomorphic function in an open neighbourhood of $\Sigma \setminus \cup_{\alpha \in Ram}\mathbb{D}_{\alpha, \epsilon_\alpha}$.
\end{proof}

\begin{proposition}\theoremname{\cite[Proposition 7.1.3]{kontsevich2017airy}}\label{sympreductionproposition}
For any $[\Sigma]\in \mathcal{B}$, the sympectic reduction of $W^{Ram}_{t(\Sigma)}$ is given by $W^{Ram}_{t(\Sigma)}//G^\perp_{\Sigma} := G_{\Sigma}/G^\perp_{\Sigma} \cong \mathcal{H}_{\Sigma} = H^1(\Sigma, \mathbb{C})$. Furthermore, the induced symplectic form $\tilde{\Omega}_{Airy}$ on the quotient $G_\Sigma/G^\perp_\Sigma$ given by
\begin{equation}
    \tilde{\Omega}_{Airy}([\xi_1],[\xi_2]) := \Omega_{Airy}(\xi_1, \xi_2)
\end{equation}
coincides with the symplectic form $\Omega_{\mathcal{H}}$ on $\mathcal{H}_\Sigma$ as given in (\ref{omegahdefinition}).
\end{proposition}
\begin{proof}
First, we will show that $G_\Sigma/G^\perp_\Sigma$ is isomorphic to the deRham cohomology group $\mathcal{H}_\Sigma = H^1(\Sigma, \mathbb{C})$.
For any set $\{\bar{\epsilon}_{\alpha\in Ram}\}$ such that $\bar{\epsilon}_\alpha < \epsilon_\alpha, \alpha\in Ram$ we have the presentation 
\begin{equation}\label{hsigmapresentation}
    \mathcal{H}_\Sigma = H^1(\Sigma,\mathbb{C}) \cong \frac{\left\{\xi \in \Gamma\left(\Sigma \setminus \cup_{\alpha \in Ram}\bar{\mathbb{D}}_{\alpha,\bar{\epsilon}_\alpha},\Omega^1_\Sigma\right)\ |\ \oint_{\partial \bar{\mathbb{D}}_{\alpha, \epsilon_\alpha}}\xi = 0\right\}}{d\Gamma(\Sigma\setminus \cup_{\alpha \in Ram}\bar{\mathbb{D}}_{\alpha, \bar{\epsilon}_\alpha}, \mathcal{O}_{\Sigma})}.
\end{equation}
To obtain this, we use Mayer–Vietoris sequence to get $H^1(\Sigma, \mathbb{C}) \cong \{\xi \in H^1(\Sigma \setminus \cup_{\alpha \in Ram}\bar{\mathbb{D}}_{\alpha, \bar{\epsilon}_\alpha}, \mathbb{C})\ |\ \oint_{\partial \bar{\mathbb{D}}_{\alpha, \epsilon_\alpha}}\xi = 0\}$. Then we use the fact that non-compact Riemann surfaces such as $\Sigma \setminus \cup_{\alpha \in Ram}\bar{\mathbb{D}}_{\alpha, \bar{\epsilon}_\alpha}$ is a Stein manifold \cite[Behnke and Stein, Section 2.2]{forstnerivc2011stein}, hence coherent sheaves are $\Gamma$-acyclic on $\Sigma \setminus \cup_{\alpha \in Ram}\bar{\mathbb{D}}_{\alpha, \bar{\epsilon}_\alpha}$ by Cartan's theorem B \cite[Theorem 2.4.1]{forstnerivc2011stein}. Therefore, $H^1(\Sigma \setminus \cup_{\alpha \in Ram}\bar{\mathbb{D}}_{\alpha, \bar{\epsilon}_\alpha}, \mathbb{C}) \cong H^1\Gamma(\Sigma \setminus \cup_{\alpha \in Ram}\bar{\mathbb{D}}_{\alpha, \bar{\epsilon}_\alpha}, \Omega^\bullet_\Sigma)$, i.e. deRham cohomology can be computed as a sheaf cohomology of $\mathbb{C}$ using the $\Gamma$-acyclic resolution $0\rightarrow \mathbb{C}\rightarrow \Omega^\bullet_\Sigma$ (see, for example, \cite[Corrollary 2.5.3]{cattani2014hodge}), which implies (\ref{hsigmapresentation}). Recall that the direct limit functor preserves short exact sequences. Therefore, taking the direct limit of (\ref{hsigmapresentation}) and compare the result to the explicit expression of $G_\Sigma$ and $G^\perp_\Sigma$ in Definition \ref{gdefinition} and Proposition \ref{gperpproposition}, we have $H^1(\Sigma, \mathbb{C}) \cong G_\Sigma/G^\perp_\Sigma$ as claimed.

Next, we show that $\tilde{\Omega}_{Airy}$ is a symplectic form on $\mathcal{H}_\Sigma$. If $\xi_1' = \xi_1 + df_1, \xi_2' = \xi_2 + df_2$ then $\tilde{\Omega}_{Airy}([\xi'_1], [\xi'_2]) = \Omega_{Airy}(\xi_1 + df_1, \xi_2 + df_2) = \Omega_{Airy}(\xi_1, \xi_2) = \tilde{\Omega}_{Airy}([\xi_1],[\xi_2])$, so $\tilde{\Omega}_{Airy}$ is well-defined. It is non-degenerates because if $[\xi] \in \mathcal{H}_\Sigma$ is such that $\tilde{\Omega}_{Airy}([\xi], [\eta]) = 0$ for all $[\eta] \in \mathcal{H}_\Sigma$ then $\Omega_{Airy}(\xi, \eta) = 0$ for all $\eta \in G_{\Sigma}$ implying that $\xi \in G^\perp_\Sigma$ and hence $[\xi] = 0$. This proves that $\mathcal{H}_\Sigma$ is a symplectic vector space with the induced symplectic form $\tilde{\Omega}_{Airy}$. Using (\ref{omegaairyong}) and Lemma \ref{symplecticformonhlemma} we can see that $\tilde{\Omega}_{Airy}$ coincides with $\Omega_{\mathcal{H}}$ as claimed. 
\end{proof}

Given $\xi \in G_\Sigma$, the cohomology class $[\xi] \in H^1(\Sigma,\mathbb{C})$ is completely determined by the $A$ and $B$ periods of $\xi$. In other words, we have a map $[.] : G_\Sigma\rightarrow \mathcal{H}_\Sigma$ given by $\xi \mapsto (H_1(\Sigma, \mathbb{C})\ni [C] \mapsto \oint_C\xi)$. Proposition \ref{sympreductionproposition} tells us that this map is surjective, in other words, there are enough differential forms on $G_\Sigma$ to represent every cohomology classes in $H^1(\Sigma, \mathbb{C})$.

\begin{remark}
Similarly, it has been shown in \cite[Proposition 7.1.3]{kontsevich2017airy} that $G^{mer}_\Sigma$ is a coisotropic subspace of $W^{|Ram|}_{Airy}$ (see Section \ref{ramproductresconstraintssubsection} for the definition of $W^{|Ram|}_{Airy}$) with
\begin{equation*}
    (G^{mer}_\Sigma)^{\perp} = \left\{df\ |\ f \in \Gamma\left(\Sigma, \mathcal{O}_{\Sigma}\left(\sum_{\alpha \in Ram}nr_\alpha\right)\right), \text{ for some } n \in \mathbb{Z}_{\geq 0}\right\} \subseteq G^{mer}_\Sigma
\end{equation*}
and we have $G^{mer}_\Sigma/(G^{mer}_\Sigma)^\perp \cong \mathcal{H}_\Sigma$.
\end{remark}

At last, let us define the vector bundle $G\rightarrow \mathcal{B}_{\Sigma_0}$ with fiber over each point $[\Sigma] \in \mathcal{B}_{\Sigma_0}$ given by $G_\Sigma$, where $\mathcal{B}_{\Sigma_0} \subseteq \mathcal{B}$ along with $\{\epsilon_{\alpha \in Ram}\}, \{M_{\alpha \in Ram}\}$ and $\mathcal{U}_{Ram}$ are chosen to satisfy Condition \ref{howtochoosebsigma0condition}. In fact, let $\mathcal{B}_{\Sigma_0}^\circ \supseteq \mathcal{B}$ be the maximal open neighbourhood of $[\Sigma_0]$, $\mathcal{B}_{\Sigma_0} \subseteq \mathcal{B}^\circ_{\Sigma_0}$ such that $\mathcal{B}^\circ_{\Sigma_0}$, $\{\epsilon_{\alpha \in Ram}\}, \{M_{\alpha \in Ram}\}$ and $\mathcal{U}_{Ram}$ satisfy C\ref{howtochoosebsigma0condition}.2 then
we will define the vector bundle $G\rightarrow \mathcal{B}^\circ_{\Sigma_0}$. Note that it is not possible to extend the definition of $G$ as a vector bundle on a larger base than $\mathcal{B}^\circ_{\Sigma_0}$ because any different choice of the collection of $(\mathcal{F}, \Omega_S)$-charts $\mathcal{U}_{Ram}$ will lead to a different definition for each fiber $G_\Sigma$.

For any given point $[\Sigma] \in \mathcal{B}^\circ_{\Sigma_0}$, there exists an open neighbourhood $\mathcal{B}_\Sigma \subset \mathcal{B}^\circ_{\Sigma_0}$ of $[\Sigma]$ such that $A, B$ symplectic basis cycles can be chosen consistently for all $[\Sigma'] \in \mathcal{B}_{\Sigma}$ (see Remark \ref{htrivializationtremark}). It follows that the normalized holomorphic forms $\{\omega_{i=1,\cdots,g}\}$ and the normalized Bergman kernel $B(p,q)$ are defined on every curve $[\Sigma'] \in \mathcal{B}_{\Sigma}$ consistently throughout $\mathcal{B}_{\Sigma}$. For each $[\Sigma'] \in\mathcal{B}_{\Sigma}$, let us denote by $\{\omega_i^{\Sigma,\Sigma'}\}$, $\{\bar{e}^{k,\alpha}_{\Sigma,\Sigma'}\}$ the set of normalized holomorphic forms and the set of meromorphic differentials on $\Sigma'$ defined by the normalized Bergman kernel (Definition \ref{endifferentialdefinition}) corresponding to the given consistent choice of $A,B$ symplectic basis cycles on $\mathcal{B}_{\Sigma}$. Let the topology on $G_\Sigma$ be induced from the norm
\begin{equation*}
    \Vert \xi \Vert_{G_\Sigma} := \sqrt{\sum_{\alpha \in Ram}\left(\sup_{p \in \mathbb{A}_{\alpha, \epsilon, M_\alpha}(\Sigma)}\left|\xi(p)\frac{z_\alpha(p)}{dz_\alpha(p)}\right|\right)^2}.
\end{equation*}
and recall from Corollary \ref{gisvoplush0corollary}, any $\xi \in G_{\Sigma'}$ can be written as $\xi = \sum_{\alpha \in Ram}\sum_{k=1}^\infty \xi_{k,\alpha}\bar{e}^{k,\alpha}_{\Sigma,\Sigma'} + \sum_{k=1}^ga^k\omega^{\Sigma, \Sigma'}_k$. Then, given $[\Sigma_1], [\Sigma_2] \in \mathcal{B}_{\Sigma}$ we define a homeomorphism $G_{\Sigma_1}\cong G_{\Sigma_2}$ by
\begin{equation}\label{homeomorphismgsigma1gsigma2}
    \sum_{\alpha \in Ram}\sum_{k=1}^\infty \xi_{k,\alpha}\bar{e}^{k,\alpha}_{\Sigma,\Sigma_1} + \sum_{k=1}^ga^k\omega^{\Sigma, \Sigma_1}_k \xleftrightarrow{\hspace{1cm}} \sum_{\alpha \in Ram}\sum_{k=1}^\infty \xi_{k,\alpha}\bar{e}^{k,\alpha}_{\Sigma,\Sigma_2} + \sum_{k=1}^ga^k\omega^{\Sigma, \Sigma_2}_k.
\end{equation}

Equip $\mathcal{B}_{\Sigma}\times G_\Sigma$ with the product topology and assign a topology to the set $G|_{\mathcal{B}_{\Sigma}} := \{([\Sigma'],\xi)\ |\ [\Sigma'] \in \mathcal{B}_{\Sigma}, \xi \in G_{\Sigma'}\}$ such that $G|_{\mathcal{B}_{\Sigma}} \cong \mathcal{B}_{\Sigma}\times G_\Sigma$, induced from (\ref{homeomorphismgsigma1gsigma2}), is a homeomorphism over the base $\mathcal{B}_{\Sigma}$. We can repeat the above construction for any sufficiently small neighbourhood $\mathcal{B}_{\tilde{\Sigma}}$ of $[\tilde{\Sigma}] \in \mathcal{B}^\circ_{\Sigma_0}$, possibly with different consistent choice of $A,B$ symplectic basis cycles. The transition function
\begin{equation*}
    \mathcal{B}_{\Sigma}\cap \mathcal{B}_{\tilde{\Sigma}}\times G_\Sigma \xrightarrow{\cong} \left(G|_{\mathcal{B}_{\Sigma}}\right)|_{\mathcal{B}_{\tilde{\Sigma}}} = \left(G|_{\mathcal{B}_{\tilde{\Sigma}}}\right)|_{\mathcal{B}_\Sigma} \xrightarrow{\cong} \mathcal{B}_{\Sigma}\cap \mathcal{B}_{\tilde{\Sigma}}\times G_{\tilde{\Sigma}}
\end{equation*}
induced from the identity on $G|_{\mathcal{B}_{\Sigma}\cap\mathcal{B}_{\tilde{\Sigma}}}$ can be seen to be continuous. It follows that $G := \bigcup_{[\Sigma] \in \mathcal{B}^\circ_{\Sigma_0}}G|_{\mathcal{B}_{\Sigma}}$ with the canonical projection $\pi : G\rightarrow \mathcal{B}^\circ_{\Sigma_0}$ is a vector bundle with fiber over each point $[\Sigma] \in \mathcal{B}^\circ_{\Sigma_0}$ given by $\pi^{-1}([\Sigma]) = G_\Sigma$ and $G|_{\mathcal{B}_{\Sigma}}\cong \mathcal{B}_{\Sigma}\times G_\Sigma$ giving a local trivialization over each $\mathcal{B}_{\Sigma}$.

\begin{definition}\label{gvectorbundledefinition}
Given a choice of $\{\epsilon_{\alpha \in Ram}\}, \{M_{\alpha \in Ram}\}$, $\mathcal{U}_{Ram}$ and let $\mathcal{B}^\circ_{\Sigma_0}$ be the maximal open neighbourhood of $[\Sigma_0]$ for which C\ref{howtochoosebsigma0condition}.2 is satisfied. We define $(G\rightarrow \mathcal{B}^\circ_{\Sigma_0}, \nabla_{\mathcal{F}})$ to be the vector bundle $G := \bigcup_{[\Sigma] \in \mathcal{B}^\circ_{\Sigma_0}}G|_{\mathcal{B}_{\Sigma}}$ with fiber $G_\Sigma$ over each point $[\Sigma] \in \mathcal{B}^\circ_{\Sigma_0}$ and equipped with the connection $\nabla_{\mathcal{F}} : \Gamma(\mathcal{B}^\circ_{\Sigma_0}, G)\rightarrow \Gamma(\mathcal{B}^\circ_{\Sigma_0}, T^*\mathcal{B}\otimes G)$ given by the differentiation along a leaf of the foliation $\mathcal{F}$.
\end{definition}

For any open subset $\mathcal{B}' \subseteq \mathcal{B}^\circ_{\Sigma_0}$ we denote by $\Gamma(\mathcal{B}', G)$ the set of holomorphic sections of $G$ on $\mathcal{B}'$.

We will mostly be interested in the small deformations of the curve $\Sigma_0$, therefore, it is sufficient to consider the restriction of $G$ to the open neighbourhood $\mathcal{B}_{\Sigma_0} \subseteq \mathcal{B}^\circ_{\Sigma_0}$ of $[\Sigma_0]$, where $\mathcal{B}_{\Sigma_0}$ along with $\{\epsilon_{\alpha \in Ram}\}, \{M_{\alpha \in Ram}\}$ and $\mathcal{U}_{Ram}$ are chosen to satisfy Condition \ref{howtochoosebsigma0condition}. Since $\mathcal{B}_{\Sigma_0}$ is contractible by C\ref{howtochoosebsigma0condition}.1, the restriction of $G$ to $\mathcal{B}_{\Sigma_0}$ is a trivial vector bundle which we denote by $(G\rightarrow \mathcal{B}_{\Sigma_0}, \nabla_{\mathcal{F}})$. Similarly, the trivial vector sub-bundles of $G\rightarrow \mathcal{B}_{\Sigma_0}$:  $G^\perp\rightarrow \mathcal{B}_{\Sigma_0}$ with fiber $G^\perp_{\Sigma}$ and $V\rightarrow \mathcal{B}_{\Sigma_0}$ with fiber $V_\Sigma$ over $[\Sigma] \in \mathcal{B}_{\Sigma_0}$ may be constructed.

Let us recast some definitions and results in this section as follows. We have an injective vector bundle morphism $i : G\hookrightarrow W^{Ram}$ covering the map $\gamma : \mathcal{B}_{\Sigma_0}\hookrightarrow Discs^{Ram}$ given by (\ref{gammamapbtodiscs}), such that $i$ restricted to the fiber over any point $[\Sigma] \in \mathcal{B}_{\Sigma_0}$ is $i : G_{\Sigma}\hookrightarrow W^{Ram}_{t(\Sigma)}$ given by (\ref{imapdefinition}) (see Figure \ref{hgwvectorbundlesrelation}). Therefore, $G$ is a vector sub-bundle of $W^{Ram}$, in fact by Proposition \ref{gperpproposition} it is a coisotropic vector sub-bundle. Proposition \ref{sympreductionproposition} can be restated as an isomorphism of vector bundles $G/G^\perp \cong \mathcal{H}$. We define $[.] : G\rightarrow \mathcal{H}$ to be the morphism of vector bundle given fiber-wise by the natural projection $G_{\Sigma}\rightarrow G_{\Sigma}/G^\perp_{\Sigma} \cong \mathcal{H}_\Sigma \cong H^1(\Sigma, \mathbb{C})$ which sends any differential form $\xi \in G_{\Sigma}$ to its cohomology class $[\xi] \in H^1(\Sigma,\mathbb{C})$ (see Figure \ref{hgwvectorbundlesrelation}).

\subsubsection{The $Ram$ product residue constraints}\label{ramproductresconstraintssubsection}
So far, the residue constraints were only introduced on a disc with only a single ramification point. Since there are more than one ramification points on $\Sigma$, we will need to consider a $|Ram|$ copies of the residue constraints. The purpose of this section is to fix the needed notations regarding the $Ram$ product residue constraints Airy structure.

Consider the Airy structure obtained as a product of the residue constraints Airy structures introduced in Section \ref{residueconstraintairystructuresection}. Set $\mathbb{I} := \mathbb{Z}_{>0}\times Ram$ and let $V_{Airy}^{|Ram|} := \prod_{\alpha \in Ram}V_{Airy} = \sum_{\alpha \in Ram}[\alpha]\otimes z^{-1}\mathbb{C}[z^{-1}]\frac{dz}{z}$ be a discrete vector space with a basis $\{e^{k=1,2,3,\cdots, \alpha \in Ram} := [\alpha]\otimes z^{-k}\frac{dz}{z}\}$. The continuous dual space is given by $(V^{|Ram|}_{Airy})^* := \prod_{\alpha \in Ram}V^*_{Airy} = \sum_{\alpha \in Ram}[\alpha]\otimes z\mathbb{C}[[z]]\frac{dz}{z}$ with a basis $\{f_{k=1,2,3,\cdots, \alpha \in Ram} := [\alpha]\otimes kz^{k}\frac{dz}{z}\}$. Define the Tate space
\begin{equation*}
    W_{Airy}^{|Ram|} := \left\{\eta = \sum_{\alpha \in Ram}[\alpha]\otimes \eta_\alpha \ |\ \eta_\alpha \in \mathbb{C}((z))dz, Res_{z=0}\eta_\alpha = 0, \alpha \in Ram\right\} = V^{|Ram|}_{Airy}\oplus (V^{|Ram|}_{Airy})^*
\end{equation*}
with the symplectic form $\Omega_{Airy}([\alpha_1]\otimes\xi_1, [\alpha_2]\otimes \xi_2) = \delta_{\alpha_1,\alpha_2}Res_{z=0}(\xi_1\int\xi_2)$.  We define the $|Ram|$ product residue constraints Airy structure $\{(H_{Airy})_{k=1,2,3,\cdots, \alpha \in Ram}\}$ on $V^{|Ram|}_{Airy}$ by 
\begin{align*}\label{productresconstraintsairystructure}
    (H_{Airy})_{2n,\alpha}(w) &= Res_{z = 0}\left(\left(z - \frac{w(z)}{2zdz}\right)[\alpha]\otimes z^{2n}d(z^2)\right),\\
    (H_{Airy})_{2n-1,\alpha}(w) &= \frac{1}{2}Res_{z = 0}\left(\left(z - \frac{w(z)}{2zdz}\right)^2 [\alpha]\otimes z^{2n-2}d(z^2)\right),\qquad n = 1,2,3,...\numberthis
\end{align*}
where $w := \sum_{\alpha \in Ram}\sum_{k=1}^\infty(x^{k,\alpha}f_{k,\alpha} + y_{k,\alpha}e^{k,\alpha}) \in W^{|Ram|}_{Airy}$. We may write (\ref{productresconstraintsairystructure}) in the standard form (\ref{airystructuredefn}):
\begin{align*}
    (H_{Airy})_{i,\alpha} &= -y_{i,\alpha} + \sum_{\beta, \gamma \in Ram}\sum_{j,k=1}^\infty(a_{Airy})_{(i,\alpha)(j,\beta)(k,\gamma)}x^{j,\beta}x^{k,\gamma}\\
    &\qquad + \sum_{\beta, \gamma \in Ram}\sum_{j,k=1}^\infty 2(b_{Airy})_{(i,\alpha)(j,\beta)}^{k,\gamma}x^{j,\beta}y_{k,\gamma} + \sum_{\beta, \gamma \in Ram}\sum_{j,k=1}^\infty(c_{Airy})_{i,\alpha}^{(j,\beta)(k,\gamma)}y_{j,\beta}y_{k,\gamma}.
\end{align*}
and express the residue constraints Airy structure as $(V_{Airy}^{|Ram|}, A_{Airy}, B_{Airy}, C_{Airy})$ where
\begin{equation*}
    A_{Airy} := ((a_{Airy})_{(i,\alpha)(j,\beta)(k,\gamma)}),\qquad B_{Airy} := ((b_{Airy})^{k,\gamma}_{(i,\alpha)(j,\beta)}), \qquad C_{Airy} := ((c_{Airy})^{(j,\beta)(k,\gamma)}_{i,\alpha}).
\end{equation*}
The corresponding quantum Airy structure is $(V^{|Ram|}_{Airy}, A_{Airy}, B_{Airy}, C_{Airy}, \epsilon_{Airy})$ where $(\epsilon_{Airy})_{i,\alpha} := \frac{1}{16}\delta_{i,3}$ and $\epsilon_{Airy} = ((\epsilon_{Airy})_{i,\alpha})$. Define the quadratic Lagrangian submanifold $L^{|Ram|}_{Airy} := \{(H_{Airy})_{i=1,2,3,\cdots,\alpha \in Ram} = 0\} \subset W^{|Ram|}_{Airy}$ with $T_0L^{|Ram|}_{Airy} := \sum_{\alpha \in Ram}[\alpha] \otimes z\mathbb{C}[[z]]\frac{dz}{z} = (V^{|Ram|}_{Airy})^*$. It is clear that $T_0L^{|Ram|}_{Airy}$ is a Lagrangian subspace of $W^{|Ram|}_{Airy}$ and that $V^{|Ram|}_{Airy}$ is a Lagrangian complement of $T_0L^{|Ram|}_{Airy}$.

\begin{remark}
Our definitions of $W^{|Ram|}_{Airy}$ and $L^{|Ram|}_{Airy}$ are exactly the same as how $W^{Ram}_{Airy}$ and $L^{Ram}_{Airy}$ were defined in \cite[Section 7.1]{kontsevich2017airy}. Our choice of symbols came from the fact that $W^{|Ram|}_{Airy}$ and $L^{|Ram|}_{Airy}$ are the products of $|Ram|$ copies of $W_{Airy}$ and $L_{Airy}$ (as given in Section \ref{residueconstraintairystructuresection} or \cite[Section 3.3]{kontsevich2017airy}) respectively, where $|Ram|$ denotes the number of ramification points. Finally, note that by taking Laurent expansion of any $\xi \in G^{mer}_{\Sigma}$ we obtain the inclusion $i: G^{mer}_{\Sigma}\hookrightarrow W^{|Ram|}_{Airy}$ can also be defined exactly as we did in (\ref{imapdefinition}).
\end{remark}

Let us consider a new Lagrangian complement of $T_0L^{|Ram|}_{Airy}$:
\begin{equation*}
    V^{mer}_{\Sigma} := \left\{\xi \in G^{mer}_{\Sigma}\ |\ \oint_{A_i}\xi = 0, i = 1,\cdots,g\right\}
\end{equation*}
where $G^{mer}_{\Sigma}$ was given in Remark \ref{gmerremark}. Essentially, $V^{mer}_{\Sigma} \subseteq V_\Sigma$ is the subspace of meromorphic differential forms in $V_\Sigma$ with basis $\{\bar{e}^{i=1,2,3,\cdots,\alpha \in Ram}\}$ where $\bar{e}^{i,\alpha}$ is defined in Definition \ref{endifferentialdefinition} and note that (\ref{ebarintermofef}) can be written as
\begin{equation*}
    \bar{e}^{k,\alpha} = e^{k,\alpha} + \sum_{\alpha'\in Ram}\sum_{k'=1}^\infty \frac{1}{kk'}P^{(k,\alpha)(k',\alpha')}f_{k',\alpha'}.
\end{equation*}
The gauge transformation $(V^{|Ram|}_{Airy}, A_{Airy}, B_{Airy}, C_{Airy}, \epsilon_{Airy}) \mapsto (V^{mer}_\Sigma, \bar{A}_{Airy}, \bar{B}_{Airy}, \bar{C}_{Airy}, \bar{\epsilon}_{Airy})$ associating with the choice of basis $\{\eta^{i=1,2,3,\cdots} := \sum_{\beta \in Ram}\sum_{j=1}^\infty d^i_{j,\beta}\bar{e}^{j,\beta}\}$ for $V^{mer}_\Sigma$ and $\{\omega_{i=1,2,3,\cdots} := \sum_{\alpha\in Ram}\sum_{i=1}^\infty c^{i,\alpha}_jf_{i,\alpha}\}$ for $T_0L^{|Ram|}_{Airy}$ for some $(c^{i,\alpha}_j) : V^{|Ram|}_{Airy} \rightarrow V^{|Ram|}_{Airy}$ and its inverse $(d^i_{j,\beta}) : V^{|Ram|}_{Airy} \rightarrow V^{|Ram|}_{Airy}$, is given by 
\begin{align*}\label{resconstraintsatgaugetranfs}
    \numberthis
    (\bar{a}_{Airy})_{i_1i_2i_3} &= (a_{Airy})_{(j_1,\alpha_1)(j_2,\alpha_2)(j_3,\alpha_3)}c^{j_1,\alpha_1}_{i_1}c^{j_2,\alpha_2}_{i_2}c^{j_3,\alpha_3}_{i_3}\\
    (\bar{b}_{Airy})_{i_1i_2}^{i_3} &= \left((b_{Airy})_{(j_1,\alpha_1)(j_2,\alpha_2)}^{j_3,\alpha_3} + (a_{Airy})_{(j_1,\alpha_1)(j_2,\alpha_2)(k,\gamma)}s^{(k,\gamma)(j_3,\alpha_3)}\right)c^{j_1,\alpha_1}_{i_1}c^{j_2,\alpha_2}_{i_2}d^{i_3}_{j_3,\alpha_3}\\
    (\bar{c}_{Airy})_{i_1}^{i_2i_3} &= \Big((c_{Airy})_{(j_1,\alpha_1)}^{(j_2,\alpha_2)(j_3,\alpha_3)} + (b_{Airy})_{(j_1,\alpha_1)(k_1,\gamma_1)}^{(j_3,\alpha_3)}s^{(k_1,\gamma_1)(j_2,\alpha_2)}\\
    &\qquad + (b_{Airy})_{(j_1,\alpha_1)(k_2,\gamma_2)}^{(j_2,\alpha_2)}s^{(k_2,\gamma_2)(j_3,\alpha_3)} + (a_{TR})_{(j_1,\alpha_1)(k_1,\gamma_1)(k_2,\gamma_2)}s^{(k_1,\gamma_1)(j_2,\gamma_2)}s^{(k_2,\gamma_2)(j_3,\gamma_3)}\Big)\\
    &\qquad\qquad\qquad\times c^{j_1,\alpha_1}_{i_1}d_{j_2,\alpha_2}^{i_2}d_{j_3,\alpha_3}^{i_3}\\
    (\bar{\epsilon}_{Airy})_i &= \left((\epsilon_{Airy})_{j_1,\alpha_1} + (a_{Airy})_{(j_1,\alpha_1)(k_1,\gamma_1)(k_2,\gamma_2)}s^{(k_1,\gamma_1)(k_2,\gamma_2)} \right)c^{j_1,\alpha_1}_{i_1}.
\end{align*}
where $s^{(i,\alpha)(j,\beta)} = \frac{1}{ij}P^{(i,\alpha)(j,\beta)}$, $(s^{(i,\alpha)(j,\beta)}) : V^{|Ram|}_{Airy}\times V^{|Ram|}_{Airy} \rightarrow \mathbb{C}$.

\begin{remark}\label{resconstraintssubairystructureremark}
It is easy to see that $((V_{Airy})^{|Ram|}_{odd}, A_{Airy}, B_{Airy}, C_{Airy}, \epsilon_{Airy})$ is a quantum Airy sub-structure of $(V_{Airy}^{|Ram|}, A_{Airy}, B_{Airy}, C_{Airy}, \epsilon_{Airy})$ where $(V_{Airy})^{|Ram|}_{odd} := \prod_{\alpha \in Ram}(V_{Airy})_{odd} \subset V_{Airy}^{|Ram|}$. Likewise, let $(V^{mer}_{\Sigma})_{odd} \subset V^{mer}_\Sigma$ be the vector space spanned by finite linear combinations of vectors $\{\bar{e}^{k=odd, \alpha \in Ram}\}$ then $((V^{mer}_\Sigma)_{odd}, \bar{A}_{Airy}, \bar{B}_{Airy}, \bar{C}_{Airy}, \bar{\epsilon}_{Airy})$ is a quantum Airy sub-structure of
$(V^{mer}_{\Sigma}, \bar{A}_{Airy}, \bar{B}_{Airy}, \bar{C}_{Airy}, \bar{\epsilon}_{Airy})$ according to Remark \ref{gaugetransfrespectssubstructuresremark}.
\end{remark}

The consideration of the analytic residue constraints Airy structure in Section \ref{analyticresconstraintssubsection} straightforwardly generalizes to the case of the $Ram$ product residue constraints Airy structure because $Ram$ is a finite set. In particular, we treat (\ref{productresconstraintsairystructure}) analytically as functions $(H_{Airy})_{i,\alpha} : W^{Ram}_{Airy} \rightarrow \mathbb{C}$ and we also have that $L^{Ram}_{Airy} = \{(H_{Airy})_{i=1,2,3,\cdots, \alpha \in Ram} = 0\} \subset W^{Ram}_{Airy}$. The analytic counter-part of $V^{|Ram|}_{Airy}$ is obviously given by $V_{Airy}^{Ram} := \prod_{\alpha \in Ram}V^{\epsilon_\alpha}_{Airy}$. It is clear that $V^{Ram}_{Airy}$ is a Lagrangian complement of $T_0L^{Ram}_{Airy}$ in $W^{Ram}_{Airy} = V^{Ram}_{Airy} \oplus T_0L^{Ram}_{Airy}$. 

According to Corollary \ref{vislagrangiancomplementcorollary}, $V_\Sigma$ is another Lagrangian complement of $T_0L^{Ram}_{Airy}$. Let us use $\{\eta^{i=1,2,3,\cdots} := \sum_{\beta \in Ram}\sum_{j=1}^\infty d^i_{j,\beta}\bar{e}^{j,\beta}\}$ as a basis of $V_\Sigma$ and $\{\omega_{i=1,2,3,\cdots} := \sum_{\alpha\in Ram}\sum_{i=1}^\infty c^{i,\alpha}_jf_{i,\alpha}\}$ as a basis of $T_0L^{Ram}_{Airy}$ for some $(c^{i,\alpha}_j), (d^i_{j,\beta})$ satisfying Condition \ref{analyticcdscondition}. Note that $(s^{(i,\alpha)(j,\beta)})$ where $s^{(i,\alpha)(j,\beta)} = \frac{1}{ij}P^{(i,\alpha)(j,\beta)}$ also satisfies Condition \ref{analyticcdscondition} since it came from the series expansion of a Bergman kernel. Therefore, the change of the canonical basis $\{e^{k,\alpha}, f_{k,\alpha}\} \mapsto \{\eta^k, \omega_k\}$ of $W^{Ram}_{Airy}$ is of the same type as considered in Section \ref{analyticresconstraintssubsection}. Suppose that $w = \sum_{\alpha \in Ram}\sum_{k=1}^\infty (x^{k,\alpha}f_{k,\alpha} + y_{k,\alpha}e^{k,\alpha}) \in W^{Ram}_{Airy}$, then by Proposition \ref{analyticgaugetransformationproposition} we have $w = \sum_{k=1}^\infty(\alpha^k\omega_k + \beta_k\eta^k)$ where
\begin{equation}\label{bermangaugecoordtransf}
    \alpha^i := \sum_{\beta \in Ram}\sum_{j=1}^\infty d^i_{j,\beta}\left(x^{j,\beta} - \sum_{\gamma\in Ram}\sum_{k=1}^\infty s^{(j,\beta),(k,\gamma)}y_{k,\gamma}\right), \qquad \beta_i := \sum_{\beta\in Ram}\sum_{j=1}^\infty c^{j,\beta}_iy_{j,\beta}
\end{equation}
are the new canonical coordinates of $W^{Ram}_{Airy}$. Also, by Proposition \ref{analyticgaugetransformationproposition} we have $(\bar{H}_{Airy})_i := \sum_{\beta\in Ram}\sum_{j=1}^\infty c^{j,\beta}_i(H_{Airy})_{j,\beta}$ given by
\begin{equation}\label{resconstraintshamiltonianinbergmangauge}
    (\bar{H}_{Airy})_i := -\beta_i + \sum_{j,k=1}^\infty (\bar{a}_{Airy})_{ijk}\alpha^j\alpha^k + 2\sum_{j,k=1}^\infty (\bar{b}_{Airy})_{ij}^k\alpha^i\beta_k + \sum_{j,k=1}^\infty(\bar{c}_{Airy})_i^{jk}\beta_j\beta_k
\end{equation}
where $\bar{A}_{Airy} = ((\bar{a}_{Airy})_{ijk}), \bar{B}_{Airy} = ((\bar{b}_{Airy})_{ij}^{k}), \bar{C}_{Airy} = ((\bar{c}_{Airy})_{i}^{jk})$ are given by (\ref{resconstraintsatgaugetranfs}). 

\begin{remark}\label{bergmangaugeremark}
Moreover, it is straightforward to generalize Proposition \ref{analyticatrproposition} to the case of the $Ram$ product residue constraints Airy structure. We conclude that if $i_\alpha\left(\sum_{i=1}^\infty\pmb{\beta}_i(\alpha^1,\cdots,\alpha^K)\eta^i\right)(z_\alpha) \in V^{\epsilon_\alpha}_{Airy}$, where $\pmb{\beta}_i(\alpha^1,\cdots,\alpha^K) := \beta_i(\{\alpha^1,\cdots,\alpha^K, \alpha^{k>K}=0\})$, is holomorphic in $\{\alpha^1,\cdots,\alpha^K, z^{-1}_\alpha\}$ on some open neighbourhood of $(\alpha^1,\cdots,\alpha^K) = (0,\cdots,0)$, $|z_\alpha| > \epsilon_\alpha$ and 
\begin{equation*}
    (\bar{H}_{Airy})_i\left(\{\alpha^1,\cdots,\alpha^K,\alpha^{k > K}=0\}, \{\pmb{\beta}_{k=1,2,3,\cdots}(\alpha^1,\cdots,\alpha^K)\}\right) = 0
\end{equation*}
then $\pmb{\beta}_i = (\partial_i S_0)|_{\alpha^{k>K} = 0}$. Where $S_0 = \sum_{n=1}^\infty S_{0,n} \in \prod_{n=1}^\infty Sym_n(V^{mer}_{\Sigma})$ is the genus zero part of the ATR output using the quantum Airy structure $(V^{mer}_{\Sigma},\bar{A}_{Airy}, \bar{B}_{Airy}, \bar{C}_{Airy}, \bar{\epsilon}_{Airy})$.
\end{remark}

\subsubsection{Connection $\nabla_{\mathcal{F}}$ and parallel transport on $G\rightarrow \mathcal{B}_{\Sigma_0}$}\label{connectionandparalleltransportongsubsection}
In this section, we are going to more carefully examine the connection $\nabla_{\mathcal{F}}$ and discuss the parallel transport on $G\rightarrow \mathcal{B}_{\Sigma_0}$. For now, let $\{\epsilon_{\alpha\in Ram}\}$, $\{M_{\alpha \in Ram}\}$, $\mathcal{U}_{Ram}$ and an open neighbourhood $\mathcal{B}_{\Sigma_0} \subseteq \mathcal{B}$ of $[\Sigma_0] \in \mathcal{B}$ be chosen such that C\ref{howtochoosebsigma0condition}.1 and C\ref{howtochoosebsigma0condition}.2 are satisfied. Recall from Definition \ref{gvectorbundledefinition} that $\nabla_{\mathcal{F}}$ is given by the differentiation along a leaf of the foliation $\mathcal{F}$. Let us verify that this indeed defines a connection on $G\rightarrow \mathcal{B}_{\Sigma_0}$ which is flat.

\begin{lemma}
For any section $\xi \in \Gamma(\mathcal{B}_{\Sigma_0}, G)$, we have $\nabla_{\mathcal{F}}\xi \in \Gamma(\mathcal{B}_{\Sigma_0}, T^*\mathcal{B}\otimes G)$.
\end{lemma}
\begin{proof}
Consider $[\Sigma] \in \mathcal{B}_{\Sigma_0}$. Let $\mathcal{U} := \{(U_{p\in\sigma}, x_{p\in\sigma}, y_{p\in\sigma})\}$ be a collection of $(\mathcal{F}, \Omega_S)$-charts covering $\Sigma\setminus \cup_{\alpha\in Ram}\mathbb{D}_{\alpha, \epsilon_\alpha}(\Sigma)$ where $\Sigma\cap U_p$ are bi-holomorphic to a simply-connected open subset of $\mathbb{C}$ with $x_p$ as a local coordinate and $\sigma$ is some finite index set. It follows that $dx_p|_{\Sigma \cap U_p}\neq 0$ and so $r_\alpha(\Sigma) \notin \cup_{p \in \sigma}\Sigma \cap U_p$ for all $\alpha \in Ram$.

Let $\{u^1,...,u^g\}$ be coordinates of $\mathcal{B}_{\Sigma_0}$ (this exists due to the condition C\ref{howtochoosebsigma0condition}.1).
Given a section $\xi \in \Gamma(\mathcal{B}_{\Sigma_0}, G)$, $\xi_\Sigma \in G_\Sigma$ and a tangent vector $v = \sum_{k=1}^gv^k\frac{\partial}{\partial u^k} \in T_{[\Sigma]}\mathcal{B}_{\Sigma_0}$ then for each $p \in \sigma$ we have $\iota_{v}\nabla_{\mathcal{F}}\xi|_{\Sigma \cap U_{p}} = \sum_{k=1}^gv^k\frac{\partial}{\partial u_k}|_{x_p = const}(\xi|_{\Sigma \cap U_p})$. Using $x_p$ as a local coordinate on $\Sigma \cap U_p$ and write $\xi|_{\Sigma\cap U_p} = \xi_p(x_p,u)dx$. It follows that $\iota_v\nabla_{\mathcal{F}}\xi|_{\Sigma \cap U_{p}} = \sum_{k=1}^g v^k\frac{\partial}{\partial u^k}(\xi_p(x_p,u))dx_p$. We patch $\iota_v\nabla_{\mathcal{F}}\xi|_{\Sigma\cap U_p}$ for each $p \in\sigma$ together to obtain $\iota_v\nabla_{\mathcal{F}}\xi$ which is holomorphic on $\Sigma \setminus \cup_{\alpha \in Ram}\bar{\mathbb{D}}_{\alpha, \bar{\epsilon}_\alpha}(\Sigma)\subset \cup_{p\in \sigma}\Sigma \cap U_p$ for some $\bar{\epsilon}_\alpha < \epsilon_\alpha, \forall \alpha \in Ram$. 

Let us show that $\oint_{\partial \bar{\mathbb{D}}_{\alpha, \epsilon_\alpha}(\Sigma)}\nabla_{\mathcal{F}}\xi = 0$. Write $\xi$ on $\mathbb{A}_{\alpha, \epsilon_\alpha, M_\alpha}(\Sigma)$ in the standard local coordinate as $\xi(z_\alpha) = \sum_{k\neq 0} \xi_{\alpha,k}z^k_\alpha \frac{dz_\alpha}{z_\alpha}$. Since $x_\alpha = a_\alpha(\Sigma) + z^2_\alpha$, we have $\nabla_{\mathcal{F}}(z_\alpha) = \sum_{k=1}^gdu^k\frac{\partial a_\alpha(\Sigma)}{\partial u^k}\frac{\partial}{\partial a_\alpha}|_{x_\alpha = const}z_\alpha = -\frac{1}{2z_\alpha}\sum_{k=1}^gdu^k\frac{\partial a_\alpha(\Sigma)}{\partial u^k}$. It follows that $\nabla_{\mathcal{F}}(z_\alpha^kdz_\alpha) = \frac{1}{2}\sum_{k=1}^gdu^k\frac{\partial a_\alpha(\Sigma)}{\partial u^k}(1-k)z^{k-2}_\alpha dz_\alpha$, therefore $\iota_v\nabla_{\mathcal{F}}\xi$ is free of $\frac{dz_\alpha}{z_\alpha}$ term. Hence $\iota_v\nabla_{\mathcal{F}}\xi_\Sigma \in G_{\Sigma}$, in other words,  $\nabla_{\mathcal{F}}\xi \in \Gamma(\mathcal{B}_{\Sigma_0},T^*\mathcal{B}\otimes G)$, as required.
\end{proof}

\begin{lemma}
The vector bundle $G\rightarrow \mathcal{B}_{\Sigma_0}$ with the connection $\nabla_{\mathcal{F}}$ is flat.
\end{lemma}
\begin{proof}
This is mostly identical to the proof of Lemma \ref{hisaflatbundlelemma}, but let us write it more concretely this time. Let $d_{\nabla_{\mathcal{F}}} : \Gamma(\mathcal{B}_{\Sigma_0}, \Omega^r_{\mathcal{B}}\otimes G) \rightarrow \Gamma(\mathcal{B}_{\Sigma_0}, \Omega^{r+1}_{\mathcal{B}}\otimes G)$ be the exterior covariant derivative. Suppose that $\xi \in \Gamma(\mathcal{B}_{\Sigma_0}, \Omega^r_{\mathcal{B}}\otimes G)$. On each $\Sigma \cap U_p$ we have $\xi|_{\Sigma\cap U_p} = \sum_{i_1,\cdots, i_r = 1}^g(\xi_{i_1\cdots i_r; p}(x_p;u)dx_p)du^{i_1}\wedge \cdots du^{i_r}$, therefore
\begin{equation*}
d_{\nabla_{\mathcal{F}}} \xi|_{\Sigma\cap U_p} = \nabla_{\mathcal{F}}\xi|_{\Sigma\cap U_p} = \sum_{k=1}^g\sum_{i_1,\cdots, i_r = 1}^g\left(\frac{\partial}{\partial u^k}\Big|_{x_p = const}(\xi_{i_1\cdots i_r;p}(x_p,u))dx_p\right)du^k\wedge du^{i_1}\wedge \cdots \wedge du^{i_r}    
\end{equation*}
 and
 \begin{equation*}
     d^2_{\nabla_{\mathcal{F}}} \xi|_{\Sigma\cap U_p} = \sum_{k,l=1}^g\sum_{i_1,\cdots, i_r = 1}^g\left(\frac{\partial^2}{\partial u^l\partial u^k}\Big|_{x_p = const}(\xi_{i_1\cdots i_r;p}(x_p,u))dx_p\right)du^l\wedge du^k\wedge du^{i_1}\wedge \cdots \wedge du^{i_r} = 0.
 \end{equation*}
Therefore, $d^2_{\nabla_{\mathcal{F}}} = 0$, so $\nabla_{\mathcal{F}}$ is flat.
\end{proof}

Next, we turn our attention to the parallel transport on $G$ using the connection $\nabla_{\mathcal{F}}$. Since the connection is flat, parallel transport is path-independent. The parallel transport of $\xi \in G_{\Sigma}$ to the fiber $G_{\Sigma'}$ only exists if $[\Sigma] \in \mathcal{B}$ is close enough to $[\Sigma'] \in \mathcal{B}$. This is much like the situation in Section \ref{theembeddingofdiscssection}: given $\xi\in W^{\epsilon, M}_{Airy}$ then $\exp(a\mathcal{L}_{\frac{1}{2z}\partial_z})\xi \in W^{\epsilon, M}_{Airy}$ only if $|a|$ is sufficiently small. The idea is to first construct a map $s_{\Sigma',\Sigma}$ sending points on $\Sigma'$ to $\Sigma$ along the foliation leaves, then the parallel transport of $\xi \in G_{\Sigma}$ to $G_{\Sigma'}$ is given by the pull-back $s^*_{\Sigma',\Sigma}\xi$. 

Let us discuss this in more detail. Suppose that $[\Sigma], [\Sigma'] \in \mathcal{B}_{\Sigma_0}$, then they are path-connected since $\mathcal{B}_{\Sigma_0}$ is contractible as required by C\ref{howtochoosebsigma0condition}. Let $\pmb{\Sigma} : [0,1] \rightarrow \mathcal{B}_{\Sigma_0}$ be one such path, $\pmb{\Sigma}(0) = [\Sigma']$, and $\pmb{\Sigma}(1) = [\Sigma]$. For some $\bar{\epsilon}'_\alpha < \epsilon_\alpha$ let us send a point $q' \in \Sigma'\setminus \cup_{\alpha \in Ram}\bar{\mathbb{D}}_{\alpha, \bar{\epsilon}'_\alpha}(\Sigma')$ traveling along a leaf of the foliation $\mathcal{F}$. If $[\Sigma']$ is sufficiently close to $\pmb{\Sigma}(t) \in \mathcal{B}_{\Sigma_0}$ for any $t \in [0,1]$, then the point $q'$ will intersect $\pmb{\Sigma}(t)$, potentially multiple times. As we move continuously along the path $\pmb{\Sigma}$, these intersection points will also move continuously in $S$. Therefore, for all $t \in [0,1]$ sufficiently small, we can define $q(t) \in \pmb{\Sigma}(t)$ to be the intersection point which moves continuously from $q(0) = q'$. Now, suppose that $[\Sigma]$, $[\Sigma']$ and every points along the path $\pmb{\Sigma}$ are sufficiently close to $[\Sigma']$, then we define the following map

\begin{definition}\label{ssigmadefinition}
The map $s_{\Sigma', \Sigma}:\Sigma'\setminus \cup_{\alpha \in Ram}\bar{\mathbb{D}}_{\alpha, \bar{\epsilon}'_\alpha}(\Sigma')\rightarrow \Sigma$ is given by sending any point $q' \in \Sigma'\setminus \cup_{\alpha \in Ram}\bar{\mathbb{D}}_{\alpha, \bar{\epsilon}'_\alpha}(\Sigma')$ along a leaf of the foliation to $q(1) \in \Sigma$.
\end{definition}

\begin{remark}
Technically, the map $s_{\Sigma', \Sigma}$ should be denoted by $s_{\Sigma', \Sigma, \{\bar{\epsilon}'_\alpha\}}$ since the set of parameters $\{\bar{\epsilon}'_{\alpha\in Ram}\}$ is a part of the definition. We are going to stick with the notation $s_{\Sigma', \Sigma}$ for simplicity, keeping in mind that $s_{\Sigma', \Sigma}$ is defined on some open subset in $\Sigma'$ containing $\Sigma'\setminus \cup_{\alpha \in Ram}\bar{\mathbb{D}}_{\alpha, \epsilon_\alpha}(\Sigma')$. 
\end{remark}

Suppose further that $[\Sigma], [\Sigma'] \in \mathcal{B}_{\Sigma_0}$ are sufficiently close such that the image of $s_{\Sigma',\Sigma}$ is contained in $\Sigma\setminus \cup_{\alpha\in Ram}\bar{\mathbb{D}}_{\alpha, \bar{\epsilon}_\alpha}(\Sigma)$ for some $\bar{\epsilon}_{\alpha} < \epsilon_\alpha$. Then given any holomorphic differential form $\xi$ on $\Sigma\setminus \cup_{\alpha\in Ram}\bar{\mathbb{D}}_{\alpha, \bar{\epsilon}_\alpha}(\Sigma)$, the pull-back $s^*_{\Sigma',\Sigma}\xi$ is a holomorphic differential form on $\Sigma'\setminus \cup_{\alpha \in Ram}\bar{\mathbb{D}}_{\alpha, \bar{\epsilon}'_\alpha}(\Sigma')$. The following lemma gives a more concrete sufficient condition for $[\Sigma], [\Sigma'] \in \mathcal{B}$ to be considered `sufficiently close' and shows that if $\xi \in G_{\Sigma}$ then $s^*_{\Sigma', \Sigma}\xi \in G_{\Sigma'}$. 

\begin{lemma}\label{sstaroperatorlemma}
Let $[\Sigma], [\Sigma'] \in \mathcal{B}$. Given $\xi \in G_{\Sigma}$, where $\xi$ is holomorphic on $\Sigma\setminus \cup_{\alpha\in Ram}\bar{\mathbb{D}}_{\alpha,\bar{\epsilon}_\alpha}(\Sigma)$ for some $\bar{\epsilon}_\alpha < \epsilon_\alpha, \forall\alpha \in Ram$. If we could cover $\Sigma\setminus \cup_{\alpha \in Ram}\mathbb{D}_{\alpha,\epsilon_\alpha}(\Sigma)$ with a collection of $(\mathcal{F},\Omega_S)$-charts $\mathcal{U} = \{(U_{p\in \sigma}, x_{p\in\sigma}, y_{p\in\sigma})\}$ such that $\Sigma\setminus \cup_{\alpha \in Ram}\mathbb{D}_{\alpha, \epsilon_\alpha}(\Sigma) \subset \cup_{p\in\sigma}\Sigma\cap U_p \subset \Sigma\setminus \cup_{\alpha \in Ram}\bar{\mathbb{D}}_{\alpha, \bar{\epsilon}_\alpha}(\Sigma)$ and the following properties are satisfied:
\begin{enumerate}
    \item For each $p \in \sigma, \Sigma\cap U_p$ is bi-holomorphic to a simply-connected open subset of $\mathbb{C}$ with $x_p$ as a local coordinate.
    \item Any $q' \in U_p$ is path-connected in $U_p$ via a leaf of the foliation $\mathcal{F}$ to a unique point $q$ on $\Sigma\cap U_p$.
    \item For each $p \in \sigma$, $\Sigma'\cap U_p$ is bi-holomorphic to a simply-connected open subset of $\mathbb{C}$ with $x_p$ as a local coordinate and $\Sigma'\setminus \cup_{\alpha \in Ram}\mathbb{D}_{\alpha,\epsilon_\alpha}(\Sigma') \subset \cup_{p\in\sigma}\Sigma'\cap U_p$.
\end{enumerate}
then $s^*_{\Sigma',\Sigma}\xi \in G_{\Sigma'}$.
\end{lemma}
\begin{proof}
Let $\xi \in G_{\Sigma}$ be a holomorphic form defined on $\Sigma\setminus \cup_{\alpha \in Ram}\bar{\mathbb{D}}_{\alpha,\bar{\epsilon}_\alpha}(\Sigma)$ for some $\bar{\epsilon}_\alpha < \epsilon_\alpha$ such that $\cup_{p\in \sigma}\Sigma\cap U_p\subset \Sigma\setminus \cup_{\alpha \in Ram}\bar{\mathbb{D}}_{\alpha,\bar{\epsilon}_\alpha}(\Sigma)$. By the properties of the covering $\mathcal{U}$, Definition \ref{ssigmadefinition} gives a well-defined map 
\begin{equation*}
    s_{\Sigma',\Sigma}: \Sigma'\setminus \cup_{\alpha \in Ram}\bar{\mathbb{D}}_{\alpha, \bar{\epsilon}'_\alpha}(\Sigma') \subseteq \cup_{p \in \sigma}\Sigma'\cap U_p\rightarrow \cup_{p\in\sigma}\Sigma\cap U_p \subseteq \Sigma\setminus \cup_{\alpha \in Ram}\bar{\mathbb{D}}_{\alpha, \bar{\epsilon}_\alpha}(\Sigma)
\end{equation*} 
for some $\bar{\epsilon}'_\alpha < \epsilon_\alpha$. Let us verify that the pull-back $s_{\Sigma',\Sigma}^*\xi$ is in $G_{\Sigma'}$. We note that $x_p$ is a local coordinate for both $\Sigma\cap U_p$ and $\Sigma'\cap U_p$. In $x_p$ local coordinate, we can express $s_{\Sigma',\Sigma}$ as
\begin{equation*}
x_p \circ s_{\Sigma',\Sigma}|_{\Sigma'\cap U_p} \circ x_p^{-1} = id_{\mathbb{C}}|_{x_p(\Sigma'\cap U_p)}: x_p(\Sigma'\cap U_p)\rightarrow x_p(\Sigma'\cap U_p)\subset x_p(\Sigma\cap U_p) \subset \mathbb{C}.
\end{equation*}
In particular, $s_{\Sigma',\Sigma}$ is locally bi-holomorphic from each $\Sigma'\cap U_p$ to its image. It follows that $s^*_{\Sigma',\Sigma}\xi|_{\Sigma'\cap U_p}$ is holomorphic on $\Sigma'\cap U_p$ for all $p \in \sigma$, hence $s^*_{\Sigma', \Sigma}\xi$ is holomorphic on $\Sigma'\setminus \cup_{\alpha \in Ram}\mathbb{D}_{\alpha,\bar{\epsilon}'_\alpha}(\Sigma') \subset \cup_{p\in \sigma}\Sigma'\cap U_p$. Lastly, we need to check that $\oint_{\partial\bar{\mathbb{D}}_{\alpha, \epsilon_\alpha}(\Sigma')}s^*_{\Sigma',\Sigma}\xi = 0$, but this follows from the lemma below:
\begin{lemma}\label{integralofparalleltransportlemma}
For any $\xi \in G_{\Sigma}$ holomorphic on $\Sigma\setminus \cup_{\alpha \in Ram}\bar{\mathbb{D}}_{\alpha, \bar{\epsilon}_\alpha}(\Sigma) \subset \cup_{p\in\sigma}\Sigma\cap U_p$ and a path $C \subset\Sigma'\setminus \cup_{\alpha \in Ram}\bar{\mathbb{D}}_{\alpha,\epsilon_\alpha}(\Sigma')$, we have
\begin{equation*}
    \int_{C}s^*_{\Sigma', \Sigma}\xi = \int_{(s_{\Sigma',\Sigma})_*C}\xi,
\end{equation*}
where $(s_{\Sigma',\Sigma})_*C \subset \Sigma\setminus \cup_{\alpha\in Ram}\bar{\mathbb{D}}_{\alpha, \bar{\epsilon}_\alpha}(\Sigma)$ is the image of $C$ under the map $s_{\Sigma',\Sigma}$.
\end{lemma}
\begin{proof}
We dissect $C$ into closed segments $C_i, i = 1,...,N$ such that $C_i \subset \Sigma'\cap U_{p_i}$ for some $p_i \in \sigma$. As we have seen, $s_{\Sigma', \Sigma}|_{\Sigma'\cap U_{p_i}}$ is a restriction of the identity, hence it is biholomorphic. So the Lemma holds for each segment $C_i$. Summing the integral over all segments $C_i$ proving the lemma.
\end{proof}

It follows from the construction of the map $s_{\Sigma', \Sigma}$ that $(s_{\Sigma',\Sigma})_*(\partial \bar{\mathbb{D}}_{\alpha,\epsilon_\alpha}(\Sigma'))$ is homotopic to $\partial \bar{\mathbb{D}}_{\alpha,\epsilon_\alpha}(\Sigma)$ because we can deform one to another by moving along a path $\pmb{\Sigma}$ linking $[\Sigma]$ and $[\Sigma']$. We conclude that $\oint_{\partial \bar{\mathbb{D}}_{\alpha,\epsilon_\alpha}(\Sigma')}s^*_{\Sigma',\Sigma}\xi = \oint_{\partial \bar{\mathbb{D}}_{\alpha,\epsilon_\alpha}(\Sigma)}\xi = 0$ because $\xi \in G_\Sigma$ which shows $s^*_{\Sigma',\Sigma}\xi \in G_{\Sigma'}$ as claimed.
\end{proof}

Next, we show that the pull-back via $s_{\Sigma',\Sigma}$ is the parallel transport on $G$. If Lemma \ref{sstaroperatorlemma} can be compared to Lemma \ref{expoperatorlemma} then the following is analogous to Lemma \ref{paralleltransportinwlemma}:
\begin{lemma}\label{paralleltransportinglemma}
Let $[\Sigma]\in \mathcal{B}_{\Sigma_0}$ and $\xi\in G_{\Sigma}$ be given, then $[\Sigma'] \mapsto \xi_{\Sigma'} := s^*_{\Sigma',\Sigma}\xi$ is a parallel vector field, defined over the subset of all $[\Sigma']\in \mathcal{B}_{\Sigma_0}$ such that the conditions of Lemma \ref{sstaroperatorlemma} are satisfied. In other words, $\xi_{\Sigma'} \in G_{\Sigma'}$ is the parallel transport of $\xi \in G_{\Sigma}$ from $[\Sigma]$ to $[\Sigma']$ on $G\rightarrow \mathcal{B}_{\Sigma_0}$ using the connection $\nabla_{\mathcal{F}}$.
\end{lemma}
\begin{proof}
Since $x_p$ is a local coordinate of $\Sigma$ on $\Sigma \cap U_p$ we can write $\xi|_{\Sigma\cap U_p} = f(x_p)dx_p$ for some holomorphic function $f$ which is constant with respect to the coordinates of $\mathcal{B}_{\Sigma_0}$. Because $s^*_{\Sigma',\Sigma}x_p = x_p$ by definition, we have $s^*_{\Sigma',\Sigma}\xi|_{\Sigma'\cap U_p} = f(x_p)dx_p$. Therefore, $\nabla_{\mathcal{F}}s^*_{\Sigma',\Sigma}\xi|_{\Sigma'\cap U_p} = 0$ for all $p \in \sigma$ and so $\nabla_{\mathcal{F}}s^*_{\Sigma',\Sigma}\xi = 0$. 
\end{proof}

Now that all the relevant results have been established, let us assume that $\{\epsilon_{\alpha\in Ram}\}, \{M_{\alpha\in Ram}\}, \mathcal{U}_{Ram}$ and $\mathcal{B}_{\Sigma_0}$ also satisfy C\ref{howtochoosebsigma0condition}.3 and C\ref{howtochoosebsigma0condition}.4 in addition to C\ref{howtochoosebsigma0condition}.1 and C\ref{howtochoosebsigma0condition}.2. Let us show in the following that the image of $\mathcal{B}_{\Sigma_0}$ under $\gamma$ defined in (\ref{gammamapbtodiscs}) is contained in $Discs^{Ram}_{t_0}$. This fact will be useful in comparing parallel transport on $W^{Ram}\rightarrow Discs^{Ram}_{t_0}$ with parallel transport on $G\rightarrow \mathcal{B}_{\Sigma_0}$.
\begin{lemma}\label{gammamaplemma}
Suppose that $\{\epsilon_{\alpha\in Ram}\}, \{M_{\alpha\in Ram}\}, \mathcal{U}_{Ram}$ and $\mathcal{B}_{\Sigma_0}$ satisfies the Condition \ref{howtochoosebsigma0condition} then the image of $\gamma : \mathcal{B}_{\Sigma_0} \hookrightarrow Discs^{Ram}$ is contained in $Discs^{Ram}_{t_0}$. In other words, we have
\begin{equation}\label{gammamapdefinition}
    \gamma : \mathcal{B}_{\Sigma_0} \hookrightarrow Discs^{Ram}_{t_{0}} \hookrightarrow Discs^{Ram}.
\end{equation}
\end{lemma}
\begin{proof}
C\ref{howtochoosebsigma0condition}.1 and C\ref{howtochoosebsigma0condition}.2 ensure that the map $\gamma : \mathcal{B}_{\Sigma_0} \rightarrow Discs^{Ram}$ is well-defined. We are going to show that C\ref{howtochoosebsigma0condition}.3 implies that if $\sum_{k=0}^\infty b_{k\alpha}(\Sigma)z_\alpha^k$ has a radius of convergence $\bar{M}_\alpha$, then $\min(\bar{M}_\alpha^2-M_\alpha^2, \epsilon_\alpha^2) > |a_\alpha(\Sigma) - a_\alpha(\Sigma_0)|$ for all $[\Sigma] \in \mathcal{B}_{\Sigma_0}$.

It is well-known that there is no point on a compact Riemann surface where all holomorphic differential forms vanish simultaneously \cite{griffiths2014principles}. In particular, for any ramification point, $r_\alpha, \alpha \in Ram$ there exists a holomorphic differential form $\omega\in \Gamma(\Sigma_0, \Omega^1_{\Sigma_0}) \subset G_{\Sigma_0}$ such that $\omega(r_\alpha) \neq 0$. For clarity, let us denote the standard local coordinate on $\Sigma_0\cap U_\alpha$ by $z_{\alpha, \Sigma_0}$ and the standard local coordinate on $\Sigma\cap U_\alpha$ by $z_\alpha$. Let us write $i_\alpha(\omega) = \sum_{k=1}^\infty\omega_{\alpha, k}z^k_{\alpha,\Sigma_0} \frac{dz_{\alpha,\Sigma_0}}{z_{\alpha,\Sigma_0}} \in T_0L^{M_\alpha}_{Airy}\subset W^{\epsilon_\alpha, M_\alpha}_{t_\alpha(\Sigma_0)}$ where $\sum_{k=1}^\infty$ converges for $|z_{\alpha,\Sigma_0}| < \bar{M}_\alpha$ for some $\bar{M}_\alpha > M_\alpha$. Using the fact that $x_\alpha = a_\alpha(\Sigma) + z^2_\alpha$ on $\Sigma\cap U_\alpha$ and $x_\alpha = a_\alpha(\Sigma_0) + z^2_{\alpha,\Sigma_0}$ on $\Sigma_0\cap U_\alpha$, it follows that 
\begin{equation*}
    \omega(z_{\alpha,\Sigma_0})|_{\mathbb{A}_{\alpha, \epsilon_\alpha, M_\alpha}(\Sigma_0)} = \sum_{k=1}^\infty \omega_{\alpha, k}\left(\sqrt{x_\alpha - a_\alpha(\Sigma_0)}\right)^{k-1}d\sqrt{x_\alpha - a_\alpha(\Sigma_0)}
\end{equation*}
implies 
\begin{equation}\label{sstarofxionsigma}
    s^*_{\Sigma, \Sigma_0}\omega(z_\alpha)|_{\mathbb{A}_{\alpha, \epsilon_\alpha, M_\alpha}(\Sigma)} = \sum_{k=1}^\infty \omega_{\alpha, k}\left(\sqrt{z^2_\alpha + a_\alpha(\Sigma) - a_\alpha(\Sigma_0)}\right)^{k-1}d\sqrt{z^2_\alpha + a_\alpha(\Sigma) - a_\alpha(\Sigma_0)}.
\end{equation}
By C\ref{howtochoosebsigma0condition}.3 we have that $s^*_{\Sigma, \Sigma_0}\omega$ is holomorphic on $\mathbb{A}_{\alpha, \epsilon_\alpha, M_\alpha}(\Sigma)$. For the expression in (\ref{sstarofxionsigma}) to be holomorphic on $\mathbb{A}_{\alpha, \epsilon_\alpha, M_\alpha}(\Sigma)$ we must have $\bar{M}^2_\alpha > M_\alpha^2 + |a_\alpha(\Sigma) - a_\alpha(\Sigma_0)|$ and either $\epsilon_\alpha^2 > |a_\alpha(\Sigma) - a_\alpha(\Sigma_0)|$ or that $\omega_{\alpha, k} = 0$ for all $k$ odd and all $\alpha \in Ram$. But the latter option is not possible since $\omega(r_\alpha) \neq 0$.
\end{proof} 

\subsubsection{Comparison of vector bundles $G\rightarrow \mathcal{B}_{\Sigma_0}$, $\mathcal{H}\rightarrow \mathcal{B}$, and $W^{Ram}\rightarrow Discs^{Ram}$}\label{comparisonofvectorbundlessubsection}
So far we have introduced covariantly constant flat (weak) symplectic vector bundles $(\mathcal{H}\rightarrow \mathcal{B},\Omega_{\mathcal{H}}, \nabla_{GM})$, $(W^{Ram}\rightarrow Discs^{Ram}, \Omega, \nabla)$, and a flat coisotropic vector bundle $(G\rightarrow \mathcal{B}_{\Sigma_0}, \nabla_{\mathcal{F}})$. We are now in a position to summarize how they are related.

\begin{proposition}\label{ghwproposition}
The morphisms of vector bundles $[.] : G\rightarrow \mathcal{H}$ given by taking the cohomology class and $i : G\rightarrow W^{Ram}$ given by Laurent series expansions are compatible with the connections in the sense that
\begin{enumerate}
    \item $i(\nabla_{\mathcal{F}}\xi) = \gamma^*\nabla i(\xi)$
    \item $[\nabla_{\mathcal{F}}\xi] = \nabla_{GM}[\xi]$
\end{enumerate}
for any sections $\xi \in \Gamma(\mathcal{B}_{\Sigma_0}, G)$. Moreover, they are compatible with parallel transport in the sense that
\begin{enumerate}
    \setcounter{enumi}{2}
    \item $i(s^*_{\Sigma',\Sigma}\xi_{\Sigma}) = \sum_{\alpha \in Ram}[\alpha]\otimes\exp\left((a(\Sigma') - a(\Sigma))\mathcal{L}_{\frac{1}{2z_\alpha}\partial_{z_\alpha}}\right)i_\alpha(\xi_{\Sigma})$
    \item $[s^*_{\Sigma',\Sigma}\xi_\Sigma] = \Gamma_{[\Sigma]}^{[\Sigma']}[\xi_{\Sigma}]$
\end{enumerate}
for any vectors in the fiber $\xi_\Sigma \in G_\Sigma$ given that $[\Sigma], [\Sigma'] \in \mathcal{B}_{\Sigma_0}$ and $\xi_\Sigma$ satisfies the conditions of Lemma \ref{sstaroperatorlemma}. 
\end{proposition}
\begin{proof}
\begin{enumerate}
\item
Let $\xi \in \Gamma(\mathcal{B}_{\Sigma_0},G)$ and let us examine $\nabla_{\mathcal{F}}\xi$ on $\Sigma\cap U_\alpha \subset \Sigma$. Working in the standard local coordinates we can write $\xi_\Sigma = \xi_\alpha(z_\alpha,u)dz_\alpha$. Using the fact that $x_\alpha = a_\alpha(\Sigma) + z^2_\alpha$, we have $\frac{\partial z_\alpha}{\partial u^k}|_{x_\alpha = const} = -\frac{1}{2z_\alpha}\frac{\partial a_\alpha(\Sigma)}{\partial u^k}$ and so
\begin{align*}
    i_\alpha (\nabla_{\mathcal{F}}\xi_\Sigma)(z_\alpha) &= i_\alpha\left(\sum_{k=1}^gdu^k\left(\left(\frac{\partial}{\partial u^k}\Big|_{z_\alpha - const}\xi_\alpha - \frac{1}{2z_\alpha}\frac{\partial a_\alpha(\Sigma)}{\partial u^k}\frac{\partial}{\partial z_\alpha}\xi_\alpha\right)dz_\alpha + \xi_\alpha d\left(\frac{1}{2z_\alpha}\frac{\partial a_\alpha(\Sigma)}{\partial u^k}\right)\right)\right)\\
    &= \sum_{k=1}^g du^k\left(\frac{\partial}{\partial u^k} -  \frac{\partial a_\alpha(\Sigma)}{\partial u^k}\mathcal{L}_{\frac{1}{2z_\alpha}\partial_{z_\alpha}}\right)i_\alpha(\xi)_{t_\alpha(\Sigma)} = \gamma^*\nabla i_\alpha(\xi)_{t_\alpha(\Sigma)}(z_\alpha)
\end{align*}
for every $\alpha \in Ram$. It follows that $i(\nabla_{\mathcal{F}}\xi) = \gamma^*\nabla i(\xi)$ as claimed.
\item
This immediately follows from Lemma \ref{integralandderivativelemma} with $D_\Sigma := \Sigma\setminus\cup_{p\in\sigma}\bar{\mathbb{D}}_{\alpha, \bar{\epsilon}_\alpha}(\Sigma)$ for some $\bar{\epsilon}_\alpha < \epsilon_\alpha, \alpha \in Ram$.
\item Let $\xi_\Sigma \in G_\Sigma$ and let us write $\xi_\Sigma$ on $\Sigma \cap U_\alpha \subset \Sigma$ in the standard local coordinates as $\xi_\Sigma = \xi_\alpha(z_\alpha)dz_\alpha$. Recall that the local coordinates of $S$ on $U_\alpha$ are $(x_\alpha, y_\alpha)$ and the foliation $\mathcal{F}$ is given by $x_\alpha = const$, so we have $s^*_{\Sigma',\Sigma}x_\alpha = x_\alpha$. For clarity, let us denote the standard local coordinate on $\Sigma'\cap U_\alpha$ by $z_{\alpha,\Sigma'}$ whereas $z_\alpha$ denotes the standard local coordinate on $\Sigma\cap U_\alpha$. Using the fact that $x_\alpha = a_\alpha(\Sigma) + z^2_\alpha$ on $\Sigma\cap U_\alpha$ and $x_\alpha = a_\alpha(\Sigma') + z^2_{\alpha, \Sigma'}$ on $\Sigma'\cap U_\alpha$, we have
\begin{equation*}
    \xi_\Sigma(z_\alpha) = \xi_\alpha(z_\alpha)dz_\alpha = \xi_\alpha\left(\sqrt{x_\alpha - a_\alpha(\Sigma)}\right)d\sqrt{x_\alpha - a_\alpha(\Sigma)}
\end{equation*}
and
\begin{align*}
    s^*_{\Sigma',\Sigma}&\xi_{\Sigma}(z_\alpha) = \xi_\alpha\left(\sqrt{x_\alpha - a_\alpha(\Sigma)}\right)d\sqrt{x_\alpha - a_\alpha(\Sigma)}\\
    &= \xi_\alpha\left(\sqrt{z^2_{\alpha, \Sigma'} + a_\alpha(\Sigma') - a_\alpha(\Sigma)}\right)d\sqrt{z^2_{\alpha, \Sigma'} + a_\alpha(\Sigma') - a_\alpha(\Sigma)} = \exp\left((a_\alpha(\Sigma') - a_\alpha(\Sigma))\mathcal{L}_{\frac{1}{2z_\alpha}\partial_{z_\alpha}}\right)(\xi_\Sigma)
\end{align*}
By construction, $s^*_{\Sigma',\Sigma}\xi_{\Sigma}$ is holomorphic and it can be represented by a Laurent series converging on $\mathbb{A}_{\alpha,\bar{\epsilon}_\alpha,\bar{M}_\alpha}(\Sigma')$ for some $\bar{\epsilon}_\alpha < \epsilon_\alpha < M_\alpha < \bar{M}_\alpha$. Repeating the above for each ramification point $\alpha \in Ram$ and perform a Laurent series expansion, the claimed formula follows.
\item Since $(s_{\Sigma',\Sigma})_* : H_1(\Sigma', \mathbb{C})\rightarrow H_1(\Sigma,\mathbb{C})$ is an isomorphism, mapping $A, B$-periods on $\Sigma'$ to $A, B$-periods on $\Sigma$, then $[s^*_{\Sigma', \Sigma}\xi_{\Sigma}] = \Gamma^{[\Sigma']}_{[\Sigma]}[\xi_\Sigma]$ follows straightforwardly from Lemma \ref{integralofparalleltransportlemma}.
\end{enumerate}
\end{proof}

\begin{figure}[h]
\begin{center}
\tikzset{every picture/.style={line width=0.75pt}} 

\begin{tikzpicture}[x=0.75pt,y=0.75pt,yscale=-1,xscale=1]
\draw  [color={rgb, 255:red, 0; green, 0; blue, 0 }  ,draw opacity=0 ][fill={rgb, 255:red, 0; green, 0; blue, 0 }  ,fill opacity=0.26 ] (60.13,217.25) .. controls (60.13,216.67) and (60.6,216.2) .. (61.18,216.2) -- (163.83,216.2) .. controls (164.4,216.2) and (164.87,216.67) .. (164.87,217.25) -- (164.87,220.38) .. controls (164.87,220.96) and (164.4,221.42) .. (163.83,221.42) -- (61.18,221.42) .. controls (60.6,221.42) and (60.13,220.96) .. (60.13,220.38) -- cycle ;
\draw    (187.85,218.81) -- (40.56,219.03) ;
\draw  [color={rgb, 255:red, 0; green, 0; blue, 0 }  ,draw opacity=0 ][fill={rgb, 255:red, 0; green, 0; blue, 0 }  ,fill opacity=0.11 ] (108.61,61.23) -- (117.52,61.23) -- (117.52,208.86) -- (108.61,208.86) -- cycle ;
\draw  [fill={rgb, 255:red, 0; green, 0; blue, 0 }  ,fill opacity=1 ] (110.94,218.81) .. controls (110.94,217.6) and (111.65,216.62) .. (112.53,216.62) .. controls (113.41,216.62) and (114.13,217.6) .. (114.13,218.81) .. controls (114.13,220.02) and (113.41,221.01) .. (112.53,221.01) .. controls (111.65,221.01) and (110.94,220.02) .. (110.94,218.81) -- cycle ;
\draw  [fill={rgb, 255:red, 0; green, 0; blue, 0 }  ,fill opacity=1 ] (150.39,218.81) .. controls (150.39,217.6) and (151.1,216.62) .. (151.98,216.62) .. controls (152.86,216.62) and (153.57,217.6) .. (153.57,218.81) .. controls (153.57,220.02) and (152.86,221.01) .. (151.98,221.01) .. controls (151.1,221.01) and (150.39,220.02) .. (150.39,218.81) -- cycle ;
\draw    (108.1,142.93) .. controls (110.27,132.05) and (115.22,144.2) .. (117.75,131.2) ;
\draw  [fill={rgb, 255:red, 0; green, 0; blue, 0 }  ,fill opacity=1 ] (113.11,137.49) .. controls (113.11,136.27) and (113.82,135.29) .. (114.7,135.29) .. controls (115.58,135.29) and (116.3,136.27) .. (116.3,137.49) .. controls (116.3,138.7) and (115.58,139.68) .. (114.7,139.68) .. controls (113.82,139.68) and (113.11,138.7) .. (113.11,137.49) -- cycle ;
\draw    (151.98,218.81) .. controls (148.9,182.98) and (122.3,172.23) .. (116.57,142.25) ;
\draw [shift={(116.25,140.4)}, rotate = 441.03] [color={rgb, 255:red, 0; green, 0; blue, 0 }  ][line width=0.75]    (10.93,-4.9) .. controls (6.95,-2.3) and (3.31,-0.67) .. (0,0) .. controls (3.31,0.67) and (6.95,2.3) .. (10.93,4.9)   ;
\draw    (28.3,40) -- (27.31,204) ;
\draw [shift={(27.3,206)}, rotate = 270.35] [color={rgb, 255:red, 0; green, 0; blue, 0 }  ][line width=0.75]    (10.93,-3.29) .. controls (6.95,-1.4) and (3.31,-0.3) .. (0,0) .. controls (3.31,0.3) and (6.95,1.4) .. (10.93,3.29)   ;
\draw  [color={rgb, 255:red, 0; green, 0; blue, 0 }  ,draw opacity=0 ][fill={rgb, 255:red, 0; green, 0; blue, 0 }  ,fill opacity=0.26 ] (488.91,217.25) .. controls (488.91,216.67) and (489.38,216.2) .. (489.96,216.2) -- (605.46,216.2) .. controls (606.03,216.2) and (606.5,216.67) .. (606.5,217.25) -- (606.5,220.38) .. controls (606.5,220.96) and (606.03,221.42) .. (605.46,221.42) -- (489.96,221.42) .. controls (489.38,221.42) and (488.91,220.96) .. (488.91,220.38) -- cycle ;
\draw    (632.3,218.81) -- (463.18,218.81) ;
\draw  [color={rgb, 255:red, 0; green, 0; blue, 0 }  ,draw opacity=0 ][fill={rgb, 255:red, 0; green, 0; blue, 0 }  ,fill opacity=0.11 ] (543.34,61.23) -- (553.34,61.23) -- (553.34,208.86) -- (543.34,208.86) -- cycle ;
\draw  [fill={rgb, 255:red, 0; green, 0; blue, 0 }  ,fill opacity=1 ] (545.95,218.81) .. controls (545.95,217.6) and (546.75,216.62) .. (547.74,216.62) .. controls (548.73,216.62) and (549.53,217.6) .. (549.53,218.81) .. controls (549.53,220.02) and (548.73,221.01) .. (547.74,221.01) .. controls (546.75,221.01) and (545.95,220.02) .. (545.95,218.81) -- cycle ;
\draw  [fill={rgb, 255:red, 0; green, 0; blue, 0 }  ,fill opacity=1 ] (590.24,218.81) .. controls (590.24,217.6) and (591.04,216.62) .. (592.03,216.62) .. controls (593.02,216.62) and (593.82,217.6) .. (593.82,218.81) .. controls (593.82,220.02) and (593.02,221.01) .. (592.03,221.01) .. controls (591.04,221.01) and (590.24,220.02) .. (590.24,218.81) -- cycle ;
\draw    (542.77,142.93) .. controls (545.2,132.05) and (550.76,144.2) .. (553.59,131.2) ;
\draw  [fill={rgb, 255:red, 0; green, 0; blue, 0 }  ,fill opacity=1 ] (548.38,137.49) .. controls (548.38,136.27) and (549.19,135.29) .. (550.17,135.29) .. controls (551.16,135.29) and (551.97,136.27) .. (551.97,137.49) .. controls (551.97,138.7) and (551.16,139.68) .. (550.17,139.68) .. controls (549.19,139.68) and (548.38,138.7) .. (548.38,137.49) -- cycle ;
\draw    (592.03,218.81) .. controls (588.57,182.98) and (558.7,172.23) .. (552.27,142.25) ;
\draw [shift={(551.91,140.4)}, rotate = 439.95] [color={rgb, 255:red, 0; green, 0; blue, 0 }  ][line width=0.75]    (10.93,-4.9) .. controls (6.95,-2.3) and (3.31,-0.67) .. (0,0) .. controls (3.31,0.67) and (6.95,2.3) .. (10.93,4.9)   ;
\draw    (211.3,232) .. controls (166.52,255.88) and (97.99,257.98) .. (37.22,229.43) ;
\draw [shift={(36.3,229)}, rotate = 385.43] [color={rgb, 255:red, 0; green, 0; blue, 0 }  ][line width=0.75]    (10.93,-3.29) .. controls (6.95,-1.4) and (3.31,-0.3) .. (0,0) .. controls (3.31,0.3) and (6.95,1.4) .. (10.93,3.29)   ;
\draw [shift={(211.3,232)}, rotate = 151.93] [color={rgb, 255:red, 0; green, 0; blue, 0 }  ][line width=0.75]      (0,0) .. controls (-3.09,0) and (-5.59,2.5) .. (-5.59,5.59) .. controls (-5.59,8.68) and (-3.09,11.18) .. (0,11.18) ;
\draw    (251.44,231.82) .. controls (303.51,260.56) and (360.79,255.16) .. (393.81,230.16) ;
\draw [shift={(395.3,229)}, rotate = 501.59] [color={rgb, 255:red, 0; green, 0; blue, 0 }  ][line width=0.75]    (10.93,-3.29) .. controls (6.95,-1.4) and (3.31,-0.3) .. (0,0) .. controls (3.31,0.3) and (6.95,1.4) .. (10.93,3.29)   ;
\draw  [color={rgb, 255:red, 0; green, 0; blue, 0 }  ,draw opacity=0 ][fill={rgb, 255:red, 0; green, 0; blue, 0 }  ,fill opacity=0.26 ] (255.96,217.25) .. controls (255.96,216.67) and (256.43,216.2) .. (257,216.2) -- (359.65,216.2) .. controls (360.23,216.2) and (360.7,216.67) .. (360.7,217.25) -- (360.7,220.38) .. controls (360.7,220.96) and (360.23,221.42) .. (359.65,221.42) -- (257,221.42) .. controls (256.43,221.42) and (255.96,220.96) .. (255.96,220.38) -- cycle ;
\draw    (360.44,218.85) -- (255.67,218.85) ;
\draw  [color={rgb, 255:red, 0; green, 0; blue, 0 }  ,draw opacity=0 ][fill={rgb, 255:red, 0; green, 0; blue, 0 }  ,fill opacity=0.11 ] (304.44,61.23) -- (313.34,61.23) -- (313.34,208.86) -- (304.44,208.86) -- cycle ;
\draw  [fill={rgb, 255:red, 0; green, 0; blue, 0 }  ,fill opacity=1 ] (306.76,218.81) .. controls (306.76,217.6) and (307.48,216.62) .. (308.36,216.62) .. controls (309.24,216.62) and (309.95,217.6) .. (309.95,218.81) .. controls (309.95,220.02) and (309.24,221.01) .. (308.36,221.01) .. controls (307.48,221.01) and (306.76,220.02) .. (306.76,218.81) -- cycle ;
\draw  [fill={rgb, 255:red, 0; green, 0; blue, 0 }  ,fill opacity=1 ] (346.21,218.81) .. controls (346.21,217.6) and (346.93,216.62) .. (347.81,216.62) .. controls (348.69,216.62) and (349.4,217.6) .. (349.4,218.81) .. controls (349.4,220.02) and (348.69,221.01) .. (347.81,221.01) .. controls (346.93,221.01) and (346.21,220.02) .. (346.21,218.81) -- cycle ;
\draw    (303.93,142.93) .. controls (306.1,132.05) and (311.05,144.2) .. (313.57,131.2) ;
\draw  [fill={rgb, 255:red, 0; green, 0; blue, 0 }  ,fill opacity=1 ] (308.93,137.49) .. controls (308.93,136.27) and (309.65,135.29) .. (310.53,135.29) .. controls (311.41,135.29) and (312.12,136.27) .. (312.12,137.49) .. controls (312.12,138.7) and (311.41,139.68) .. (310.53,139.68) .. controls (309.65,139.68) and (308.93,138.7) .. (308.93,137.49) -- cycle ;
\draw    (347.81,218.81) .. controls (344.73,182.98) and (318.12,172.23) .. (312.4,142.25) ;
\draw [shift={(312.07,140.4)}, rotate = 441.03] [color={rgb, 255:red, 0; green, 0; blue, 0 }  ][line width=0.75]    (10.93,-4.9) .. controls (6.95,-2.3) and (3.31,-0.67) .. (0,0) .. controls (3.31,0.67) and (6.95,2.3) .. (10.93,4.9)   ;
\draw    (228.37,43.52) -- (227.31,204) ;
\draw [shift={(227.3,206)}, rotate = 270.38] [color={rgb, 255:red, 0; green, 0; blue, 0 }  ][line width=0.75]    (10.93,-3.29) .. controls (6.95,-1.4) and (3.31,-0.3) .. (0,0) .. controls (3.31,0.3) and (6.95,1.4) .. (10.93,3.29)   ;
\draw    (419.98,43.52) -- (419.31,205) ;
\draw [shift={(419.3,207)}, rotate = 270.24] [color={rgb, 255:red, 0; green, 0; blue, 0 }  ][line width=0.75]    (10.93,-3.29) .. controls (6.95,-1.4) and (3.31,-0.3) .. (0,0) .. controls (3.31,0.3) and (6.95,1.4) .. (10.93,3.29)   ;
\draw    (246.53,26.88) -- (402.52,26.2) ;
\draw [shift={(404.52,26.19)}, rotate = 539.75] [color={rgb, 255:red, 0; green, 0; blue, 0 }  ][line width=0.75]    (10.93,-3.29) .. controls (6.95,-1.4) and (3.31,-0.3) .. (0,0) .. controls (3.31,0.3) and (6.95,1.4) .. (10.93,3.29)   ;
\draw    (209.89,26.69) -- (46.3,26.01) ;
\draw [shift={(44.3,26)}, rotate = 360.24] [color={rgb, 255:red, 0; green, 0; blue, 0 }  ][line width=0.75]    (10.93,-3.29) .. controls (6.95,-1.4) and (3.31,-0.3) .. (0,0) .. controls (3.31,0.3) and (6.95,1.4) .. (10.93,3.29)   ;

\draw (92.26,139.39) node   [align=left] {$\displaystyle \mathcal{L}_{\Sigma _{0}}$};
\draw (114.71,47.12) node   [align=left] {$\displaystyle \mathcal{H}_{\Sigma _{0}}$};
\draw (72.27,203.61) node   [align=left] {$\displaystyle \mathcal{B}_{\Sigma _{0}}$};
\draw (26.86,219.15) node   [align=left] {$\displaystyle \mathcal{B}$};
\draw (27.86,25.36) node   [align=left] {$\displaystyle \mathcal{H}$};
\draw (147.36,161.59) node   [align=left] {$\displaystyle \mathbf{\Phi }_{\Sigma _{0}}$};
\draw (115.86,230.35) node   [align=left] {$\displaystyle [ \Sigma _{0}]$};
\draw (156.3,230.31) node   [align=left] {$\displaystyle [ \Sigma ]$};
\draw (309.84,48.08) node   [align=left] {$\displaystyle G_{\Sigma _{0}}$};
\draw (521.77,140.36) node   [align=left] {$\displaystyle L^{Ram}_{Airy}$};
\draw (550.18,47.12) node   [align=left] {$\displaystyle W^{Ram}_{t_{0}}$};
\draw (428.36,218.63) node   [align=left] {$\displaystyle Discs^{Ram}$};
\draw (434.74,27.92) node   [align=left] {$\displaystyle W^{Ram}$};
\draw (583.6,162) node   [align=left] {$\displaystyle \Phi _{t_{0}}$};
\draw (550.92,230.35) node   [align=left] {$\displaystyle [ \Sigma _{0}]$};
\draw (596.32,230.31) node   [align=left] {$\displaystyle [ \Sigma ]$};
\draw (518.23,205.64) node   [align=left] {$\displaystyle Discs^{Ram}_{t_{0}}$};
\draw (121.74,14.4) node    {$[ .]$};
\draw (323.69,15.02) node    {$i$};
\draw (326.34,258.91) node    {$\gamma $};
\draw (233.91,219.1) node   [align=left] {$\displaystyle \mathcal{B}_{\Sigma _{0}}$};
\draw (342.07,161.59) node   [align=left] {$\displaystyle \Phi 
_{\Sigma _{0}}$};
\draw (311.69,230.35) node   [align=left] {$\displaystyle [ \Sigma _{0}]$};
\draw (352.12,230.31) node   [align=left] {$\displaystyle [ \Sigma ]$};
\draw (228.37,26.4) node   [align=left] {$\displaystyle G$};

\end{tikzpicture}
\end{center}
\begin{caption}
\text{Comparison of vector bundles $G\rightarrow \mathcal{B}_{\Sigma_0}$, $\mathcal{H}\rightarrow \mathcal{B}$, and $W^{Ram}\rightarrow Discs^{Ram}$. The vector bundle morphism $[.] : G\rightarrow \mathcal{H}$ covering the inclusion map $\mathcal{B}_{\Sigma_0}\hookrightarrow \mathcal{B}$ is given fiber-wise by the natural projection $G_{\Sigma}\rightarrow G_{\Sigma}/G^\perp_{\Sigma} \cong \mathcal{H}_{\Sigma}$, which is the same as taking a cohomology class. The vector bundle morphism $i : G\rightarrow W^{Ram}$ covering the map $\gamma : \mathcal{B}_{\Sigma_0}\hookrightarrow Discs^{Ram}$ is given by taking Laurent expansions around ramification points. In this sense, the vector bundle $G$ bridges the local point of view given by $W^{Ram}$ with the global point of view given by $\mathcal{H}$. This Figure also illustrates the parallel construction of the map $\pmb{\Phi}_{\Sigma_0} : \mathcal{B}_{\Sigma_0} \rightarrow \mathcal{H}_{\Sigma_0}$, $\Phi_{\Sigma_0} : \mathcal{B}_{\Sigma_0} \rightarrow G_{\Sigma_0}$, and $\Phi_{t_0} : Discs^{Ram}_{t_0} \rightarrow W^{Ram}_{t_0}$. Each map embeds some neighbourhood $\mathcal{B}_{\Sigma_0} \ni [\Sigma_0]$ or $Discs^{Ram}_{t_0} \ni t_0$ into the fiber $\mathcal{H}_{\Sigma_0}$, $G_{\Sigma_0}$ or $W^{Ram}_{t_0}$ over the reference point by sending the zero vector at an arbitrary point to the reference point via parallel transport using the affine connection $\nabla_{GM} + \pmb{\phi}$, $\nabla_{\mathcal{F}} + \phi$ or $\nabla + \phi$.}
\end{caption}
\label{hgwvectorbundlesrelation}
\end{figure}

\subsubsection{The embedding of $\mathcal{B}_{\Sigma_0}$ into $G_{\Sigma_0}$ and the residue constraints}\label{relationtoresconstraintssubsection}
In this section, we will construct the embedding map $\Phi_{\Sigma_0} : \mathcal{B}_{\Sigma_0} \rightarrow G_{\Sigma_0}$ analogous to $\pmb{\Phi}_{\Sigma_0} : \mathcal{B}_{\Sigma_0}\rightarrow \mathcal{H}_{\Sigma_0}$ defined in Section \ref{embeddingofbinhsubsection} and $\Phi_{t_0} : Discs_{t_0}^{\epsilon, M} \rightarrow W^{\epsilon, M}_{t_0}$ defined in Section \ref{embeddingofdiscssubsection}. The main outcome of this section is Proposition \ref{localvsglobalproposition} comparing $\Phi_{\Sigma_0}$ to $\Phi_{t_0}$ and $\pmb{\Phi}_{\Sigma_0}$, which is the key result to be used in proving Theorem \ref{prepotentialvsomega0ntheorem} later. To construct $\Phi_{\Sigma_0}$, we are going to need $\phi \in \Gamma(\mathcal{B}_{\Sigma_0},T^*\mathcal{B}\otimes G)$ and find $\theta \in \Gamma(\mathcal{B}_{\Sigma_0}, G)$ such that $\nabla_{\mathcal{F}}\theta = \phi$. Therefore, most of this section leading up to Proposition \ref{localvsglobalproposition} will be allocated to discuss the construction and properties of $\phi$ and $\theta$. We are also going to assume $\{\epsilon_{\alpha\in Ram}\}, \{M_{\alpha \in Ram}\}$, $\mathcal{U}_{Ram}$ and $\mathcal{B}_{\Sigma_0}$ are chosen such that Condition \ref{howtochoosebsigma0condition} is satisfied.

Let us begin by discussing $\phi \in \Gamma(\mathcal{B}_{\Sigma_0}, T^*\mathcal{B}\otimes G)$ as follows. At any point $[\Sigma] \in \mathcal{B}_{\Sigma_0}$, the map $\phi_\Sigma : T_{[\Sigma]}\mathcal{B}_{\Sigma_0} \xrightarrow{\cong} \Gamma(\Sigma, \Omega^1_\Sigma)\hookrightarrow G_{\Sigma}$ has already been defined previously in Section \ref{deformationofcurvessection}, but let us recall it here. Any given vector $v \in T_{[\Sigma]}\mathcal{B}_{\Sigma_0}$ can be mapped into a global holomorphic section of normal bundle $n_v \in \Gamma(\Sigma, \nu_\Sigma)$. Then $\phi_\Sigma(v) := \Omega_S(n_v,.)|_\Sigma \in \Gamma(\Sigma, \Omega^1_\Sigma)$. Let us cover $\Sigma$ with a collection of $(\mathcal{F},\Omega_S)$-charts $\mathcal{U} = \{(U_{p\in\sigma}, x_{p\in\sigma}, y_{p\in\sigma})\}$ then $\phi_{\Sigma}(v)$ is written locally on each $\Sigma\cap U_p$ as
\begin{equation}\label{phiBvalue}
    \phi_\Sigma(v)|_{\Sigma\cap U_p} = -\iota_v\nabla_{\mathcal{F}}(y_pdx_p)|_{\Sigma\cap U_p}.
\end{equation}
Now, let us compare this to $\pmb{\phi} \in \Gamma(\mathcal{B}_{\Sigma_0}, T^*\mathcal{B}\otimes\mathcal{H})$ from Section \ref{deformationofcurvesinfamilysubsection} and
\begin{equation*}
    \phi := \sum_{\alpha \in Ram}[\alpha]\otimes \phi_\alpha \in \Gamma(Discs^{Ram}_{t_0}, Hom(TDiscs^{Ram}_{t_0},W^{Ram})),
\end{equation*}
where $\phi_{\alpha} \in \Gamma(Discs^{\epsilon_\alpha, M_\alpha}_{t_{0\alpha}}, Hom(TDiscs^{\epsilon_\alpha, M_\alpha}_{t_{0\alpha}},W^{\epsilon_\alpha,M_\alpha}))$ are as given in Section \ref{theembeddingofdiscssection}. For example, let $t = (t_{\alpha \in Ram}) \in Discs^{Ram}_{t_0}, t_\alpha \in Discs^{\epsilon_\alpha, M_\alpha}_{t_{0\alpha}}$ and $v = (v_{\alpha \in Ram}) \in T_{t}Discs^{Ram}_{t_0}, v_\alpha \in T_{t_\alpha}Discs^{\epsilon_\alpha, M_\alpha}_{t_{0\alpha}}$, we define $\phi_t(v) := \sum_{\alpha \in Ram}[\alpha]\otimes \phi_{t_\alpha}(v_\alpha)$. We have the following lemma compare $\phi_\Sigma$ to $\phi_{t(\Sigma)}$ and $\pmb{\phi}_{\Sigma}$:

\begin{lemma}\label{phiongtophionhandwlemma}
Let $\phi \in \Gamma(\mathcal{B}_{\Sigma_0}, T^*\mathcal{B}\otimes G)$ be given as above, then
\begin{enumerate}
    \item $[\phi_\Sigma] = \pmb{\phi}_{\Sigma}$
    \item $i(\phi_\Sigma) = \phi_{t(\Sigma)}\circ d\gamma$.
\end{enumerate}
We summarize these relationships by the following commutative diagram
\begin{equation*}
    \begin{tikzcd}
    G_{\Sigma}/G^\perp_{\Sigma} & & & & W^{Ram}_{Airy}\\
    \mathcal{H}_{\Sigma}\arrow[rightarrow]{u}{\cong}\arrow[leftarrow]{rr}{[.]} & & G_{\Sigma}\arrow[rightarrow]{rr}{i} & & W^{Ram}_{t(\Sigma)}\arrow[rightarrow]{u}{\cong}\\
    \Gamma(\Sigma,\Omega^1_\Sigma)\arrow[hookrightarrow]{u} & & \Gamma(\Sigma,\Omega^1_\Sigma)\arrow[hookrightarrow]{u} & & T_0L^{Ram}_{Airy}\arrow[hookrightarrow]{u}\\
    T_{[\Sigma]}\mathcal{B}_{\Sigma_0}\arrow[rightarrow]{u}{\cong}\arrow[leftarrow]{rr}{\cong}\arrow[bend left=60, hookrightarrow]{uu}{\pmb{\phi}_{\Sigma}} & & T_{[\Sigma]}\mathcal{B}_{\Sigma_0}\arrow[rightarrow]{u}{\cong}\arrow[rightarrow]{rr}{d\gamma}\arrow[bend left=60, hookrightarrow]{uu}{\phi_{\Sigma}} & & T_{t(\Sigma)}Discs^{Ram}_{t_0}\arrow[rightarrow]{u}\arrow[bend left=60, hookrightarrow]{uu}{\phi_{t(\Sigma)}}
    \end{tikzcd}
\end{equation*}
\end{lemma}
\begin{proof}
\begin{enumerate}
    \item This is obvious since $\pmb{\phi}_{\Sigma} = \sum_{k=1}^g[\Omega_S(n_{v_k},.)|_{\Sigma}]du^k$ by definition.
    \item Given $t_\alpha \in Discs^{\epsilon_\alpha, M_\alpha}_{t_0\alpha}, t_\alpha = (x_\alpha = z^2_\alpha + a_\alpha, y_\alpha = \sum_{k=0}^\infty b_{k\alpha}z_\alpha^k)$ and vector $v_\alpha = A_\alpha\frac{\partial}{\partial a_\alpha} + \sum_{k=0}^\infty B_{k\alpha}\frac{\partial}{\partial b_{k\alpha}} \in T_{t_\alpha}Discs^{\epsilon_\alpha, M_\alpha}_{t_{0\alpha}}$ we can also compute $\phi_{\alpha, t_\alpha}$ defined in Section \ref{theembeddingofdiscssection} by keeping $x_\alpha = const$ when taking derivatives and use $\frac{\partial z_\alpha}{\partial a_\alpha}|_{x_\alpha = const} = -\frac{1}{2z_\alpha}$. We have
    \begin{align*}
    \phi_{\alpha,t_\alpha}(v_\alpha) &= \Omega_S\left(\left(A_\alpha\frac{\partial y_\alpha}{\partial a_\alpha}\Big|_{x_\alpha = const} + \sum_{k=0}^\infty B_{k\alpha}\frac{\partial y_\alpha}{\partial b_{k\alpha}}\Big|_{x_\alpha = const}\right)\frac{\partial}{\partial y_\alpha},.\right)\Big|_{\mathbb{D}_{t,M_\alpha}}\\
    &= -\left(\sum_{k = 0}^\infty B_{k\alpha}z^k_\alpha - \frac{A_\alpha}{2z_\alpha}\sum_{k=1}^\infty kb_{k\alpha}z^{k-1}_\alpha\right)dx_\alpha\Big|_{\mathbb{D}_{t,M_\alpha}} = \left(A_\alpha\sum_{k=1}^\infty kb_kz^{k-1}_\alpha - 2\sum_{k = 0}^\infty B_{k\alpha}z^{k+1}_\alpha\right)dz_\alpha
    \end{align*}
    which should be compared to (\ref{phioftangentvectdiscs}). Setting $t_\alpha = t_\alpha(\Sigma) = \gamma_\alpha(\Sigma)$ and
    \begin{equation*}
    v_\alpha = d\gamma_\alpha(v) = \sum_{k=1}^gv^k\left(\frac{\partial a_\alpha(\Sigma)}{\partial u^k}\frac{\partial}{\partial a_\alpha} + \sum_{l=0}^\infty \frac{\partial b_{l\alpha}(\Sigma)}{\partial u^k}\frac{\partial}{\partial b_{l\alpha}}\right)    
    \end{equation*}
    shows that $i_\alpha \phi_\Sigma = \gamma^*_\alpha\phi_\alpha = \phi_\alpha \circ d\gamma_\alpha$.
\end{enumerate}
\end{proof}

Since $\mathcal{B}_{\Sigma_0}$ is chosen to be contractible (C\ref{howtochoosebsigma0condition}.1), $G\rightarrow \mathcal{B}_{\Sigma_0}$ is a flat vector bundle and $d_{\nabla_{\mathcal{F}}}\phi = 0$, therefore Poincar\'{e} Lemma suggests the existence of $\theta \in \Gamma(\mathcal{B}_{\Sigma_0}, G)$ such that $\nabla_{\mathcal{F}} \theta = \phi$. $\theta_\Sigma$ is a holomorphic differential form on $\Sigma\setminus \cup_{\alpha \in Ram}\bar{\mathbb{D}}_{\alpha, \bar{\epsilon}_\alpha}(\Sigma)$ for some $\bar{\epsilon}_\alpha < \epsilon_\alpha$, $\oint_{\partial\bar{\mathbb{D}}_{\alpha, \epsilon_\alpha}}\theta_\Sigma = 0$ and any infinitesimal variation of $\theta_\Sigma$ in the moduli space directions produce a global holomorphic form on $\Sigma$, i.e. a holomorphic form on $\Sigma\setminus \cup_{\alpha \in Ram}\bar{\mathbb{D}}_{\alpha, \bar{\epsilon}_\alpha}(\Sigma)$ which can be analytically extended to the entire curve $\Sigma$. 

\begin{lemma}\label{thetaingtothetaonhandwlemma}
Let $\theta \in \Gamma(\mathcal{B}_{\Sigma_0}, G)$ be a section satisfying $\nabla_{\mathcal{F}}\theta = \phi$,  then 
\begin{enumerate}
    \item $[\theta_{\Sigma}] = \pmb{\theta}_{\Sigma} + \Gamma_{[\Sigma_0]}^{[\Sigma]}[\xi_0] \in \mathcal{H}_{\Sigma}$ for some $[\xi_0] \in \mathcal{H}_{\Sigma_0}$.
    \item $i(\theta_{\Sigma}) = \sum_{\alpha \in Ram}[\alpha]\otimes\left(\theta_{t_\alpha(\Sigma)} + \exp\left((a_\alpha(\Sigma) - a_\alpha(\Sigma_0))\mathcal{L}_{\frac{1}{2z_\alpha}\partial_{z_\alpha}}\right)(\xi_{0\alpha})\right) \in W^{Ram}_{t(\Sigma)}$ for some $\xi_{0\alpha} \in W^{\bar{\epsilon}_\alpha, \bar{M}_\alpha}_{t_{0\alpha}} \subset W^{\epsilon_\alpha, M_\alpha}_{t_{0\alpha}}$ where 
    \begin{equation*}
        \bar{\epsilon}_\alpha := \sqrt{\epsilon^2_\alpha - \sup_{[\Sigma]\in \mathcal{B}_{\Sigma_0}}|a_\alpha(\Sigma) - a_\alpha(\Sigma_0)|}, \qquad \bar{M}_\alpha := \sqrt{M^2_\alpha + \sup_{[\Sigma]\in \mathcal{B}_{\Sigma_0}}|a_\alpha(\Sigma) - a_\alpha(\Sigma_0)|}.
    \end{equation*}
\end{enumerate}
where $\pmb{\theta}_\Sigma \in \mathcal{H}_\Sigma$ is given by (\ref{definingthetainh}) and $\theta_{t_\alpha(\Sigma)} \in W^{\epsilon_\alpha, M_\alpha}_{t_\alpha(\Sigma)}$ is given by (\ref{definingthetainw}).
\end{lemma}
\begin{proof}
\begin{enumerate}
    \item By taking cohomology class, apply Proposition \ref{ghwproposition} and Lemma \ref{phiongtophionhandwlemma}, we find that $\nabla_{GM}[\theta_{\Sigma}] = [\nabla_{\mathcal{F}}\theta_\Sigma] = [\phi_\Sigma] = \pmb{\phi}_{\Sigma}$. Therefore we have $[\theta] = \pmb{\theta} + \pmb{c}$ for some parallel section $\pmb{c} \in \Gamma(\mathcal{B}_{\Sigma_0}, \mathcal{H})$, $\nabla_{GM}\pmb{c} = 0$. But then we we can write $\pmb{c}_\Sigma = \Gamma_{[\Sigma_0]}^{[\Sigma]}[\xi_0]$ where $[\xi_0] := \pmb{c}_{\Sigma_0} \in \mathcal{H}_{\Sigma_0}$.
    \item By applying the map $i : G_\Sigma \rightarrow W^{Ram}_{t(\Sigma)}$ follows by Proposition \ref{ghwproposition} and Lemma \ref{phiongtophionhandwlemma} we have $\gamma^*\nabla i(\theta_\Sigma) = i(\nabla_{\mathcal{F}}\theta_\Sigma) = i(\phi_\Sigma) = \phi_{t(\Sigma)}\circ d\gamma$. Therefore, we have $i(\theta_\Sigma) = \sum_{\alpha \in Ram}[\alpha]\otimes\theta_{t_\alpha(\Sigma)} + c$ for some parallel section $c \in \Gamma(\mathcal{B}_{\Sigma_0}, \gamma^*W^{Ram})$, $\gamma^*\nabla c = 0$. It is not hard to check that
    \begin{equation*}
        c_{t(\Sigma)} = \sum_{\alpha \in Ram}[\alpha]\otimes\exp\left((a_\alpha(\Sigma) - a_\alpha(\Sigma_0))\mathcal{L}_{\frac{1}{2z_\alpha}\partial_{z_\alpha}}\right)(\xi_{0\alpha})
    \end{equation*}
    where $\sum_{\alpha\in Ram}[\alpha]\otimes\xi_{0\alpha} = c_{t_0} \in W^{Ram}_{t_0}$.
\end{enumerate}
\end{proof}

Let us show explicitly how $\theta \in \Gamma(\mathcal{B}_{\Sigma_0}, G)$ such that $\nabla_{\mathcal{F}}\theta = \phi$ can be constructed. We choose a collection of $(\mathcal{F},\Omega_S)$-charts $\mathcal{U}_0 = \{(U_{0,p\in \sigma_0}, x_{p\in\sigma_0}, y_{p\in \sigma_0})\}$ which allows the parallel transport of any holomorphic form $\xi \in \Gamma(\Sigma_0,\Omega^1_{\Sigma_0}) \subset G_{\Sigma_0}$ to $G_{\Sigma}$ for any $[\Sigma]\in\mathcal{B}_{\Sigma_0}$ according to Lemma \ref{paralleltransportinglemma}. Such a collection $\mathcal{U}_0$ exists due to C\ref{howtochoosebsigma0condition}.3. From (\ref{phiBvalue}) we know that $\theta_\Sigma|_{\Sigma\cap U_{0,p}}$ must take the form $-(y_p + f_p(x_p))dx_p|_{\Sigma\cap U_{0,p}}$ for us to have $\nabla_{\mathcal{F}}\theta_\Sigma|_{\Sigma\cap U_{0,p}} = \phi_\Sigma|_{\Sigma\cap U_{0,p}}$, where $f_p(x_p)dx_p$ is an integration constant. Write $y_p(x_p;\Sigma)dx_p := y_pdx_p|_{\Sigma\cap U_{0,p}}$ for each $p \in \sigma$. In other words, $\Sigma \cap U_{0,p} = \{(x_p, y_p(x_p;\Sigma)) \in U_p \subset S\ |\ x_p \in x_p(\Sigma\cap U_{0,p})\}$ is the graph of $y_p = y_p(x_p;\Sigma)$. Then, on each open subset $\Sigma\cap U_{0,p}$, let us define
\begin{equation*}
    \theta_\Sigma|_{\Sigma\cap U_{0,p}} := -y_p(x_p;\Sigma)dx_p + y_p(x_p;\Sigma_0)dx_p.
\end{equation*}
If each $\theta_\Sigma|_{\Sigma\cap U_{0,p}}$ can be patched together to get $\theta_\Sigma \in G_\Sigma$, then it will automatically follows that $\nabla_{\mathcal{F}}\theta_\Sigma = \phi_\Sigma$ as this relation holds on every $\Sigma\cap U_{0,p}$ and we are done. The following guarantees that this is the case:
\begin{lemma}\label{patchingthetalemma}
Each $\theta_\Sigma|_{\Sigma\cap U_{0,p}}, p \in \sigma_0$ can be patched together as $\theta_\Sigma \in G_\Sigma$.
\end{lemma}
\begin{proof}
First, let us show that each $\theta_{\Sigma}|_{\Sigma\cap U_{0,p}}, p \in \sigma_0$ can be patched together to form a global holomorphic form $\theta_\Sigma$ on $\Sigma\setminus \cup_{\alpha \in Ram}\bar{\mathbb{D}}_{\alpha, \bar{\epsilon}_\alpha}(\Sigma)$ for some $\bar{\epsilon}_\alpha < \epsilon_\alpha$. Using the coordinates transformation on $U_{p,0}\cap U_{0,p'}$ between $(x_p, y_p)$ and $(x_{p'}, y_{p'})$ as given in (\ref{coordtransformationpreservingomegaandfoliation}), we find that $x_{p'} = F_{p'}(x_p)$ and
\begin{equation*}
    y_{p'}(x_{p'};\Sigma) = \frac{y_{p}(F^{-1}_{p'}(x_{p'});\Sigma)}{F'_{p'}(F^{-1}(x_{p'}))} - G'_{p'}(F^{-1}(x_{p'})), \qquad y_{p'}(x_{p'};\Sigma_0) = \frac{y_{p}(F^{-1}_{p'}(x_{p'});\Sigma_0)}{F'_{p'}(F^{-1}(x_{p'}))} - G'_{p'}(F^{-1}(x_{p'})).
\end{equation*}
This means
\begin{align*}
    \left(\theta_{\Sigma}|_{\Sigma\cap U_{0,p'}}\right)&|_{\Sigma\cap U_{0,p}\cap U_{0,p'}} = -y_{p'}(x_{p'};\Sigma)dx_{p'} + y_{p'}(x_{p'};\Sigma_0)dx_{p'}\\
    &= -\left(\frac{y_{p}(F^{-1}_{p'}(x_{p'});\Sigma)}{F'_{p'}(F^{-1}(x_{p'}))} - G'_{p'}(F^{-1}(x_{p'}))\right)dx_{p'} + \left(\frac{y_{p}(F^{-1}_{p'}(x_{p'});\Sigma_0)}{F'_{p'}(F^{-1}(x_{p'}))} - G'_{p'}(F^{-1}(x_{p'}))\right)dx_{p'}\\
    &= -\left(\frac{y_{p}(x_p;\Sigma)}{F'_{p'}(x_p)} - G'_{p'}(x_p)\right)dF(x_p) + \left(\frac{y_{p}(x_p;\Sigma_0)}{F'_{p'}(x_p)} - G'_{p'}(x_p)\right)dF(x_p)\\
    &= -y_p(x_p;\Sigma)dx_p + y_p(x_p;\Sigma_0)dx_p\\
    &= \left(\theta_{\Sigma}|_{\Sigma\cap U_{0,p}}\right)|_{\Sigma\cap U_{0,p}\cap U_{0,p'}}.
\end{align*}
Therefore, we can patch together each $\theta_{\Sigma}|_{\Sigma\cap U_{0,p}}$ for $p \in \sigma_0$ to get a well-defined holomorphic form $\theta_\Sigma$ on $\cup_{p \in \sigma_0}\Sigma \cap U_{0,p} \supset \Sigma\setminus \cup_{\alpha \in Ram}\mathbb{D}_{\alpha, \epsilon_\alpha}(\Sigma)$.

To finish showing that $\theta_\Sigma \in G_\Sigma$ we need to check that $\oint_{\partial\bar{\mathbb{D}}_{\alpha,\epsilon_\alpha}}\theta_\Sigma = 0$. Since $\theta_\Sigma$ we have so far is well-defined on the annulus $\mathbb{A}_{\alpha, \epsilon_\alpha, M_\alpha}(\Sigma)$, it can be written in the standard local coordinate $z_\alpha$ of $\Sigma \cap U_\alpha$ as 
\begin{align*}
    i_\alpha\theta_\Sigma(z_\alpha) &= -y_\alpha(z_\alpha;\Sigma)dx_\alpha(z_\alpha) + y_\alpha(z_{\alpha,\Sigma_0};\Sigma_0)dx_\alpha(z_{\alpha,\Sigma_0})\\
    &= \theta_{t_\alpha(\Sigma)}(z_\alpha) - \exp\left((a_\alpha(\Sigma) - a_\alpha(\Sigma_0))\mathcal{L}_{\frac{1}{2z_\alpha}\partial_{z_\alpha}}\right)\theta_{t_\alpha(\Sigma_0)}(z_\alpha).
\end{align*}
For clarity, we have denoted by $z_{\alpha, \Sigma_0}$ the standard local coordinate on $\Sigma_0\cap U_\alpha$ which is related to $z_\alpha$ by $a_\alpha(\Sigma) + z_\alpha^2 = x_\alpha = a_\alpha(\Sigma_0) + z_{\alpha, \Sigma_0}^2$. Let us give more explanation to the second equality. By (\ref{definingthetainw}) we have $\theta_{t_\alpha(\Sigma)} = -y_\alpha(z_\alpha, \Sigma)dx_\alpha(z_\alpha) \in W^{\epsilon_\alpha, M_\alpha}_{t_\alpha(\Sigma)}$ and $\theta_{t_\alpha(\Sigma_0)} = -y_\alpha(z_{\alpha, \Sigma_0}; \Sigma_0)dx_\alpha(z_{\alpha, \Sigma_0}) \in W^{\epsilon_\alpha, M_\alpha}_{t_\alpha(\Sigma_0)}$. In particular,
\begin{equation*}
    y_\alpha(z_{\alpha,\Sigma_0};\Sigma_0)dx_\alpha(z_{\alpha,\Sigma_0}) = \exp\left((a_\alpha(\Sigma) - a_\alpha(\Sigma_0))\mathcal{L}_{\frac{1}{2z_{\alpha,\Sigma_0}}\partial_{z_{\alpha, \Sigma_0}}}\right)\left(y_\alpha(z_{\alpha, \Sigma_0}; \Sigma_0)dx_\alpha(z_{\alpha, \Sigma_0})\right)\Big|_{z_{\alpha, \Sigma_0}\mapsto z_\alpha}.
\end{equation*}
Clearly, $i_\alpha\theta_\Sigma$ contains no $\frac{dz_\alpha}{z_\alpha}$ term according to Lemma \ref{expoperatorlemma} and so $\oint_{\partial\bar{\mathbb{D}}_{\alpha,\epsilon_\alpha}}\theta_\Sigma = 0$. 
\end{proof}

Let us embed the neighbourhood $\mathcal{B}_{\Sigma_0}$ of $[\Sigma_0] \in \mathcal{B}$ into the fiber $G_{\Sigma_0}$, using the same procedure we did to embed $\mathcal{B}_{\Sigma_0}$ into $\mathcal{H}_{\Sigma_0}$ in Section \ref{deformationofcurvessection} and $Discs^{\epsilon,M}_{t_0}$ into $W^{\epsilon, M}_{t_0}$ in Section \ref{theembeddingofdiscssection} but this time with the vector bundle $(G\rightarrow \mathcal{B}_{\Sigma_0}, \nabla_{\mathcal{F}})$. The embedding map $\Phi_{\Sigma_0} : \mathcal{B}_{\Sigma_0} \rightarrow G_{\Sigma_0}$ is given by sending the zero vector $0_{\Sigma} \in G_{\Sigma}$ at $[\Sigma] \in \mathcal{B}_{\Sigma_0}$ into $G_{\Sigma_0}$ at $[\Sigma_0]\in \mathcal{B}_{\Sigma_0}$ via parallel transport using the connection $\nabla_{\mathcal{F}} + \phi$. Since $\nabla_{\mathcal{F}}\theta = \phi$, the parallel transport of $0_\Sigma$ from $[\Sigma] \in \mathcal{B}_{\Sigma'}$ to $[\Sigma'] \in \mathcal{B}$ using $\nabla_{\mathcal{F}} + \phi$ is path-independent and it is uniquely given by 
\begin{equation*}
    v_{\Sigma}(\Sigma') = -\theta_{\Sigma'} + s^*_{\Sigma',\Sigma}\theta_{\Sigma}.
\end{equation*}
Therefore,
\begin{equation*}
    \Phi_{\Sigma_0}(\Sigma) := v_{\Sigma}(\Sigma_0) = -\theta_{\Sigma_0} + s^*_{\Sigma_0, \Sigma}\theta_\Sigma \in G_{\Sigma_0},
\end{equation*}
where the existence of parallel transport $s^*_{\Sigma_0, \Sigma}\theta_{\Sigma} \in G_{\Sigma_0}$ is due to C\ref{howtochoosebsigma0condition}.4 as $\theta_\Sigma$ is holomorphic on $\cup_{p \in \sigma_0}\Sigma\cap U_{0,p} \supset \Sigma \setminus \cup_{\alpha \in Ram}\mathbb{D}_{\alpha,\epsilon_\alpha}(\Sigma)$. 

Given $t_0 = (t_{0\alpha}) \in Discs^{Ram}$, let us define the $Ram$ product version of $\Phi_{t_0} : Discs^{Ram}_{t_0}\rightarrow W^{Ram}_{t_0}$ by $\Phi_{t_0}(t) := \sum_{\alpha \in Ram}[\alpha]\otimes \Phi_{\alpha, t_{0\alpha}}(t_\alpha)$ for any $t = (t_\alpha) \in Discs^{Ram}_{t_0}$ where $\Phi_\alpha : Discs^{\epsilon_\alpha, M_\alpha}_{t_{0\alpha}}\rightarrow W^{\epsilon_\alpha, M_\alpha}_{t_{0\alpha}}$ is as given in Section \ref{theembeddingofdiscssection}:
\begin{equation*}
    \Phi_{\alpha, t_{0\alpha}}(t_\alpha) := -\theta_{t_{0\alpha}} + \exp((a_\alpha(t_{0\alpha}) - a_\alpha(t_\alpha))\mathcal{L}_{\frac{1}{2z_\alpha}\partial_{z_\alpha}})\theta_{t_\alpha} \in W^{\epsilon_\alpha, M_\alpha}_{t_{0\alpha}}.
\end{equation*}

Finally, we have the following proposition:
\begin{proposition}\theoremname{\cite[Proposition 7.1.2]{kontsevich2017airy}}\label{localvsglobalproposition}
The embedding $\Phi_{\Sigma_0} : \mathcal{B}_{\Sigma_0} \rightarrow G_{\Sigma_0}$ satisfies $[\Phi_{\Sigma_0}(\Sigma)] = \pmb{\Phi}_{\Sigma_0}(\Sigma)$ and $i(\Phi_{\Sigma_0}(\Sigma)) = \Phi_{t_0}\circ\gamma(\Sigma) = \Phi_{t_0}(t(\Sigma))$, so $(\mathbf{a})$ and $(\mathbf{b})$ squares in the diagram are commutative.
Moreover, if each $(\mathcal{F},\Omega_S)$-charts $(U_\alpha, x_\alpha, y_\alpha)$ in $\mathcal{U}_{Ram}$ are chosen such that $\gamma_\alpha(\Sigma_0) = (x_\alpha = z_\alpha^2, y_\alpha = z_\alpha)$, then $\Phi_{\Sigma_0}$ factors through $L^{Ram}_{Airy}$ and the triangle $(\mathbf{c})$ is also commutative.
\begin{equation*}
    \begin{tikzcd}
    & & G_{\Sigma_0}/G^\perp_{\Sigma_0} & & & & W^{Ram}_{Airy}\\
    & & \mathcal{H}_{\Sigma_0}\arrow[rightarrow]{u}{\cong}\arrow[leftarrow]{rr}{[.]} & & G_{\Sigma_0}\arrow[rightarrow]{rr}{i} & & W^{Ram}_{t_0}\arrow[rightarrow]{u}{\cong}\arrow[hookleftarrow]{drr} \\
    \mathcal{L}_{\Sigma_0}\arrow[hookrightarrow]{urr} & & & (\mathbf{a}) & & (\mathbf{b}) & & (\mathbf{c}) & L^{Ram}_{Airy}\\
    & & \mathcal{B}_{\Sigma_0}\arrow[hookrightarrow]{uu}{\pmb{\Phi}_{\Sigma_0}}\arrow[rightarrow]{ull}{\cong}\arrow[leftarrow]{rr}{\cong} & & \mathcal{B}_{\Sigma_0}\arrow[hookrightarrow]{uu}{\Phi_{\Sigma_0}}\arrow[rightarrow]{rr}{\gamma} & & Discs^{Ram}_{t_0}\arrow[hookrightarrow]{uu}{\Phi_{t_0}}\arrow[rightarrow]{urr}
    \end{tikzcd}
\end{equation*}
\end{proposition}
We note that the left-most triangle is merely a restatement of the definition $\mathcal{L}_{\Sigma_0} := \text{im}\pmb{\Phi}_{\Sigma_0} \subset \mathcal{H}_{\Sigma_0}$ as given in Section \ref{embeddingofbinhsubsection} and we include it here for completeness.

\begin{proof}
The proof is mostly an application of Proposition \ref{ghwproposition} and Lemma \ref{thetaingtothetaonhandwlemma}. From Lemma \ref{thetaingtothetaonhandwlemma} we know that $[\theta_{\Sigma}] = \pmb{\theta}_{\Sigma} + \Gamma^{[\Sigma]}_{[\Sigma_0]}[\xi_0]$ for some $[\xi_0] \in \mathcal{H}_{\Sigma_0}$ and 
\begin{equation*}
    i_\alpha(\theta_{\Sigma}) = \theta_{t_\alpha(\Sigma)} + \exp\left((a_\alpha(\Sigma) - a_\alpha(\Sigma_0))\mathcal{L}_{\frac{1}{2z_\alpha}\partial_{z_\alpha}}\right)(\xi_{0\alpha})
\end{equation*}
for some $\xi_{0\alpha} \in W^{\epsilon_\alpha, M_\alpha}_{t_{0\alpha}}$. Therefore, we have
\begin{align*}
    [\Phi_{\Sigma_0}(\Sigma)] &= -[\theta_{\Sigma_0}] + [s^*_{\Sigma_0, \Sigma}\theta_{\Sigma}] = -[\theta_{\Sigma_0}] + \Gamma^{[\Sigma_0]}_{[\Sigma]}[\theta_{\Sigma}]\\
    &= -(\pmb{\theta}_{\Sigma_0} + [\xi_0]) + \Gamma^{[\Sigma_0]}_{[\Sigma]}(\pmb{\theta}_{\Sigma} + \Gamma_{[\Sigma_0]}^{[\Sigma]}[\xi_0]) = -\pmb{\theta}_{\Sigma_0} - [\xi_0] + \Gamma_{[\Sigma]}^{[\Sigma_0]}\pmb{\theta}_{\Sigma} + [\xi_0] = -\pmb{\theta}_{\Sigma_0} + \Gamma_{[\Sigma]}^{[\Sigma_0]}\pmb{\theta}_{\Sigma} = \pmb{\Phi}_{\Sigma_0}(\Sigma)
\end{align*}
and so $(\mathbf{a})$ is commutative.
Similarly, 
\begin{align*}
    i(\Phi_{\Sigma_0}(\Sigma)) &= -i(\theta_{\Sigma_0}) + i(s^*_{\Sigma_0, \Sigma}\theta_{\Sigma})\\
    &= -i(\theta_{\Sigma_0}) + \sum_{\alpha \in Ram}[\alpha] \otimes \exp\left((a_\alpha(\Sigma_0) - a_\alpha(\Sigma))\mathcal{L}_{\frac{1}{2z_\alpha}\partial_{z_\alpha}}\right)i_\alpha(\theta_\Sigma)\\
    &= -\sum_{\alpha \in Ram}[\alpha]\otimes \left(\theta_{t_\alpha(\Sigma_0)} + \xi_{0\alpha}\right)\\
    &\qquad + \sum_{\alpha \in Ram}[\alpha]\otimes \exp\left((a_\alpha(\Sigma_0) - a_\alpha(\Sigma))\mathcal{L}_{\frac{1}{2z_\alpha}\partial_{z_\alpha}}\right)\left(\theta_{t_\alpha(\Sigma)} + \exp\left((a_\alpha(\Sigma) - a_\alpha(\Sigma_0))\mathcal{L}_{\frac{1}{2z_\alpha}\partial_{z_\alpha}}\right)(\xi_{0\alpha})\right)\\
    &= \sum_{\alpha \in Ram}[\alpha]\otimes \left(-\theta_{t_\alpha(\Sigma_0)} - \xi_{0\alpha} + \exp\left((a_\alpha(\Sigma_0) - a_\alpha(\Sigma))\mathcal{L}_{\frac{1}{2z_\alpha}\partial_{z_\alpha}}\right)\theta_{t_\alpha(\Sigma)} + \xi_{0\alpha}\right)\\
    &= \sum_{\alpha \in Ram}[\alpha]\otimes \left(-\theta_{t_\alpha(\Sigma_0)} + \exp\left((a_\alpha(\Sigma_0) - a_\alpha(\Sigma))\mathcal{L}_{\frac{1}{2z_\alpha}\partial_{z_\alpha}}\right)\theta_{t_\alpha(\Sigma)}\right)\\
    &= \sum_{\alpha \in Ram} [\alpha]\otimes \Phi_{\alpha, t_\alpha(\Sigma_0)}(t_\alpha(\Sigma)) = \Phi_{t_0}(t(\Sigma)) = \Phi_{t_0}\circ\gamma(\Sigma)
\end{align*}
and so $(\mathbf{b})$ is commutative. Lastly, if $\gamma_\alpha(\Sigma_0) = (x_\alpha = z^2_\alpha, y_\alpha = z_\alpha)$ then it follows from Proposition \ref{phiembeddiscsinwproposition} that $\text{im}\Phi_{\alpha, t_{0\alpha}} \subseteq L^{M_\alpha}_{Airy}$ and therefore $\text{im} \Phi_{t_0} \subseteq L^{Ram}_{Airy}$ which shows that $(\mathbf{c})$ is commutative.
\end{proof} 

\subsection{From Airy Structures to Topological Recursion}\label{relationstotrsection}
As we have mentioned earlier in Section \ref{residueconstraintairystructuresection}, the main reason for our interest in the residue constraints Airy structure is due to its connection to topological recursion. This connection is well-known and has been studied in various places such as \cite{borot2018higher, andersen2017abcd, kontsevich2017airy}, however we include this review here for completeness. 
\subsubsection{Review of Topological Recursion}
We recall the following basic facts and setup of the original Eynard-Orantin topological recursion \cite{eynard2007invariants,eynard2008algebraic}. A \emph{spectral curve} $(\Sigma_0, x, y, B)$ is given by a smooth compact Riemann surface $\Sigma_0$, global functions $x,y : \Sigma_0 \rightarrow \mathbb{P}^1$ and a choice of Bergman kernel $B = B(p,q)$ on $\Sigma_0\times \Sigma_0$. We assume that $dx(p)$ has simple zeros at a finite number of points $p \in \{r_{\alpha \in Ram} \in \Sigma_0\}$ called \emph{ramification points} and that $y \sim y(r_\alpha) + C\sqrt{x - x(r_\alpha)}$ near each $r_\alpha$ (it follows that $dy(r_\alpha) \neq 0$). We equip to each ramification point the \emph{involution}  map $\sigma_\alpha$ defined locally in some neighbourhood of each $r_\alpha$ satisfying $\sigma_\alpha(r_\alpha) = r_\alpha$, $x\circ \sigma_\alpha = x$ and $d\sigma_\alpha|_{r_\alpha} = -id : T_{r_\alpha}\Sigma_0 \rightarrow T_{r_\alpha}\Sigma_0$. We set the initial condition of the recursion to be
\begin{equation*}
    \omega_{0,1}(p) := ydx, \qquad \omega_{0,2}(p,q) = B(p,q).
\end{equation*}
For $g\geq 0, n\geq 3$ or $g\geq 1, n\geq 1$ we compute the multi-differential $\omega_{g,n} \in \Gamma\left((\Sigma_0\setminus \cup_{\alpha\in Ram} r_\alpha)^n, (\Omega_{\Sigma_0}^1)^{\boxtimes n}\right)$ using the following recursion formula
\begin{align*}\label{trrecursion}
    \omega_{g,n}(p_1,...,p_n) &= \sum_{\alpha\in Ram}Res_{p = r_\alpha}K(p_1,p)\omega_{g-1,n+1}(p,\sigma_\alpha(p),p_2,...,p_n)\\
    &\qquad + \sum_{\alpha \in Ram}\sum^*_{\substack{g_1 + g_2 = g\\I_1\coprod I_2 = \{2,...,n\}}}Res_{p=r_\alpha}K(p_1,p)\omega_{g_1,1 + |I_1|}(p,p_{I_1})\omega_{g_2,1+|I_2|}(\sigma_\alpha(p),p_{I_2})\numberthis
\end{align*}
where $\sum^*$ indicates that we exclude all terms involving $\omega_{0,1}$ from the summation. We also define the \emph{recursion kernel} $K = K(p_1,p)$ for $p_1 \in \Sigma$ and $p$ in the vicinity of a ramification point by
\begin{equation*}
    K(p_1,p) := -\frac{1}{2}\frac{\int_{p' = \sigma_\alpha(p)}^{p' = p}\omega_{0,2}(p_1,p')}{\omega_{0,1}(p) - \omega_{0,1}(\sigma_\alpha(p))}.
\end{equation*}
We note that $K(p_1,p)$ is a global meromorphic differential on $\Sigma_0$ in $p_1$ with simple poles at $p_1 = p, \sigma_\alpha(p)$. On the other hand, $K(p_1,p)$ is only defined as an inverse of differential locally for $p$ close to $r_\alpha$ with simple poles at $p = r_\alpha$ for any $\alpha \in Ram$. The formula (\ref{trrecursion}) expresses $\omega_{g,n}$ only in terms of $\omega_{g',n'}$ for $2g'-2+n' < 2g-2+n$ and therefore the recursion is guaranteed to terminate.

\begin{remark}\label{omegagnpropertiesremark}
It is known that $\omega_{g,n}(p_1,\cdots,p_n)$ is a meromorphic differential on $\Sigma_0$ symmetric in each of its variable $p_i$ with poles only at ramification points of $\Sigma_0$ \cite{eynard2007invariants, eynard2008algebraic} and zero residues $\oint_{p_i = r_\alpha}\omega_{g,n}(p_1,\cdots,p_n) = 0$ \cite[Corollary 4.1]{eynard2007invariants}. Moreover, if we have chosen $B = B(p,q)$ to be the normalized Bergman kernel, then $\oint_{p_i\in A_k}\omega_{g,n}(p_1,\cdots,p_n) = 0$ \cite[Theorem 4.3]{eynard2007invariants}. 
\end{remark}

We observe that we have only used local information of $\omega_{0,1}$ in the vicinity of each $r_\alpha$. This suggests a slight generalization of the original version of topological recursion. We consider a smooth compact Riemann surface $\Sigma_0$ equipped with a choice of Bergman kernel $B = B(p,q)$ and a set of distinct points $\{r_{\alpha \in Ram}\}$ indexed by the set $Ram$. We let $x$ and $y$ be locally defined functions on some open neighbourhood $D_\alpha$ of each $r_\alpha \in Ram$ such that $dx(r_\alpha) = 0$ and $\tau_\alpha := \sqrt{x - x(r_\alpha)}$ is a local coordinate on $D_\alpha$ of $r_\alpha$. Then in general, we have
\begin{equation*}
    x|_{D_\alpha} = \tau_\alpha^2 + a_\alpha, \qquad y|_{D_\alpha} = b_{0\alpha} + b_{1\alpha}\tau_\alpha + b_{2\alpha}\tau_\alpha^2 + \cdots
\end{equation*}
and the involution is given by $\sigma_\alpha(\tau_\alpha) = -\tau_\alpha$. As before, we set the initial condition of the recursion to be $\omega_{0,1}(p) = ydx$ and $\omega_{0,2}(p,q) = B(p,q)$, notice that $\omega_{0,1}$ is now only a locally defined differential form in the vicinity of each $r_\alpha$. The rest of the recursion is the same as above. This variance of topological recursion sometimes called \emph{local topological recursion} \cite{dunin2014identification} and it is the version we are going to be interested in.

\subsubsection{Relations to Airy structures}
Let us introduce the following pre-Airy structure $\{(H_{TR})_{i=1,2,3,\cdots,\alpha \in Ram}\}$ on $V^{|Ram|}_{Airy}$ which is a slight modification of $\{(H_{Airy})_{n,\alpha}\}$ given in  (\ref{productresconstraintsairystructure}):
\begin{align*}\label{resconstraintstr}
    (H_{TR})_{2n,\alpha}(w) &= Res_{z=0}\left(\left(z-\frac{w(z)}{2zdz}\right)[\alpha]\otimes z^{2n-2}d(z^2)\right), \\ (H_{TR})_{2n-1,\alpha}(w) &= \frac{1}{2}Res_{z=0}\left(\left(z-\frac{w(z)}{2zdz}\right)\left(z+\frac{w(-z)}{2zdz}\right)[\alpha]\otimes z^{2n-2}d(z^2)\right)\numberthis
\end{align*}
for $n\geq 1$ where $w := \sum_{\alpha \in Ram}\sum_{k=1}^\infty(x^{k,\alpha}f_{k,\alpha} + y_{k,\alpha}e^{k,\alpha}) \in W^{|Ram|}_{Airy}$. We note the only change from $\{(H_{n,\alpha})_{Airy}\}$ is that $w(z)$ in one of the factor of $(H_{TR})_{2n-1,\alpha}$ has been replaced by $-w(-z)$. Expanding (\ref{resconstraintstr}) we get
\begin{equation}\label{resconstraintsairystructuretr}
    (H_{TR})_{i,\alpha} = -y_{i,\alpha} + (a_{TR})_{(i,\alpha)(j,\beta)(k,\gamma)}x^{j,\beta}x^{k,\gamma} + 2(b_{TR})_{(i,\alpha)(j,\beta)}^{k,\gamma}x^{j,\beta}y_{k,\gamma} + (c_{TR})_{i,\alpha}^{(j,\beta)(k,\gamma)}y_{j,\beta}y_{k,\gamma}
\end{equation}
where
\begin{align*}
    (a_{TR})_{(i,\alpha)(j,\beta)(k,\gamma)} &= Res_{z=0}\left(-\frac{1}{4i}\frac{1}{z(dz)^2}f_{i,\alpha}(z)f_{j,\beta}(z)f_{k,\gamma}(-z)\right)\delta_{i,odd}, \qquad A_{TR} := ((a_{TR})_{(i,\alpha)(j,\beta)(k,\gamma)})\\
    (b_{TR})_{(i,\alpha)(j,\beta)}^{k,\gamma} &= Res_{z=0}\left(-\frac{1}{4i}\frac{1}{z(dz)^2}f_{i,\alpha}(z)f_{j,\beta}(z)e^{k,\gamma}(-z)\right)\delta_{i,odd}, \qquad B_{TR} := ((b_{TR})_{(i,\alpha)(j,\beta)}^{k,\gamma})\\
    (c_{TR})_{i,\alpha}^{(j,\beta)(k,\gamma)} &= Res_{z=0}\left(-\frac{1}{4i}\frac{1}{z(dz)^2}f_{i,\alpha}(z)e^{j,\beta}(z)e^{k,\gamma}(-z)\right)\delta_{i,odd},\qquad C_{TR} := ((c_{TR})_{i,\alpha}^{(j,\beta)(k,\gamma)}).
\end{align*}
The expression (\ref{resconstraintsairystructuretr}) should be compared to (\ref{resconstraintsairystructure}).
Additionally, we define $(\epsilon_{TR})_{i,\alpha} = \frac{1}{16}\delta_{i,3}, \epsilon_{TR} = ((\epsilon_{TR})_{i,\alpha}) = \epsilon_{Airy}$. The key point is that: $(V_{Airy}^{|Ram|}, A_{TR}, B_{TR}, C_{TR}, \epsilon_{TR})$ is a quantum pre-Airy structure. Additionally, since $-w(-z) = w(z)$ if $x^{i=even, \alpha \in Ram} = y_{i=even, \alpha \in Ram} = 0$, it is clear from (\ref{resconstraintstr}) and Remark \ref{resconstraintssubairystructureremark} that both $(V^{|Ram|}_{Airy}, A_{TR}, B_{TR}, C_{TR}, \epsilon_{TR})$ and $(V^{|Ram|}_{Airy}, A_{Airy}, B_{Airy}, C_{Airy}, \epsilon_{Airy})$ contain the same quantum Airy sub-structure $((V_{Airy})^{|Ram|}_{odd}, A_{Airy}, B_{Airy}, C_{Airy}, \epsilon_{Airy})$.

We will now present the main result of this section.

\begin{proposition}\theoremname{\cite[Theorem G]{borot2018higher},\cite[Lemma 9.1]{andersen2017abcd},\cite[Section 3.1]{kontsevich2017airy}}\label{atrvstrproposition}
Consider a smooth compact Riemann surface $\Sigma_0$ with a set of distinct points $\{r_{\alpha \in Ram}\}$ and a neighbourhood $D_\alpha \ni r_\alpha$ for each $\alpha \in Ram$ with local coordinate $z_\alpha$. Let $\left\{\omega_{g,n} \in \Gamma\left((\Sigma_0\setminus \cup_{\alpha\in Ram} r_\alpha)^n, (\Omega_{\Sigma_0}^1)^{\boxtimes n}\right)\right\}$ be the multi-differentials produced from topological recursion on $\Sigma_0$ using:
\begin{enumerate}
    \item The involution $\sigma_\alpha(z_\alpha) := -z_\alpha$.
    \item $\omega_{0,1}(p)$ locally defined on each $D_\alpha$ for $\alpha \in Ram$ such that $\omega_{0,1}(z_\alpha) - \omega_{0,1}(\sigma_\alpha(z_\alpha)) = 4z_\alpha^2dz_\alpha$.
    \item $\omega_{0,2}(p,q) = B(p,q)$ the normalized Bergman kernel.
\end{enumerate}
Suppose that $V_{\Sigma_0}^{mer} \subset W^{|Ram|}_{Airy}$ is a Lagrangian complement of $T_0L_{Airy}^{|Ram|}$ with basis $\{\eta^{i=1,2,3,\cdots} := \sum_{\beta \in Ram}\sum_{j=1}^\infty d^i_{j,\beta}\bar{e}^{j,\beta}\}$, for some $(d^i_{j,\beta}) : V^{|Ram|}_{Airy}\rightarrow V^{|Ram|}_{Airy}$ and  $\{\bar{e}^{k=1,2,3,\cdots,\alpha \in Ram} = e^{k,\alpha} + \sum_{j,\beta}s^{(k,\alpha)(j,\beta)}f_{j,\beta}\}$ are given by Definition \ref{endifferentialdefinition}. Let $(V^{mer}_{\Sigma_0},\bar{A}_{Airy}, \bar{B}_{Airy}, \bar{C}_{Airy}, \bar{\epsilon}_{Airy})$ be the quantum Airy structure obtained from the gauge transformation of the quantum Airy structure $(V^{|Ram|}_{Airy}, A_{Airy}, B_{Airy}, C_{Airy}, \epsilon_{Airy})$ (see Section \ref{ramproductresconstraintssubsection}) corresponding to the change of the canonical basis $\{e^{k,\alpha}, f_{k,\alpha}\} \mapsto \{\eta^k, \omega_k\}$ of $W^{|Ram|}_{Airy}$. Also, let $\{S_{g,n} \in Sym_n(V^{mer}_{\Sigma_0})\}$ be the output of the ATR using $(V^{mer}_{\Sigma_0},\bar{A}_{Airy}, \bar{B}_{Airy}, \bar{C}_{Airy}, \bar{\epsilon}_{Airy})$. Then for $g\geq 0,n\geq 3$ or $g\geq 1, n\geq 1$, we have:
\begin{equation}\label{omegagnandsgn}
    \omega_{g,n}(p_1,...,p_n) = \sum_{i_1,\cdots, i_n = 1}^\infty S_{g,n;i_1,\cdots,i_n}\eta^{i_1}(p_1)\cdots\eta^{i_n}(p_n),
\end{equation}
where $\eta^{i_1}(p_1)\cdots\eta^{i_n}(p_n)$ denotes symmetric tensor product $\sum_{\sigma \in S_n}\frac{1}{n!}\eta^{i_{\sigma(1)}}(p_{\sigma(1)})\otimes ... \otimes \eta^{i_{\sigma(n)}}(p_{\sigma(n)})$. 
\end{proposition}
\begin{proof}
We let the basis of $T_0L_{Airy}^{|Ram|}$ be given by $\{\omega_{i=1,2,3,\cdots} := \sum_{\beta \in Ram}\sum_{j=1}^\infty c^{j,\beta}_if_{j,\beta}\}$, where $(c^{j,\beta}_i) = (d^i_{j,\beta})^{-1}$.
From Remark \ref{omegagnpropertiesremark}, we may write for $g\geq 0, n\geq 3$ or $g\geq 1, n\geq 1$:
\begin{equation*}
    \omega_{g,n}(p_1,...,p_n) = \sum_{i_1,\cdots,i_n=1}^\infty(S_{TR})_{g,n;i_1,...,i_n}\eta^{i_1}(p_1)\cdots\eta^{i_n}(p_n).
\end{equation*}
In the following let us assume that $p$ is close to one of the ramification point but $p_1$ is further away in the sense that if $p, p_1 \in U_\alpha \cap \Sigma_0$ were in the neighbourhood of $r_\alpha, \alpha \in Ram$ then in standard local coordinates we have $|z_\alpha(p)| << |z_\alpha(p_1)|$. 
The Laurent expansion of $\omega_{0,2}(p_1,p)$ when $p$ is close to a ramification point in the standard local coordinates can be written as
\begin{equation*}
    \omega_{0,2}(p_1, p) = \sum_{\alpha \in Ram, k > 0}\bar{e}^{k,\alpha}(p_1)f_{k,\alpha}(z_\alpha) = \sum_{k > 0}\eta^k(p_1)\omega_k(z_\alpha),
\end{equation*}
where $z$ is the standard local coordinate corresponding to $p$. Similarly, for $K(p_1,p)$ we have
\begin{equation*}
    K(p_1,p) = -\frac{1}{4}\frac{1}{z_\alpha(dz_\alpha)^2}\sum_{\substack{\alpha \in Ram, k > 0\\ k\ odd}} \frac{1}{k}\bar{e}^{k,\alpha}(p_1)f_{k,\alpha}(z_\alpha).
\end{equation*}
Now we have everything we need to evaluate the recursion formula. For $g=0,n=3$ we have
\begin{align*}
    \omega_{0,3}(p_1,p_2,p_3) = 2\sum_{\alpha \in Ram}Res_{p=r_\alpha}\left(K(p_1,p)\omega_{0,2}(p,p_2)\omega_{0,2}(\sigma_{\alpha}(p),p_3)\right)
\end{align*}
or,
\begin{align*}
    \sum_{i_1,i_2,i_3>0}&(S_{TR})_{0,3;i_1i_2i_3}\eta_0^{i_1}(p_1)\eta_0^{i_2}(p_2)\eta_0^{i_3}(p_3)\\
    &= 2\sum_{\substack{\alpha_1\in Ram, i_1,i_2,i_3 > 0\\i_1\ odd}}Res_{z_\alpha=0}\left(-\frac{1}{4i_1}\frac{1}{z_\alpha(dz_\alpha)^2}f_{i_1,\alpha_1}(z_\alpha)\omega_{i_2}(z_\alpha)\omega_{i_3}(-z_\alpha)\right)\bar{e}^{i_1,\alpha_1}(p_1)\eta^{i_2}(p_2)\eta^{i_3}(p_3).
\end{align*}
Comparing the coefficients, we have
\begin{align*}
    (S_{TR})_{0,3;i_1i_2i_3} &= 2\sum_{\substack{\alpha_1,\alpha_2,\alpha_3\in Ram, j_1,j_2,j_3 > 0\\j_1\ odd}}Res_{z_\alpha=0}\left(-\frac{1}{4j_1}f_{j_1,\alpha_1}(z_\alpha)f_{j_2,\alpha_2}(z_\alpha)f_{j_3,\alpha_3}(-z_\alpha)\right)c^{j_1,\alpha_1}_{i_1}c^{j_2,\alpha_2}_{i_2}c^{j_3,\alpha_3}_{i_3}\\
    &= 2\sum_{\substack{\alpha_1,\alpha_2,\alpha_3\in Ram, j_1,j_2,j_3 > 0\\j_1\ odd}} (a_{TR})_{(j_1,\alpha_1)(j_2,\alpha_2)(j_3,\alpha_3)}c^{j_1,\alpha_1}_{i_1}c^{j_2,\alpha_2}_{i_2}c^{j_3,\alpha_3}_{i_3}.
\end{align*}
For $g = 1, n = 1$ we have
\begin{align*}
    \omega_{1,1}&(p_1) = \sum_{\alpha \in Ram}Res_{p = r_\alpha}(K(p_1,p)\omega_{0,2}(p,\sigma_{\alpha}(p)))\\
    \sum_{i_1 > 0}(S_{TR})_{1,1;i_1}&\eta^{i_1}_0(p_1) = \lim_{z'_\alpha \rightarrow z_\alpha}\sum_{\alpha \in Ram}Res_{z_\alpha = r_\alpha}(K(p_1,p)\omega_{0,2}(z_\alpha,-z'_\alpha))\\
    =& \lim_{z'_\alpha \rightarrow z_\alpha}\sum_{\substack{\alpha_1,\alpha_2\in Ram, i_1,i_2 > 0\\i_1\ odd}}Res_{z_\alpha=0}\left(-\frac{1}{4i_1}\frac{1}{z_\alpha(dz_\alpha)^2}f_{i_1,\alpha_1}(z_\alpha)f_{i_2,\alpha_2}(z_\alpha)\bar{e}^{i_2,\alpha_2}(-z'_\alpha)\right)\bar{e}^{i_1,\alpha_1}(p_1)\\
    =& \lim_{z'_\alpha \rightarrow z_\alpha}\sum_{\substack{\alpha_1,\alpha_2\in Ram, i_1,i_2 > 0\\i_1\ odd}}Res_{z_\alpha=0}\left(-\frac{1}{4i_1}\frac{1}{z_\alpha(dz_\alpha)^2}f_{i_1,\alpha_1}(z_\alpha)f_{i_2,\alpha_2}(z_\alpha)e^{i_2,\alpha_2}(-z'_\alpha)\right)\bar{e}^{i_1,\alpha_1}(p_1)\\
    +\lim_{z'_\alpha \rightarrow z_\alpha}&\sum_{\substack{\alpha_1,\alpha_2\in Ram, i_1,i_2 > 0\\i_1\ odd}}s^{(i_2,\alpha_2)(j_2,\beta_2)}Res_{z_\alpha=0}\left(-\frac{1}{4i_1}\frac{1}{z_\alpha(dz_\alpha)^2}f_{i_1,\alpha_1}(z_\alpha)f_{i_2,\alpha_2}(z_\alpha)f_{j_2,\beta_2}(-z'_\alpha)\right)\bar{e}^{i_1,\alpha_1}(p_1)\\
    =& \sum_{\substack{\alpha_1\in Ram,i_1 > 0\\i_1\ odd}}Res_{z_\alpha = 0}\left(\frac{1}{4z_\alpha}\left(z^{i_1}_\alpha\frac{dz_\alpha}{z_\alpha}\right)\frac{1}{4z^2_\alpha}\right)\bar{e}^{i_1,\alpha_1}(p_1)\\
    &+ \sum_{\substack{\alpha_1,\alpha_2\in Ram, i_1,i_2 > 0\\i_1\ odd}}s^{(i_2,\alpha_2)(j_2,\beta_2)}Res_{z=0}\left(-\frac{1}{4i_1}\frac{1}{z_\alpha(dz_\alpha)^2}f_{i_1,\alpha_1}(z_\alpha)f_{i_2,\alpha_2}(z_\alpha)f_{j_2,\beta_2}(-z_\alpha)\right)\bar{e}^{i_1,\alpha_1}(p_1)
\end{align*}
Where the limit $\lim_{z'_\alpha \rightarrow z_\alpha}$ is taken from the direction $|z'_\alpha| > |z_\alpha|$. We can easily see that the first term of the last line is simply $\sum_{\alpha \in Ram}\frac{1}{16}\bar{e}^{3,\alpha}(p_1) = \epsilon_{TR}$ and the residue in the second term can be recognized as $A_{TR}$. Comparing the coefficients, we have
\begin{align*}
    (S_{TR})_{1,1;i_1} &= \sum_{\substack{\alpha_1 \in Ram, j_1>0\\
    j_1\ odd}}\left((\epsilon_{TR})_{j_1,\alpha_1} + \sum_{\alpha_2,\alpha_3\in Ram,  j_2, j_3 > 0} (a_{TR})_{(j_1,\alpha_1)(j_2,\alpha_2)(j_3,\alpha_3)}s^{(j_2,\alpha_2)(j_3,\alpha_3)}\right)c^{j_1,\alpha_1}_{i_1}.
\end{align*}
For $g\geq 0, n\geq 4$ or $g\geq 1, n\geq 2$ or $g\geq 2, n\geq 1$ we have from the recursion formula
\begin{align*}
    &\sum_{i_1,\cdots,i_n > 0}(S_{TR})_{g,n;i_1,...,i_n}\eta^{i_1}(p_1)\cdots\eta^{i_n}(p_n)\\
    &= \sum_{\substack{\alpha_1\in Ram, i_1,\cdots, i_n > 0\\i_1\ odd}}\Bigg[\sum_{j_1,j_2 > 0}Res_{z_\alpha=0}\left(-\frac{1}{4i_1}\frac{1}{z_\alpha(dz_\alpha)^2}f_{i_1,\alpha_1}(z_\alpha)\eta^{j_1}(z_\alpha)\eta^{j_2}(-z_\alpha)\right)(S_{TR})_{g-1,n+1;j_1j_2i_2...i_n}\\
    &\qquad + 2\sum_{k=2}^n\sum_{j_1>0}Res_{z_\alpha=0}\left(-\frac{1}{4i_1}\frac{1}{z_\alpha(dz_\alpha)^2}f_{i_1,\alpha_1}(z_\alpha)\omega_{i_k}(z_\alpha)\eta^{j_1}(-z_\alpha)\right)(S_{TR})_{g,n-1;j_1i_{\{2,..,n\}\setminus \{k\}}}\\
    &+\sum_{\substack{g_1+g_2=g\\I_1\coprod I_2 = \{2,...,n\}}}\sum_{j_1,j_2 > 0}Res_{z_\alpha=0}\left(-\frac{1}{4i_1}\frac{1}{z_\alpha(dz_\alpha)^2}f_{i_1,\alpha_1}(z_\alpha)\eta^{j_1}(z_\alpha)\eta^{j_2}(z_\alpha)\right)(S_{TR})_{g_1,1+|I_1|;j_1i_{I_1}}(S_{TR})_{g_2,1+|I_2|;j_2i_{I_2}}\Bigg]\\
    &\qquad\qquad\qquad\qquad\qquad\qquad\times\bar{e}^{i_1,\alpha_1}(p_1)\eta^{i_2}(p_2)...\eta^{i_n}(p_n).
\end{align*}
Comparing the coefficients, we find that the recursion relations (\ref{trrecursion}) becomes
\begin{align*}
    (S_{TR})_{g,n;i_1,...,i_n} &= 2\sum_{k=2}^n\sum_{j>0}(\bar{b}_{TR})_{i_1i_k}^{j}(S_{TR})_{g,n-1;ji_{\{2,...,n\}\setminus\{k\}}} + \numberthis\\
    &\qquad + \sum_{\substack{g_1+g_2=g\\I_1\coprod I_2 = \{2,...,n\}}}\sum_{j_1,j_2 > 0}(\bar{c}_{TR})_{i_1}^{j_1j_2}(S_{TR})_{g_1,1+|I_1|;j_1i_{I_1}}(S_{TR})_{g_2,1+|I_2|;j_2i_{I_2}}\\
    &\qquad + \sum_{j_1,j_2 > 0}(\bar{c}_{TR})_{i_1}^{j_1j_2}(S_{TR})_{g-1,n+1;j_1j_2i_2...i_n},
\end{align*}
for $g\geq 0, n\geq 4$ or $g\geq 2,n\geq 1$ or $g\geq 1, n\geq 2$ with the initial condition
\begin{equation}\label{tratrinitialconditions}
    (S_{TR})_{0,3;i_1i_2i_3} = 2(\bar{a}_{TR})_{i_1i_2i_3},\qquad (S_{TR})_{1,1;i} = (\bar{\epsilon}_{TR})_i, \qquad (S_{TR})_{0,n\leq 2} = (S_{TR})_{g\geq 0, 0} = 0
\end{equation}
where
\begin{align*}\label{tratrgaugetranfs}
    \numberthis
    (\bar{a}_{TR})_{i_1i_2i_3} &= (a_{TR})_{(j_1,\alpha_1)(j_2,\alpha_2)(j_3,\alpha_3)}c^{j_1,\alpha_1}_{i_1}c^{j_2,\alpha_2}_{i_2}c^{j_3,\alpha_3}_{i_3}\\
    (\bar{b}_{TR})_{i_1i_2}^{i_3} &= \left((b_{TR})_{(j_1,\alpha_1)(j_2,\alpha_2)}^{j_3,\alpha_3} + (a_{TR})_{(j_1,\alpha_1)(j_2,\alpha_2)(k,\gamma)}s^{(k,\gamma)(j_3,\alpha_3)}\right)c^{j_1,\alpha_1}_{i_1}c^{j_2,\alpha_2}_{i_2}d^{i_3}_{j_3,\alpha_3}\\
    (\bar{c}_{TR})_{i_1}^{i_2i_3} &= \Big((c_{TR})_{(j_1,\alpha_1)}^{(j_2,\alpha_2)(j_3,\alpha_3)} + (b_{TR})_{(j_1,\alpha_1)(k_1,\gamma_1)}^{(j_3,\alpha_3)}s^{(k_1,\gamma_1)(j_2,\alpha_2)}\\
    &\qquad + (b_{TR})_{(j_1,\alpha_1)(k_2,\gamma_2)}^{(j_2,\alpha_2)}s^{(k_2,\gamma_2)(j_3,\alpha_3)} + (a_{TR})_{(j_1,\alpha_1)(k_1,\gamma_1)(k_2,\gamma_2)}s^{(k_1,\gamma_1)(j_2,\gamma_2)}s^{(k_2,\gamma_2)(j_3,\gamma_3)}\Big)\\
    &\qquad\qquad\qquad\times c^{j_1,\alpha_1}_{i_1}d_{j_2,\alpha_2}^{i_2}d_{j_3,\alpha_3}^{i_3}\\
    (\bar{\epsilon}_{TR})_i &= \left((\epsilon_{TR})_{j_1,\alpha_1} + (a_{TR})_{(j_1,\alpha_1)(k_1,\gamma_1)(k_2,\gamma_2)}s^{(k_1,\gamma_1)(k_2,\gamma_2)} \right)c^{j_1,\alpha_1}_{i_1}.
\end{align*}
In (\ref{tratrgaugetranfs}) the summation is over positive odd integers $j_1$ and over all positive integers for all other indices. Equivalently, we may take all summations in (\ref{tratrgaugetranfs}) to be over all positive integers since components of $A_{TR}, B_{TR}, C_{TR}, \epsilon_{TR}$ with $j_1 = even$ are zero. Let $\bar{A}_{TR} = ((\bar{a}_{TR})_{i_1i_2i_3}), \bar{B}_{TR} = ((\bar{b}_{TR})_{i_1i_2}^{i_3}), \bar{C}_{TR} = ((\bar{c}_{TR})_{i_1}^{i_2i_3})$ and $\bar{\epsilon}_{TR} = ((\bar{\epsilon}_{TR})_i)$.
We recognize (\ref{tratr}) together with (\ref{tratrinitialconditions}) as the ATR of quantum pre-Airy structure $(V^{mer}_{\Sigma_0},\bar{A}_{TR},\bar{B}_{TR},\bar{C}_{TR},\bar{\epsilon}_{TR})$ which is the gauge transform of the quantum pre-Airy structure $(V^{|Ram|}_{Airy}, A_{TR}, B_{TR}, C_{TR}, \epsilon_{TR})$ corresponding to the change of the canonical basis $\{e^{k,\alpha}, f_{k,\alpha}\} \mapsto \{\eta^k, \omega_k\}$ of $W^{|Ram|}_{Airy}$. Using Remark \ref{gaugetransfrespectssubstructuresremark} and Remark \ref{resconstraintssubairystructureremark} we have that $(V^{mer}_{\Sigma_0},\bar{A}_{TR},\bar{B}_{TR},\bar{C}_{TR},\bar{\epsilon}_{TR})$ and $(V^{mer}_{\Sigma_0}, \bar{A}_{Airy}, \bar{B}_{Airy}, \bar{C}_{Airy}, \bar{\epsilon}_{Airy})$ both contain the quantum Airy sub-structure $((V^{mer}_{\Sigma_0})_{odd}, \bar{A}_{Airy}, \bar{B}_{Airy}, \bar{C}_{Airy}, \bar{\epsilon}_{Airy})$. It follows that $(S_{TR})_{g,n} = S_{g,n} \in Sym_n((V^{mer}_{\Sigma_0})_{odd})$ (see Remark \ref{atrandprestructuresremark}), which completes the proof. 
\end{proof}

\subsection{Topological Recursion and Prepotentials}\label{prepotentialandtrsection}
In this section, we are going to prove the main theorem which relates the prepotential $\mathfrak{F}_{\Sigma_0}$ (see Definition \ref{prepotentialdefinition}), defined on some open neighbourhood $\mathcal{B}_{\Sigma_0} \ni [\Sigma_0]$ of the moduli space $\mathcal{B}$ of $\mathcal{F}$-transversal genus $g$ curves in $(S, \Omega_S,\mathcal{F})$, to the genus zero part of topological recursion on $\Sigma_0$. The basic idea is as follows. The map $\pmb{\Phi}_{\Sigma_0} : \mathcal{B}_{\Sigma_0}\rightarrow \mathcal{H}_{\Sigma_0}$ embeds $\mathcal{B}_{\Sigma_0}$ into the fiber $\mathcal{H}_{\Sigma_0}$ as the Lagrangian submanifold $\mathcal{L}_{\Sigma_0}$ with $\mathfrak{F}_{\Sigma_0}$ as the generating function. The image of $\Phi_{t_0} : Discs^{Ram}_{t_0}\rightarrow W^{Ram}_{t_0}$ is contained inside the Lagrangian submanifold $L^{Ram}_{Airy}$ with $S_0$, the $g=0$ part of the ATR output of the $Ram$ product residue constraints Airy structure $\{(H_{Airy})_{i=1,2,3,\cdots, \alpha \in Ram}\}$, as the generating function. In Section \ref{localtoglobalsection} we constructed the map $\Phi_{\Sigma_0} : \mathcal{B}_{\Sigma_0}\rightarrow G_{\Sigma_0}$ relating $\pmb{\Phi}_{\Sigma_0}$ to $\Phi_{t_0}$, therefore also relating $\mathfrak{F}_{\Sigma_0}$ to $S_0$. Finally, we use Proposition \ref{atrvstrproposition} to relate $S_0$ to $\omega_{0,n}$.

The important step in simplifying the proof of this result is to use the right gauge when working with the Airy structure $\{(H_{Airy})_{i=1,2,3,\cdots, \alpha \in Ram}\}$. For example, the canonical basis 
\begin{equation*}
    \left\{e^{k=1,2,3,\cdots,\alpha \in Ram} := [\alpha]\otimes z^{-k}_\alpha\frac{dz_\alpha}{z_\alpha}, f_{k=1,2,3,\cdots,\alpha \in Ram} := [\alpha]\otimes kz^k_\alpha\frac{dz_\alpha}{z_\alpha}\right\}
\end{equation*}
for $W^{Ram}_{Airy}$ will not make the comparison between $S_0 = S_0(\{x^{k=1,2,3,\cdots,\alpha \in Ram}\})$ and $\mathfrak{F}_{\Sigma_0} = \mathfrak{F}_{\Sigma_0}(\mathbf{a}^1,\cdots,\mathbf{a}^g)$ easy. Each $x^{k,\alpha}$ is a coefficient of the Laurent series expansion of $\Phi_{\Sigma_0}(\Sigma)$ whereas each $\mathbf{a}^i$ relates to $A$-periods of $\Phi_{\Sigma_0}(\Sigma)$. Therefore, we start by choosing the new canonical basis for $W^{Ram}_{Airy}$ which makes the cohomological information of $\Phi_{\Sigma_0}(\Sigma)$ becomes more visible from the point of view of $W^{Ram}_{Airy}$.

\begin{lemma}\label{goodbasislemma}
Let $[\Sigma_0] \in \mathcal{B}$ and let $\{\bar{e}^{k=1,2,3,\cdots,\alpha \in Ram}\}$ be the basis of the Lagrangian complement $V_{\Sigma_0}$ as given in Definition \ref{endifferentialdefinition}. There exists a canonical basis 
\begin{equation*}
    \left\{\eta^{i=1,2,3,\cdots} := \sum_{\beta \in Ram}\sum_{j=1}^\infty d^i_{j,\beta}\bar{e}^{j,\beta}, \omega_{i=1,2,3,\cdots} := \sum_{\beta \in Ram}\sum_{j=1}^\infty c^{j,\beta}_if_{j,\beta}\right\}
\end{equation*}
of $W^{Ram}_{t(\Sigma_0)} \cong W^{Ram}_{Airy}$,  
\begin{equation*}
    \Omega_{Airy}(\omega_i,\omega_j) = \Omega_{Airy}(\eta^i,\eta^j) = 0, \qquad \Omega_{Airy}(\eta^i,\omega_j) = \delta^i_j,
\end{equation*}
corresponding to the Lagrangian complement $V_{\Sigma_0}$ such that $\omega_1,\cdots,\omega_g \in \Gamma(\Sigma_0, \Omega^1_{\Sigma_0})$, $\eta^{i=1,2,3,\cdots} \in V_{\Sigma_0}$ and $\omega_{i > g}$ are locally defined holomorphic forms on $\sqcup_{\alpha \in Ram}\mathbb{D}_{\alpha, \bar{M}_\alpha}(\Sigma_0)\subset \Sigma_0$ for some $\bar{M}_\alpha > M_\alpha$. Furthermore, we have that for $i,j = 1,\cdots,g$:
\begin{equation}\label{abperiodsofgoodbasis}
    \oint_{A_i}\omega_j = \delta_j^i, \qquad \oint_{B_i}\omega_j = \tau_{ij}(\Sigma_0), \qquad \oint_{A_i}\eta^j = 0, \qquad \oint_{B_i}\eta^j = 2\pi i\delta^j_i,
\end{equation}
and all $A,B$-periods of $\eta^{i>g}$ vanish. Lastly, $(c^{i,\alpha}_j) : V^{|Ram|}_{Airy} \rightarrow V^{|Ram|}_{Airy}$, $(d^i_{j,\beta}) : V^{|Ram|}_{Airy}\rightarrow V^{|Ram|}_{Airy}$ both satisfy Condition \ref{analyticcdscondition} and $\{\bar{e}^{i=1,2,3,\cdots,\alpha \in Ram}\}$ is given by Definition \ref{endifferentialdefinition}.
\end{lemma}
\begin{remark}
Equivalently, we describe the canonical basis in \ref{goodbasislemma} as follows. Choose $\{\omega_1,\cdots,\omega_g\}$ to be the set of normalized holomorphic differentials which is a basis for $\Gamma(\Sigma_0,\Omega^1_{\Sigma_0})$. The set $\{\omega_1,\cdots,\omega_g,\omega_{i>g}\}$ is a basis for $T_0L^{Ram}_{Airy}$, where $\omega_{i>g}$ are locally defined. The set $\{\eta^1,\cdots,\eta^g,\eta^{i>g}\}$ is a basis for $V_{\Sigma_0}$ because $\oint_{A_i}\eta^j = 0$ for all $j = 1,2,3,\cdots$ and $\{\eta^{i>g}\}$ is a basis for $G^{\perp}_{\Sigma_0}$ because all $A,B$-periods of $\eta^{i>g}$ vanish.

It follows that $\{\omega_1,\cdots,\omega_g,\eta^1,\cdots,\eta^g,\eta^{i>g}\}$ is a basis of $G_{\Sigma_0}$ and $\{\omega_1,\cdots,\omega_g, \eta^1,\cdots,\eta^g\}$ is a basis for $\mathcal{H}_{\Sigma_0} = H^1(\Sigma_0,\mathbb{C}) = G_{\Sigma_0}/G^{\perp}_{\Sigma_0}$. 
Finally, $\oint_{B_i}\eta^j = 2\pi i\delta_{i}^j$ for $j = 1,\cdots,g$ can be deduced from $\Omega_{Airy}(\eta^j,\omega_i) = \delta^i_j$.
\end{remark}
\begin{proof}
Let $\{\omega_1,\cdots,\omega_g\}$ be the normalized holomorphic differentials on $\Sigma_0$, $\oint_{A_i}\omega_j = \delta^i_j, \oint_{B_i}\omega_j = \tau_{ij}(\Sigma_0)$. It is clear that $\{\omega_1,\cdots,\omega_g\}$ is linearly independent in $T_0L^{Ram}_{Airy} \subset W^{Ram}_{Airy}$. Define 
\begin{equation*}
    (T_0L_{Airy}^{Ram})_{> n} := \left\{\xi = \sum_{\alpha\in Ram}[\alpha]\otimes \xi_\alpha\ |\ \xi_\alpha(z) \in z^{n+1}_\alpha \mathbb{C}[[z_\alpha]]\frac{dz_\alpha}{z_\alpha}, \begin{array}{c}
    \text{$\xi_\alpha(z)$ converges for $|z_\alpha| < \bar{M}_\alpha$}\\ 
    \text{for some $\bar{M}_\alpha > M_\alpha$}\end{array}\right\} \subset T_0L^{Ram}_{Airy}.
\end{equation*}
We take $\{f^{i>n,\alpha \in Ram}\}$ to be the basis of $(T_0L^{Ram}_{Airy})_{>n}$. It follows that $T_0L^{Ram}_{Airy}/(T_0L^{Ram}_{Airy})_{>n}$ is an $n|Ram|$-dimensional vector space. By choosing a large enough $n$, the image of $\{\omega_1,\cdots,\omega_g\}$ in $T_0L^{Ram}_{Airy}/(T_0L^{Ram}_{Airy})_{>n}$ is linearly independent and we can complement it with $\{\omega_{g+1},\cdots,\omega_{n|Ram|}\}$ such that the image of $\{\omega_1,\cdots,\omega_{n|Ram|}\}$ in $T_0L^{Ram}_{Airy}/(T_0L^{Ram}_{Airy})_{>n}$ is a basis for $T_0L^{Ram}_{Airy}/(T_0L^{Ram}_{Airy})_{>n}$. Let $\{\omega_{i > n|Ram|}\}$ be given by $\{f_{i>n,\alpha \in Ram}\}$ rearranged in some order, for example, let $r:\{0,\cdots,|Ram|-1\}\rightarrow Ram$ be an arbitrary bijection then we define $\omega_i := f_{\lfloor i/|Ram|\rfloor, r(i\bmod{|Ram|})}$. Then $\{\omega_{i=1,2,3,\cdots}\}$ is a basis of $(T_0L^{Ram}_{Airy}/(T_0L^{Ram}_{Airy})_{>n})\oplus ((T_0L^{Ram}_{Airy})_{>n} \cong T_0L^{Ram}_{Airy}$ and $i(\omega_k) = \sum_{\beta\in Ram}\sum_{j=1}^\infty c^{j,\beta}_kf_{j,\beta}$ for some $(c^{j,\beta}_i)$ satisfying Condition \ref{analyticcdscondition}. In particular, $(c^{j,\beta}_i)$ satisfies C\ref{analyticcdscondition}.3. 

It follows that its inverse $(d^i_{j,\beta})$ also satisfies Condition \ref{analyticcdscondition} and that $\{\eta^{i=1,2,3,\cdots}, \omega_{i=1,2,3,\cdots}\}$ gives a canonical basis for $W^{Ram}_{Airy}$. Since $\eta^i = \sum_{\beta\in Ram}\sum_{j=1}^\infty d^i_{j,\beta}\bar{e}^{j,\beta}$ is a finite summation for each $i$, we have $\eta^i \in V_{\Sigma_0}$ and in particular that $\oint_{A_i}\eta^j = 0$. Suppose that $i > g$ and $j = 1,\cdots, g$ then by using (\ref{omegaairyong}) we have $\oint_{B_j}\eta^i = 2\pi i\Omega_{Airy}(\eta^i, \omega_j) = 0$. Similarly, by using (\ref{omegaairyong}) for $i,j = 1,\cdots,g$ we have $\oint_{B_j}\eta^i = 2\pi i\Omega_{Airy}(\eta^i,\omega_j) = 2\pi i\delta^i_j$. 
\end{proof}

The main theorem is a joint work with Paul Norbury, Michael Swaddle and Mehdi Tavakol which was proven in \cite{chaimanowong2020airystructures} using a different approach. Here we are stating the theorem and prove it using an analytical framework:

\begin{theorem}\label{prepotentialvsomega0ntheorem}
Let $\mathcal{B}$ be the moduli space of $\mathcal{F}$-transversal genus $g$ curves in a foliated symplectic surface $(S,\Omega_S,\mathcal{F})$. Then for some small open neighbourhood $\mathcal{B}_{\Sigma_0} \subset \mathcal{B}$ of the reference point $[\Sigma_0] \in \mathcal{B}$, the prepotential $\mathfrak{F}_{\Sigma_0} = \mathfrak{F}_{\Sigma_0}(\mathbf{a}^1,\cdots, \mathbf{a}^g)$ such that 
\begin{equation*}
    \mathbf{a}^i := \mathbf{a}^i(\Sigma) = \oint_{A_i}\pmb{\theta}_{\Sigma}, \qquad \mathbf{b}_i := \mathbf{b}_i(\Sigma) = \oint_{B_i}\pmb{\theta}_\Sigma = \frac{\partial \mathfrak{F}_{\Sigma_0}(\{\mathbf{a}^k\})}{\partial\mathbf{a}^i}, \qquad i = 1,\cdots, g.
\end{equation*}
can be expressed as
\begin{align*}
    \mathfrak{F}_{\Sigma_0}&(\mathbf{a}^1,\cdots,\mathbf{a}^g) = \mathfrak{F}_{\Sigma_0}(\mathbf{a}^1_0,\cdots,\mathbf{a}^g_0)\\ 
    &+ \sum_{i=1}^g(\mathbf{a}^i - \mathbf{a}^i_0)\frac{\partial \mathfrak{F}_{\Sigma_0}}{\partial \mathbf{a}^i}(\mathbf{a}^1_0,\cdots,\mathbf{a}^g_0) + \frac{1}{2}\sum_{i,j=1}^g(\mathbf{a}^i - \mathbf{a}^i_0)(\mathbf{a}^j - \mathbf{a}^j_0)\frac{\partial^2\mathfrak{F}_{\Sigma_0}}{\partial\mathbf{a}^i\partial\mathbf{a}^j}(\mathbf{a}^1_0,\cdots,\mathbf{a}^g_0)\\
    &+ \sum_{n=3}^\infty\frac{1}{n!}\left(\frac{1}{2\pi i}\right)^{n-1}\sum_{i_1,\cdots,i_n = 1}^g\left(\oint_{p_1\in B_{i_1}}\cdots \oint_{p_n \in B_{i_n}}\omega_{0,n}(p_1,\cdots,p_n)\right)(\mathbf{a}^{i_1}-\mathbf{a}^{i_1}_0)\cdots (\mathbf{a}^{i_n}-\mathbf{a}^{i_n}_0)\label{prepotentialvsomega0nformula}\numberthis
\end{align*}
for all $(\mathbf{a}^1,\cdots, \mathbf{a}^g) = [\Sigma]\in \mathcal{B}_{\Sigma_0}$. Where $\mathbf{a}_0^k := \mathbf{a}^k(\Sigma_0)$ and the multi-differentials $\left\{\omega_{g,n} \in \Gamma\left((\Sigma_0\setminus \cup_{\alpha\in Ram} r_\alpha)^n, (\Omega_{\Sigma_0}^1)^{\boxtimes n}\right)\right\}$ are produced from topological recursion on $\Sigma_0$ using: 
\begin{enumerate}
    \item The involution $\sigma_\alpha(z_\alpha) := -z_\alpha$.
    \item $\omega_{0,1}(p)$ locally defined on each $\Sigma_0\cap U_\alpha$ for $\alpha \in Ram$ such that $\omega_{0,1}(z_\alpha) - \omega_{0,1}(\sigma_\alpha(z_\alpha)) = 4z_\alpha^2dz_\alpha$.
    \item $\omega_{0,2}(p,q) = B(p,q)$ the normalized Bergman kernel.
\end{enumerate}
with $z_\alpha$ for $\alpha \in Ram$ denotes the standard local coordinates corresponding to the collection of $(\mathcal{F},\Omega_S)$-charts $\mathcal{U}_{Ram} = \{(U_{\alpha \in Ram}, x_{\alpha \in Ram}, y_{\alpha \in Ram})\}$, $r_\alpha(\Sigma_0) \in \Sigma_0\cap U_\alpha$ such that $\gamma_\alpha(\Sigma_0) = (x_\alpha = z^2_\alpha, y_\alpha = z_\alpha)$.
\end{theorem}
\begin{proof}
Let $\{\epsilon_{\alpha\in Ram}\}, \{M_{\alpha \in Ram}\}$, $\mathcal{U}_{Ram}$ and $\mathcal{B}_{\Sigma_0}$ be chosen such that Condition \ref{howtochoosebsigma0condition} is satisfied. In particular, the collection of $(\mathcal{F}, \Omega_S)$-charts $\mathcal{U}_{Ram}$ is chosen such that $\gamma_\alpha(\Sigma_0) = (x_\alpha = z_\alpha^2, y_\alpha = z_\alpha)$. Consider the embedding map $\Phi_{\Sigma_0} : \mathcal{B}_{\Sigma_0}\rightarrow G_{\Sigma_0}$ given by $\Phi_{\Sigma_0}(\Sigma) := -\theta_{\Sigma_0} + s^*_{\Sigma_0}\theta_\Sigma$ (see Section \ref{relationtoresconstraintssubsection}). It follows from Proposition \ref{localvsglobalproposition} that $i(\Phi_{\Sigma_0}(\Sigma))$ satisfies the $Ram$ product analytic residue constraints:
\begin{equation}\label{iphisatisfiesramprodresconstraints}
    i(\Phi_{\Sigma_0}(\Sigma)) = \Phi_{t(\Sigma_0)}(\gamma(\Sigma)) = \sum_{\alpha \in Ram}\sum_{k=1}^\infty(x^{k,\alpha}f_{k,\alpha} + y_{k,\alpha}e^{k,\alpha}) \in L^{Ram}_{Airy} \subset W^{Ram}_{Airy}.
\end{equation}
Let us change the canonical basis of $W^{Ram}_{Airy}$ to $\{\eta^{i=1,2,3,\cdots}, \omega_{i=1,2,3,\cdots}\}$ as provided by Lemma \ref{goodbasislemma}. Then,
\begin{equation}\label{phisigma0innewbasis}
\Phi_{\Sigma_0}(\Sigma) = \sum_{k=1}^\infty \beta_k\eta^k + \sum_{k=1}^g\alpha^k\omega_k,    
\end{equation}
by Proposition \ref{analyticgaugetransformationproposition} where $\{\alpha^{k=1,2,3,\cdots}\}, \{\beta_{k=1,2,3,\cdots}\}$ are given by (\ref{bermangaugecoordtransf}) and $\alpha^{i>g} = 0$ since $\omega_{i>g}\notin G_{\Sigma_0}$ (alternatively, we could apply Corollary \ref{gisvoplush0corollary}). Computing the $A$-periods of $\Phi_{\Sigma_0}(\Sigma)$ using (\ref{abperiodsofgoodbasis}), we obtain for $i = 1,\cdots, g$:
\begin{align*}
    \alpha^i &= \oint_{A_i}\left(\sum_{k=1}^\infty \beta_k\eta^k + \sum_{k=1}^g\alpha^k\omega_k\right)\\
    &= \oint_{A_i}\Phi_{\Sigma_0}(\Sigma) = \oint_{A_i}\pmb{\Phi}_{\Sigma_0}(\Sigma) = -\oint_{A_i}\pmb{\theta}_{\Sigma_0} + \oint_{A_i}\pmb{\theta}_{\Sigma} = \mathbf{a}^i(\Sigma) - \mathbf{a}^i(\Sigma_0)\label{aperiodsofphi}\numberthis.
\end{align*}
Similarly, computing $B$-periods of $\Phi_{\Sigma_0}(\Sigma)$ using (\ref{abperiodsofgoodbasis}), we obtain for $i = 1,\cdots,g$:
\begin{align*}
    2\pi i\beta_i + \sum_{k=1}^\infty\tau_{ik}(\Sigma_0)\alpha^k &= \oint_{B_i}\left(\sum_{k=1}^\infty \beta_k\eta^k + \sum_{k=1}^g\alpha^k\omega_k\right)\\
    &= \oint_{B_i}\Phi_{\Sigma_0}(\Sigma) = \oint_{B_i}\pmb{\Phi}_{\Sigma_0}(\Sigma) = -\oint_{B_i}\pmb{\theta}_{\Sigma_0} + \oint_{B_i}\pmb{\theta}_{\Sigma} = \mathbf{b}_i(\Sigma) - \mathbf{b}_i(\Sigma_0)\label{bperiodsofphi}\numberthis.
\end{align*}
Where we have used Lemma \ref{fromlocalvtoglobalvlemma} in conjunction with the fact that $(d^i_{j,\beta})$ satisfies C\ref{analyticcdscondition}.3 to justify the term-by-term integration. In particular for each $i = 1,2,3,\cdots$, $\eta^i$ is a finite linear combination of $\bar{e}^{i,\alpha}$. We also have used the relation $[\Phi_{\Sigma_0}(\Sigma)] = \pmb{\Phi}_{\Sigma_0}(\Sigma)$ from Proposition \ref{localvsglobalproposition}. From now on we will write $\mathbf{a}^i := \mathbf{a}^i(\Sigma), \mathbf{b}_i := \mathbf{b}_i(\Sigma)$ and $\mathbf{a}^i_0 := \mathbf{a}^i(\Sigma_0), \mathbf{b}^0_i := \mathbf{b}_i(\Sigma_0)$. 

Let us cover $\mathcal{B}_{\Sigma_0}$ with a coordinates chart $(\mathcal{B}_{\Sigma_0}, u^1,\cdots,u^g)$. Then it is clear from construction that for each $\alpha \in Ram$, $i_\alpha(\Phi_{\Sigma_0}(\Sigma))$ is holomorphic in $\{u^1,\cdots,u^g, z_\alpha\}$ for all $(u^1,\cdots,u^g,z_\alpha) \in \mathcal{B}_{\Sigma_0}\times\mathbb{A}_{\alpha,\epsilon_\alpha, M_\alpha}$. Recall from Lemma \ref{aiscoordforblemma} that $\{\mathbf{a}^1,\cdots,\mathbf{a}^g\}$ is also a coordinates system of $\mathcal{B}_{\Sigma_0}$. From (\ref{phisigma0innewbasis}) and (\ref{aperiodsofphi}), it follows that $\beta_i$ are holomorphic functions in $\{\alpha^1,\cdots,\alpha^g\}$. It follows that $i_\alpha\left(\sum_{k=1}^\infty\beta_k(\alpha^1,\cdots,\alpha^g)\eta^k\right)(z_\alpha)$ is holomorphic in $\{\alpha^1,\cdots,\alpha^g, z_\alpha^{-1}\}$ on the open neighbourhood $\mathcal{B}_{\Sigma_0}$ of $(\alpha^1,\cdots,\alpha^g) = (0,\cdots,0)$ and $|z_\alpha| > \epsilon_\alpha$. Using (\ref{iphisatisfiesramprodresconstraints}) together with Proposition \ref{analyticgaugetransformationproposition} implies that
\begin{equation*}
    (\bar{H}_{Airy})_i\left(\{\alpha^1,\cdots,\alpha^g,\alpha^{k > g} = 0\}, \{\beta_{k=1,2,3,\cdots}(\alpha^1,\cdots,\alpha^g)\}\right) = 0,
\end{equation*}
where $(\bar{H}_{Airy})_i$ is the gauge transformed $Ram$ product residue constraints Airy structure as given in (\ref{resconstraintshamiltonianinbergmangauge}). Using the obvious generalization of Proposition \ref{analyticatrproposition} (see also Remark \ref{bergmangaugeremark}), we conclude that $\beta_i(\alpha^1,\cdots,\alpha^g) = (\partial_iS_0)|_{\alpha^{k > g}=0}$ where $S_0 = \sum_{n=1}^\infty S_{0,n} \in \prod_{n=1}^\infty Sym_n(V^{mer}_{\Sigma_0})$ is the output of the ATR using the quantum Airy structure $(V^{mer}_{\Sigma_0},\bar{A}_{Airy}, \bar{B}_{Airy}, \bar{C}_{Airy}, \bar{\epsilon}_{Airy})$. Here, $(V^{mer}_{\Sigma_0},\bar{A}_{Airy}, \bar{B}_{Airy}, \bar{C}_{Airy}, \bar{\epsilon}_{Airy})$ is the quantum Airy structure obtained from the gauge transformation of the quantum Airy structure $(V^{|Ram|}_{Airy}, A_{Airy}, B_{Airy}, C_{Airy}, \epsilon_{Airy})$, corresponding to the change of a canonical basis of $W^{|Ram|}_{Airy}$ from $\{e^{k=1,2,3,\cdots,\alpha \in Ram}, f_{k=1,2,3,\cdots,\alpha \in Ram}\}$ to $\{\eta^{k=1,2,3\cdots}, \omega_{k=1,2,3,\cdots}\}$ (see Section \ref{ramproductresconstraintssubsection}). In particular, let $\mathcal{S}_0(\alpha^1,\cdots,\alpha^g) := S_0(\{\alpha^1,\cdots,\alpha^g, \alpha^{k>g}=0\})$ then $\mathcal{S}_0$ is holomorphic on $\mathcal{B}_{\Sigma_0}$ and $\beta_i(\alpha^1,\cdots,\alpha^k) = \frac{\partial \mathcal{S}_0}{\partial \alpha^i}(\alpha^1,\cdots,\alpha^k)$ for $i = 1,\cdots,g$. Equivalently, $\beta_i = \frac{\partial \mathcal{S}_0}{\partial \mathbf{a}^i}(\mathbf{a}^1-\mathbf{a}^1_0,\cdots, \mathbf{a}^g-\mathbf{a}^g_0)$ for $i = 1,\cdots,g$. Using the fact that $\mathbf{b}_i = \frac{\partial \mathfrak{F}_{\Sigma_0}}{\partial \mathbf{a}^i}(\mathbf{a}^1,\cdots,\mathbf{a}^g)$ and $\tau_{ij}(\Sigma) = \frac{\partial^2\mathfrak{F}_{\Sigma_0}}{\partial\mathbf{a}^i\partial\mathbf{a}^j}(\mathbf{a}^1,\cdots,\mathbf{a}^g)$ then (\ref{aperiodsofphi}) and (\ref{bperiodsofphi}) implies
\begin{equation*}
    2\pi i\frac{\partial \mathcal{S}_0}{\partial \mathbf{a}^i}(\mathbf{a}^1-\mathbf{a}^1_0,\cdots, \mathbf{a}^g-\mathbf{a}^g_0) + \sum_{k=1}^g (\mathbf{a}^k-\mathbf{a}^k_0)\frac{\partial^2\mathfrak{F}_{\Sigma_0}}{\partial\mathbf{a}^k\partial\mathbf{a}^i}(\mathbf{a}^1_0,\cdots,\mathbf{a}^g_0) = \frac{\partial\mathfrak{F}_{\Sigma_0}}{\partial\mathbf{a}^i}(\mathbf{a}^1,\cdots,\mathbf{a}^g) - \frac{\partial\mathfrak{F}_{\Sigma_0}}{\partial\mathbf{a}^i}(\mathbf{a}^1_0,\cdots,\mathbf{a}^g_0)
\end{equation*}
Integrating both sides, we find that for all $(\mathbf{a}^1,\cdots,\mathbf{a}^g) = [\Sigma] \in \mathcal{B}_{\Sigma_0}$ we have
\begin{align*}
    \mathfrak{F}_{\Sigma_0}(\mathbf{a}^1,\cdots,\mathbf{a}^g) &= \mathfrak{F}_{\Sigma_0}(\mathbf{a}^1_0,\cdots,\mathbf{a}^g_0) + \sum_{i=1}^g(\mathbf{a}^i-\mathbf{a}^i_0)\frac{\partial\mathfrak{F}_{\Sigma_0}}{\partial\mathbf{a}^i}(\mathbf{a}^1_0,\cdots,\mathbf{a}^g_0)\\
    &\qquad + \sum_{i,j=1}^g (\mathbf{a}^i-\mathbf{a}^i_0)(\mathbf{a}^j-\mathbf{a}^j_0)\frac{\partial^2\mathfrak{F}_{\Sigma_0}}{\partial\mathbf{a}^i\partial\mathbf{a}^j}(\mathbf{a}^1_0,\cdots,\mathbf{a}^g_0) + 2\pi i\mathcal{S}_0(\mathbf{a}^1-\mathbf{a}_0^1,\cdots,\mathbf{a}^g-\mathbf{a}^g_0)\label{prepotentialvsmathcals0}\numberthis.
\end{align*}
From (\ref{sgnpolynomial}), $S_0 = \sum_{n=1}^\infty \frac{1}{n!}\sum_{i_1,\cdots,i_n = 1}^\infty S_{0,n;i_1\cdots i_n}\alpha^{i_1}\cdots\alpha^{i_n}$ it follows that
\begin{equation}\label{mathcals0}
    \mathcal{S}_0(\mathbf{a}^1-\mathbf{a}^1_0,\cdots,\mathbf{a}^g-\mathbf{a}^g_0) = \sum_{n=1}^\infty\frac{1}{n!}\sum_{i_1,\cdots,i_n}^gS_{0,n;i_1,\cdots,i_n}(\mathbf{a}^{i_1}-\mathbf{a}^{i_1}_0)\cdots(\mathbf{a}^{i_n}-\mathbf{a}^{i_n}_0).
\end{equation}
On the other hand, Proposition \ref{atrvstrproposition} and the fact that $\oint_{B_i}\eta^j = 0$ if $j > g$ and $\oint_{B_i}\eta^j = 2\pi i\delta_{i}^j$ if $j \leq g$ implies that for $i_1,\cdots,i_n = 1,\cdots,g$ we have
\begin{equation}\label{s0vsomegagn}
    S_{0,n;i_1,\cdots,i_n} = \left(\frac{1}{2\pi i}\right)^n\oint_{p_1 \in B_{i_1}}\cdots \oint_{p_n \in B_{i_n}}\omega_{0,n}(p_1,\cdots,p_n).
\end{equation}
Combining (\ref{prepotentialvsmathcals0}), (\ref{mathcals0}) and (\ref{s0vsomegagn}) the theorem is proven.
\end{proof}

\begin{remark}
In Theorem \ref{prepotentialvsomega0ntheorem}, we can take $\omega_{0,1}(z_\alpha) = y_\alpha dx_\alpha$ but other choices for $\omega_{0,1}(p)$ are equally valid as long as $\omega_{0,1}(z_\alpha) - \omega_{0,1}(-z_\alpha) = 4z_\alpha^2dz_\alpha$ is satisfied. Note that although $\omega_{0,1}(z_\alpha) = y_\alpha dx_\alpha$ looks very similar to $-\theta_{\Sigma_0}$ which locally takes the form $\theta_{\Sigma_0}|_{\Sigma_0\cap U_p} = -(y_p + f_p(x_p))dx_p$, in general $\omega_{0,1}$ and $-\theta_{\Sigma_0}$ cannot be the same. Firstly, $\omega_{1,0}$ is usually defined only locally near ramification points whereas $\theta_{\Sigma_0}$ is defined on $\Sigma_0\setminus \cup_{\alpha \in Ram}\mathbb{D}_{\alpha,\epsilon_\alpha}(\Sigma_0)$. Moreover, it is generally not true that $\theta_{\Sigma_0}(-z_\alpha) - \theta_{\Sigma_0}(z_\alpha) = 4z_\alpha^2dz_\alpha$ so we cannot choose $\omega_{1,0} = -\theta_{\Sigma_0}$. For example, let $[\Sigma_1] \in \mathcal{B}_{\Sigma_0}$ be a nearby curve of $\Sigma_0$ with the local parameterization:
\begin{equation*}
    \gamma_\alpha(\Sigma_1) = \left(x_\alpha = z_\alpha^2+a_\alpha(\Sigma_1), y_\alpha = b_{0\alpha}(\Sigma_1) + b_{1\alpha}(\Sigma_1)z_\alpha + b_{2\alpha}(\Sigma_1)z^2_\alpha + \cdots\right).
\end{equation*}
Let us construct $\theta \in \Gamma(\mathcal{B}_{\Sigma_0}, G)$ by patching together $\theta_{\Sigma}|_{\Sigma\cap U_p} := -y_p(x_p;\Sigma)dx_p + y_p(x_p;\Sigma_1)dx_p$ as explained in Section \ref{relationtoresconstraintssubsection} (replacing $\mathcal{B}_{\Sigma_0}$ by $\mathcal{B}_{\Sigma_0}\cap \mathcal{B}_{\Sigma_1}$ if necessary). As in the proof of Lemma \ref{patchingthetalemma}, on the annulus $\mathbb{A}_{\alpha, \epsilon_\alpha, M_\alpha}(\Sigma_0) \subset \Sigma_0\cap U_\alpha$, we have
\begin{align*}
    \theta_{\Sigma_0}(z_\alpha) &= \theta_{t_\alpha(\Sigma_0)}(z_\alpha) - \exp\left(-a_\alpha(\Sigma_1)\mathcal{L}_{\frac{1}{2z_\alpha}\partial_{z_\alpha}}\right)\theta_{t_\alpha(\Sigma_1)}(z_\alpha)\\
    &= -z_\alpha d(z^2_\alpha) + \sum_{k=0}^\infty b_{k\alpha}(\Sigma_1)z^k_\alpha\left(\sqrt{1 - \frac{a_\alpha(\Sigma_1)}{z_\alpha^2}}\right)^kd(z_\alpha^2).
\end{align*}
From which it follows that
\begin{equation*}
    \theta_{\Sigma_0}(-z_\alpha) - \theta_{\Sigma_0}(z_\alpha) = 4z^2_\alpha dz_\alpha - 2\sum_{k=0}^\infty b_{2k+1,\alpha}(\Sigma_1)\left(\sqrt{z_\alpha^2 - a_\alpha(\Sigma_1)}\right)^{2k+1}d(z^2_\alpha) \neq 4z_\alpha^2 dz_\alpha.
\end{equation*}
However, in some exceptional cases such as those we are going to examine in the examples below, we can technically set $\omega_{1,0} := -\theta_{\Sigma_0}$. Suppose that there exists a meromorphic differential form $dS(\Sigma)$ on each $[\Sigma] \in \mathcal{B}_{\Sigma_0}$, varying holomorphically over $\mathcal{B}_{\Sigma_0}$, such that $\nabla_{\mathcal{F}}dS(\Sigma) = \phi_\Sigma$ and $dS(\Sigma)|_{\Sigma\cap U_\alpha} = -(y_\alpha + f_\alpha(x_\alpha))dx_\alpha|_{\Sigma\cap U_\alpha}$ for some holomorphic function $f_\alpha(x_\alpha)$ for all $\alpha \in Ram$. Then we can define the map $\Phi_{\Sigma_0} : \mathcal{B}_{\Sigma_0}\rightarrow G_{\Sigma_0}$ by $\Phi_{\Sigma_0}(\Sigma) := -dS(\Sigma_0) + s^*_{\Sigma_0, \Sigma}dS(\Sigma)$. If $dS(\Sigma)$ has no residues, then it effectively plays the role of $\theta_{\Sigma}$, even though $dS(\Sigma)$ might not be in $G_{\Sigma}$ because it may have poles at non-ramification points on $\Sigma$. In particular, we can set $\pmb{\theta}_{\Sigma_0} := [dS(\Sigma_0)] \in\mathcal{H}_{\Sigma_0}$. Since $(dS(\Sigma_0))(-z_\alpha) - dS(\Sigma_0))(z_\alpha) = 4z_\alpha^2dz_\alpha$ we can apply Theorem \ref{prepotentialvsomega0ntheorem} with $\omega_{1,0} := dS(\Sigma_0)$. An example of this case is the Seiberg-Witten family of curves given in Example \ref{swexample}. If $dS(\Sigma)$ turns out to be holomorphic, then we can define $\theta_{\Sigma} = dS(\Sigma) \in G_{\Sigma}$ and apply Theorem \ref{prepotentialvsomega0ntheorem} with $\omega_{0,1} := -\theta_{\Sigma_0}$. Example \ref{hitchinexample} below provides an example of this scenario.
\end{remark}

\begin{example}\label{hitchinexample}
Let us consider the family of curves arising from the \emph{Hitchin systems} \cite{hitchin1987stable, baraglia2017special}. Let $\mathcal{M}_{n,d}$ be the moduli space of rank $n$ degree $d$ semi-stable \emph{Higgs Bundles} on a compact genus $\bar{g}$ Riemann surface $\mathcal{C}$. Given a holomorphic vector bundle $E$, we define the \emph{degree} by $\deg E := \int_{\mathcal{C}}c_1(E)$. A rank $n$ degree $d$ semi-stable Higgs Bundle is a pair $(E,\Phi)$, where $E$ is a rank $n$ holomorphic vector bundle on $\mathcal{C}$ and $\Phi : E\rightarrow E\otimes \Omega_{\mathcal{C}}^1$ is a bundle map, such that
    \begin{equation*}
        \frac{\deg F}{\rank F} \leq \frac{\deg E}{\rank E}
    \end{equation*}
for every sub-bundle $F \subset E$ such that $\Phi|_F : F \rightarrow F\otimes \Omega^1_{\mathcal{C}}$. We say that $(E,\Phi)$ is stable if the inequality is strict.
Let $\mathcal{B} := \bigoplus_{k=1}^n \Gamma(\mathcal{C},\Omega^{\otimes k}_{\mathcal{C}})$ be the base of the \emph{Hitchin map} $h : \mathcal{M}_{n,d} \rightarrow \mathcal{B}$ given by $h(E,\Phi) = (p_1(\Phi),\cdots,p_n(\Phi))$ where $x^n + p_1(\Phi)x^{n-1} + \cdots + p_n(\Phi) = \det(x\cdot Id - \Phi)$. A fiber of the Hitchin map over $b = (b_1,\cdots,b_n) \in \mathcal{B}$ can be described by a spectral curve $\Sigma_b \subset T^*\mathcal{C}$ given by the zero divisor of the following global section:
\begin{equation*}
    \lambda^n + \pi^*b_1 \lambda^{n-1} + \cdots + \pi^*b_n \in \Gamma(T^*\mathcal{C}, \pi^*(\Omega^1_{\mathcal{C}})^{\otimes n}).
\end{equation*}
Where $\pi : T^*\mathcal{C}\rightarrow \mathcal{C}$ is the projection morphism and $\lambda \in \Gamma(T^*\mathcal{C},\pi^*\Omega^1_{\mathcal{C}})$ denotes the tautological section: $\lambda(p) = p \in (\pi^*\Omega_{\mathcal{C}}^1)_p$. 

It follows that $\mathcal{B}$ is a family of genus $g := 2n^2(\bar{g}-1)+2$ compact Riemann surfaces $\Sigma$ embedded inside the foliated symplectic surface $(S,\Omega_S,\mathcal{F})$ where $S := T^*\mathcal{C}$, $\mathcal{F}$ is given by the cotangent fibers and $\Omega_S := -d\Theta$ where $\Theta \in \Gamma(T^*\mathcal{C},\pi^*\Omega^1_{\mathcal{C}})$ is the tautological $1$-form. Denote by $\mathcal{B}^{reg}$, the regular locus of $\mathcal{B}$ where $\Sigma$ are smooth for all $[\Sigma] \in \mathcal{B}^{reg}$. For each $[\Sigma] \in \mathcal{B}^{reg}$ we define $\theta_{\Sigma} := -\Theta|_{\Sigma} \in \Gamma(\Sigma,\Omega^1_{\Sigma})$. Let $U\subset \mathcal{C}$ be an open subset with local coordinate $q$, then $(T^*U,q,p), T^*U \subset T^*\mathcal{C}$ is a $(\mathcal{F},\Omega_S)$-chart because $\Theta|_{T^*U} = pdq$ means $\Omega_S|_{T^*U} = dq\wedge dp$ and the foliation is given by $p = const$. Moreover, we have $\theta_{\Sigma}|_{\Sigma \cap T^*U} = -pdq|_{\Sigma \cap T^*U}$ which implies $\nabla_{\mathcal{F}}\theta_{\Sigma} = \phi_{\Sigma}$. Let
\begin{equation*}
    z^i := \oint_{A_i}\theta_\Sigma, \qquad w_i := \oint_{B_i}\theta_\Sigma = \frac{\partial\mathfrak{F}_{\Sigma_0}}{\partial z^i}(z^1,\cdots,z^g)
\end{equation*}
where $\mathfrak{F}_{\Sigma_0}$ is the prepotential. In fact, in this case the prepotential is simple enough to be written down explicitly: $\mathfrak{F}_{\Sigma_0}(z^1,\cdots,z^g) = \sum_{i=1}^gz^iw_i = \frac{1}{2}\sum_{i,j=1}^g\tau_{ij}(z^1,\cdots,z^g)z^iz^j$ \cite[Proposition 5.11]{baraglia2017special}. 

Now, consider the topological recursion on $[\Sigma_0] \in \mathcal{B}^{reg}$ with $\omega_{0,1} := -\theta_{\Sigma_0}$ and $\omega_{0,2} = B$ the normalized Bergman kernel. If we picked a $(\mathcal{F},\Omega_S)$-chart $(T^*U_\alpha, p_\alpha,q_\alpha)$ such that $r_\alpha(\Sigma_0) \in T^*U_\alpha$ and $\gamma_\alpha(\Sigma_0) = (q_\alpha = z_\alpha^2, p_\alpha = z_\alpha)$, then $\theta_{\Sigma_0}|_{\Sigma_0\cap T^*U_\alpha} = -p_\alpha dq_\alpha|_{\Sigma\cap T^*U_\alpha} = -z_\alpha d(z_\alpha^2)$ and therefore $\omega_{0,1}(z_\alpha) - \omega_{0,1}(-z_\alpha) = 4z_\alpha^2dz_\alpha$. It follows from Theorem \ref{prepotentialvsomega0ntheorem} that formula (\ref{prepotentialvsomega0nformula}) relating $\mathfrak{F}_{\Sigma_0} = \mathfrak{F}_{\Sigma_0}(z^1,\cdots,z^g)$ and $\omega_{0,n}$ holds in this setting over some open neighbourhood $\mathcal{B}_{\Sigma_0} \subset \mathcal{B}^{reg}$ of $[\Sigma_0]$. The exact same result was obtained earlier by Baraglia-Huang \cite[Theorem 7.4]{baraglia2017special} using a different technique. Note that in comparing our (\ref{prepotentialvsomega0nformula}) with \cite[Theorem 7.4]{baraglia2017special}, we need to replace $z^i$ by $\bar{z}^i := -z^i$ and $\mathfrak{F}_{\Sigma_0} = \mathfrak{F}_{\Sigma_0}(z^1,\cdots,z^g)$ by $\bar{\mathfrak{F}}_{\Sigma_0} = \bar{\mathfrak{F}}_{\Sigma_0}(\bar{z}^1,\cdots,\bar{z}^g) := \mathfrak{F}_{\Sigma_0}(-\bar{z}^1,\cdots,-\bar{z}^g)$ because $\theta_\Sigma$ in \cite{baraglia2017special} was defined to be $+\Theta|_\Sigma$ instead of $-\Theta|_\Sigma$.
\end{example}

\begin{example}\label{swexample}
Let $\mathcal{B}$ be the family of smooth Seiberg-Witten curves (see Section \ref{intro.swsection}):
\begin{equation*}
    \Sigma(u) : \Lambda^{g+1}\left(w + \frac{1}{w}\right) = z^{g+1} + u^gz^{g-1} + ... + u^1 =: P(z;u).
\end{equation*}
It can be shown that each Seiberg-Witten curve $\Sigma(u)$ can be embedded inside a foliated symplectic surface $(S,\Omega_S,\mathcal{F})$ where $S$ is a certain compactification of $\mathbb{C}^2$. 
Let us consider the dense subset $U := (\mathbb{C}\setminus (-\infty,0])\times \mathbb{C} \subset S$ with coordinates $(\log w, z)$ and suppose without loss of generality that $\Sigma_0 = \Sigma(u_0)$ is chosen such that all ramification points $r_\alpha(\Sigma_0)$ are contained in $U$ (we simply pick a different branch-cut for $\log$ otherwise). On $U$, the symplectic form is $\Omega_S|_U = \frac{dw}{w}\wedge dz$ and the foliation is given by $w = const$, hence $(U,\log w, z)$ is a $(\mathcal{F},\Omega_S)$-chart. The Seiberg-Witten differential is given by $dS_{SW} := z\frac{dw}{w}$, or $dS_{SW}(\Sigma(u)) = z\frac{dw}{w}|_{\Sigma(u)}$ if we would like to emphasize that we are considering $dS_{SW}$ as a differential form on $\Sigma(u)$. We define $A, B$-periods and the Seiberg-Witten prepotential $\mathfrak{F}^{SW}_{\Sigma_0} = \mathfrak{F}^{SW}_{\Sigma_0}(a^1,\cdots,a^g)$ by 
\begin{equation*}
    a^i := \oint_{A_i}dS_{SW}(\Sigma(u)), \qquad b_i := \oint_{B_i}dS_{SW}(\Sigma(u)) = \frac{\partial \mathfrak{F}^{SW}_{\Sigma_0}}{\partial a^i}(a^1,\cdots, a^g),
\end{equation*}
over some open neighbourhood $\mathcal{B}_{\Sigma_0}$ of $[\Sigma_0]\in \mathcal{B}$. By definition of the symplectic form $\Omega_S$, we can see that $\nabla_{\mathcal{F}}dS_{SW}(\Sigma(u)) = -\phi_{\Sigma(u)}$ and hence $\pmb{\theta}_{\Sigma(u)} := -[dS_{SW}(\Sigma(u))] \in \mathcal{H}_{\Sigma(u)}$. The negative sign means that we will need to identify $\mathbf{a}^i := -a^i$ and $\mathfrak{F}^{SW}_{\Sigma_0}(a^1,\cdots,a^g) = \mathfrak{F}_{\Sigma_0}(\mathbf{a}^1,\cdots,\mathbf{a}^g)$ after our application of Theorem \ref{prepotentialvsomega0ntheorem} later on. Choose a new $(\mathcal{F},\Omega_S)$-chart $(U_\alpha, x_\alpha, y_\alpha)$ around every ramification points $r_\alpha$ of $\Sigma_0$ such that $\gamma_\alpha(\Sigma_0) = (x_\alpha = z^2_\alpha, y_\alpha = z_\alpha)$. Then $(\log w, z)$ and $(x_\alpha, y_\alpha)$ must be related by the coordinate transformation of the form (\ref{coordtransformationpreservingomegaandfoliation}):
\begin{equation*}
    \log w = F_\alpha(x_\alpha), \qquad z = \frac{y_\alpha}{F'(x_\alpha)} + G'(x_\alpha).
\end{equation*}
Therefore, 
\begin{align*}
    (dS_{SW}(\Sigma_0))(z_\alpha) - (dS_{SW}(\Sigma_0))(-z_\alpha) &= \left(\frac{y_\alpha(z_\alpha)}{F'(x_\alpha(z_\alpha))} + G'(x_\alpha(z_\alpha))\right)F'(x_\alpha(z_\alpha))dx_\alpha(z_\alpha)\\
    &\qquad - \left(\frac{y_\alpha(-z_\alpha)}{F'(x_\alpha(-z_\alpha))} + G'(x_\alpha(-z_\alpha))\right)F'(x_\alpha(-z_\alpha))dx_\alpha(-z_\alpha)\\
    &= y_\alpha(z_\alpha)dx_\alpha(z_\alpha) - y_\alpha(-z_\alpha)dx_\alpha(-z_\alpha) = 4z_\alpha^2dz_\alpha.
\end{align*}
Using Theorem \ref{prepotentialvsomega0ntheorem}, we find that on some $\mathcal{B}_{\Sigma_0}\subseteq \mathcal{B}$ the relation (\ref{prepotentialvsomega0nformula}) holds between the Seiberg-Witten prepotential $\mathfrak{F}^{SW}_{\Sigma_0}$ and $\omega_{0,n}$ produced from topological recursion on $\Sigma_0$ using $\omega_{0,1} := dS_{SW}(\Sigma_0)$ and $\omega_{0,2} = B$, where $B = B(p,q)$ is the normalized Bergman kernel on $\Sigma_0$. In Section \ref{swexamplechapter} we are going to study this example in more detail. In particular, we will construct the needed foliated symplectic surface $(S,\Omega_S,\mathcal{F})$ and explicitly show the calculations leading to (\ref{prepotentialvsomega0nformula}) without referring to Theorem \ref{prepotentialvsomega0ntheorem}.
\end{example}

\section{Seiberg-Witten Prepotential and Topological Recursion}\label{swexamplechapter}

In this section we demonstrate the Kontsevich-Soibelman approach to topological recursion we have studied in Section \ref{airystructurechapter} using an explicit example of the genus $g\geq 1$ Seiberg-Witten family of curves, parameterized by $u = (u_1,\cdots,u_g)$:
\begin{equation*}
    \Lambda^{g+1}\left(w + \frac{1}{w}\right) = z^{g+1} + u_gz^{g-1} + \cdots + u_1 =: P(z;u).
\end{equation*} 
In the end, we obtain the formula (\ref{swprepotentialvsomega0nformula}) relating the Seiberg-Witten prepotential $\mathfrak{F}^{SW}_{\Sigma_0} = \mathfrak{F}^{SW}_{\Sigma_0}(\{a^i\})$ to the genus zero part of topological recursion on a Seiberg-Witten curve with $\omega_{0,1} = dS_{SW}(\Sigma)$ on some open neighbourhood $\mathcal{B}_{\Sigma_0}\subseteq \mathcal{B}$ of $[\Sigma_0] \in \mathcal{B}$.
Note that the energy scale parameter $\Lambda$ plays no role here and we will set $\Lambda = 1$. Alternatively, we can simply absorb $\Lambda$ into $u_k$ and $z$ by re-defining $\frac{u_k}{\Lambda^{g+2-k}}\mapsto u_k$ and $\frac{z}{\Lambda} \mapsto z$. First,  we will show in Section \ref{embeddingofsigmaswinssection} that the family of smooth Seiberg-Witten curves
\begin{equation*}
    \mathcal{B} := \left\{u = (u_1,\cdots,u_g) \in \mathbb{C}^g\ |\ \text{$P^2(z;u) - 4 = 0$ has $2g+2$ distinct roots}\right\}
\end{equation*}
is the moduli space of $\mathcal{F}$-transversal genus $g$ curves in some foliated symplectic surface $(S,\Omega_S,\mathcal{F})$. Once this is done, using Theorem \ref{prepotentialvsomega0ntheorem}, the formula (\ref{swprepotentialvsomega0nformula}) immediately follows. However, the purpose of this section is to demonstrate various calculations in Section \ref{airystructurechapter} in the light of the Seiberg-Witten family of curves. Therefore, in the remainder of this section, we will not refer to Theorem \ref{prepotentialvsomega0ntheorem}, in fact, we will re-prove the Seiberg-Witten case of Theorem \ref{prepotentialvsomega0ntheorem} in Section \ref{swprepotentialandtrsection}. Some other relevant details will be presented in Section \ref{fromswfamilytoairysection}, such as the choice of $\{\epsilon_{\alpha \in Ram}\}, \{M_{\alpha\in Ram}\}, \mathcal{U}_{Ram}, \mathcal{B}_{\Sigma_0}$ and the verification of Proposition \ref{localvsglobalproposition} in the context of Seiberg-Witten family of curves.

\subsection{Embedding of Seiberg-Witten Family of Curves into a Foliated Symplectic Surface}\label{embeddingofsigmaswinssection}

\subsubsection{Outline of the construction of the foliated symplectic surface $(S, \Omega_S, \mathcal{F})$}
To apply the technique we have discussed in Section \ref{airystructurechapter} we need to find a foliated symplectic surface $(S,\Omega_S,\mathcal{F})$, such that any smooth genus $g\geq 1$ Seiberg-Witten curve $\Sigma(u)$, in the family $\mathcal{B}$, can be embedded into $S$ as an $\mathcal{F}$-transversal curve. We require that $S$ contains a dense open subset $U = \mathbb{C}^*\times \mathbb{C}$ where $\mathbb{C}^* := \mathbb{C}\setminus \{0\}$, which can be covered using two $(\mathcal{F},\Omega_S)$-charts $(U_i, \log w, z), i = 1,2$ given by any two different branch-cuts of $\log$, $\Omega_S|_U = \frac{dw}{w}\wedge dz$ and the foliation given by $w = const$. This is so that the section $\theta \in \Gamma(\mathcal{B}_{\Sigma_0}, G)$ introduced in Section \ref{relationtoresconstraintssubsection} will be related to the Seiberg-Witten differential $dS_{SW} = z\frac{dw}{w}$. In other words: if $dS_{SW}$ is treated as a differential form on $U$ then $d(dS_{SW})|_{U} = \Omega_S|_{U}$.

\begin{remark}
We cannot simply set $S = \mathbb{C}^*\times\mathbb{C}$ or $\mathbb{C}^2$ because Seiberg-Witten curves $\Sigma(u)$ are non-compact as curves in $\mathbb{C}^*\times\mathbb{C}$ or $\mathbb{C}^2$.
\end{remark}

There are many ways to construct $(S, \Omega_S, \mathcal{F})$ but let us now describe one of the possible ways. We embed Seiberg-Witten curves $\Sigma(u)$ into $\mathbb{P}^2$ as follows 
\begin{equation*}
    \Sigma(u) := \left\{[w:z:x] \in \mathbb{P}^2\ |\ w^2x^g - w(z^{g+1}+u_gx^2z^{g-1}+ \cdots +u_1x^{g+1}) + x^{g+2} = 0\right\} \subset \mathbb{P}^2.
\end{equation*}
When $g > 1$ the curve $\Sigma(u)$ embeds non-smoothly into $\mathbb{P}^2$ with a singular point of multiplicity $g$ at $[w:z:x] = [1:0:0]$. There also exists no globally defined holomorphic symplectic forms on $\mathbb{P}^2$ because the canonical divisor of $\mathbb{P}^2$ is $K_{\mathbb{P}^2} \sim -3H$, where $H$ is the hyperplane divisor. In other words, any symplectic form on $\mathbb{P}^2$ is meromorphic with pole divisor a degree 3 curve. If $\Sigma(u)$ intersects $K_{\mathbb{P}^2}$ trivially we could still consider $S$ to be an open subset of $\mathbb{P}^2$ containing $\Sigma(u)$ but not $K_{\mathbb{P}^2}$, hence it is symplectic. However, Bézout's theorem guarantees that $\Sigma(u)$ intersects $K_{\mathbb{P}^2}$ at $3(g+2)$ points counting multiplicity. For these reasons, we cannot simply take $S$ to be $\mathbb{P}^2$ or any open subset of $\mathbb{P}^2$. 

In particular, let us consider $\Omega_{\mathbb{P}^2} := \frac{d(w/x)}{w/x}\wedge d\left(\frac{z}{x}\right)$. We define divisors $H_w := \{w = 0\}, H_x := \{x = 0\}$ and $H_z := \{z = 0\}$ on $\mathbb{P}^2$. Then it is clear that $\Omega_{\mathbb{P}^2}$ reduces to $dz\wedge \frac{dw}{w}$ on $\mathbb{C}^2 = \mathbb{P}^2\setminus H_x$ where we can set $x = 1$. We can see that $\Omega_{\mathbb{P}^2}$ has order $1$ pole along the hyperplane $H_w \subset \mathbb{P}^2$ and order $2$ pole along $H_x \subset \mathbb{P}^2$. Hence the canonical divisor is
\begin{equation}\label{canonicaldivisorforsw}
K_{\mathbb{P}^2} = (\Omega_{\mathbb{P}^2}) = -H_w - 2H_x
\end{equation}
and $\Sigma(u)$ intersects $K_{\mathbb{P}^2}$ at $[w:z:x] = [0:1:0]$ and $[w:z:x] = [1:0:0]$.

We are going to show that the symplectic surface $S$ can be obtained by blowing-up $\mathbb{P}^2$ repeatedly starting from points $[0:1:0]$ and $[1:0:0]$ until $\Sigma(u)$ no longer intersects the canonical divisor. 

\begin{figure}[h]
\begin{center}
\begin{tikzpicture}[x=0.75pt,y=0.75pt,yscale=-0.8,xscale=0.8]
\draw    (180,50) -- (520,50) ;
\draw    (500,20) -- (315.83,301) ;
\draw    (190,19) -- (209.61,51.41) -- (360,300) ;

\draw[thick]   (160,40) .. controls (175.3,32.2) and (197.3,28.2) .. (209.61,51.41) .. controls (221.91,74.62) and (274.7,127.8) .. (300,100) .. controls (325.3,72.2) and (393.3,122.2) .. (410,90) .. controls (426.7,57.8) and (464.7,49.8) .. (480,50) .. controls (495.3,50.2) and (513.3,58.2) .. (510,70) ;

\draw (343.5,32) node   {$H_{x}$};
\draw (422,181) node   {$H_{z}$};
\draw (263.5,179) node   {$H_{w}$};
\draw (335.83,222.73) node  [align=left] {};
\draw (342.25,63.25) node   {$2$};
\draw (290,157) node   {$1$};
\draw (397,160) node   {$0$};
\draw (331.5,110) node   {$\Sigma ( u)$};
\end{tikzpicture}
\end{center}
\begin{caption}
\text{$\mathbb{P}^2$ with divisors $H_w := \{w = 0\}, H_x := \{x = 0\}$ and $H_z := \{z = 0\}$. The numbers are order of poles of $\Omega_{\mathbb{P}^2}$ on each respective divisor. We notice that $\Sigma(u)$ intersects the point $[w:z:x] = [1:0:0]$ tangentially to $H_x$ and intersects $[w:z:x] = [0:1:0]$ tangentially to $H_w$.}
\end{caption}
\label{p2beforeblowup}
\end{figure}
 
\subsubsection{Some algebraic geometry background and intuition}
The adjunction formula states that for a non-singular complex curve $C$ embedded in a complex surface $X$ then 
\begin{equation}\label{adjuctionformula}
    \Omega_C^1 \cong \Omega_X^2|_C \otimes \nu_C
\end{equation}
where $\nu_C$ is the holomorphic normal bundle of $C$, $\Omega^1_C$ is the holomorphic canonical line bundle of $C$ and $\Omega^2_X$ is the holomorphic canonical sheaf of $X$. Therefore, if a genus $g$ Seiberg-Witten curve $\Sigma(u)\subset \mathbb{P}^2$ was smooth and intersects $K_{\mathbb{P}^2}$ trivially, then $\Omega^2_{\mathbb{P}^2}|_{\Sigma(u)} \cong \mathcal{O}_{\Sigma(u)}$ and (\ref{adjuctionformula}) implies that $\nu_{\Sigma(u)} \cong \Omega_{\Sigma(u)}^1$.
Instead, $\Sigma(u)$ is singular and intersects non-trivially with $K_{\mathbb{P}^2}$ and we have from Bézout's theorem that
\begin{equation*}
\deg \nu_{\Sigma(u)} = \Sigma(u)\cdot\Sigma(u) = (\deg \Sigma(u))^2 = (g+2)^2 > 2g - 2 = \deg K_{\Sigma(u)}
\end{equation*}
where $N_{\Sigma(u)}$ denotes the normal bundle divisor, and so $\nu_{\Sigma(u)} \ncong \Omega^1_{\Sigma(u)}$.
From Riemann-Roch, the deformation space of $\Sigma(u)$ inside $\mathbb{P}^2$ has dimension
\begin{align*}
\dim \Gamma(\Sigma(u),\nu_{\Sigma(u)}) &= \dim \Gamma(\Sigma(u), \Omega^1_{\Sigma(u)}\otimes \nu^*_{\Sigma(u)}) + \deg \nu_{\Sigma(u)} - g + 1\\
&= 0 + (g+2)^2 - g + 1 = g^2 + 3g + 5 > g
\end{align*}
whereas the Seiberg-Witten family has dimension $g$ = number of parameters $\{u_1,...,u_g\}$. So the problem is that the deformation of $\Sigma(u)$ within the Seiberg-Witten family is more restrictive. In particular, $\Sigma(u)$ will always pass through the point $[w:z:x] = [0:1:0]$ and $[w:z:x] = [1:0:0]$ for every value of $\{u_1,...,u_g\}$, whereas if we are allowed to deform $\Sigma(u)$ freely it will not be fixed at these two points. 

It suggests that the problem can be solved if we construct $S$ by performing a series of blow-ups starting from $\mathbb{P}^2$ until no such fix points exists for the Seiberg-Witten family and each curve $\Sigma(u)$ is non-singular. Conceptually, after each blow-up, the neighbourhood of $\Sigma(u)$ will become more restrictive and the dimension of its deformation space will be reduced. This can be seen as follows. Given a curve $C$ inside a surface $X$ and let $p \in X$. Then blowing up $p$ we have $\pi : \tilde{X}\rightarrow X$ and 
\begin{equation*}
\pi^* C = \tilde{C} + mE, \qquad \pi^* C \cdot E = 0.
\end{equation*}
Where $\tilde{C}$ is the proper transform of $C$, $E$ is the exceptional divisor and $m$ is the multiplicity of the point $p$. In particular, $m = 0$ if $p \notin C$, $m = 1$ if $p\in C$ is a non-singular point and $m > 1$ if $p \in C$ is a singular point of multiplicity $m$. The important thing for us is that
\begin{equation}\label{blowupdeg} 
\deg \nu_{\tilde{C}} = \tilde{C}\cdot \tilde{C} = (\pi^* C - mE)\cdot (\pi^* C - mE) = C\cdot C - m^2 = \deg\nu_{C} - m^2.
\end{equation}
Therefore, we have $\dim \Gamma(\tilde{C}, \nu_{\tilde{C}}) \leq \dim \Gamma(C,\nu_C)$ if $\deg \Omega^1_{\tilde{C}} \leq \deg \nu_{\tilde{C}}$ by the application of Riemann-Roch.

It is also well-known what happens to the canonical class after the blow-up. We have
\begin{equation}\label{blowupcanonical}
    K_{\tilde{X}} = \pi^*K_X + E.
\end{equation}
We will apply (\ref{blowupcanonical}) successively to $K_{\mathbb{P}^2}$ as given in (\ref{canonicaldivisorforsw}) and show that eventually, the canonical class will intersect the proper transform of the Seiberg-Witten curve $\Sigma(u)$ trivially. Then we define $S$ to be $\mathbb{P}^2$, after a series of blow-ups, excluding the canonical divisor. We will also check for consistency using (\ref{blowupdeg}) that the degree of $\nu_{\Sigma(u)}$ drops to $2g-2$. In the end, we conclude that the moduli space $\mathcal{B}$ of $\mathcal{F}$-transversal genus $g$ curves in the resulting foliated symplectic surface $(S,\Omega_S,\mathcal{F})$ coincides with the family of smooth Seiberg-Witten curves.

\begin{figure}[h]
\begin{center}
\begin{tikzpicture}[x=0.75pt,y=0.75pt,yscale=-1,xscale=1]

\draw    (230,38) -- (434,38) ;

\draw    (570,374) -- (290,534) ;

\draw    (90,470) -- (380,520) ;

\draw    (400,20) -- (470,130) ;

\draw    (430,100) -- (570,140) ;

\draw    (460,354) -- (564,409) ;

\draw    (480,254) -- (480,384) ;

\draw    (460,284) -- (580,230) ;

\draw    (370,320) -- (500,320) ;

\draw    (370,170) -- (400,340) ;

\draw    (200,130) -- (270,20) ;

\draw    (130,130) -- (240,100) ;

\draw    (140,270) -- (70,380) ;

\draw    (140,320) -- (80,200) ;

\draw    (80,330) -- (120,500) ;

\draw    (80,421.1) -- (230,350) ;

\draw    (220,190) -- (200,390) ;

\draw[thick]    (190,250) .. controls (200.45,244.48) and (224.27,213.27) .. (250,214.79) .. controls (263.24,215.56) and (275.21,221.24) .. (285.53,226.51) .. controls (295.26,231.48) and (303.53,236.1) .. (310,235.91) .. controls (323.34,235.53) and (340.68,203.91) .. (360,200.7) .. controls (379.32,197.49) and (418.6,206.72) .. (430,200.7) ;

\draw (332.5,19) node   {$H_{x}$};
\draw (458,461) node   {$H_{z}$};
\draw (203.5,509) node   {$H_{w}$};
\draw (523,371) node   {$E_{w,1}$};
\draw (457,71) node   {$E_{w,2}$};
\draw (517,109) node   {$E_{w,3}$};
\draw (543,269) node   {$E_{w,g}$};
\draw (524,319) node   {$E_{w,g+1}$};
\draw (436.5,305) node   {$E_{w,g+2}$};
\draw (413.5,251) node   {$E_{w,g+3}$};
\draw (215,69) node   {$E_{z,1}$};
\draw (165,101) node   {$E_{z,2}$};
\draw (55,421) node   {$E_{z,g+2}$};
\draw (88,309) node   {$E_{z,g+1}$};
\draw (82,251) node   {$E_{z,g}$};
\draw (142,371) node   {$E_{z,g+3}$};
\draw (182,299) node   {$E_{z,g+4}$};
\draw (301.5,210) node   {$\Sigma ( u)$};
\draw (333,50) node   {$2$};
\draw (253,70) node   {$2$};
\draw (173,130) node   {$2$};
\draw (117,240) node   {$2$};
\draw (117,330) node   {$2$};
\draw (113,420) node   {$2$};
\draw (216,477) node   {$1$};
\draw (157,400) node   {$1$};
\draw (223,280) node   {$0$};
\draw (417,440) node   {$0$};
\draw (503,390) node   {$1$};
\draw (473,330) node   {$2$};
\draw (513,240) node   {$2$};
\draw (505,134) node   {$2$};
\draw (425,77) node   {$2$};
\draw (433,330) node   {$1$};
\draw (373,250) node   {$0$};
\draw (102,159) node [rotate=-302.48]  {$...$};
\draw (573.81,182.19) node [rotate=-82.8]  {$...$};
\end{tikzpicture}
\end{center}
\begin{caption}
\text{Summary of the construction of the foliated symplectic surface $(S,\Omega_S,\mathcal{F})$. We notice that $\Omega_S$ is holomorphic on $E_{z,g+4}$ and $E_{w,g+3}$ which $\Sigma(u)$ intersects with transversely.}
\end{caption}
\label{p2afterblowup}
\end{figure}

\subsubsection{Constructing $S$. Part 1: Blowing up $[w:z:x] = [1:0:0]$}
Let us start by considering the blow-up starting from $[w:z:x] = [1:0:0]$.
We consider the open set $U_w := \{w \neq 0\} = \mathbb{A}^2 \subset \mathbb{P}^2$. Set $w = 1$, the local coordinates on $U_w$ are $(z,x)$ and we have
\begin{equation}\label{sigmainw1set}
\Sigma(u)\cap U_w = \left\{(z,x) \in U_w\ |\ x^g - z^{g+1} - u_gx^2z^{g-1} - ... - u_1x^{g+1} + x^{g+2} = 0\right\} \subset U_w.
\end{equation}
Let $\tilde{\mathbb{P}}^2_1\rightarrow \mathbb{P}^2$ be the blow-up of $\mathbb{P}^2$ at $[w:z:x] = [1:0:0]$. Locally on $U_w$ we write this as $\tilde{U}_{w}\rightarrow U_w$ where
\begin{equation*}
    \tilde{U}_{w} = \left\{(z,x,[\tilde{z}:\tilde{x}]) \in U_w\times \mathbb{P}^1\ |\ z\tilde{x} - \tilde{z}x = 0\right\} \subset U_w\times \mathbb{P}^1.
\end{equation*}
Let $\tilde{\Sigma}_1(u)$ be the proper transform of $\Sigma(u)$. Consider the open subset $U_{w,1} := \{\tilde{z}\neq 0\} \subset \tilde{U}_w$, i.e. we set $\tilde{z} = 1$, so $x = z\tilde{x}$ and we can take $(z_1,x_1) := (z,\tilde{x})$ to be the local coordinates on $U_{w,1}$. Then
\begin{equation*}
\tilde{\Sigma}_1(u)\cap U_{w,1} = \left\{(z_1,x_1)\in U_{w,1}\ |\ x^g_1 - z_1 - u_gx^2_1z_1 - ... - u_1x^{g+1}_1z_1 + x^{g+2}_1z^2_1 = 0\right\} \subset U_{w,1}.
\end{equation*}
We can see that for any choices of $u$ the curve $\tilde{\Sigma}_1(u)$ always passes through the point $(z_1,x_1) = (0,0)$. We can check that there is no such basepoint on the other open set $V_{w,1} := \{\tilde{x} \neq 0\} \subset \tilde{U}_{w}$. This confirms that $\Sigma(u)$ intersects $[w:z:x] = [1:0:0]$ tangentially to $H_x$ as illustrated by Figure \ref{p2beforeblowup}. To continue we need to repeat the process by blowing up again at $(z_1,x_1) = (0,0)$. Let us remark that the blow-up has removed the singular point and $\tilde{\Sigma}_1(u)$ is a non-singular curve. Let $E_{w,1} := \{z = 0, x = 0\} \subset \tilde{U}_{w}$ be the exceptional divisor. Using (\ref{blowupcanonical}) we have found that
\begin{equation*}
    K_{\tilde{\mathbb{P}}^2_1} = \pi^*K_{\mathbb{P}^2} + E_{w,1} = \pi^*(-H_w - 2H_x) + E_{w,1} = -H_w - 2(H_x + E_{w,1}) + E_{w,1} = -H_w - 2H_x - E_{w,1},
\end{equation*}
therefore $\tilde{\Sigma}_1(u)$ still intersects $K_{\tilde{\mathbb{P}}^2_1}$ non-trivially.
Moreover, $\deg N_{\tilde{\Sigma}_1(u)} = \deg N_{\Sigma(u)} - g^2$. Since the remaining blow-ups will occur on non-singular points, the degree of normal bundle will always go down by $1$ for each blow-up.

Let $\tilde{\mathbb{P}}^2_2\rightarrow \tilde{\mathbb{P}}^2_1$ be the blow-up of $\tilde{\mathbb{P}}^2_1$ at $(z_1,x_1) = (0,0) \in U_{w,1} \subset \tilde{\mathbb{P}}^2_1$. Locally on $U_{w,1}$ we write this as $\tilde{U}_{w,1}\rightarrow U_{w,1}$ where
\begin{equation*}
    \tilde{U}_{w,1} = \left\{(z_1,x_1,[\tilde{z}_1,\tilde{x}_1]) \in U_{w,1}\times \mathbb{P}^1\ |\ z_1\tilde{x}_1 - \tilde{z}_1x_1 = 0\right\} \subset U_{w,1}\times \mathbb{P}^1.
\end{equation*}
Let $\tilde{\Sigma}_2(u)$ be the proper transform of $\tilde{\Sigma}_1(u)$. Consider the open subset $U_{w,2} := \{\tilde{x}_1 \neq 0\} \subset \tilde{U}_{w,1}$ with coordinates $(z_2,x_2) := (\tilde{z}_1,x_1)$. For any choice of $u$ the curve $\tilde{\Sigma}_2(u)$ always pass the point $(z_2,x_2) = (0,0)$. We can check that there is no such fixed point on the other open set $V_{w,1} := \{\tilde{z}_1 \neq 0\}$. We repeat this process $g$ times. In more detail, for $i = 1,...,g$, we let $\tilde{\mathbb{P}}^2_{i+1}\rightarrow \tilde{\mathbb{P}}^2_{i}$ be the blow-up of $\tilde{\mathbb{P}}^2_{i}$ at $(z_{i},x_{i}) = (0,0) \in U_{w,i}\subset \tilde{\mathbb{P}}^2_{i}$. Locally on $U_{w,i}$ we write this as $\tilde{U}_{w,i} \rightarrow U_{w,i}$. Let $\tilde{\Sigma}_{i+1}(u)$ be the proper transform of $\tilde{\Sigma}_{i}(u)$. Set $U_{w,i+1} := \{\tilde{x}_{i} \neq 0\} \subset \tilde{U}_{w,i}$ with coordinates $(z_{i+1},x_{i+1}) := (\tilde{z}_{i},x_{i})$. We find that
\begin{equation*}
    \tilde{\Sigma}_{g+1}(u)\cap U_{w,g+1} = \left\{(z_{g+1},x_{g+1})\in U_{w,{g+1}}\ |\ 1 - z_{g+1} - u_gx^2_{g+1}z_{g+1} - ... - u_1x^{g+1}_{g+1}z_{g+1} + x^{2g+2}_{g+1}z^2_{g+1} = 0\right\}.
\end{equation*}
On the other hand, repeated application of (\ref{blowupcanonical}) gives
\begin{equation*}
    K_{\tilde{\mathbb{P}}^2_{g+1}} = -H_w - 2H_x - E_{w,1} - 2\sum_{i=1}^gE_{w,i+1}
\end{equation*}
where $E_{w,i+1} := \{z_i = 0, x_i = 0\}\subset \tilde{U}_{w,i}$ are exceptional divisors.
This time we observe that $\tilde{\Sigma}_{g+1}(u)$ passes through the point $(z_{g+1},x_{g+1}) = (1,0)$ for any choice of $u$. Let us define the new coordinates by $(z'_{g+1}, x'_{g+1}) := (z_{g+1} - 1, x_{g+1})$. We continue by letting $\tilde{\mathbb{P}}^2_{g+2} \rightarrow \tilde{\mathbb{P}}^2_{g+1}$ to be the blow-up at $(z'_{g+1}, x'_{g+1}) = (0,0)$. Locally, we write this as $\tilde{U}_{w,g+1} \rightarrow U_{w,g+1}$ where
\begin{equation*}
    \tilde{U}_{w,g+1} = \left\{(z'_{g+1},x'_{g+1},[\tilde{z}'_{g+1},\tilde{x}'_{g+1}]) \in U_{w,g+1}\times \mathbb{P}^1\ |\ z'_{g+1}\tilde{x}'_{g+1} - \tilde{z}'_{g+1}x'_{g+1} = 0\right\} \subset U_{w,g+1}\times \mathbb{P}^1.
\end{equation*}
Let $\tilde{\Sigma}_{g+2}(u)$ be the proper transform of $\tilde{\Sigma}_{g+1}(u)$ and let $E_{w,g+2} := \{z'_{g+1} = 0, x'_{g+1} = 0\} \subset \tilde{U}_{w,g+1}$ be the exceptional divisor.
Consider the open set $U_{w,g+2} := \{\tilde{x}'_{g+1} \neq 0\}\subset \tilde{U}_{w,g+1}$ and let the local coordinates be $(z_{g+2},x_{g+2}) := (\tilde{z}'_{g+1}, x'_{g+1})$. For any choices of $u$, the curve $\tilde{\Sigma}_{g+2}(u)$ always passes through the point $(z_{g+2},x_{g+2}) = (0,0)$. Blow up at the point $(z_{g+2},x_{g+2}) = (0,0)$ to get $\tilde{\mathbb{P}}^2_{g+3} \rightarrow \tilde{\mathbb{P}}^2_{g+2}$. Let $\tilde{\Sigma}_{g+3}(u)$ be the proper transform of $\tilde{\Sigma}_{g+2}(u)$. We can check that we no longer have any fixed points. Consider the open set $U_{w,g+3} := \{\tilde{x}_{g+2} \neq 0\} \subset \tilde{U}_{w,g+2}$ and let $(z_{g+3}, x_{g+3}) := (\tilde{z}_{g+2}, x_{g+2})$ be the local coordinates. Then
\begin{align*}
    \tilde{\Sigma}_{g+3}(u)\cap U_{w,g+3} &= \Big\{(z_{g+3},x_{g+3}) \in U_{w, g+3}\ |\\
    &\qquad-z_{g+3} - (u_g + u_{g-1}x_{g+3} + ... + u_1x^{g-1}_{g+3})(x^2_{g+3}z_{g+3} + 1) + x^{2g}_{g+3}(x^2_{g+3}z_{g+3}+1)^2 = 0\Big\}
\end{align*}
intersects the exceptional divisor $E_{w,g+3} := \{z_{g+2} = 0, x_{g+2} = 0\} \subset \tilde{U}_{w,g+2}$ at $(z_{g+3},x_{g+3}) = (-u_g,0)$ which varies with $u$. Using (\ref{blowupcanonical}) we find that the canonical divisor of $\tilde{\mathbb{P}}^2_{g+3}$ is 
\begin{equation*}
    K_{\tilde{\mathbb{P}}^2_{g+3}} = -H_w - 2H_x - E_{w,1} - 2\sum_{i=2}^{g+1}E_{w,i} - E_{w,g+2}.
\end{equation*}
Let $U_z := \{z\neq 0\} \subset \tilde{\mathbb{P}}^2_{g+3}$ be an open set with coordinates $(w,x)$. We notice that $\tilde{\Sigma}_{g+3}(u) \subset \tilde{\mathbb{P}}^2_{g+3}$ only intersects $K_{\tilde{\mathbb{P}}^2_{g+3}}$ on $U_z$ at $(w,x) = 0$. This is the point $[w:z:x] = [0:1:0] \in \mathbb{P}^2$ if we blow-down $\tilde{\mathbb{P}}^2_{g+3}$ to $\mathbb{P}^2$. So the next step is to perform another series of blow-ups starting from $(w,x) = (0,0)$. 

\subsubsection{Constructing $S$. Part 2: Blowing up $[w:z:x] = [0:1:0]$}
On the open set $U_z = \{z \neq 0\} \subset \tilde{\mathbb{P}}^2_{g+3}$ we set $z = 1$ and it follows that
\begin{equation}
\tilde{\Sigma}_{g+3}(u)\cap U_z = \left\{(w,x) \in U_z\ |\ w^2x^g - w(1 + u_gx^2 + ... + u_1x^{g+1}) + x^{g+2} = 0\right\}.
\end{equation}
Let $\tilde{\mathbb{P}}^2_{g+3;1}\rightarrow \tilde{\mathbb{P}}^2_{g+3}$ be the blow-up of $\tilde{\mathbb{P}}^2_{g+3}$ at $(w,x)=(0,0) \in U_z$. Locally on $U_z$ we write this as $\tilde{U}_z\rightarrow U_z$ where
\begin{equation*}
    \tilde{U}_z = \left\{(w,x,[\tilde{w},\tilde{x}]) \in U_z\times \mathbb{P}^1\ |\ w\tilde{x} - \tilde{w}x = 0\right\}.
\end{equation*}
Let $\tilde{\Sigma}_{g+3;1}(u)$ be the proper transform of $\tilde{\Sigma}_{g+3}(u)$ and $E_{z,1} := \{w = 0, x = 0\} \subset \tilde{U}_z$ be the exceptional divisor. Consider the open subset $U_{z,1} := \{\tilde{x} \neq 0\} \subset \tilde{U}_z$ with coordinates $(w_1,x_1) := (\tilde{w},x)$. For any choice of $u$, the curve $\tilde{\Sigma}_{g+3;1}(u)$ passes through the point $(w_1,x_1) = (0,0)$. This confirms that $\Sigma(u)$ intersects $[w:z:x] = [0:1:0]$ tangentially to $H_x$ as illustrated by Figure \ref{p2beforeblowup}. We repeat this process $g+2$ times. In more detail, for $i = 1,...,g+1$, we let $\tilde{\mathbb{P}}^2_{g+3;i+1} \rightarrow \tilde{\mathbb{P}}^2_{g+3;i}$ be the blow-up at $(w_i,x_i) = (0,0) \in U_{z,i} \subset \tilde{\mathbb{P}}^2_{g+3;i}$. Locally on $U_{z,i}$ we write this as $\tilde{U}_{z,i} \rightarrow U_{z,i}$. Let $\tilde{\Sigma}_{g+3;i+1}(u)$ be the proper transform of $\tilde{\Sigma}_{g+3;i}(u)$ and let $E_{z,i+1} := \{w_i=0,x_i=0\} \subset \tilde{U}_{z,i}$. Set $U_{z,i+1} := \{\tilde{x}_i \neq 0\} \subset \tilde{U}_{z,i}$ with coordinates $(w_{i+1},x_{i+1}) := (\tilde{w}_i,x_i)$. Then we find that
\begin{equation*}
    \tilde{\Sigma}_{g+3;g+2}(u)\cap U_{z,g+2}=\left\{(w_{g+2},x_{g+2})\in U_{z,g+2}\ |\ w^2_{g+2}x^{2g+2}_{g+2} - w_{g+2}(1 + u_gx^2_{g+2} + ... + u_1x^{g+1}_{g+2}) + 1 = 0\right\}.
\end{equation*}
Therefore $\tilde{\Sigma}_{g+3;g+2}(u)$ passes through the point $(w_{g+2},x_{g+2}) = (1,0)$ for all $u$.  We define new coordinates $(w'_{g+2}, x'_{g+2}) := (w_{g+2} - 1, x_{g+2})$ and continue by letting $\tilde{\mathbb{P}}^2_{g+3;g+3}\rightarrow \tilde{\mathbb{P}}^2_{g+3;g+2}$ be the blow-up at $(w'_{g+2}, x'_{g+2}) = (0,0)$. 
Locally, we write this as $\tilde{U}_{z,g+2} \rightarrow U_{z,g+2}$ where
\begin{equation*}
    \tilde{U}_{z,g+2} = \left\{(w'_{g+2},x'_{g+2},[\tilde{w}'_{g+2},\tilde{x}'_{g+2}]) \in U_{z,g+2}\times \mathbb{P}^1\ |\ w'_{g+2}\tilde{x}'_{g+2} - \tilde{w}'_{g+2}x'_{g+2} = 0\right\} \subset U_{z,g+2}\times \mathbb{P}^1.
\end{equation*}
Let $\tilde{\Sigma}_{g+3;g+3}(u)$ be the proper transform of $\tilde{\Sigma}_{g+3;g+2}(u)$ and let $E_{z,g+3} := \{w'_{g+2} = 0, x'_{g+2} = 0\} \subset \tilde{U}_{z,g+2}$ be the exceptional divisor. Consider the open set $U_{z, g+3} := \{\tilde{x}_{g+2}'\neq 0\} \subset \tilde{U}_{z,g+2}$ and let the local coordinates be $(w_{g+3},x_{g+3}) := (\tilde{w}'_{g+2},x'_{g+2})$. For any choice of $u$, the curve $\tilde{\Sigma}_{g+3;g+3}(u)$ always passes through the point $(w_{g+3}, x_{g+3}) = (0,0)$. Blow up at the point $(w_{g+3}, x_{g+3}) = (0,0)$ to get $\tilde{\mathbb{P}}^2_{g+3;g+4}\rightarrow \tilde{\mathbb{P}}^2_{g+3;g+3}$. Let $\tilde{\Sigma}_{g+3;g+4}(u)$ be the proper transform of $\tilde{\Sigma}_{g+3;g+3}(u)$. We can check that we no longer have any basepoint. Consider the open set $U_{z,g+4} := \{\tilde{x}_{g+3}\neq 0\} \subset \tilde{U}_{z,g+3}$ and let $(w_{g+4}, x_{g+4}) := (\tilde{w}_{g+3}, x_{g+3})$ be the local coordinates. Then
\begin{align*}
    \tilde{\Sigma}_{g+3;g+4}(u)&\cap U_{z,g+4} = \Big\{(w_{g+4},x_{g+4}) \in U_{z, g+4}\ |\\
    &(x^2_{g+4}w_{g+4} + 1)^2x^{2g}_{g+4} - w_{g+4} - w_{g+4}(u_gx^2_{g+4} + ... + u_1x^{g+1}_{g+4}) - (u_g + u_{g-1}x_{g+4}... + u_1x^{g-1}_{g+4}) = 0\Big\}
\end{align*}
intersects the exceptional divisor $E_{z,g+4} := \{w_{g+3} = 0, x_{g+3} = 0\} \subset \tilde{U}_{z,g+3}$ at $(w_{g+4},x_{g+4}) = (-u_g,0)$ which varies with $u$.

Finally, using (\ref{blowupcanonical}) we find that the canonical divisor of $\tilde{\mathbb{P}}^2_{g+3;g+4}$ is
\begin{equation*}
    K_{\tilde{\mathbb{P}}^2_{g+3;g+4}} = -H_w - 2H_x - E_{w,1} - 2\sum_{i=2}^{g+1}E_{w,i} - E_{w,g+2} - 2\sum_{i=1}^{g+2}E_{z,i} - E_{z,g+3}.
\end{equation*}
Evidently, the curve $\tilde{\Sigma}_{g+3;g+4}(u)$ intersects $K_{\tilde{\mathbb{P}}^2_{g+3;g+4}}$ trivally and therefore, we have $\nu_{\tilde{\Sigma}_{g+3;g+4}} \cong \Omega^1_{\tilde{\Sigma}_{g+3;g+4}}$ by (\ref{adjuctionformula}). In particular, the moduli space now has the correct dimension 
\begin{equation*}
    \dim \Gamma(\tilde{\Sigma}_{g+3;g+4}(u), \nu_{\tilde{\Sigma}_{g+3;g+4}(u)}) = \dim \Gamma(\tilde{\Sigma}_{g+3;g+4}(u), \Omega^1_{\tilde{\Sigma}_{g+3;g+4}(u)}) = g.
\end{equation*}
We can check the consistency by calculating the degree of $\nu_{\tilde{\Sigma}_{g+3;g+4}(u)}$ using (\ref{blowupdeg}):
\begin{equation*}
    \deg \nu_{\tilde{\Sigma}_{g+3;g+4}(u)} = \deg \nu_{\Sigma(u)} - g^2 - (g+2) - (g+4) = (g+2)^2 - (g^2 + 2g + 6) = 2g-2.
\end{equation*}
\subsubsection{Summary}
The Figure \ref{p2afterblowup} assists the visualization of this construction. We define $S := \tilde{\mathbb{P}}^2_{g+3;g+4}\setminus K_{\tilde{\mathbb{P}}^2_{g+3;g+4}}$. We take $\tilde{\Sigma}_{g+3;g+4}(u)$ to be the embedding of the Seiberg-Witten curve $\Sigma(u)$ in $S$ which we will simply denote by $\Sigma(u) \subset S$ from now on. It is clear that $\mathbb{C}^*\times \mathbb{C}$ is an open dense subset of $S$.
The symplectic form is $\Omega_S := \frac{d(w/x)}{w/x}\wedge d\left(\frac{z}{x}\right)$ where we have extended the domain of functions $\frac{z}{x}$ and $\frac{w}{x}$ from $\mathbb{P}^2$ to $S \subset \tilde{\mathbb{P}}^2_{g+3;g+4}$ in the obvious way. The foliation $\mathcal{F}$ is given by $\frac{w}{x} = const$. On the open set $U_{w,g+3}$, the foliation $\mathcal{F}$ is given in local coordinates $(z_{g+3},x_{g+3})$ by 
\begin{equation*}
    (z_{g+3}x^2_{g+3} + 1)x^{g+1}_{g+3} = const.
\end{equation*}
We find that $E_{w, g+3}\cap U_{w,g+3} = \{x_{g+3} = 0\}$ is a leaf of the foliation with $const = 0$ and so $\Sigma(u)$ is transverse to the foliation at the intersection $\Sigma(u) \cap E_{w,g+3} = \{(z_{g+3}, x_{g+3}) = (-u_g,0)\}$. Similarly, on the open set $U_{z,g+4}$ for foliation $\mathcal{F}$ is given in local coordinates $(w_{g+4}, x_{g+4})$ by 
\begin{equation*}
    (w_{g+4}x^2_{g+4} + 1)x^{g+1}_{g+4} = const.
\end{equation*}
We find that $E_{z,g+4}\cap U_{z,g+4} = \{x_{g+4} = 0\}$ is a leaf of the foliation with $const = 0$ and so $\Sigma(u)$ is transverse to the foliation at the intersection $\Sigma(u)\cap E_{z, g+4} = \{(w_{g+4},x_{g+4} = (-u_g,0)\}$.

It follows that all ramification points of $\Sigma(u) \subset S$ are contained inside the dense subset $\mathbb{C}^*\times \mathbb{C} \subset S$.

\subsection{From the Seiberg-Witten Family of Curves to Airy Structures}\label{fromswfamilytoairysection}
Let us consider the genus $g$ family $\mathcal{B}$ of smooth Seiberg-Witten curves $\Sigma(u)$ embedded in $(S,\Omega_S,\mathcal{F})$ as constructed in the last section. Since all ramification points of $\Sigma(u)$ are contained inside $\mathbb{C}^*\times \mathbb{C} \subset S \rightarrow \mathbb{P}^2$, for all practical purposes we can work exclusively in this subset and set $x = 1$. We now have $\Omega_S = \frac{dw}{w}\wedge dz$ and $\Sigma(u)$ is defined by the algebraic equation $w + \frac{1}{w} = P(z;u)$. Throughout this section, we will denote by  
\begin{equation*}
z_i(u),\qquad i = 1,...,g    
\end{equation*}
the roots of $P'(z;u) = 0$ and we will fix the reference point to be $[\Sigma_0] := [\Sigma(u_0)] \in \mathcal{B}$. 

We begin in Section \ref{swrampointssubsection} by locating the ramification points $\{r_{\alpha \in Ram}\}$ of $\Sigma(u)$ and write down the local parameterization of $\Sigma(u)$ in some neighbourhood of each $r_\alpha$. Recall that the image $\gamma(\Sigma_0)$ of $\Sigma_0$ in $Discs^{Ram}$ needs to take the form $(x_\alpha = z_\alpha^2, y_\alpha = z_\alpha), \alpha \in Ram$ for the image of the embedding $\Phi_{\Sigma_0} : \mathcal{B}_{\Sigma_0} \rightarrow G_{\Sigma_0}$ to satisfies the $Ram$ product residue constraints (see Proposition \ref{localvsglobalproposition}). Therefore, in Section \ref{swchoosinglocalcoordsubsection}, we explicitly find the $(\mathcal{F},\Omega_S)$-local coordinates transformation which brings $\gamma(\Sigma_0)$ into the needed form. This gives us a collection of $(\mathcal{F},\Omega_S)$-charts $\mathcal{U}_{Ram}$. Then in Section \ref{swchoosingepsilonmubsubsection}, we will also explain how the parameters $\{\epsilon_{\alpha \in Ram}\}$, $\{M_{\alpha \in Ram}\}$ and an open neighbourhood $\mathcal{B}_{\Sigma_0}$ of $[\Sigma_0]\in\mathcal{B}$ can be chosen. Now that $\{\epsilon_{\alpha \in Ram}\}, \{M_{\alpha\in Ram}\}, \mathcal{U}_{Ram}$ and $\mathcal{B}_{\Sigma_0}$ has been selected, we are going to write down the embedding $\gamma : \mathcal{B}_{\Sigma_0} \hookrightarrow Discs^{Ram}_{t(\Sigma_0)}$ follows by $Discs^{Ram}_{t(\Sigma_0)}\hookrightarrow W^{Ram}_{t(\Sigma_0)}$ as given in Section \ref{embeddingofdiscssubsection} and the embedding $\mathcal{B}_{\Sigma_0}\hookrightarrow G_{\Sigma_0}$ in Section \ref{swtheembeddingbtowsubsection} and Section \ref{swtheembeddingbtogsubsection} respectively. These two embeddings will be compared, giving a proof of Proposition \ref{localvsglobalproposition} in the specific example of the Seiberg-Witten family of curves.

\subsubsection{Ramification points and local parameterization of Seiberg-Witten curves}\label{swrampointssubsection}
Ramification points of $\Sigma_0$ are 
\begin{equation}
    Ram = \left\{r_{i\pm} = \left(w = \frac{-P\left(z_i(u_0);u_0\right) \pm \sqrt{P\left(z_i(u_0);u_0\right)^2 - 4}}{2}, z = z_i(u_0)\right)\ \Big|\ i = 1,...,g\right\}.
\end{equation}
Let us choose the local coordinates $(w_{i\pm}, z_{i\pm})$ of $S$ around $r_\pm$ to be 
\begin{equation}\label{coordtransformationwztounbar}
    (w_{i\pm},z_{i\pm}) = \left(w + \frac{1}{w}, \frac{z}{w - \frac{1}{w}}\right).
\end{equation}
We have $\Omega_S = dw_{i\pm}\wedge dz_{i\pm}$ and the foliation is $w_{i\pm} = const$. Of course, this is not a unique choice, alternatively we can also use $(\log w, z)$ but it will be harder to write down the standard local coordinates. Given a curve $[\Sigma(u)] \in \mathcal{B}_{\Sigma_0}$, we find a standard local coordinates $\eta^2_{i\pm} := P(z;u) - P(z_i(u);u) = \frac{1}{2}P''(z_i(u);u)(z - z_i(u))^2 + ...$ of $\Sigma(u)$ around $r_{i\pm}(\Sigma(u))$. The parameterization of $\Sigma(u)$ is given in the same $(w_{i\pm},z_{i\pm})$ coordinates of $S$ by 
\begin{equation}\label{discatsigmau}
    t_{i\pm}(\Sigma(u)) = \left(w_{i\pm} = \eta_{i\pm}^2 + P(z_i(u);u), z_{i\pm} = \pm\frac{z(\eta_{i\pm};u)}{\sqrt{(\eta_{i\pm}^2 + P(z_i(u);u))^2 - 4}}\right).
\end{equation}
Where $z(\eta_{i\pm};u) = z_i(u) + \left(\frac{2}{P''(z_i(u);u)}\right)^{\frac{1}{2}}\eta_{i\pm} + ...$ is the inverse of $\eta_{i\pm}^2 = P(z;u) - P(z_i(u);u)$ for $z \approx z_i(u)$ or $\eta_{i\pm} \approx 0$. When $u = u_0$ the parameterization of $\Sigma_0$ near $r_{i\pm}(\Sigma_0)$ using the standard local coordinates $\eta_{i\pm}$ is given in the $(w_{i\pm},z_{i\pm})$ coordinates by
\begin{equation}\label{discatsigma0}
    t_{i\pm}(\Sigma_0) = \left(w_{i\pm} = \eta_{i\pm}^2 + P(z_i(u_0);u_0), z_{i\pm} = \pm\frac{z(\eta_{i\pm};u_0)}{\sqrt{(\eta_{i\pm}^2 + P(z_i(u_0);u_0))^2 - 4}}\right).
\end{equation}
The problem is that in $(w_{i\pm},z_{i\pm})$ local coordinates $t_{i\pm}(\Sigma_0)$ does not equal to $(w_{i\pm} = \eta^2_{i\pm}, z_{i\pm} = \eta_{i\pm})$. 

\subsubsection{Choosing the $(\mathcal{F}, \Omega_S)$-local coordinates}\label{swchoosinglocalcoordsubsection}
Lemma \ref{picoveringmaplemma} tells us that we can transform into a new local coordinates 
\begin{equation}\label{coordtransformationunbartobar}
    (\bar{w}_{i\pm}, \bar{z}_{i\pm}) = \left(F_{i\pm}(w_{i\pm}), \frac{z_{i\pm}}{F'_{i\pm}(w_{i\pm})} - G'_{i\pm}(w_{i\pm})\right)
\end{equation}
of $S$ around $r_{i\pm}(\Sigma_0)$ such that $\Omega_S = \bar{w}_{i\pm}\wedge \bar{z}_{i\pm}$ and the foliation is $\bar{w}_{i\pm} = const$ and $t_{i\pm}(\Sigma_0) = (\bar{w}_{i\pm} = \bar{\eta}_{i\pm}^2, \bar{z}_{i\pm} = \bar{\eta}_{i\pm})$ where the new standard local coordinates $\bar{\eta}_{i\pm}$ is given by $\bar{\eta}_{i\pm}^2 := F_{i\pm}(\eta_{i\pm}^2 + F^{-1}_{i\pm}(0))$. Comparing (\ref{discafterchangecoordinates}) with (\ref{discatsigma0}) determines what function $F_{i\pm}$ and $G_{i\pm}$ have to be:
\begin{align}
F^{-1}_{i\pm}(0) &= P(z_i(u_0);u_0)\label{findfgeqn1}\\
F'_{i\pm}\left(\eta^2_{i\pm} + F^{-1}_{i\pm}(0)\right)\left(\left(F_{i\pm}(\eta^2_{i\pm} + F^{-1}_{i\pm}(0)\right)^{1/2} + G'_{i\pm}\left(\eta^2_{i\pm} + F^{-1}_{i\pm}(0)\right)\right) &= \pm\frac{z(\eta_{i\pm};u_0)}{\sqrt{(\eta_{i\pm}^2 + P(z_i(u_0);u_0))^2 - 4}}\label{findfgeqn2}.
\end{align}
We separate $z(\eta_{i\pm};u_0)$ into the sum of its odd and even components $z(\eta_{i\pm};u_0) = z_{odd}(\eta_{i\pm};u_0) + z_{even}(\eta_{i\pm};u_0)$ where
\begin{equation*}
    z_{odd}(\eta_{i\pm};u_0) := \frac{z(+\eta_{i\pm};u_0) - z(-\eta_{i\pm};u_0)}{2}, \qquad z_{even}(\eta_{i\pm};u_0) := \frac{z(+\eta_{i\pm};u_0) + z(-\eta_{i\pm};u_0)}{2}.
\end{equation*}
Comparing the odd components on both sides of (\ref{findfgeqn2}), we have
\begin{align*}
    \frac{1}{2\eta_{i\pm}}\left(F_{i\pm}\left(\eta^2_{i\pm} + F^{-1}_{i\pm}(0)\right)\right)^{1/2}\frac{d}{d\eta_{i\pm}}F_{i\pm}\left(\eta^2_{i\pm} + F^{-1}_{i\pm}(0)\right) &= \pm\frac{z_{odd}(\eta_{i\pm};u_0)}{\sqrt{(\eta^2_{i\pm} + P(z_i(u_0);u_0))^2 - 4}}\\
    \implies \frac{1}{3}\frac{d}{d\eta_{i\pm}}\left(F_{i\pm}(\eta^2_{i\pm} + F^{-1}_{i\pm}(0))\right)^{3/2} &= \pm\frac{\eta_{i\pm}z_{odd}(\eta_{i\pm};u_0)}{\sqrt{(\eta^2_{i\pm} + P(z_i(u_0);u_0))^2 - 4}}.
\end{align*}
Integrating both sides gives us
\begin{equation}\label{fsoln}
    \left(F_{i\pm}(\eta^2_{i\pm} + F^{-1}_{i\pm}(0))\right)^{3/2} = \pm 3 \int_{0}^{\eta_{i\pm}}\frac{\tau z_{odd}(\tau;u_0)}{\sqrt{(\tau^2 + P(z_i(u_0);u_0))^2 - 4}}d\tau.
\end{equation}
We note that the right-hand-side of (\ref{fsoln}) takes the form $\sim\eta_{i\pm}^3(1 + O(\eta^2_{i\pm}))$ because the integrand is a function of $\tau^2$ without a constant term. It follows after raising (\ref{fsoln}) to the power of $2/3$ that $F_{i\pm}(\eta^2_{i\pm} + F^{-1}(0)) \sim \eta^2_{i\pm}(1 + O(\eta^2_{i\pm}))$. For each ramification point $i\pm \in Ram$, there are $3$ choices of $F_{i\pm}$ corresponding to each third root of $(\pm 1)$. For simplicity, let us choose a third root so that $(\pm 1)^{2/3} = +1$. Combining (\ref{fsoln}), (\ref{findfgeqn1}) we have
\begin{equation}\label{fsoln2}
    F_{i\pm}(w_{i\pm}) = \left(3 \int_{0}^{\sqrt{w_{i\pm} - P(z_i(u_0);u_0)}}\frac{\tau z_{odd}(\tau;u_0)}{\sqrt{(\tau^2 + P(z_i(u_0);u_0))^2 - 4}}d\tau\right)^{2/3}.
\end{equation}
With our choice of the third-root we have $F_{i+} = F_{i-}$, similar to how we have $w_{i+} = w + \frac{1}{w} = w_{i-}$. Using the expansion $z(\eta_{i\pm};u) = z_i(u) + \left(\frac{2}{P''(z_i(u);u)}\right)^{\frac{1}{2}} + ...$, the expansion of $F_{i\pm}(w_{i\pm})$ around the ramification point $r_{i\pm}$ is given by 
\begin{equation}\label{fexpansion}
    F_{i\pm}(w_{i\pm}) = (w_{i\pm} - P(z_i(u_0);u_0))\left(\left(\frac{2}{P''(z_i(u_0);u_0)}\right)^{\frac{1}{2}} + O\left((w_{i\pm} - P(z_i(u_0);u_0))^2\right)\right)^{\frac{2}{3}}
\end{equation}
Next, we compare the even components on both sides of (\ref{findfgeqn2}). We have
\begin{equation*}
    F'_{i\pm}\left(\eta^2_{i\pm} + F^{-1}_{i\pm}(0)\right) G'_{i\pm}\left(\eta^2_{i\pm} + F^{-1}_{i\pm}(0)\right) = \pm\frac{z_{even}(\eta_{i\pm};u_0)}{\sqrt{(\eta^2_{i\pm} + P(z_i(u_0);u_0))^2 - 4}}
\end{equation*}
Substituting the expression for $F_{i\pm}$ from (\ref{fsoln2}) we have
\begin{equation}\label{gsoln}
    G'_{i\pm}\left(w_{i\pm}\right) = \frac{z_{even}(\sqrt{w_{i\pm} - P(z_i(u_0);u_0)};u_0)}{z_{odd}(\sqrt{w_{i\pm} - P(z_i(u_0);u_0)};u_0)}\left(F_{i\pm}(w_{i\pm})\right)^{1/2}.
\end{equation}
Finally, using (\ref{coordtransformationwztounbar}) and (\ref{coordtransformationunbartobar}) we can write down the coordinates transformation from $(\log w, z)$ to $(\bar{w}_{i\pm}, \bar{z}_{i\pm})$ in terms of the function $F_{i\pm}$ we have found in (\ref{fsoln2}):
\begin{align*}\label{coordstransfunbartobar}
    \bar{w}_{i\pm} &= F_{i\pm}\left(w + \frac{1}{w}\right)\\
    \bar{z}_{i\pm} &= \left( \frac{z - z_{even}\left(\sqrt{w + \frac{1}{w} - P(z_i(u_0);u_0)};u_0\right)}{z_{odd}\left(\sqrt{w + \frac{1}{w} - P(z_i(u_0);u_0)};u_0\right)}\right)\left(F_{i\pm}\left(w + \frac{1}{w}\right)\right)^{1/2}\numberthis.
\end{align*}
Using the standard local coordinate
\begin{equation}\label{stdlocalcoordinatesswformula}
    \bar{\eta}_{i\pm} := \sqrt{F_{i\pm}(\eta^2_{i\pm} + P(z_i(u);u)) - F_{i\pm}(P(z_i(u_0);u_0))}
\end{equation}
corresponding to the local coordinates $(\bar{w}_{i\pm}, \bar{z}_{i\pm})$, we find that $t_{i\pm}(\Sigma(u))$ as given in (\ref{discatsigmau}) becomes 
\begin{align*}
t_{i\pm}(\Sigma(u)) = \Bigg(\bar{w}_{i\pm} &= \bar{\eta}_{i\pm}^2 + F_{i\pm}(P(z_i(u);u)), \\
\bar{z}_{i\pm} &= \frac{z\left(\sqrt{F^{-1}_{i\pm}\left(\bar{\eta}_{i\pm}^2 + F_{i\pm}\left(P(z_i(u);u)\right)\right) - P(z_i(u);u)};u\right)}{z_{odd}\left(\sqrt{F^{-1}_{i\pm}\left(\bar{\eta}_{i\pm}^2 + F_{i\pm}\left(P(z_i(u);u)\right)\right) - P(z_i(u_0);u_0)}; u_0\right)}\sqrt{\bar{\eta}_{i\pm}^2 + F_{i\pm}\left(P(z_i(u);u)\right)}\\
&\qquad - \frac{z_{even}\left(\sqrt{F^{-1}_{i\pm}\left(\bar{\eta}_{i\pm}^2 + F_{i\pm}\left(P(z_i(u);u)\right)\right) - P(z_i(u_0);u_0)}; u_0\right)}{z_{odd}\left(\sqrt{F^{-1}_{i\pm}\left(\bar{\eta}_{i\pm}^2 + F_{i\pm}\left(P(z_i(u);u)\right)\right) - P(z_i(u_0);u_0)}; u_0\right)}\sqrt{\bar{\eta}_{i\pm}^2 + F_{i\pm}\left(P(z_i(u);u)\right)}\\
&=: \sum_{k=0}^\infty b_{i\pm,k}(\Sigma)\bar{\eta}_{i\pm}^k\Bigg)\in Discs^{M_{i\pm}}
\end{align*}
When $u=u_0$, we have $F_{i\pm}(P(z_i(u);u)) = 0$ and evidently
\begin{equation*}
    t_{i\pm}(\Sigma(u = u_0)) = t_{i\pm}(\Sigma_0) = (\bar{w}_{i\pm} = \bar{\eta}_{i\pm}^2, \bar{z}_{i\pm} = \bar{\eta}_{i\pm}) \in Discs^{M_{i\pm}}
\end{equation*}
as expected. We note that $\bar{z}_{i\pm}$ is a well-defined holomorphic function of $\bar{\eta}_{i\pm}$ for all $|\bar{\eta}_{i\pm}| < M_{i\pm}$ (for some constant $M_{i\pm}$ which we are going to discuss soon) and for all $u$ close to $u_0$ (or $[\Sigma] \in \mathcal{B}_{\Sigma_0}$ close to $[\Sigma_0]$) despite all of appearance of square-roots. This is because $z_{odd}$ factors are $\sim\sqrt{\bar{\eta}_{i\pm}^2 + F_{i\pm}\left(P(z_i(u);u)\right)}\left(1 + O\left(\bar{\eta}_{i\pm}^2 + F_{i\pm}\left(P(z_i(u);u)\right)\right)\right)$ as $F^{-1}_{i\pm}(\bar{w}_{i\pm}) = P(z_i(u_0);u_0) + O(\bar{w}_{i\pm})$ and the factor $\sqrt{\bar{\eta}_{i\pm}^2 + F_{i\pm}\left(P(z_i(u);u)\right)}$ cancels with the same factor in the numerator. On the other hand, $z_{even}$ is free of square-roots because it only contains the even power terms. Lastly, $F^{-1}_{i\pm}\left(\bar{\eta}_{i\pm}^2 + F_{i\pm}\left(P(z_i(u);u)\right)\right) - P(z_i(u);u)$ vanishes when $\bar{\eta}^2_{i\pm} = 0$ and so it is $\sim \bar{\eta}_{i\pm}^2\left(1 + O(\bar{\eta}^2_{i\pm})\right)$. 

\subsubsection{Choosing $\{\epsilon_{i\pm}\}$, $\{M_{i\pm}\}$, $\mathcal{U}_{Ram}$ and $\mathcal{B}_{\Sigma_0}$}\label{swchoosingepsilonmubsubsection}
So far we have not chosen explicitly what $\{\epsilon_{i\pm}\}$, $\{M_{i\pm}\}$, $\mathcal{U}_{Ram} := \{(U_{i\pm}, \bar{w}_{i\pm}, \bar{z}_{i\pm})\}$ and $\mathcal{B}_{\Sigma_0}$ are. As we have mentioned, Condition \ref{howtochoosebsigma0condition} merely gave a list of sufficient conditions on how to choose $\{\epsilon_{i\pm}\}$, $\{M_{i\pm}\}$, $\mathcal{U}_{Ram}$, and $\mathcal{B}_{\Sigma_0}$. It is often easier to deal with the situation on a case-by-case basis and therefore in the following we will not refer ourselves to Condition \ref{howtochoosebsigma0condition}. 
From (\ref{fexpansion}) we can see that $F'_{i\pm}(P(z_i(u_0);u_0)) = \left(\frac{2}{P''(z_i(u_0);u_0)}\right)^{\frac{1}{3}} \neq 0$. Therefore, by the Inverse Function Theorem, there exists $R_{i\pm} > 0$ such that $F_{i\pm}$ is bi-holomorphic for all $w_{i\pm}$ whenever $|w_{i\pm} - P(z_i(u_0);u_0)| < R_{i\pm}$. 

Let us choose
\begin{equation*}
U_{i\pm} := \left\{(w,z) \in \mathbb{C}\times \mathbb{C} \subset \mathbb{P}^2\ |\ \left|w + \frac{1}{w} - P(z_i(u_0);u_0)\right| < R_{i\pm}\right\} \ni r_{i\pm}(\Sigma_0).
\end{equation*}
In particular, the open set $U_{i\pm}$ does not contain the ramification points of $z : \Sigma_0 \rightarrow \mathbb{P}^1$ given by $w + \frac{1}{w} = \pm 2$ as $F_{i\pm}(w_{i\pm})$ is singular at $w_{i\pm} = w + \frac{1}{w} = \pm 2$. Let $\bar{R}_{i\pm} > 0$ be such that $F^{-1}_{i\pm}(\bar{w}_{i\pm}) \in U_{i\pm}$ for all $|\bar{w}_{i\pm}| < \bar{R}_{i\pm}$. Choose $\epsilon_{i\pm}, M_{i\pm}$ to be any real numbers such that $0 < \epsilon_{i\pm} < M_{i\pm}$ and $M_{i\pm}^2 + \epsilon_{i\pm} < \bar{R}_{i\pm}$. Finally, we let $\mathcal{B}_{\Sigma_0}$ to be any contractible space such that 
\begin{equation*}
    \mathcal{B}_{\Sigma_0} \subset \left\{[\Sigma(u)] \in \mathcal{B}\ |\ |F_{i\pm}(P(z_i(u);u))| < \epsilon_{i\pm}, \forall i = 1,...,g\right\}.
\end{equation*}
It is clear that $[\Sigma_0]\in\mathcal{B}_{\Sigma_0}$ since $F_{i\pm}(P(z_i(u_0);u_0)) = 0$. Finally, we note that for any $\Sigma \in \mathcal{B}_{\Sigma_0}$ we have $\bar{w}_{i\pm} = \bar{\eta}_{i\pm}^2 + F_{i\pm}(P(z_i(u);u))$ and so $|\bar{w}_{i\pm}| < |\bar{\eta}_{i\pm}|^2 + |F_{i\pm}(P(z_i(u);u))| = M^2_{i\pm} + \epsilon_{i\pm} < \bar{R}_{i\pm}$ which implies 
\begin{equation*}
    \bar{\mathbb{D}}_{i\pm,M_{i\pm}}(\Sigma(u)) := \{p \in \Sigma(u) \cap U_{i\pm}\ |\ |\bar{\eta}_{i\pm}(p)| < M_{i\pm}\} \subset \Sigma\cap U_{i\pm}
\end{equation*}
for all $[\Sigma(u)] \in \mathcal{B}_{\Sigma_0}$. This allows us to define the map $\gamma : \mathcal{B}_{\Sigma_0} \rightarrow Discs^{Ram}_{t_0}$ and $i : G\rightarrow W^{Ram}$. As we are about to see, this is also sufficient for us to define the section $\theta \in \Gamma(\mathcal{B}_{\Sigma_0}, G)$ and to perform any relevant parallel transports in $G$.
\subsubsection{The embedding $\mathcal{B}_{\Sigma_0}\hookrightarrow Discs^{M_{i\pm}}_{t_{i\pm}(\Sigma_0)} \hookrightarrow W^{\epsilon_{i\pm}, M_{i\pm}}_{t_{i\pm}(\Sigma_0)}$}\label{swtheembeddingbtowsubsection}
Let $\theta_{t_{i\pm}(\Sigma)} = -\bar{z}_{i\pm}d\bar{w}_{i\pm} = -\sum_{k=0}^\infty b_{i\pm,k}(\Sigma)\bar{\eta}^k_{i\pm}d\left(\bar{\eta}^2_{i\pm} + F_{i\pm}\left(P(z_i(u);u)\right)\right) \in T_0L_{Airy}^{M_{i\pm}} \subset W^{\epsilon_{i\pm},M_{i\pm}}_{t_{i\pm}(\Sigma)}$.
The embedding $\Phi_{t_{i\pm}(\Sigma_0)}\circ \gamma : \mathcal{B}_{\Sigma_0} \hookrightarrow Discs^{M_{i\pm}}_{t_{i\pm}(\Sigma_0)} \hookrightarrow W^{\epsilon_{i\pm}, M_{i\pm}}_{t_{i\pm}(\Sigma_0)} \cong W^{\epsilon_{i\pm}, M_{i\pm}}_{Airy}$ is given by
\begin{align*}
    \Phi_{t_{i\pm}(\Sigma_0)}(t_{i\pm}(\Sigma)) &= -\theta_{t_{i\pm}(\Sigma_0)} + \exp\left(-F_{i\pm}\left(P(z_i(u);u)\right)\mathcal{L}_{\frac{1}{2\bar{\eta}_{i\pm}}\partial_{\bar{\eta}_{i\pm}}}\right)\theta_{t_{i\pm}(\Sigma)}\\
    &= \bar{\eta}_{i\pm}d\bar{\eta}_{i\pm}^2 - \exp\left(-F_{i\pm}\left(P(z_i(u);u)\right)\mathcal{L}_{\frac{1}{2\bar{\eta}_{i\pm}}\partial_{\bar{\eta}_{i\pm}}}\right)\left(\sum_{k=0}^\infty b_{i\pm,k}(\Sigma)\bar{\eta}^k_{i\pm} d(\bar{\eta}^2_{i\pm})\right)\\
    &:= \sum_{k\neq 0}J^{i\pm}_{k}(u)\bar{\eta}_{i\pm}^k\frac{d\bar{\eta}_{i\pm}}{\bar{\eta}_{i\pm}}\in L^{M_{i\pm}}_{Airy} \subset W^{\epsilon_{i\pm}, M_{i\pm}}_{Airy} = W^{\epsilon_{i\pm},M_{i\pm}}_{t_{i\pm}(\Sigma_0)}.
\end{align*}
We note that $\exp\left(-F_{i\pm}\left(P(z_i(u);u)\right)\mathcal{L}_{\frac{1}{2\bar{\eta}_{i\pm}}\partial_{\bar{\eta}_{i\pm}}}\right)\theta_{t_{i\pm}(\Sigma)} \in W^{\epsilon_{i\pm},M_{i\pm}}_{t_{i\pm}(\Sigma_0)}$ according to Lemma \ref{expoperatorlemma} because\\ $|F_{i\pm}\left(P(z_i(u);u)\right)| < \epsilon_{i\pm}$ for all $[\Sigma(u)] \in \mathcal{B}_{\Sigma_0}$.

\subsubsection{The embedding $\mathcal{B}_{\Sigma_0}\hookrightarrow G_{\Sigma_0}$}\label{swtheembeddingbtogsubsection}
To obtain the embedding $\Phi_{\Sigma_0}:\mathcal{B}_{\Sigma_0}\hookrightarrow G_{\Sigma_0}$ we will first define $\theta \in \Gamma(\mathcal{B}_{\Sigma_0}, G)$ by
\begin{equation*}
    \theta_{\Sigma(u)} := dS_{SW}(\Sigma(u)) - s^*_{\Sigma(u), \Sigma_0}dS_{SW}(\Sigma_0) \in G_{\Sigma(u)}, \qquad \forall [\Sigma(u)] \in \mathcal{B}_{\Sigma_0}.
\end{equation*}
Let us elaborate on the definition of $s^*_{\Sigma(u), \Sigma_0}dS_{SW}(\Sigma_0)$ in this context.
Let $p \in \Sigma(u) \setminus \{r_{\alpha \in Ram}\}$, then $dw_{i\pm}(p) \neq 0$. Hence, there exists an open neighbourhood $V_p \subset \Sigma(u)$ of $p$ such that $w_{i\pm}$ is a local coordinate of $V_p$. 
Then we define $s^*_{\Sigma(u), \Sigma_0}dS_{SW}(\Sigma_0)|_{V_p\cap \Sigma} := z_p\left(w + \frac{1}{w};u_0\right)\frac{dw}{w}$.
Where $z_p\left(w + \frac{1}{w};u_0\right) = z_p(w_{i\pm};u_0)$ is the root of $P(z;u_0) - w_{i\pm} = 0$ such that $\left(w,z_p\left(w + \frac{1}{w};u_0\right)\right) \in \Sigma_0$ varies continuously to $\left(w,z_p\left(w+\frac{1}{w};u\right)\right) \in V_p \subset \Sigma(u)$ as we move continuously from $u_0$ to $u$, where $z_p\left(w_{i\pm};u\right)$ is a root of $P(z;u) - w_{i\pm} = 0$. It follows that on $V_p$ we have $\theta_{\Sigma(u)}$ given by
\begin{equation*}
    \theta_{\Sigma(u)}|_{V_p} = z_p\left(w + \frac{1}{w}; u\right)\frac{dw}{w} - z_p\left(w + \frac{1}{w};u_0\right)\frac{dw}{w}.
\end{equation*}
Of course, $\theta_{\Sigma(u)}$ is not a single-valued differential form on $\Sigma(u)$ because if $V_p$ contains a ramification point $r_{i\pm}(\Sigma(u)) \in \Sigma(u)$ then
\begin{align*}
    z_p\left(w + \frac{1}{w};u_0\right) &= z_i(u_0) + \left(\frac{2}{P''(z_i(u_0);u_0)}\right)^{\frac{1}{2}}\sqrt{w + \frac{1}{w} - P(z_i(u_0);u_0)} + ...\\
    &= z_i(u_0) + \left(\frac{2}{P''(z_i(u_0);u_0)}\right)^{\frac{1}{2}}\sqrt{\eta^2_{i\pm} + P(z_i(u);u) - P(z_i(u_0);u_0)} + ....
\end{align*}
Where $\eta_{i\pm}$ is the standard local coordinates on $\Sigma\cap U_{i\pm}$ corresponding to the local coordinates $(w_{i\pm}, z_{i\pm})$ and $w_{i\pm} = \eta^2_{i\pm} + P(z_i(u);u)$ (see (\ref{discatsigma0})).
The term $\sqrt{\eta^2_{i\pm} + P(z_i(u);u) - P(z_i(u_0);u_0)}$ acquires a negative sign whenever $\eta_{i\pm}$ moves around the point $\pm\sqrt{P(z_i(u);u) - P(z_i(u_0);u_0)}$.

What happening is, for each fixed $w$ the value of $z$ can be any one of $g+1$ distinct roots of $P(z;u_0) - w - \frac{1}{w} = 0$ except when $w + \frac{1}{w}$ is a critical value of $P(z;u_0)$ two of the roots will coincide and we are left with only $g$ distinct roots. As $w + \frac{1}{w}$ moves around the critical value $z(w;u_0)$ will switch between the two roots of $P(z;u_0) - w - \frac{1}{w} = 0$ that ramify when $w + \frac{1}{w}$ is the critical value. However, $\theta_{\Sigma(u)}$ becomes single-valued once we remove the discs $\mathbb{D}_{i\pm,\epsilon_{i\pm}}(\Sigma(u))$ from the curve $\Sigma(u)$ i.e. $\theta_{\Sigma(u)}$ is a single-valued differential form on $\Sigma(u) \setminus \cup_{i = 1}^g\left(\bar{\mathbb{D}}_{i+,\bar{\epsilon}_{i+}} \cup \bar{\mathbb{D}}_{i-,\bar{\epsilon}_{i-}}\right)$ for some $\bar{\epsilon}_{i\pm} < \epsilon_{i\pm}$.

Note that although $dS_{SW}(\Sigma_0) \notin G_{\Sigma_0}$, we indeed have $\theta_{\Sigma(u)} \in G_{\Sigma(u)}$.
Before the blow-ups, $dS_{SW}(\Sigma_0)$ and $dS_{SW}(\Sigma(u))$ both have poles of order $2$ at $\infty_{+} := (w = 0, z = \infty), \infty_{-} := (w = \infty, z = \infty)$, on $\Sigma_0$ and $\Sigma(u)$ respectively, without residues and they are holomorphic elsewhere. The poles at $\infty_{\pm}$ from $dS_{SW}(\Sigma_0)$ and $dS_{SW}(\Sigma(u))$ will cancel out in the expression of $\theta_{\Sigma(u)}$. The same remains true after blow-ups. So $\theta_{\Sigma(u)}$ is holomorphic on $\Sigma(u) \setminus \cup_{i = 1}^g\left(\bar{\mathbb{D}}_{i+,\bar{\epsilon}_{i+}} \cup \bar{\mathbb{D}}_{i-,\bar{\epsilon}_{i-}}\right)$ for some $\bar{\epsilon}_\alpha < \epsilon_\alpha$. Lastly, we also have 
\begin{equation*}
    \oint_{\partial \bar{\mathbb{D}}_{i\pm, \epsilon_{i\pm}}}\theta_{\Sigma(u)} = \oint_{\partial \bar{\mathbb{D}}_{i\pm, \epsilon_{i\pm}}}dS_{SW}(\Sigma(u)) - \oint_{\partial \bar{\mathbb{D}}_{i\pm, \epsilon_{i\pm}}}s^*_{\Sigma(u), \Sigma_0}dS_{SW}(\Sigma_0) = 0
\end{equation*}
using the same argument presented in the proof of Lemma \ref{integralofparalleltransportlemma} and the fact that $dS_{SW}$ has no poles around the ramification points $\{r_{\alpha \in Ram}\}$. Therefore, we conclude that $\theta_\Sigma \in G_\Sigma$. Now we can define $\Phi_{\Sigma_0}$ by
\begin{equation}\label{phisw}
    \Phi_{\Sigma_0}(\Sigma(u)) := \theta_{\Sigma_0} - s^*_{\Sigma_0, \Sigma(u)}\theta_{\Sigma(u)} = dS_{SW}(\Sigma_0) - s^*_{\Sigma_0, \Sigma(u)}dS_{SW}(\Sigma(u)) \in G_{\Sigma_0}.
\end{equation}
\begin{example}
In the Elliptic case $g = 1$, $\Sigma : w + \frac{1}{w} = z^2 + u$, the equation (\ref{phisw}) simplifies to 
\begin{equation*}
    \Phi_{\Sigma_0}(\Sigma(u)) = \sqrt{w + \frac{1}{w} - u_0}\frac{dw}{w} - \sqrt{w + \frac{1}{w} - u}\frac{dw}{w}.
\end{equation*}
\end{example}

It is clear by choosing $A$ and $B$-cycles on $\Sigma_0$ avoiding the discs $\bar{\mathbb{D}}_{i\pm, \epsilon_{i\pm}}(\Sigma_0)$ we have
\begin{equation*}
    \oint_{A_i}\Phi_{\Sigma_0}(\Sigma(u)) = a_{0}^i - a^i, \qquad \oint_{B_i}\Phi_{\Sigma_0}(\Sigma(u)) = b^{0}_i - b_i = \frac{\partial \mathfrak{F}^{SW}_{\Sigma_0}}{\partial a^i}(a_0) - \frac{\partial \mathfrak{F}^{SW}_{\Sigma_0}}{\partial a^i}(a), \qquad i = 1,...,g
\end{equation*}
where
\begin{equation*}
    a^i = a^i(\Sigma(u)) := \oint_{A_i}dS_{SW}(\Sigma(u)), \qquad b_i = b_i(\Sigma(u)) := \oint_{B_i}dS_{SW}(\Sigma(u)) = \frac{\partial \mathfrak{F}^{SW}_{\Sigma_0}}{\partial a^i}(a),
\end{equation*}
$a^i_0 := a^i(\Sigma_0), b_i^0 := b_i(\Sigma_0)$ and $\mathfrak{F}^{SW}_{\Sigma_0}: \mathcal{B}_{\Sigma_0} \rightarrow \mathbb{C}$ is the Seiberg-Witten prepotential which is well-defined on a contractible set $\mathcal{B}_{\Sigma_0}$.

Let us compute $i\circ \Phi_{\Sigma_0}(\Sigma(u)) \in W^{Ram}_{Airy}$. On $\mathbb{A}_{i\pm, \epsilon_{i\pm}, M_{i\pm}}\cap V_p \subset \Sigma_0\cap U_{i\pm} \cap V_p$ we have $z_p(w_{i\pm};u) = z(\eta_{i\pm};u) = z(\sqrt{w_{i\pm} - P(z_i(u);u)};u)$. The square-root does not introduce any ambiguity because we are restricted to the subset $\mathbb{A}_{i\pm, \epsilon_{i\pm}, M_{i\pm}}\cap V_p$. We rewrite (\ref{phisw}) on $\mathbb{A}_{i\pm, \epsilon_{i\pm}, M_{i\pm}}\cap V_p$ in the local coordinates $(\bar{w}_{i\pm}, \bar{z}_{i\pm})$ as 
\begin{align*}
    \Phi_{\Sigma_0}(\Sigma(u))|_{\mathbb{A}_{i\pm, \epsilon_{i\pm}, M_{i\pm}}\cap V_p} = \frac{z\left(\sqrt{F^{-1}_{i\pm}(\bar{w}_{i\pm}) - P(z_i(u_0);u_0)};u_0\right)}{F'_{i\pm}\left(F^{-1}_{i\pm}(\bar{w}_{i\pm})\right)\sqrt{F^{-1}_{i\pm}(\bar{w}_{i\pm})^2 - 4}}d\bar{w}_{i\pm} - \frac{z\left(\sqrt{F^{-1}_{i\pm}(\bar{w}_{i\pm}) - P(z_i(u);u)};u\right)}{F'_{i\pm}\left(F^{-1}_{i\pm}(\bar{w}_{i\pm})\right)\sqrt{F^{-1}_{i\pm}(\bar{w}_{i\pm})^2 - 4}}d\bar{w}_{i\pm}.
\end{align*}
On $\mathbb{A}_{i\pm, \epsilon_{i\pm}, M_{i\pm}} \subset \Sigma_0$ we have $\bar{w}_{i\pm} = \bar{\eta}_{i\pm}^2 + F_{i\pm}\left(P(z_i(u_0);u_0)\right) = \bar{\eta}_{i\pm}^2$ and using (\ref{fsoln2}) it follows that
\begin{align*}
    &i_{i\pm}\circ\Phi_{\Sigma_0}(\Sigma(u))|_{\mathbb{A}_{i\pm, \epsilon_{i\pm}, M_{i\pm}}}\\
    &\qquad\qquad= \frac{z\left(\sqrt{F^{-1}_{i\pm}(\bar{\eta}^2_{i\pm}) - P(z_i(u_0);u_0)};u_0\right)\bar{\eta}_{i\pm}}{z_{odd}\left(\sqrt{F^{-1}_{i\pm}(\bar{\eta}_{i\pm}^2) - P(z_i(u_0);u_0)};u_0\right)}d\bar{\eta}^2_{i\pm} - \frac{z\left(\sqrt{F^{-1}_{i\pm}(\bar{\eta}^2_{i\pm}) - P(z_i(u);u)};u\right)\bar{\eta}_{i\pm}}{z_{odd}\left(\sqrt{F^{-1}_{i\pm}(\bar{\eta}_{i\pm}^2) - P(z_i(u_0);u_0)};u_0\right)}d\bar{\eta}^2_{i\pm}\\
    &\qquad\qquad= \bar{\eta}_{i\pm}d\bar{\eta}_{i\pm}^2 - \frac{\left(z\left(\sqrt{F^{-1}_{i\pm}(\bar{\eta}^2_{i\pm}) - P(z_i(u);u)};u\right) - z_{even}\left(\sqrt{F^{-1}_{i\pm}(\bar{\eta}^2_{i\pm}) - P(z_i(u_0);u_0)};u_0\right)\right)\bar{\eta}_{i\pm}}{z_{odd}\left(\sqrt{F^{-1}_{i\pm}(\bar{\eta}_{i\pm}^2) - P(z_i(u_0);u_0)};u_0\right)}d\bar{\eta}_{i\pm}^2\\
    &\qquad\qquad= \bar{\eta}_{i\pm}d\bar{\eta}_{i\pm}^2 - \sum_{k=0}^\infty b_{i\pm,k}(\Sigma)\left(1 - \frac{F_{i\pm}(P(z_i(u);u))}{\bar{\eta}^2_{i\pm}}\right)^{\frac{k}{2}}\bar{\eta}_{i\pm}^kd\bar{\eta}_{i\pm}^2\\
    &\qquad\qquad= \bar{\eta}_{i\pm}d\bar{\eta}_{i\pm}^2 - \exp\left(-F_{i\pm}\left(P(z_i(u);u)\right)\mathcal{L}_{\frac{1}{2\bar{\eta}_{i\pm}}\partial_{\bar{\eta}_{i\pm}}}\right)\left(\sum_{k=0}^\infty b_{i\pm,k}(\Sigma)\bar{\eta}^k_{i\pm} d(\bar{\eta}^2_{i\pm})\right)\\
    &\qquad\qquad= \Phi_{t_{i\pm}(\Sigma_0)}(t_{i\pm}(\Sigma)).
\end{align*}
Evidently, we have an agreement: $i\circ\Phi_{\Sigma_0}(\Sigma(u)) = \Phi_{t(\Sigma_0)}\circ \gamma(\Sigma)$. In particular, each $i_{i\pm}\circ\Phi_{\Sigma_0}(\Sigma(u)) = \sum_{k\neq 0}J^{i\pm}_k(u)\bar{\eta}_{i\pm}^k\frac{d\bar{\eta}_{i\pm}}{\bar{\eta}_{i\pm}}$ satisfies the residue constrains and therefore, $i\circ\Phi_{\Sigma_0}(\Sigma(u)) = \sum_{i=1}^g\sum_{k\neq 0}J^{i\pm}_k(u)[i\pm]\otimes \bar{\eta}_{i\pm}^k\frac{d\bar{\eta}_{i\pm}}{\bar{\eta}_{i\pm}} \in L^{Ram}_{Airy} \in W^{Ram}_{Airy}$. 

\subsection{Seiberg-Witten Prepotential and Topological Recursion}\label{swprepotentialandtrsection}

Let us now express $\Phi_{\Sigma_0}(\Sigma(u))$ in terms of the normalized holomorphic differential $\{\omega_{i = 1,...,g}\}$ and meromorphic differentials $\{\bar{e}^{k,i\pm}\}$ (see Definition \ref{endifferentialdefinition}) on $\Sigma_0$. From Corollary \ref{gisvoplush0corollary}, the fact that $\Phi_{\Sigma_0}(\Sigma(u))\in G_{\Sigma_0}$ and the principal part of $i_{i\pm}\circ \tilde{\Phi}_{\Sigma_0}(\Sigma(u))$ is the principal part of $i_{i\pm}\circ \Phi_{\Sigma_0}(\Sigma(u)) = \Phi_{t_{i\pm}(\Sigma_0)}(t_{i\pm}(\Sigma))$ which is $\sum_{k > 0}J^{i\pm}_k(u)\bar{\eta}^k_{i\pm}\frac{d\bar{\eta}_{i\pm}}{\bar{\eta}_{i\pm}}$, we have
\begin{equation}\label{phiinglobaldiffformbasis1}
    \Phi_{\Sigma_0}(\Sigma(u)) = \sum_{i=1}^g\sum_{k > 0}^\infty \left(J^{i+}_k(u)\bar{e}^{k,i+} + J^{i-}_k(u)\bar{e}^{k,i-}\right) + \sum_{i=1}^g(a_{0}^i - a^i)\omega_i.
\end{equation}
Another way to look at this is that the principal part of $\sum_{i=1}^g\sum_{k > 0}^\infty \left(J^{i+}_k(u)\bar{e}^{k,i+} + J^{i-}_k(u)\bar{e}^{k,i-}\right)$ exactly matches that of $\Phi_{\Sigma_0}(\Sigma(u))$. Therefore, $\Phi_{\Sigma_0}(\Sigma(u)) - \sum_{i=1}^g\sum_{k > 0}^\infty \left(J^{i+}_k(u)\bar{e}^{k,i+} + J^{i-}_k(u)\bar{e}^{k,i-}\right)$ can be analytically extended to a meromorphic differential form defined on the entire curve $\Sigma_0$, and it is, in fact, a holomorphic on $\Sigma_0$ because neither term has singularity anywhere except at ramification points $r_\alpha, \alpha \in Ram$. Subtracting $\sum_{i=1}^g(a_{0}^i - a^i)\omega_i$, we get a global holomorphic differential form on $\Sigma_0$ with vanishing $A$-periods, hence must be identically zero and (\ref{phiinglobaldiffformbasis1}) follows. 

Let us extend the set $\{\omega_1,...,\omega_g\}$ into a new canonical basis of $W^{Ram}_{Airy}$. First, let us compute the $\Omega_{Airy}$ symplectic pairings between $\{e^{k,i\pm}\}$ and $\{\omega_i\}$. Suppose that
\begin{equation*}
    i(\omega_j) = \sum_{i=1}^g\sum_{k=1}^\infty\left(c^{k,i+}_j [i+]\otimes k\bar{\eta}_{i+}^k\frac{d\bar{\eta}_{i+}}{\bar{\eta}_{i+}} + c^{k,i-}_j [i-]\otimes k\bar{\eta}_{i-}^k\frac{d\bar{\eta}_{i-}}{\bar{\eta}_{i-}}\right)
\end{equation*}
then
\begin{align*}
    \oint_{p \in B_j}\bar{e}^{k,i\pm}(p) &= \frac{1}{2\pi ik}\oint_{p \in B_j}\oint_{\bar{\eta}_{i\pm} = 0} \frac{B(p, q(\bar{\eta}_{i\pm}))}{\bar{\eta}^k_{i\pm}} = \frac{1}{2\pi i k}\oint_{\bar{\eta}_{i\pm}=0}\frac{1}{\bar{\eta}_{i\pm}^k}\oint_{p \in B_j}B(p, q(\bar{\eta}_{i\pm}))\\
    &= \frac{1}{k}\oint_{\bar{\eta}_{i\pm} = 0}\frac{1}{\bar{\eta}_{i\pm}^k}\omega_j(q(\bar{\eta}_{i\pm})) = 2\pi ic^{k,i\pm}_j.
\end{align*}
Then it follows that the symplectic pairings are
\begin{equation*}
    \Omega_{Airy}(\bar{e}^{k,i\pm}, \omega_j) = \frac{1}{2\pi i}\sum_{l=1}^g\left(\oint_{A_l}\omega_j\oint_{B_l}\bar{e}^{i\pm}_k - \oint_{B_l}\omega_j\oint_{A_l}\bar{e}^{i\pm}_k\right) = \frac{1}{2\pi i}\oint_{B_j}\bar{e}^{k,i\pm} = c^{k,i\pm}_j.
\end{equation*}
This shows that $\{\omega_i, e^{k,i\pm}\ |\ i = 1,...,g, k \geq 2\}$ cannot be extended to a canonical basis. Let us examine some properties of the coefficients $c^{k,i\pm}_j$.

\begin{lemma}
$c^{k,i\pm}_j = -c^{k,i\mp}_j$.
\end{lemma}
\begin{proof}
We note that each normalized holomorphic differential form $\omega_i$ can be written as a linear combination of $\{\frac{z^{i-1}dz}{w - \frac{1}{w}}\ |\ i = 1,...,g\}$. Since we can map the neighbourhood $\Sigma_0 \cap U_{i+}$ of each ramification point $r_{i+}$ to the corresponding neighbourhood $\Sigma_0 \cap U_{i-}$ of $r_{i-}$ via the involution $p = (z,w) \mapsto \hat{p} = (z,\frac{1}{w})$, it follows that $\omega_k(p) = -\omega_k(\hat{p})$. Let us examine the involuation in term of the standard local coordinates $\bar{\eta}_{i\pm}$ as defined in (\ref{stdlocalcoordinatesswformula}). Suppose that $p \in \Sigma_0\cap U_{i+}$. First, we note that we have $\eta_{i+}(p) = \eta_{i-}(\hat{p})$ because $\eta^2_{i\pm} = P(z;u) - P(z_i(u);u)$ and $z(p) = z(\hat{p})$. Because we have chosen $F_{i+}$ and $F_{i-}$ to be the same (see (\ref{fsoln2})), therefore the relationship between $\bar{\eta}_{i\pm}$ and $\eta_{i\pm}$ as given in (\ref{stdlocalcoordinatesswformula}) implies that $\bar{\eta}_{i+}(p) = \bar{\eta}_{i-}(\hat{p})$. Writing $\omega_{k}(p) = -\omega_k(\hat{p})$ in the standard local coordinates we have
\begin{equation*}
    \sum_{k=0}^\infty c^{k,i+}_jk\bar{\eta}_{i+}^k(p)\frac{d\bar{\eta}_{i+}(p)}{\bar{\eta}_{i+}(p)} = i_{i+}\omega_j(p) = -i_{i-}\omega_j(\hat{p}) = -\sum_{k=0}^\infty c^{k,i-}_jk\bar{\eta}_{i-}^k(\hat{p})\frac{d\bar{\eta}_{i-}(\hat{p})}{\bar{\eta}_{i-}(\hat{p})} = -\sum_{k=0}^\infty c^{k,i-}_jk\bar{\eta}_{i+}^k(p)\frac{d\bar{\eta}_{i+}(p)}{\bar{\eta}_{i+}(p)},
\end{equation*}
from which we conclude that $c^{k,i\pm}_j = -c^{k,i\mp}_j$.
\end{proof}

Let $(c^i_j)$ be a $g\times g$ matrix where $c^i_j := c^{1,i+}_j = -c^{1,i-}_j$.
\begin{lemma}
The matrix $(c^i_j)$ is invertible. 
\end{lemma}
\begin{proof}
Let $(a^j) \in \mathbb{C}^g$ be a vector such that $\sum_{j=1}^g c^i_ja^j = 0$. Then $\omega := \sum_{j=1}^g a^j\omega_j$ is a holomorphic differential form on $\Sigma_0$ such that 
\begin{equation*}
    \omega(r_{i\pm}) = \frac{1}{2\pi i}\oint_{\bar{\eta}_{i\pm}=0}\frac{1}{\bar{\eta}_{i\pm}}\omega(q(\bar{\eta}_{i\pm})) = \sum_{j=1}^g c^{1,i\pm}_ja^j = \pm \sum_{j=1}^g c_j^ia^j = 0
\end{equation*}
for all $i\pm \in Ram$. So $\omega$ vanishes at every ramification points and therefore has at least $2g$ zeros and has no poles. But $\deg K_{\Sigma_0} = 2g - 2 < 2g$ which is a contradiction, unless $(a^j) = 0$. In other words, $(c_j^i)$ is invertible.
\end{proof}

Let $(b^i_j)$ be the inverse of $(c^i_j)$, i.e. $\sum_{k=1}^g b^{k}_ic_{k}^j = \delta_i^j$.
Now, let us choose a new canonical basis for $W^{Ram}_{Airy}$ to be
\begin{equation}\label{newbasissw}
\{\tau^i,\bar{\tau}^i,\tau^{k,i\pm}, \omega_i, \bar{\omega}_i, \omega_{k,i\pm}\ |\ i = 1,...,g, k\geq 2\}
\end{equation}
where
\begin{align*}
    \tau^i &:= \frac{1}{2}\sum_{j=1}^gb^{i}_j\left(\bar{e}^{1,j+} - \bar{e}^{1,j-}\right), \qquad \bar{\tau}^i := \frac{1}{2}\sum_{j=1}^gb^{i}_j\left(\bar{e}^{1,j+} + \bar{e}^{1,j-}\right), \qquad \tau^{k,i\pm} := \bar{e}^{k,i\pm} - \sum_{j,l=1}^gc^{k,i\pm}_jb^{j}_l\bar{e}^{1,l\pm}\\
    \bar{\omega}_i &:= \sum_{j=1}^g\sum_{k=1}^\infty\left(c^{k,j+}_i[j+]\otimes k\bar{\eta}_{j+}^k\frac{d\bar{\eta}_{j+}}{\bar{\eta}_{j+}} - c^{k,j-}_i[j-]\otimes k\bar{\eta}_{j-}^k\frac{d\bar{\eta}_{j-}}{\bar{\eta}_{j-}}\right), \qquad \omega^{k,i\pm} := [i\pm]\otimes k\bar{\eta}_{i\pm}^k\frac{d\bar{\eta}_{i\pm}}{\bar{\eta}_{i\pm}}.
\end{align*}
Note that this is the type of gauge transformation we studied in Section \ref{analyticresconstraintssubsection} and Section \ref{ramproductresconstraintssubsection}.
\begin{remark}
Observe that $\tau^i, \bar{\tau}^i, \tau^{k,i\pm}$ and $\omega_i$ are all global meromorphic differential forms on $\Sigma_0$ belonging to $G_{\Sigma_0}$ while $\bar{\omega}_i, \omega_{k,i\pm} \in W^{Ram}_{Airy}$ are only locally defined near each ramification points $r_{i\pm}$.
\end{remark}

To confirm that we have defined a canonical basis we can check that for $i,j = 1,\cdots, g, \alpha, \beta \in Ram$ and $l,k \geq 2$ we have
\begin{align*}
    \Omega_{Airy}(\tau^i, \omega_j) = \delta_j^i, \qquad \Omega_{Airy}(\bar{\tau}^i, \bar{\omega}_j) = \delta_j^i, \qquad \Omega_{Airy}(\tau^{k,\alpha}, \omega_{l,\beta}) = \delta^k_l\delta^\alpha_\beta
\end{align*}
and all other pairings vanish. In particular,  
\begin{equation}\label{abperiodsofnewbasis}
    \oint_{A_i}\omega_j = \delta_{j}^i,\qquad \oint_{B_i}\omega_j = \tau_{ij}(a_0),\qquad \oint_{A_i}\tau^j = 0,\qquad \oint_{B_i}\tau^j = 2\pi i \delta^{j}_i
\end{equation}
and all $A,B$-periods of $\bar{\tau}^i, \bar{e}^{k,i\pm}$ are zero. 

Substitute $\bar{e}^{1,i\pm} = \sum_{j=1}^gc^i_j(\bar{\tau}^j \pm \tau^j), \bar{e}^{k,i\pm} = \tau^{k,i\pm} + \sum_{j,l = 1}^gc^{k,i\pm}_j(\bar{\tau}^j \pm \tau^j)$ into (\ref{phiinglobaldiffformbasis1}) and recollect terms, we can rewrite (\ref{phiinglobaldiffformbasis1}) as
\begin{align*}
    \Phi_{\Sigma_0}(\Sigma(u)) &= \sum_{i=1}^g\left(\alpha^i\omega_i + \bar{x}^i\bar{\omega}_i + \sum_{k=2}^\infty\left(\bar{x}^{k,i+}\omega_{k,i+} + \bar{x}^{k,i-}\omega_{k,i-}\right)\right)\\
    &\qquad + \sum_{i=1}^g\left(\beta_i\tau^i + \bar{y}_i\bar{\tau}^i + \sum_{k=2}^\infty \left(\bar{y}_{k,i+}\tau^{k,i+} + \bar{y}_{k,i-}\tau^{k,i-}\right)\right)\label{phiinglobaldiffformbasis2}\numberthis
\end{align*}
where
\begin{align*}
    \alpha^i &= a_{0}^i - a^i, \qquad \bar{x}^i = 0, \qquad \bar{x}^{k,i\pm} = 0,\qquad
    \beta_i = \sum_{j=1}^g\sum_{k=1}^\infty \left(c^{k,j+}_iJ^{j+}_k(u) - c^{k,j-}_iJ^{j-}_k(u)\right), \\ 
    \bar{y}_i &= \sum_{j=1}^g\sum_{k=1}^\infty \left(c^{k,j+}_iJ^{j+}_k(u) + c^{k,j-}_iJ^{j-}_k(u)\right),\qquad
    \bar{y}_{k,i\pm} = J^{i\pm}_k(u).
\end{align*}
There is no technical difficulties involved in recollection of terms because $\bar{e}^{k,i\pm}$ are written as a finite sum of $\tau^i, \bar{\tau}^i, \tau^{k,i\pm}$ and infinite sums commute with finite sums. 

Let $(V^{mer}_{\Sigma_0}, \bar{A}_{Airy}, \bar{B}_{Airy}, \bar{C}_{Airy}, \bar{\epsilon}_{Airy})$ be the gauge transformed quantum Airy structure of\\
$(V^{Ram}_{Airy}, A_{Airy}, B_{Airy}, C_{Airy}, \epsilon_{Airy})$ corresponding to the basis (\ref{newbasissw}) of $W^{|Ram|}_{Airy}$. The ATR output is given by $S = \sum_{g\geq 0}\hbar^{g-1}S_g$ where $S_g = \sum_{n=1}^\infty S_{g,n}$,
\begin{align*}
    S_{g,n} &= S_{g,n}\left(\{\alpha^i\}, \{\bar{x}^i\}, \{\bar{x}^{k,i\pm}\}\right)\\
    &= \sum_{\substack{n_1,n_2,n_3 \geq 0\\n_1 + n_2 + n_3 = n}}\frac{1}{n_1!n_2!n_3!}\sum_{\substack{i_1,...,i_{n_1} = 1,...,g\\j_1,...,j_{n_2} = 1,...,g\\k_1,...,k_{n_3} \geq 2\\\alpha_1,...,\alpha_{n_3} \in Ram}}S_{g,n;\alpha^{i_1}...\alpha^{i_{n_1}}\bar{x}^{j_1}...\bar{x}^{j_{n_2}}\bar{x}^{k_1;\alpha_1}...\bar{x}^{k_{n_3};\alpha_{n_3}}}\alpha^{i_1}...\alpha^{i_{n_1}}\bar{x}^{j_1}...\bar{x}^{j_{n_2}}\bar{x}^{k_1,\alpha_1}...\bar{x}^{k_{n_3},\alpha_{n_3}}
\end{align*}
where
\begin{equation*}
    S_{g,n;\alpha^{i_1}...\alpha^{i_{n_1}}\bar{x}^{j_1}...\bar{x}^{j_{n_2}}\bar{x}^{k_1;\alpha_1}...\bar{x}^{k_{n_3};\alpha_{n_3}}} := \frac{\partial}{\partial\alpha^{i_1}}...\frac{\partial}{\partial\alpha^{i_{n_1}}}\frac{\partial}{\partial\bar{x}^{j_1}}...\frac{\partial}{\partial\bar{x}^{j_{n_2}}}\frac{\partial}{\partial\bar{x}^{k_1,\alpha_1}}...\frac{\partial}{\partial\bar{x}^{k_{n_3};\alpha_{n_3}}}S_g\Bigg|_{\substack{\alpha^i = \bar{x}^i = \bar{x}^{k,i\pm} = 0\\\forall k \geq 2, i = 1,...,g}}.
\end{equation*}
Since $\Phi_{\Sigma_0}(\Sigma(u)) \in L^{Ram}_{Airy}$, it follows that we can re-write (\ref{phiinglobaldiffformbasis2}) as
\begin{align*}\label{phiinglobaldiffformbasis3}
    \Phi_{\Sigma_0}(\Sigma(u)) = \sum_{i=1}^g(a_{0}^i - a^i)\omega_i &+ \sum_{i=1}^g\left(\frac{\partial S_0}{\partial \alpha^i}\Big|_{\substack{\alpha^i = a_{0}^i-a^i\\\bar{x}^i = \bar{x}^{k,i\pm}=0\\\forall k \geq 2, i = 1,..,g}}\tau^i + \frac{\partial S_0}{\partial \bar{x}^i}\Big|_{\substack{\alpha^i = a_{0}^i-a^i\\\bar{x}^i = \bar{x}^{k,i\pm}=0\\\forall k \geq 2, i = 1,..,g}}\bar{\tau}^i\right)\\
    &+ \sum_{i=1}^g\sum_{k = 2}^\infty \left(\frac{\partial S_0}{\partial \bar{x}^{k,i+}}\Big|_{\substack{\alpha^i = a_{0}^i-a^i\\\bar{x}^i = \bar{x}^{k,i\pm}=0\\\forall k \geq 2, i = 1,..,g}}\bar{e}^{k,i+} + \frac{\partial S_0}{\partial \bar{x}^{k,i-}}\Big|_{\substack{\alpha^i = a_{0}^i-a^i\\\bar{x}^i = \bar{x}^{k,i\pm}=0\\\forall k \geq 2, i = 1,..,g}}\bar{e}^{k,i-}\right)\numberthis
\end{align*}
Let $\mathcal{S}_0 = \mathcal{S}_0(\alpha^1,...,\alpha^g) := S_{0}(\{\alpha^i\}, \{\bar{x}^i = 0\}, \{\bar{x}^{k,i\pm} = 0\})$. Compute the $B$-periods of $\Phi_{\Sigma_0}(\Sigma(u))$ using (\ref{phiinglobaldiffformbasis3}), (\ref{abperiodsofnewbasis}) and the fact that $\frac{\partial}{\partial \alpha^i} = -\frac{\partial}{\partial a^i}$ we get
\begin{equation*}
    \frac{\partial \mathfrak{F}^{SW}_{\Sigma_0}}{\partial a^i}(a_{0}^1,...,a_{0}^g) - \frac{\partial \mathfrak{F}^{SW}_{\Sigma_0}}{\partial a^i}(a^1,...,a^g) = \oint_{B_i}\Phi_{\Sigma_0}(\Sigma(u)) = \sum_{j=1}^g(a_{0}^j - a^j)\tau_{ji}(a_0) - 2\pi i\frac{\partial \mathcal{S}_0(a_{0}^1-a^1,...,a_{0}^g - a^g)}{\partial a^i}.
\end{equation*}
Integrating both sides we obtain
\begin{align*}
    \mathfrak{F}^{SW}_{\Sigma_0}(a^1,...,a^g) &= \mathfrak{F}^{SW}_{\Sigma_0}(a_{0}^1,...,a_{0}^g) + \sum_{i=1}^g(a^i - a_{0}^i)\frac{\partial \mathfrak{F}^{SW}_{\Sigma_0}}{\partial a^i}(a_{0}^1,...,a_{0}^g)\\
    &\qquad + \frac{1}{2}\sum_{i,j=1}^g(a^i - a_{0}^i)(a^j - a_{0}^j)\frac{\partial^2\mathfrak{F}^{SW}_{\Sigma_0}}{\partial a^i\partial a^j}(a_{0}^1,...,a_{0}^g) + 2\pi i\mathcal{S}_0(a_{0}^1-a^1,...,a_{0}^g-a^g)\numberthis.
\end{align*}

Let us consider a topological recursion with spectral curve $(\Sigma_0, \log w, z, B)$ where $\Sigma_0$ is the Seiberg-Witten curve given by $w + \frac{1}{w} = P(z;u_0)$ and $B = B(p,q)$ is the normalized Bergman kernel on $\Sigma_0\times \Sigma_0$. The initial condition of the recursion shall be given by 
\begin{equation*}
\omega_{0,1}(p) = z(p)\frac{dw(p)}{w(p)} = dS_{SW}(p),\qquad \omega_{0,2}(p,q) = B(p,q).
\end{equation*}
It is easy to check using (\ref{coordstransfunbartobar}) that in the standard local coordinate $\bar{\eta}_{i\pm}$ in the neighbourhood $\Sigma_0\cap U_{i\pm}$ of each ramification point $r_{i\pm}(\Sigma_0)$ we have
\begin{equation*}
    \omega_{0,1}(\bar{\eta}_{i\pm}) - \omega_{0,1}(-\bar{\eta}_{i\pm}) = 4\bar{\eta}_{i\pm}^2d\bar{\eta}_{i\pm}.
\end{equation*}
From Proposition \ref{atrvstrproposition}, the output meromorphic multi-differential forms are given by
\begin{align*}
    \omega^{SW}_{g,n}(p_1,...,p_n) &= \sum_{\sigma \in S_n}\sum_{\substack{n_1,n_2,n_3 \geq 0\\n_1 + n_2 + n_3 = n}}\frac{1}{n_1!n_2!n_3!}\sum_{\substack{i_1,...,i_{n_1} = 1,...,g\\j_1,...,j_{n_2} = 1,...,g\\k_1,...,k_{n_3} \geq 2\\\alpha_1,...,\alpha_{n_3} \in Ram}}\Bigg(S_{g,n;\alpha^{i_1}...\alpha^{i_{n_1}}\bar{x}^{j_1}...\bar{x}^{j_{n_2}}\bar{x}^{k_1,\alpha_1}...\bar{x}^{k_{n_3},\alpha_{n_3}}}\\
    &\qquad\times \tau^{i_1}(p_{\sigma(1)})...\tau^{i_{n_1}}(p_{\sigma(n_1)})\bar{\tau}^{j_1}(p_{\sigma(n_1 + 1)})...\bar{\tau}^{j_{n_2}}(p_{\sigma(n_1 + n_2)})\bar{e}^{k_1,\alpha_1}(p_{\sigma(n_1+n_2+1)})...\bar{e}^{k_{n_3},\alpha_{n_3}}(p_{\sigma(n)})\Bigg)
\end{align*}
where the superscript $SW$ is added to signify that $\omega^{SW}_{g,n}$ are produced from topological recursion using Seiberg-Witten curve as initial data. Then we have
\begin{equation*}
    \oint_{p_1 \in B_{i_1}}...\oint_{p_n\in B_{i_n}}\omega^{SW}_{g,n}(p_1,...,p_n) = (2\pi i)^n S_{g,n;\alpha^{i_1}...\alpha^{i_n}}.
\end{equation*}
By definition we have
\begin{align*}
    &2\pi i\mathcal{S}_0(a_{0}^1-a^1,...,a_{0}^g-a^g) = \sum_{n = 3}^\infty\frac{1}{n!}\sum_{i_1,...,i_n = 1}^gS_{0,n;\alpha^{i_1}...\alpha^{i_n}}(a_{0}^{i_1} - a^{i_1})...(a_{0}^{i_n} - a^{i_n})\\
    &\qquad\qquad= \sum_{n=3}^\infty\frac{(-1)^n}{n!}\frac{1}{(2\pi i)^{n-1}}\sum_{i_1,...,i_n = 1}^g\left(\oint_{p_1 \in B_{i_1}}...\oint_{p_n\in B_{i_n}}\omega^{SW}_{0,n}(p_1,...,p_n)\right)(a^{i_1} - a_{0}^{i_1})...(a^{i_n} - a_{0}^{i_n}).
\end{align*}
Therefore, we arrive at the relation between the Seiberg-Witten prepotential $\mathfrak{F}^{SW}_{\Sigma_0}$ and meromorphic multi-differential forms $\omega_{0,n}$ produced from topological recursion:
\begin{align*}
    \mathfrak{F}^{SW}_{\Sigma_0}&(a^1,...,a^g) = \mathfrak{F}^{SW}_{\Sigma_0}(a_{0}^1,...,a_{0}^g)\\ 
    & + \sum_{i=1}^g(a^i - a_{0}^i)\frac{\partial \mathfrak{F}^{SW}_{\Sigma_0}}{\partial a^i}(a_{0}^1,...,a_{0}^g) + \frac{1}{2}\sum_{i,j=1}^g(a^i - a_{0}^i)(a^j - a_{0}^j)\frac{\partial^2\mathfrak{F}^{SW}_{\Sigma_0}}{\partial a^i\partial a^j}(a_{0}^1,...,a_{0}^g)\\
    & - \sum_{n=3}^\infty\frac{1}{n!}\left(\frac{i}{2\pi}\right)^{n-1}\sum_{i_1,...,i_n = 1}^g\left(\oint_{p_1 \in B_{i_1}}...\oint_{p_n\in B_{i_n}}\omega^{SW}_{0,n}(p_1,...,p_n)\right)(a^{i_1} - a_{0}^{i_1})...(a^{i_n} - a_{0}^{i_n})\label{swprepotentialvsomega0nformula}\numberthis,
\end{align*}
on the open neighbourhood $\mathcal{B}_{\Sigma_0}\subseteq \mathcal{B}$ of $[\Sigma_0] \in \mathcal{B}$. This agrees with the result of Theorem \ref{prepotentialvsomega0ntheorem} after identifying $\mathbf{a}^i := -a^i$ and $\mathfrak{F}_{\mathcal{B}}(\mathbf{a}^1,\cdots,\mathbf{a}^g) = \mathfrak{F}^{SW}_{\Sigma_0}(a^1,\cdots,a^g)$.

\printbibliography
\end{document}